\documentclass{article}

\usepackage[final]{neurips_2024}

\usepackage[utf8]{inputenc} 
\usepackage[T1]{fontenc}    
\usepackage{hyperref}       
\usepackage{url}            
\usepackage{booktabs}       
\usepackage{amsfonts}       
\usepackage{nicefrac}       
\usepackage{microtype}      
\usepackage{xcolor}         

\newcommand{\newmethod}{M3}
\newcommand{\explain}{M3 = MARINA-P + Momentum + MARINA}

\newcommand{\parens}[1]{\left( #1 \right)}
\newcommand{\brac}[1]{\left\{ #1 \right\}}

\usepackage{amsmath}
\usepackage{amssymb}
\usepackage{mathtools}
\usepackage{amsthm}
\usepackage{thm-restate}
\usepackage{adjustbox}

\allowdisplaybreaks

\usepackage[capitalize,noabbrev]{cleveref}

\usepackage{graphicx}
\usepackage{apptools}
\usepackage[flushleft]{threeparttable}
\usepackage{array,booktabs,makecell}
\usepackage{multirow}
\usepackage{nicefrac}
\usepackage{amsthm}
\usepackage{amsmath,amsfonts,bm}
\usepackage{wrapfig}
\usepackage{caption}
\usepackage{subcaption}
\usepackage{siunitx}
\usepackage{thm-restate}
\usepackage{nccmath}
\usepackage{bbm}
\usepackage[shortlabels]{enumitem}

\usepackage{algorithmic}
\usepackage{algorithm}

\usepackage{color}
\usepackage{colortbl}
\definecolor{bgcolor}{rgb}{0.76,0.88,0.50}
\definecolor{bgcolor0}{rgb}{0.93,0.99,1}
\definecolor{bgcolor1}{rgb}{0.8,1,1}
\definecolor{bgcolor2}{rgb}{0.8,1,0.8}
\definecolor{bgcolor3}{rgb}{0.50,0.90,0.50}
\usepackage{tcolorbox}
\usepackage{pifont}

\newcommand{\norm}[1]{\left\| #1 \right\|}

\newcommand{\R}{\mathbb{R}} 
\newcommand{\N}{\mathbb{N}} 

\newcommand{\Exp}[1]{{\rm \mathbb{E}}\left[#1\right]}
\newcommand{\ExpSub}[2]{{\rm \mathbb{E}}_{#1}\left[#2\right]}

\newcommand{\Prob}[1]{\mathbb{P}\left(#1\right)} 
\newcommand{\ProbCond}[2]{\mathbb{P}\left(#1\middle\vert#2\right)}

\newcommand{\cA}{\mathcal{A}}

\newcommand{\cC}{\mathcal{C}}

\newcommand{\cF}{\mathcal{F}}

\newcommand{\cO}{\mathcal{O}}
\newcommand{\cQ}{\mathcal{Q}}

\newcommand{\mA}{\mathbf{A}}
\newcommand{\mX}{\mathbf{X}}

\newcommand{\mE}{\mathbf{E}}

\newcommand{\mI}{\mathbf{I}}

\newcommand{\mD}{\mathbf{D}}

\newcommand{\mQ}{\mathbf{Q}}

\newcommand{\eqdef}{:=} 

\makeatletter
\newcommand{\vast}{\bBigg@{4}}

\theoremstyle{plain}
\newtheorem{theorem}{Theorem}[section]

\newtheorem{lemma}[theorem]{Lemma}
\newtheorem{corollary}[theorem]{Corollary}
\theoremstyle{definition}
\newtheorem{definition}[theorem]{Definition}
\newtheorem{assumption}[theorem]{Assumption}
\theoremstyle{remark}
\newtheorem{remark}[theorem]{Remark}
\newtheorem{example}[theorem]{Example}

\newcommand{\algnamebig}[1]{{\sf #1}}
\newcommand{\algname}[1]{{\small \sf #1}}
\newcommand{\algnamesmall}[1]{{\scriptsize \sf #1}}

\makeatletter
\newcommand{\setword}[2]{%
  \phantomsection
  #1\def\@currentlabel{\unexpanded{#1}}\label{#2}%
}
\makeatother

  {\list{}{\leftmargin=0.2in\rightmargin=0.2in}\item[]}%
  {\endlist}

\usepackage[textsize=tiny]{todonotes}

\makeatletter
\newenvironment{protocol}[1][htb]{%
    \renewcommand{\ALG@name}{Protocol}
   \begin{algorithm}[#1]%
  }{\end{algorithm}}
\makeatother

\usepackage{tcolorbox}

\newtcolorbox{theorembox}{
  colback=gray!20,
  colframe=gray!20,
  boxrule=0.8pt,
  before skip=10pt,
  after skip=10pt
}

\newcommand{\alglinelabelmain}{%
  \addtocounter{ALC@line}{-1}
  \refstepcounter{ALC@line}
  \label
}

\newcommand{\alglinelabel}{%
  \addtocounter{ALC@line}{-1}
  \refstepcounter{ALC@line}
  \label
}

\newcommand{\alglinelabelfirst}{%
  \addtocounter{ALC@line}{-1}
  \refstepcounter{ALC@line}
  \label
}

\hypersetup{
  colorlinks   = true, 
  urlcolor     = blue, 
  linkcolor    = {red!75!black}, 
  citecolor   = {blue!50!black} 
}

\title{Improving the Worst-Case Bidirectional Communication Complexity for Nonconvex Distributed Optimization under Function Similarity}

\author{%
  Kaja Gruntkowska \\
  KAUST\thanks{King Abdullah University of Science and Technology, Thuwal, Saudi Arabia} \\
  \And
  Alexander Tyurin \\
  KAUST\footnotemark[1],\,\,AIRI\thanks{AIRI, Moscow, Russia},\,\,Skoltech\thanks{Skolkovo Institute of Science and Technology, Moscow, Russia} \\
  \And
  Peter Richt\'{a}rik \\
  KAUST\footnotemark[1] \\
}

\begin{document}

\maketitle

\begin{abstract}
    Effective communication between the server and workers plays a key role in distributed optimization. In this paper, we focus on optimizing communication, uncovering inefficiencies in prevalent downlink compression approaches.
Considering first the pure setup where the uplink communication costs are negligible, we introduce \algname{MARINA-P}, a novel method for downlink compression, employing a collection of correlated compressors. 
Theoretical analysis demonstrates that \algname{MARINA-P} with permutation compressors can achieve a server-to-worker communication complexity improving with the number of workers, thus being provably superior to existing algorithms.
We further show that \algname{MARINA-P} can serve as a starting point for extensions such as methods supporting bidirectional compression: we introduce \algname{M3}, a method combining \algname{MARINA-P} with uplink compression and a momentum step, achieving bidirectional compression with provable improvements in total communication complexity as the number of workers increases.
Theoretical findings align closely with empirical experiments, underscoring the efficiency of the proposed algorithms.
\end{abstract}

\section{Introduction}
\label{sec:introduction}

In federated learning \citep{mcmahan2017communication, konevcny2016federated} and large-scale machine learning \citep{ramesh2021zero,OpenAI_GPT4_2023}, a typical environment consists of multiple devices working together to train a model. Facilitating this collaborative process requires the transmission of substantial information (e.g., gradients, current model) between these devices. In the centralized framework, communication takes place via a server. As a result, practical challenges arise due to the large size of machine learning models and network speed limitations, potentially creating a communication bottleneck \citep{kairouz2021advances,wang2023cocktailsgd}. One possible strategy to reduce this communication burden is to use \emph{lossy compression} \citep{Seide2014,alistarh2017qsgd}. Our paper focuses on this research direction.

We consider the following nonconvex distributed optimization task:
\begin{align}
\label{eq:main_task}
\textstyle \min \limits_{x \in \R^d} \left\{f(x) \eqdef \frac{1}{n} \sum\limits_{i=1}^n f_i(x)\right\},
\end{align}
where $x\in\R^d$ is the vector of parameters of the model, $n$ is the number of workers and $f_i \,:\, \R^d \rightarrow \R$, $i \in [n] \eqdef \{1, \dots, n\}$ are smooth nonconvex functions. We investigate the scenario where the functions $f_i$ are stored on $n$ distinct workers, each directly connected to the server via some communication port \citep{kairouz2021advances}. At present, we operate under the following generic assumptions:
\begin{assumption}
    \label{ass:lipschitz_constant}
    The function $f$ is $L$--smooth, i.e., $\norm{\nabla f(x) - \nabla f(y)} \leq L \norm{x - y}$ $\forall x, y \in \R^d.$
\end{assumption}

\begin{assumption}
    \label{ass:lower_bound}
    There exists $f^* \in \R$ such that $f(x) \geq f^*$ $\forall x \in \R^d.$
\end{assumption}

In the nonconvex world, our goal is to find a (possibly) random point $\bar{x}$ such that ${\rm \mathbb{E}}[\norm{\nabla f(\bar{x})}^2] \leq \varepsilon.$ We refer to such a point an $\varepsilon$--stationary point.

\subsection{Related Work}
\label{sec:related_work}
Before we discuss more advanced optimization methods, let us consider the simplest baseline: the gradient descent (\algname{GD}) \citep{lan2020first}, which iteratively performs updates $x^{t+1} = x^{t} - \gamma \nabla f(x^t) = x^{t} - \nicefrac{\gamma}{n} \sum_{i=1}^{n} \nabla f_i(x^t)$. In the distributed setting, the method can be implemented as follows: each worker calculates $\nabla f_i(x^t)$ and sends it to the server where the gradients are aggregated, after which the server takes the step and broadcasts $x^{t+1}$ back to the workers. With step size $\gamma = \nicefrac{1}{L},$ \algname{GD} finds an $\varepsilon$--stationary point after $\cO\left(\nicefrac{\delta^0 L}{\varepsilon}\right)$ steps, where $\delta^0 \eqdef f(x^0) - f^*$ for a starting point $x^0$. Since at each step the workers and the server send $\Theta(d)$ coordinates/bits, the worker-to-server (w2s, uplink) and server-to-worker (s2w, downlink) communication costs are 
\begin{align}
    \label{eq:comm_gd}
    \textstyle \cO\left(\frac{d \delta^0 L}{\varepsilon}\right).
\end{align}
\begin{definition}
    The \emph{worker-to-server (w2s)} and \emph{server-to-worker (s2w)} communication complexities of a method are the expected number of coordinates/floats that a worker sends to the server and that the server sends to a worker, respectively, to find an $\varepsilon$--solution. The \emph{total communication complexity} is the sum of these complexities.
\end{definition}
\textbf{Unbiased compressors.}
In this work, to perform lossy compression, we employ mappings from the following family:
\begin{definition}
\label{def:unbiased_compression}
A stochastic mapping $\cC\,:\,\R^d \rightarrow \R^d$ is an \textit{unbiased compressor} if
there exists $\omega \geq 0$ such that
\begin{align}
    \label{eq:compressor}
    \textstyle \Exp{\cC(x)} = x, \, \Exp{\norm{\cC(x) - x}^2} \leq \omega \norm{x}^2 \, \forall x \in \R^d.
\end{align}
\end{definition}

We denote the family of such mappings by $\mathbb{U}(\omega)$. A canonical example is the Rand$K \in \mathbb{U}(\nicefrac{d}{K} - 1)$ sparsifier, which preserves $K$ random coordinates of a vector scaled by $\nicefrac{d}{K}$ \citep{beznosikov2020biased}. More examples can be found in \citet{wangni2018gradient,beznosikov2020biased,szlendak2021permutation,horvoth2022natural}. A larger family of compressors, called \emph{biased compressors}, also exists (see Section~\ref{sec:biased_compressors}). In this paper, we implicitly assume that compressors are mutually independent \emph{across iterations} of algorithms.

\textbf{Worker-to-server compression scales with $n$.}
Many previous works ignore the s2w communication costs and focus solely on w2s compression, assuming that \emph{broadcasting is free}. 
For nonconvex objective functions, the current state-of-the-art w2s communication complexities are achieved by the \algname{MARINA} and \algname{DASHA} methods \citep{MARINA,szlendak2021permutation,tyurin2022dasha}.
Here, two additional assumptions are needed:
\begin{assumption}
    \label{ass:local_lipschitz_constant}
    The function $f_i$ is $L_i$--smooth. We define $\widehat{L}^2 \eqdef \frac{1}{n} \sum_{i=1}^n L_i^2$ and $L_{\max} \eqdef \max_{i \in [n]} L_i.$
\end{assumption}
\begin{assumption}
    \label{ass:independent}
    For all $\cC \in \mathbb{U}(\omega),$ all calls of $\cC$ are mutually independent.\footnote{This assumptions means that if an algorithm calls a compressor $\cC$ at some points $x_1, \dots, x_m,$ then $\cC(x_1), \dots, \cC(x_m)$ are \emph{i.i.d.}}
\end{assumption}
Under Assumptions~\ref{ass:lipschitz_constant}, \ref{ass:lower_bound}, \ref{ass:local_lipschitz_constant}, \ref{ass:independent}, and considering the Rand$K$ compressor with $K \leq \nicefrac{d}{\sqrt{n}}$ as an example, the w2s communication complexity of both methods is
\begin{align}
    \label{eq:comm_marina}
    \textstyle \underbrace{\textstyle K}_{\textnormal{\# of sent coord.}} \times \underbrace{\textstyle \cO\left(\frac{\delta^0}{\varepsilon} (L + \frac{\omega}{\sqrt{n}}\widehat{L})\right)}_{\textnormal{\# of iterations}} = \cO\left(\frac{d\delta^0\widehat{L}}{\sqrt{n} \varepsilon}\right),
\end{align}
where we use the facts that $L \leq \widehat{L}$ and $\omega = \nicefrac{d}{K} - 1$ for Rand$K$. The key observation is that when comparing \eqref{eq:comm_gd} and \eqref{eq:comm_marina}, one sees that \eqref{eq:comm_marina} can be $\sqrt{n}$ times smaller if $\widehat{L} \approx L.$
Consequently, the communication complexity of \mbox{\algname{MARINA}/\algname{DASHA}} scales with the number of workers $n,$ and can \emph{provably} improve the \emph{worker-to-server} communication complexity $\cO\left(\nicefrac{d \delta^0 L}{\varepsilon}\right)$ achieved by \algname{GD}.

\textbf{Server-to-worker compression does not scale with $n$.} In certain applications, the significance of s2w communication cannot be ignored. In 4G LTE and 5G networks, w2s and s2w communication speeds can be almost the same \citep{huang2012close} or differ by at most a factor of $10$ \citep{narayanan2021variegated}. Although important, this issue is often overlooked and that is why it is the s2w communication that this work places a central emphasis on.

There exist many papers using communication compression techniques to reduce the s2w communication \citep{zheng2019communication,liu2020double,philippenko2021preserved,fatkhullin2021ef21,gruntkowska2023ef21,tyurin20232direction}. However, to the best of our knowledge, under Assumptions~\ref{ass:lipschitz_constant}, \ref{ass:lower_bound}, \ref{ass:local_lipschitz_constant}, and \ref{ass:independent}, in the worst case, all previous theoretical s2w communication guarantees \emph{are greater or equal} to \eqref{eq:comm_gd}.
As an example, let us consider the result from \citet{gruntkowska2023ef21}[Theorem E.3]. If the server employs operators from $\mathbb{U}(\omega)$ and we ignore w2s compression, the method from \citet{gruntkowska2023ef21} converges in $\cO\left(\nicefrac{(\omega + 1) \delta^0 L}{\varepsilon}\right)$ iterations. Thus, with Rand$K$, the s2w communication complexity is $\cO\left(K \times \nicefrac{(\omega + 1) \delta^0 L}{\varepsilon}\right) = \cO\left(\nicefrac{d \delta^0 L}{\varepsilon}\right)$.
Another method, called \algname{CORE}, proposed by \citet{yue2023core}, achieves s2w and w2s communication complexities equal to $\cO\left(\nicefrac{r_1(f) \delta^0 L}{\varepsilon}\right)$, where $r_1(f)$ is a uniform upper bound of the trace of the Hessian. When $r_1(f) \leq d L,$ \algname{CORE} can improve on \algname{GD}. However, this complexity does not scale with $n$ and requires an additional assumption about the Hessian of $f$.

\section{Contributions}
In our work, we aim to investigate whether the \emph{server-to-worker and total communication complexities} \eqref{eq:comm_gd} of the vanilla \algname{GD} method can be improved.
We make the following contributions:\\\\
1. We start by proving the impossibility of devising a method where the server communicates with the workers using unbiased compressors $\mathbb{U}(\omega)$ (or biased compressors from Section~\ref{sec:biased_compressors}) and achieves an \emph{iteration rate} faster than $\Omega\left(\nicefrac{(\omega + 1) L \delta^0}{\varepsilon}\right)$ (Theorem~\ref{theorem:simple_lower_bound}) under Assumptions~\ref{ass:lipschitz_constant}, \ref{ass:lower_bound} and \ref{ass:independent}. This result provides a lower bound for any method that applies such compressors to vectors sent from the server to the workers in every iteration. Moreover, we prove a more general iteration lower bound of $\Omega\left(\nicefrac{(\omega + 1) L \delta^0}{\varepsilon}\right)$ for all methods where the server zeroes out a coordinate with probability $\nicefrac{1}{(\omega+1)}$ (see Remark \ref{remark:lower_bound}).\\\\
2. In view of this result, it is clear that an extra assumption is needed to break the lower bound $\Omega\left(\nicefrac{(\omega + 1) L \delta^0}{\varepsilon}\right).$ In response, we introduce a novel assumption termed ``Functional $(L_A, L_B)$ Inequality'' (see Assumption~\ref{ass:functional}). We prove that this assumption is relatively weak and holds, for instance, under the local smoothness of the functions $f_i$ (see Assumption~\ref{ass:local_lipschitz_constant}). \\\\
3. We develop a new method for downlink compression, \algname{MARINA-P}, and show that under our new assumption, along with Assumptions~\ref{ass:lipschitz_constant} \ref{ass:lower_bound}, and \ref{ass:independent}, it can achieve the iteration rate of
\begin{align*}
    \textstyle \cO\parens{\frac{\delta^0 L}{\varepsilon} + \frac{\delta^0 L_A (\omega+1)}{\varepsilon} + \frac{\delta^0 L_B (\omega+1)}{\sqrt{n} \varepsilon}}
\end{align*}
(see Theorem \ref{thm:marinap_gen_a} with $p=\nicefrac{1}{(\omega+1)}$ + Lemma \ref{lemma:theta_indep_same}). Notably, when $L_A$ is small and $n \gg 1,$ this complexity is \emph{provably} superior to $\Theta(\nicefrac{(\delta^0 L (\omega+1))}{\varepsilon})$ and the complexities of the previous compressed methods. In this context, $L_A$ serves as a measure of the similarity between the functions $f_i$, and can be bounded by the ``variance'' of the Hessians of the functions $f_i$ (see Theorem~\ref{theorem:a_b_hess}). Thus, \algname{MARINA-P} is \emph{the first method whose iteration complexity can provably improve with the number of workers $n$}. \\\\
4. Moreover, \algname{MARINA-P} can achieve the s2w communication complexity of
\begin{align*}
    \textstyle \cO\parens{\frac{d \delta^0 L}{n \varepsilon} + \frac{d \delta^0 L_A}{\varepsilon}}.
\end{align*}
When $L_A$ is small and $n \gg 1,$ this communication complexity is provably superior to \eqref{eq:comm_gd} and the communication complexities of the previous compressed methods.\\\\
5. Our theoretical improvements can be combined with techniques enhancing the \emph{w2s communication complexities}. In particular, by combining \algname{MARINA-P} with \algname{MARINA} \citep{MARINA} and adding the crucial momentum step, we develop a new method, \algname{\newmethod}, that guarantees a \emph{total communication complexity} (s2w + w2s) of
\begin{align*}
    \textstyle \cO\left(\frac{d \delta^0 L_{\max}}{n^{1/3} \varepsilon} + \frac{d \delta^0 L_A}{\varepsilon}\right).
\end{align*}
When $n \gg 1$ and in the close-to-homogeneous regime, i.e., when $L_A$ is small, this complexity is better than \eqref{eq:comm_gd} and the complexities of the previous bidirectionally compressed methods. \\\\
6. Our theoretical results are supported by numerical experiments (see Section~\ref{sec:experimetns_main}).

\section{Lower Bound under Smoothness}
\label{sec:lower_bound_main}
Let us first investigate the possibility of improving the iteration complexity $\cO\left(\nicefrac{(\omega+1) \delta^0 L}{\varepsilon}\right)$ under Assumptions~\ref{ass:lipschitz_constant},\ref{ass:lower_bound} and \ref{ass:independent}. In Section~\ref{sec:lower_bound}, we consider a family of methods that include those proposed in \citet{zheng2019communication,liu2020double,philippenko2021preserved,fatkhullin2021ef21,gruntkowska2023ef21}, where the server communicates with workers using unbiased/biased compressors, and establish that
\begin{theorem}[Slightly Less Formal Reformulation of Theorem~\ref{theorem:lower_bound}]
    \label{theorem:simple_lower_bound}
    Under Assumptions~\ref{ass:lipschitz_constant}, \ref{ass:lower_bound} and \ref{ass:independent}, all methods in which the server communicates with clients using different and independent unbiased compressors from $\mathbb{U}\left(\omega\right)$ and sends one compressed vector to each worker cannot converge before $\Omega\left(\nicefrac{(\omega + 1) L \delta^0}{\varepsilon}\right)$ iterations.
\end{theorem}

\begin{remark}\label{remark:lower_bound}
    The theorem remains applicable to biased compressors $\mathbb{B}\left(\alpha\right)$ (see Section~\ref{sec:biased_compressors}) with a lower bound of $\Theta\left(\nicefrac{L \delta^0}{\alpha \varepsilon}\right)$. This is because if $\cC \in \mathbb{U}(\omega),$ then $(\omega + 1)^{-1} \cC \in \mathbb{B}\left((\omega + 1)^{-1}\right).$ We also establish a more general result (Theorem~\ref{theorem:lower_bound_generic}): 
    ``all methods in which the server zeroes out a coordinate with probability $\leq p$ independently \emph{across iterations} cannot converge before $\Omega\left(\nicefrac{L \delta^0}{p \varepsilon}\right)$ iterations.''
\end{remark}

This lower bound is tight up to a constant factor. For instance, under exactly the same assumptions, the \algname{EF21-P} mechanism from \citet{gruntkowska2023ef21} converges after $\Theta\left(\nicefrac{(\omega + 1) L \delta^0}{\varepsilon}\right)$ iterations. Unlike~\eqref{eq:comm_marina}, this convergence rate does not scale with $n,$ and Theorem~\ref{theorem:simple_lower_bound} leaves no room for improvement. Consequently, breaking the lower bound requires an additional assumption about the structure of the problem. Before presenting our candidate assumption, we first introduce the ingredients needed to leverage it to the fullest extent: our novel downlink compression method and the type of compressors we shall employ.

\section{The \algnamebig{MARINA-P} Method}
\label{sec:marina_p}

Let us first recall the \algname{MARINA} method \citep{MARINA, szlendak2021permutation}:
\begin{equation}
\begin{aligned}
    \label{eq:marina}
    x^{t+1} &= x^t - \gamma g^t, \qquad c^t \sim \textnormal{Bernoulli}(p)\\
    g_i^{t+1} &= 
    \begin{cases}
        \nabla f_i(x^{t+1}) & \text{if } c^t = 1,\\
        g^{t} + \cC_i^t(\nabla f_i(x^{t+1}) - \nabla f_i(x^{t})) & \text{if } c^t = 0
    \end{cases}
    \quad \textnormal{for all } i \in [n], \\ g^{t+1} &= \frac{1}{n} \sum_{i=1}^n g_i^{t+1},
\end{aligned}
\end{equation}
where $g^{0} = \nabla f(x^0)$.
Motivated by \algname{MARINA}, we design its primal counterpart, \algname{MARINA-P} (Algorithm~\ref{alg:marinap}), operating in the primal space of the model parameters, as outlined in \eqref{eq:marina_p}.

\begin{figure}[t]
\begin{theorembox}
\centerline{The \algname{MARINA-P} Method: }
Initialize vectors $x_0, w_1^0,\dots,w_n^0 \in \R^d$, step size $\gamma > 0$, probability $0 < p \leq 1$ and compressors $\cC_1^t,\ldots,\cC_n^t \in \mathbb{U}(\omega_P)$ for all $t \geq 0.$ The method iterates
\begin{equation}
\begin{aligned}
    \label{eq:marina_p}
    g^t &= \frac{1}{n} \sum_{i=1}^n \nabla f_i(w_i^t), \\
    x^{t+1} &= x^t - \gamma g^t, \\
    c^t &\sim \textnormal{Bernoulli}(p), \\
    w^{t+1}_i &= 
    \begin{cases}
        x^{t+1} & \text{if } c^t = 1,\\
        w_i^t + \cC_i^t(x^{t+1} - x^t) & \text{if } c^t = 0
    \end{cases} \\
    \textnormal{for al}&\textnormal{l } i \in [n].
\end{aligned}
\end{equation}
We denote $w^{t} \eqdef \nicefrac{1}{n} \sum_{i=1}^n w_i^t.$ See the implementation in Algorithm~\ref{alg:marinap}.
\end{theorembox}
\end{figure}

At each iteration of \algname{MARINA-P}, the workers calculate $\nabla f_i(w_i^t)$ and transmit it to the server. The server then averages the gradients and updates the global model $x^t$. Subsequently, with some (typically small) probability $p$, the master sends the non-compressed vector $x^{t+1}$ to all workers. Otherwise, the $i$\textsuperscript{th} worker receives a compressed vector $\cC_i^t(x^{t+1} - x^t)$. Each worker then uses the received message to compute $w^{t+1}_i$ locally. Importantly, $\cC_1^t(x^{t+1} - x^t), \dots, \cC_n^t(x^{t+1} - x^t)$ \emph{can differ}, and this distinction will form the basis of our forthcoming advancements.

Comparing \eqref{eq:marina} and \eqref{eq:marina_p}, \algname{MARINA-P} and \algname{MARINA} are dual methods: both learn control variables ($w^{t}_i$ and $g_i^{t}$), compress the differences ($x^{t+1} - x^t$ and $\nabla f_i(x^{t+1}) - \nabla f_i(x^{t})$), and with some probability~$p$ send non-compressed vectors ($x^{t+1}$ and $\nabla f_i(x^{t+1})$). However, unlike \algname{MARINA}, which compresses vectors sent \emph{from workers to server} and operates in the \emph{dual} space of gradients, \algname{MARINA-P} compresses messages sent \emph{from server to workers} and operates in the \emph{primal} space of arguments.

Let us take Rand$K \in \mathbb{U}(\nicefrac{d}{K} - 1)$ as an example. If we set $p = \parens{\omega + 1}^{-1} = \nicefrac{K}{d}$ to balance heavy communications of $x^{t+1}$ and light communications of $\cC^t_i$ in \eqref{eq:marina_p}, \algname{MARINA-P} averages sending $p d + (1 - p) K \leq 2 K$ coordinates per iteration. Then, the lower bound from Theorem~\ref{theorem:lower_bound_generic} implies that at least $\Omega\left(\nicefrac{(\omega + 1)\delta^0 L}{\varepsilon}\right)$ iterations of the algorithm are needed.

At first glance, it seems that \algname{MARINA-P} does not offer any extra benefits compared to previous methods, and that is true -- we could not expect to break the lower bound. However, as we shall soon see, under an extra assumption, \algname{MARINA-P} can achieve a communication complexity that improves with~$n$.

\subsection{Three ways to compress}

Existing algorithms performing s2w compression share a common characteristic: at each iteration, the server broadcasts \emph{the same} message to all workers \citep{zheng2019communication, liu2020double, fatkhullin2021ef21, gruntkowska2023ef21, tyurin20232direction}.\footnote{A notable exception form this rule is the \algname{MCM} method \citep{philippenko2021preserved} - see Appendix~\ref{sec:3compr}.} In contrast, in w2s compression methods, each worker sends to the server a \emph{different} message, specific to the data stored on that particular device.
An analogous approach can be taken in the s2w communication: intuitively, sending~$n$ distinct messages would convey more information, potentially leading to theoretical improvements. This indeed proves to be the case. While the usual approach of the server broadcasting the same vector to all clients does not lead to an improvement over \eqref{eq:comm_gd}, allowing these vectors to differ enables a well-crafted method to achieve communication complexity that improves with $n$ (see Corollary~\ref{cor:main_corolary_3comp}).

In Appendix \ref{sec:3compr} we provide a detailed discussion of the topic and compare the theoretical complexities of \algname{MARINA-P} when the server employs three different compression techniques: a) uses \emph{one compressor} and sends the same vector to all clients, b) uses a \emph{collection of independent compressors}, or c) uses a \emph{collection of correlated compressors}. We now turn to presenting the technique that gives the best theoretical s2w communication complexity out of these, namely the use of a set of \emph{correlated compressors}.

\subsection{Recap: permutation compressors Perm$K$}
\label{sec:permk}

\citet{szlendak2021permutation} propose compressors that will play a key role in our new theory. For clarity of presentation, we shall assume that $d \geq n$ and $n|d$.\footnote{The general definition of Perm$K$ for $d \bmod n \neq 0$ is presented in \citep{szlendak2021permutation}[App. I].}

\begin{definition}[Perm$K$ (for $d \geq n$ and $n|d$)]
\label{def:PermK-1}
Assume that $d \geq n$ and $d  = q n$, where $q\in\N_{> 0}$. Let $\pi = (\pi_1,\dots,\pi_d)$ be a random permutation of $\{1, \dots, d\}$. For all $x \in \R^d$ and each $i\in \{1,2,\dots,n\}$, we define
\begin{equation*}
    \textstyle \cC_i(x) \eqdef n \times \sum \limits_{j = q (i - 1) + 1}^{q i} x_{\pi_j} e_{\pi_j}.
\end{equation*}
\end{definition}
Unpacking this definition: when the server compresses a vector using a Perm$K$ compressor, it randomly partitions its coordinates across the workers, so that each client receives a sparse vector containing a random subset of entries of the input vector. Like Rand$K$, Perm$K$ is also a sparsifier. However, unlike Rand$K$, it does not allow flexibility in choosing $K$, as it is fixed to $\nicefrac{d}{n}$. Furthermore, it can be shown (Lemma~\ref{lemma:3compr_omega_theta}) that $\cC_i \in \mathbb{U}(n - 1)$ for all $i \in [n]$.

An appealing property of Perm$K$ is the fact that
\begin{align}
    \label{eq:restore}
    \textstyle \frac{1}{n} \sum\limits_{i=1}^n \cC_i(x) = x
\end{align}
for all $x \in \R^d$ deterministically. Here, it is important to note that by design, compressors $\cC_i$ from Definition~\ref{def:PermK-1} are \emph{correlated}, and do not satisfy Assumption~\ref{ass:independent}. This correlation proves advantageous -- \citet{szlendak2021permutation} show that \algname{MARINA} with Perm$K$ compressors performs provably better than with i.i.d. Rand$K$ compressors.

\subsection{Warmup: homogeneous quadratics}
\label{sec:homogeneous_quadratics}
We are finally ready to present our first result showing that the s2w communication complexity can scale with the number of workers~$n$.
To explain the intuition behind our approach, let us consider the simplest (and somewhat impractical) choice of functions $f_i$ -- the homogeneous quadratics:
\begin{align}
    \label{eq:homog_quad}
    \textstyle f_i(x) = \frac{1}{2}x^\top\mA x + b^\top x + c, \quad i \in [n],
\end{align}
where $\mA \in \R^{d\times d}$ is a symmetric but not necessarily positive semidefinite matrix, $b \in \R^{d}$ and $c \in \R$. 
We now investigate the operation of \algname{MARINA-P} with Perm$K$ compressors. With probability~$p$, we have $w^{t+1} = x^{t+1}$. Otherwise
    $w^{t+1} = w^t + \frac{1}{n} \sum_{i=1}^{n} \cC_i^t(x^{t+1} - x^t) \overset{\eqref{eq:restore}}{=} x^{t+1} + (w^t - x^t).$
Hence, if we initialize $w_i^{0} = x^0$ for all $i \in [n],$ an inductive argument shows that $w^{t} = x^{t}$ deterministically for all $t \geq 0$.
Then, substituting the gradients of $f_i$ to \eqref{eq:marina_p}, one gets 
\begin{align*}
    \textstyle g^{t} = \frac{1}{n} \sum\limits_{i=1}^n (\mA w_i^{t} + b) = \mA w^{t} + b = \mA x^{t} + b = \nabla f(x^{t})
\end{align*}
for all $t \geq 0.$
Therefore, \algname{MARINA-P} with Perm$K$ compressor in this setting is essentially a smart implementation of vanilla \algname{GD}! Indeed, for $p \leq \nicefrac{1}{n}$, \algname{MARINA-P} with Perm$K$ sends on average $\leq \nicefrac{2d}{n}$ coordinates to each worker, so the s2w communication complexity is
\begin{align*}
    \textstyle \frac{2d}{n} \times \underbrace{\textstyle \cO\left(\frac{\delta^0 L}{\varepsilon}\right)}_{\textnormal{\algnamesmall{GD} rate}} = \cO\left(\frac{d \delta^0 L}{n \varepsilon}\right),
\end{align*} which is $n$ times smaller than in \eqref{eq:comm_gd}!

\begin{table}[t]
\caption{\textbf{The worst case \emph{communication complexities} to find an $\varepsilon$--stationary point.} For simplicity, we compare the complexities with non-homogeneous quadratics: $f_i(x) = \frac{1}{2}x^\top \mA_ix + b_i^\top x + c_i,$ where $\mA_i \in \R^{d\times d}$ is symmetric but not necessarily positive semidefinite, $b_i \in \R^d$ and $c_i \in \R$ for $i \in [n].$ We denote $\mA = \frac{1}{n} \sum_{i=1}^n \mA_i.$}
\label{table:complexities}
\centering 
\begin{adjustbox}{width=1.0\columnwidth,center}
\begin{threeparttable}
\begin{tabular}{cc}
\begin{tabular}[t]{cccccc}
  \toprule
  \multicolumn{2}{c}{Server-to-Workers Communication Complexities (s2w)} \\
  \midrule
  \bf  Method & \bf Complexity \\
  \midrule
  \makecell{\scriptsize \algnamesmall{GD} \\ \scriptsize and other compressed methods\textsuperscript{{\color{blue}(a)}}} & $\geq \frac{d \delta^0 \norm{\mA}}{\varepsilon}$ \\
  \midrule
  \makecell{\scriptsize \algnamesmall{CORE} \\ \scriptsize \citep{yue2023core}} & $\frac{\delta^0 \textnormal{tr} \mA}{\varepsilon}$ \\
  \midrule
  \makecell{\scriptsize \algnamesmall{MARINA-P} \\ \scriptsize with independent Rand$K$\textsuperscript{{\color{blue}(b)}} \\ \scriptsize (Corollary~\ref{cor:main_corolary_3comp})} & \makecell{$\frac{d \delta^0 \frac{1}{n} \sum_{i=1}^n \norm{\mA_i}}{\sqrt{n} \varepsilon} + \frac{d \delta^0 \max_{i\in[n]} \norm{\mA_i - \mA}}{\varepsilon}$}\\
  \midrule
  \makecell{\scriptsize \algnamesmall{MARINA-P} with Perm$K$\textsuperscript{{\color{blue}(b)}} \\ \scriptsize (Corollary~\ref{cor:main_corolary})} & \makecell{$\frac{d \delta^0 \norm{\mA}}{n \varepsilon} + \frac{d \delta^0 \max_{i\in[n]} \norm{\mA_i - \mA}}{\varepsilon}$}\\
  \bottomrule
  \end{tabular}
  \begin{tabular}[t]{cccccc}
    \toprule
    \multicolumn{2}{c}{Total Communication Complexities (s2w + w2s)} \\
    \midrule
    \bf  Method & \bf Complexity \\
    \midrule
    \makecell{\scriptsize \algnamesmall{GD} \\ \scriptsize and other compressed \\ \scriptsize methods\textsuperscript{{\color{blue}(a)}}} & $\geq \frac{d \delta^0 \norm{\mA}}{\varepsilon}$ \\
    \midrule
    \makecell{\scriptsize \algnamesmall{CORE} \\ \scriptsize \citep{yue2023core}} & $\frac{\delta^0 \textnormal{tr} \mA}{\varepsilon}$ \\
    \midrule
    \makecell{\scriptsize \algnamesmall{\newmethod} \\ \scriptsize with Perm$K$ and Rand$K$\textsuperscript{{\color{blue}(b)}} \\ \scriptsize (Theorem~\ref{thm:marinap_gen_m})} & \makecell{$\frac{d \delta^0 \max\limits_{i \in [n]}\norm{\mA_i}}{n^{1/3} \varepsilon} + \frac{d \delta^0 \max\limits_{i\in[n]} \norm{\mA_i - \mA}}{\varepsilon}$}\\
    \bottomrule
  \end{tabular}
  \end{tabular}
  \begin{tablenotes}
  \scriptsize
  \item The complexities of \algnamesmall{MARINA-P} and \algnamesmall{\newmethod} with Perm$K$ are better when $n > 1$ and in close-to-homogeneous regimes, i.e., when $\max_{i\in[n]} \norm{\mA_i - \mA}$ is small.
  \item [{\color{blue}(a)}] including \algnamesmall{EF21-P} \citep{gruntkowska2023ef21}, \algnamesmall{dist-EF-SGD} \citep{zheng2019communication}, \algnamesmall{DORE} \citep{liu2020double}, \algnamesmall{MCM} \citep{philippenko2021preserved}, and \algnamesmall{EF21-BC} \citep{fatkhullin2021ef21}.
  \item [{\color{blue}(b)}] This table only showcases the results for Rand$K$ and Perm$K.$ A more general result for all compressors is provided in Sections~\ref{sec:marinap_convergence_theory} and \ref{sec:m3_convergence_theory}. 
\end{tablenotes}
\end{threeparttable}
\end{adjustbox}
\end{table}

\subsection{Functional $(L_A, L_B)$ Inequality}
From the discussion in Section~\ref{sec:lower_bound_main}, we know that to improve~\eqref{eq:comm_gd}, an extra assumption about the structure of the problem is needed.
Building on the example from Section~\ref{sec:homogeneous_quadratics}, we introduce the \emph{Functional $(L_A, L_B)$ Inequality}.
\begin{assumption}[Functional $(L_A, L_B)$ Inequality]\label{as:AB_assumption}
    \label{ass:functional}
    There exist constants $L_A,L_B \geq 0$ such that
    \begin{align}\label{eq:functional}
        &\textstyle \norm{\frac{1}{n} \sum\limits_{i=1}^n (\nabla f_i(x+u_i) - \nabla f_i(x))}^2 \leq L_A^2\left(\frac{1}{n} \sum\limits_{i=1}^n \norm{u_i}^2\right) + L_B^2 \norm{\frac{1}{n} \sum\limits_{i=1}^n u_i}^2
    \end{align}
    for all $x, u_1, \dots, u_n \in \R^d.$
\end{assumption}
\begin{remark}
    \label{remark:previous_assumption}
    A similar assumption, termed ``Heterogeneity-driven Lipschitz Condition on Averaged Gradients'', is proposed in \citet{wang2023new}. 
    Our assumption aligns with theirs when $L_B = 0$. However, our formulation proves to be more powerful. The possibility that $L_B > 0$ becomes instrumental in driving the enhancements we introduce.
\end{remark}
Assumption~\ref{ass:functional} is defined for all functions together, and intuitively, it tries to capture the similarities between the functions $f_i$. For $n = 1$, inequality \eqref{eq:functional} reduces to 
\begin{align*}
    \textstyle \norm{\nabla f(x) - \nabla f(y)}^2 \leq \left(L_A^2 + L_B^2\right)\norm{x - y}^2 \forall x,y \in \R^d,
\end{align*}
equivalent to standard $L$-smoothness (Assumption~\ref{ass:lipschitz_constant}) with $L^2 = L_A^2 + L_B^2$. The Functional $(L_A, L_B)$ Inequality is reasonably weak also for $n>1$, as the next theorem shows.
\begin{restatable}{theorem}{THEOREMLMAX}
    \label{theorem:a_b_max}
    For all $i \in [n],$ assume that the functions $f_i$ are $L_i$--smooth (Assumption~\ref{ass:local_lipschitz_constant}). Then, Assumption~\ref{ass:functional} holds with $L_A = L_{\max}$ and $L_B = 0$.
\end{restatable}
Therefore, Assumption~\ref{ass:functional} holds whenever the functions $f_i$ are smooth, which is a standard assumption in the literature. 
Now, returning to the example from Section~\ref{sec:homogeneous_quadratics},
\begin{restatable}{theorem}{THEOREMLQUAD}
    \label{theorem:a_b_quad}
    For all $i \in [n],$ assume that the functions $f_i$ are homogeneous quadratics defined in~\eqref{eq:homog_quad}. Then, Assumption~\ref{ass:functional} holds with $L_A = 0$ and $L_B = \norm{\mA}.$
\end{restatable}
Under Assumption~\ref{ass:local_lipschitz_constant}, no information about the similarity of the functions $f_i$ is available, yielding $L_B = 0$ and $L_A > 0$ in Theorem~\ref{theorem:a_b_max}. However, once we have some information limiting heterogeneity, $L_A$ can decrease. Notably, $L_A = 0$ for homogeneous quadratics. As we shall see in Section~\ref{sec:convergence}, the values $L_A$ and $L_B$ significantly influence the s2w communication complexity of \algname{MARINA-P}, with lower $L_A$ values leading to greatly improved performance.

\subsection{The Convergence Theory of \algnamebig{MARINA-P} with Perm$K$}
\label{sec:convergence}
We are ready to present our main convergence result, focusing on the Perm$K$ compressor from Section~\ref{sec:permk}. This choice simplifies the presentation, but our approach generalizes to a much larger class of compression operators. The full theoretical framework, covering all unbiased compressors, is detailed in Appendix~\ref{sec:marinap_convergence_theory}.

\begin{theorem}\label{thm:marinap_gen}
    Let Assumptions~\ref{ass:lipschitz_constant}, \ref{ass:lower_bound} and \ref{ass:functional} be satisfied. Set $w_i^0 = x^0$ for all $i \in [n].$ Take Perm$K$ as $\cC_i^t$ and $\gamma = \big(L + L_A \sqrt{\omega_P(\nicefrac{1}{p}-1)}\big)^{-1},$ where $\omega_P = n - 1$ (Lemma~\ref{lemma:3compr_omega_theta}). Then, \algname{MARINA-P} finds an $\varepsilon$--stationary point after
    \begin{align*}
        \textstyle \cO\left(\frac{\delta^0}{\varepsilon} \left(L + L_A \sqrt{\nicefrac{\omega_P}{p}}\right)\right)
    \end{align*}
    iterations.
\end{theorem}
\begin{corollary}
    \label{cor:main_corolary}
    Let $p = \nicefrac{K}{d} \equiv \nicefrac{1}{n}$. Then, in the view of Theorem~\ref{thm:marinap_gen}, the average s2w communication complexity of \algname{MARINA-P} with Perm$K$ compressor is
    \begin{align}
        \label{eq:main_compl}
        \textstyle \cO\left(\frac{d \delta^0 L}{n \varepsilon} + \frac{d \delta^0 L_A}{\varepsilon}\right).
    \end{align}
\end{corollary}
The key observation is that \eqref{eq:main_compl} is independent  of $L_B,$ and only depends on $L_A$. This particular property is specific to \emph{correlated compressors} with parameter $\theta=0$ (defined in Appendix~\ref{sec:3compr}), such as Perm$K$. A similar result holds for \emph{independent} Rand$K$ compressors (see 
Corollary~\ref{cor:main_corolary_3comp}), but the convergence rate is worse and depends on $L_B$. Nevertheless, this dependence improves with $n$.

When $L_A = 0,$ which is the case for homogeneous quadratics, the step size bound from Theorem~\ref{thm:marinap_gen} simplifies to $\gamma \leq \nicefrac{1}{L}$, the standard \algname{GD} stepsize (recall that in this case our method reduces to \algname{GD}). Most importantly, \eqref{eq:main_compl} scales with the number of workers $n$! Even when $L_A > 0,$ for sufficiently big $n,$ \eqref{eq:main_compl} can improve \eqref{eq:comm_gd} to $\cO\left(\nicefrac{d \delta^0 L_A}{\varepsilon}\right).$

Let us now investigate how the constants $L_A$ and $L_B$ change in the general case.

\subsection{Estimating $L_A$ and $L_B$ in the General Case}
\label{sec:general_case}

It is clear from Corollary~\ref{cor:main_corolary} that \algname{MARINA-P} with Perm$K$ shines when $L_A$ is small. 
To gain further insights into what values $L_A$ may take, we now provide an analysis based on the Hessians of the functions $f_i$.

\begin{restatable}{theorem}{THEOREMLHESS}
    \label{theorem:a_b_hess}
    Assume that the functions $f_i$ are twice continuously differentiable, $L_i$--smooth (Assumption~\ref{ass:local_lipschitz_constant}), and that there exist $D_i\geq 0$ such that
    \begin{align}
        \label{eq:hess_diff}
        \textstyle \sup\limits_{z_1,\ldots,z_n\in\R^d} \norm{\nabla^2 f_i(z_i) - \frac{1}{n} \sum_{j=1}^n \nabla^2 f_j(z_j)} \leq D_i
    \end{align}
    for all $i \in [n]$.
    Then, Assumption~\ref{ass:functional} holds with $L_A = \sqrt{2} \max_{i\in[n]} D_i \leq 2 \sqrt{2} \max_{i \in [n]} L_i$ and $L_B = \sqrt{2} \left(\frac{1}{n} \sum_{i=1}^n L_i\right).$
\end{restatable}

Intuitively, \eqref{eq:hess_diff} measures the similarity between the functions $f_i$. The above theorem yields a more refined result than Theorem~\ref{theorem:a_b_max}: it is always true that $\max_{i\in[n]} D_i \leq 2 \max_{i \in [n]} L_i$, and, in fact, $\max_{i\in[n]} D_i$ can be much smaller, as the next result shows.

\begin{restatable}{theorem}{THEOREMLGENQAUD}
    \label{theorem:gen_quad}
    Assume that 
        $f_i(x) = \frac{1}{2}x^\top\mA_ix + b_i^\top x + c_i,$
    where $\mA_i \in \R^{d\times d}$ is symmetric but not necessarily positive semidefinite, $b_i \in \R^d$ and $c_i \in \R$ for $i \in [n].$ Define $\mA = \frac{1}{n} \sum_{i=1}^n \mA_i.$
    Then, Assumption~\ref{ass:functional} holds with $L_A = \sqrt{2} \max_{i\in[n]} \norm{\mA_i - \mA}$ and $L_B = \sqrt{2} \left(\frac{1}{n} \sum_{i=1}^n \norm{\mA_i}\right).$
\end{restatable}

Thus, $L_A$ is less than or equal to $\sqrt{2} \max_{i\in[n]} \norm{\mA_i - \mA},$ which serves as a measure of similarity between the matrices. The smaller the values of $\norm{\mA_i - \mA}$ (indicating greater similarity among the functions $f_i$), the smaller the $L_A$ value.

In the view of this theorem, the s2w communication complexity of \algname{MARINA-P} with Perm$K$ on non-homogeneous quadratics is
\begin{align}
    \label{eq:non_homog_quad}
    \textstyle \cO\left(\frac{d \delta^0 \norm{\mA}}{n \varepsilon} + \frac{d \delta^0 \max_{i\in[n]} \norm{\mA_i - \mA}}{\varepsilon}\right).
\end{align}
Since the corresponding complexity of \algname{GD} is
\begin{align}
    \label{eq:non_homog_quad_gd}
    \textstyle \cO\left(\frac{d \delta^0 \norm{\mA}}{\varepsilon}\right),
\end{align}
in the close-to-homogeneous regimes (i.e., when $\max_{i\in[n]} \norm{\mA_i - \mA}$ is small), the complexity \eqref{eq:non_homog_quad} can be \emph{provably} much smaller than \eqref{eq:non_homog_quad_gd}. The same reasoning applies to the general case when the functions $f_i$ are not quadratics: \algname{MARINA-P} improves with the number of workers $n$ in the regimes when $D_i$ are small (see Theorem~\ref{theorem:a_b_hess}).

Let us note that there is another method, \algname{CORE}, by \citet{yue2023core}, that can also provably outperform \algname{GD}, achieving the s2w communication complexity of $\Omega\left(\nicefrac{\delta^0 \textnormal{tr}{\mA}}{\varepsilon}\right) $ on non-homogeneous quadratics. Neither their method nor ours universally provides the best possible communication guarantees. Our method excels in the close-to-homogeneous regimes: for example, if we take $\mA_i = L_i \mI$ for all $i \in [n],$ and define $L = \nicefrac{1}{n}\sum_{i=1}^n L_i$, then the complexity of \algname{CORE} is $\Omega\left(\nicefrac{d \delta^0 L}{\varepsilon}\right),$ while ours is $\cO\left(\frac{d \delta^0 L}{n \varepsilon} + \frac{d \delta^0 \max_{i\in[n]} \lvert L_i - L \rvert}{\varepsilon}\right).$ Hence, our guarantees are superior in regimes where $\max_{i\in[n]} \lvert L_i - L\rvert \ll L$. One interesting research direction is to develop a universally better method combining the benefits of both approaches.

\section{\algnamebig{\newmethod}: A New Bidirectional Method}\label{sec:m3}

In the previous sections, we introduce a new method that provably improves the \emph{server-to-worker} communication, but ignores the \emph{worker-to-server} communication overhead. Our aim now is to treat \algname{MARINA-P} as a starting point for developing methods applicable to more practical scenarios, by combining it with techniques that compress in the opposite direction.
Since the theoretical state-of-the-art w2s communication complexity is obtained by \algname{MARINA} (see Section~\ref{sec:related_work}), our next research step was to combine the two and analyze ``\algname{MARINA} + \algname{MARINA-P}'', but this naive approach did not yield communication complexity guarantees surpassing \eqref{eq:comm_gd} in any regime. It became apparent that some ``buffer'' step between these two techniques is needed, and this step turned out to be the momentum. Our new method, \algname{\newmethod} (Algorithm \ref{alg:m3}), is described in \eqref{eq:mthree}.

\begin{figure}[t]
\begin{theorembox}
\centerline{The \algname{\newmethod} Method}
\centerline{(\algname{\explain}): }
Initialize vectors $x_0, w_i^0, g_i^0, z_i^0 \in \R^d$ for all $i \in [n]$, step size $\gamma > 0$, probabilities $0 < p_P, p_D \leq 1$ and compressors $\cC_1^t,\ldots,\cC_n^t \in \mathbb{U}(\omega_P) \cap \mathbb{P}(\theta)$\footnote{By $\mathbb{P}(\theta)$ we denote a family of \emph{correlated compressors} (defined in Appendix \ref{sec:3compr}). It includes, among others, Perm$K$ compressors.}, $\cQ_1^t,\ldots,\cQ_n^t \in \mathbb{U}(\omega_D)$ for all $t \geq 0.$ The method iterates
\begin{equation}
\begin{aligned}
    \label{eq:mthree}
    x^{t+1} &= x^t - \gamma g^t, \\
    w^{t+1}_i &= 
    \begin{cases}
        x^{t+1} & \textnormal{with probability } p_P,\\
        w_i^t + \cC_i^t(x^{t+1} - x^t) & \textnormal{with probability } 1 - p_P,
    \end{cases} \\
    z_i^{t+1} &= \beta w_i^{t+1} + (1 - \beta) z_i^{t} \qquad \textnormal{\algname{(Momentum)}}\\
    g^{t+1}_i &= 
    \begin{cases}
        \nabla f_i(z_i^{t+1}) & \textnormal{with probability } p_D,\\
        g^{t}_i + \cQ_i^t(\nabla f_i(z_i^{t+1}) - \nabla f_i(z_i^{t})) & \textnormal{with probability } 1 - p_D
    \end{cases} \\
    \textnormal{for al}&\textnormal{l } i \in [n],
\end{aligned}
\end{equation}
where the probabilistic decisions are the same for all $i \in [n],$ i.e., one coin is tossed for all workers (as in \eqref{eq:marina} and \eqref{eq:marina_p}), and the coins for the first and second probabilistic decisions with $p_P$ and $p_D$ are independent.
We denote $w^{t} \eqdef \nicefrac{1}{n} \sum_{i=1}^n w_i^t,$ $g^{t} \eqdef \nicefrac{1}{n} \sum_{i=1}^n g_i^t,$ $z^{t} \eqdef \nicefrac{1}{n} \sum_{i=1}^n z_i^t.$ See the implementation in Algorithm~\ref{alg:m3}.
\end{theorembox}
\end{figure}
\algname{\newmethod} combines \eqref{eq:marina}, \eqref{eq:marina_p}, and the momentum step $z_i^{t+1} = \beta w_i^{t+1} + (1 - \beta) z_i^{t},$ which is the key to our improvements. A similar technique is used to reduce the variance in \citet{fatkhullin2023momentum}.
Let us explain how \algname{\newmethod} works in practice. First, the server calculates $x^{t+1}.$ Depending on the first probabilistic decision, it sends either $x^{t+1}$ or $\cC_i^t(x^{t+1} - x^{t})$ to the workers, who then calculate $w^{t+1}_i$ locally. Next, the workers compute $z^{t+1}_i,$ and depending on the second probabilistic decision, they send either $\nabla f_i(z_i^{t+1})$ or $\cQ_i^t(\nabla f_i(z_i^{t+1}) - \nabla f_i(z_i^{t}))$ back to the server. The server aggregates the received vectors and calculates $g^{t+1}.$ As in \algname{MARINA}, $p_P$ and $p_D$ are chosen in such a way that the non-compressed communication does not negatively affect the communication complexity. Therefore, the method predominantly transmits compressed information, with only a marginal probability of sending uncompressed vectors.

\subsection{The Convergence Theory of \algnamebig{\newmethod}}
\label{sec:convergence_m}

For simplicity, we consider Perm$K$ in the role of $\cC_i^t$ and Rand$K$ in the role of $\cQ_i^t$. The general theory for all unbiased compressors is presented in Section~\ref{sec:m3_convergence_theory}.

\begin{theorem}\label{thm:marinap_gen_m}
    Let Assumptions~\ref{ass:lipschitz_constant}, \ref{ass:lower_bound}, \ref{ass:local_lipschitz_constant} and \ref{ass:functional} be satisfied. Take $\gamma = \big(L + 34 \parens{n L_A + n^{2/3} L_B + n^{2/3} L_{\max}}\big)^{-1}$, $p_D = p_P = \nicefrac{1}{n}$, $\beta = n^{-2/3}$, $w_i^0 = z_i^0 = x^0$ and $g_i^0 = \nabla f_i(x^0)$ for all $i \in [n]$. Then \algname{MARINA-P} with $\cC_i^t = $Perm$K$ and $\cQ_i^t =$ Rand$K$ with $K = \nicefrac{d}{n}$ finds an $\varepsilon$--stationary point after
    $\cO\big(\frac{\delta^0}{\varepsilon} \left(n^{2/3} L_{\max} + n L_A\right)\big)$
    iterations. The total communication complexity is
    \begin{align}
        \label{eq:total_comm}
        \textstyle
        \cO\left(\frac{d \delta^0 L_{\max}}{n^{1/3} \varepsilon} + \frac{d \delta^0 L_A}{\varepsilon}\right).
    \end{align}
\end{theorem}

Once again, we observe improvement with the number of workers $n$, and the obtained complexity~\eqref{eq:total_comm} can be provably smaller than \eqref{eq:comm_gd}. Indeed, in scenarios like federated learning, where the number of workers (e.g., mobile phones) is typically large \citep{kairouz2021advances, chowdhery2023palm}, the first term can be significantly smaller than $\nicefrac{d \delta^0 L}{\varepsilon}$. The second term can also be small in close-to-homogeneous regimes (see Section~\ref{sec:marina_p}).

\section{Experimental Highlights}\label{sec:marinap_quadratic}
This section presents insights from the experiments, with further details and additional results in Appendix~\ref{sec:experimetns_main}.
The experiment aims to empirically test the theoretical results from Section~\ref{sec:marina_p}. We consider a quadratic optimization problem, where the functions $f_i$ are
as defined in Theorem \ref{theorem:gen_quad} and $\mA_i\in\R^{300\times300}$.
We compare \algname{GD}, \algname{MARINA-P} sending the same message compressed using a single Rand$K$ compressor to all workers (``SameRand$K$'' from Appendix \ref{sec:3compr}), \algname{MARINA-P} with independent Rand$K$ compressors, \algname{MARINA-P} with Perm$K$ compressors, and \algname{EF21-P} with Top$K$ compressor. We consider $n \in \{10, 100, 1000\}$ and fine-tune the step size for each algorithm.
The results, presented in Figure~\ref{fig:marinap_quadratic}, align closely with the theory, with \algname{MARINA-P} using Perm$K$ compressors consistently performing best. Moreover, the convergence rate of \algname{MARINA-P} with Perm$K$ and independent Rand$K$ compressors improves with $n$. Since this is not the case for \algname{EF21-P}, even though it outperforms \algname{MARINA-P} with independent Rand$K$ compressors for $n=10$, it falls behind for $n\in\{100,1000\}$.

\begin{figure}[t]
\centering
\begin{subfigure}{0.33\columnwidth}
  \centering
  \includegraphics[width=\columnwidth]{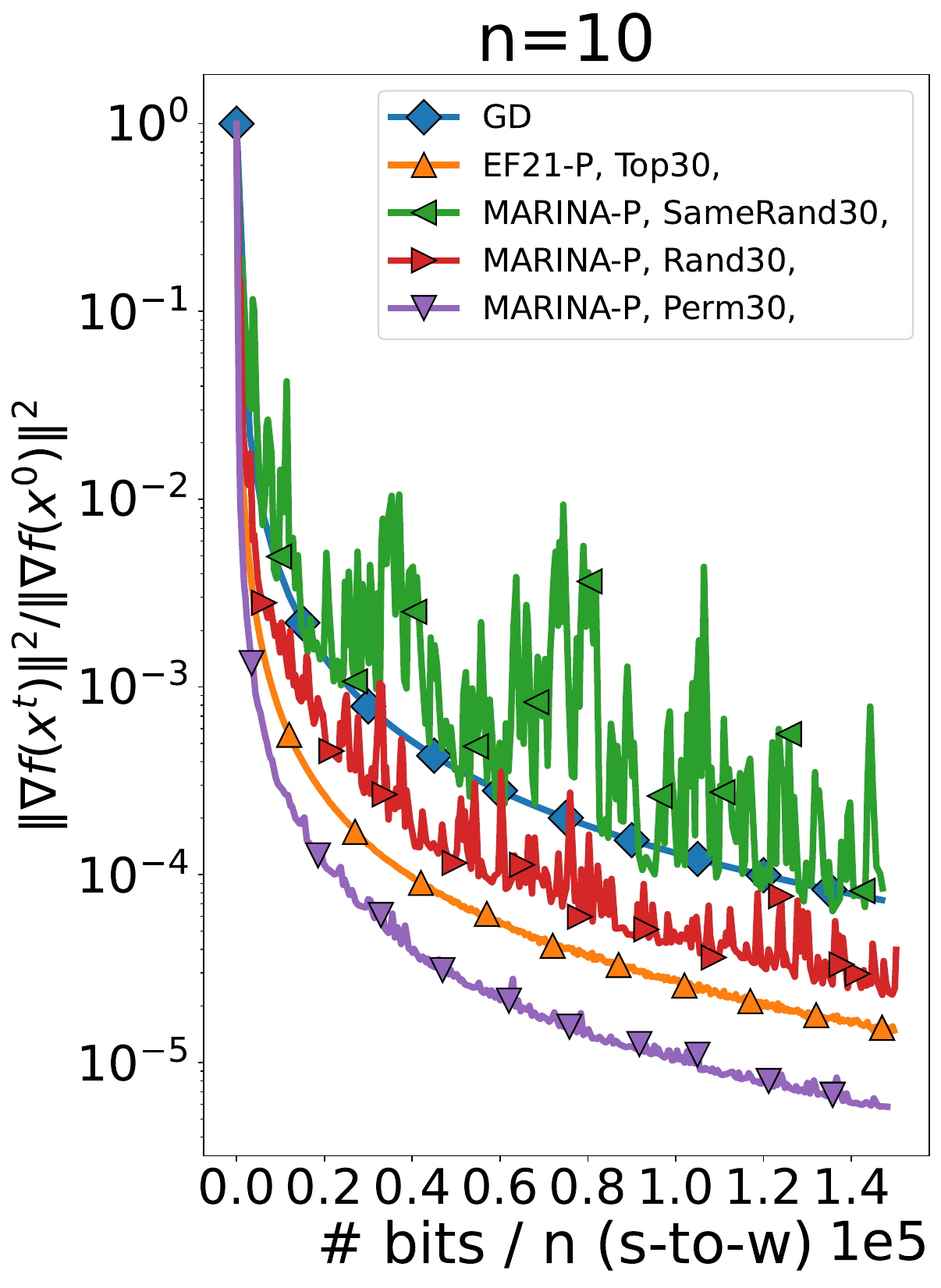}
\end{subfigure}
\begin{subfigure}{0.32\columnwidth}
  \centering
  \includegraphics[width=\columnwidth]{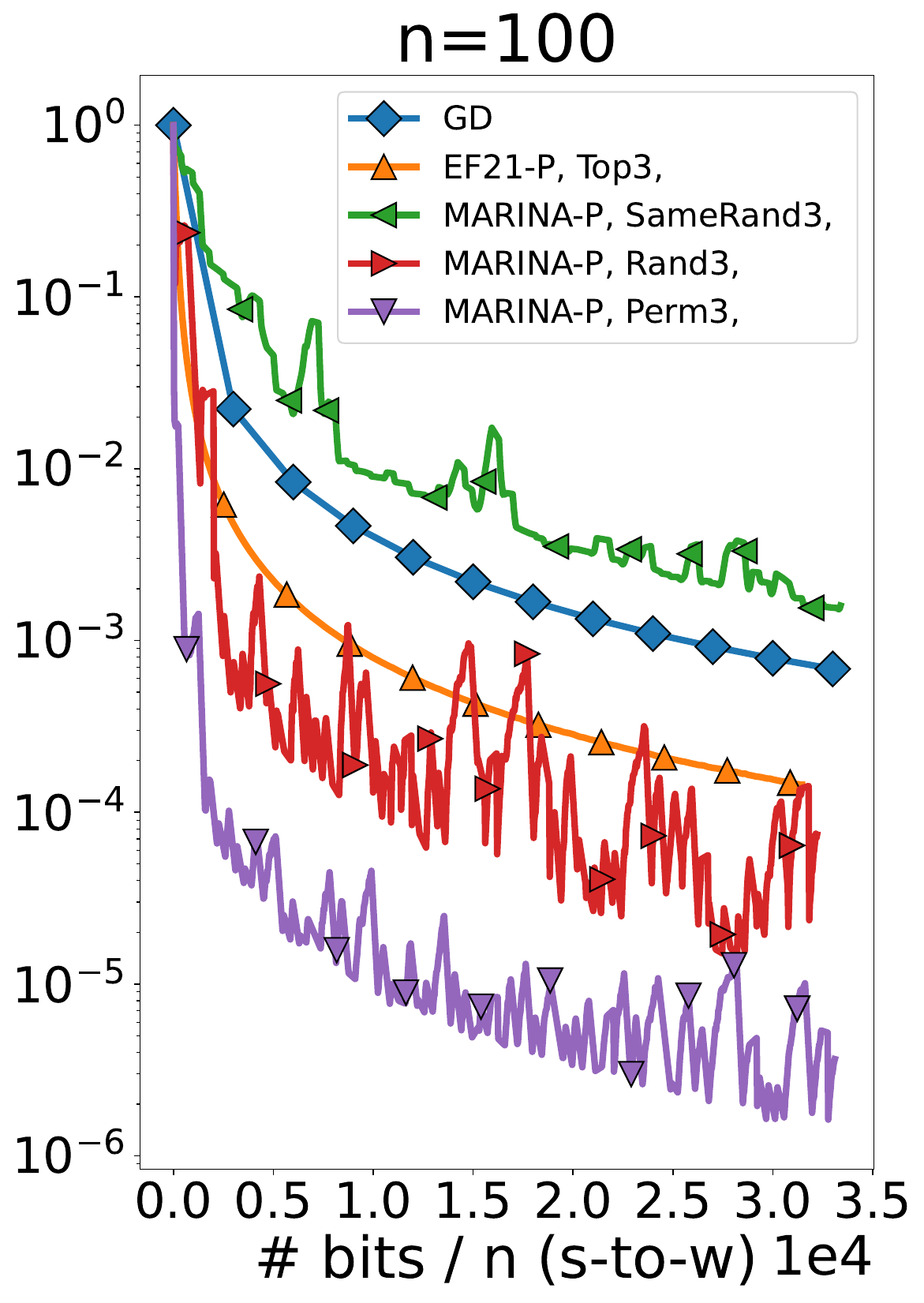}
\end{subfigure}
\begin{subfigure}{0.31\columnwidth}
  \centering
  \includegraphics[width=\columnwidth]{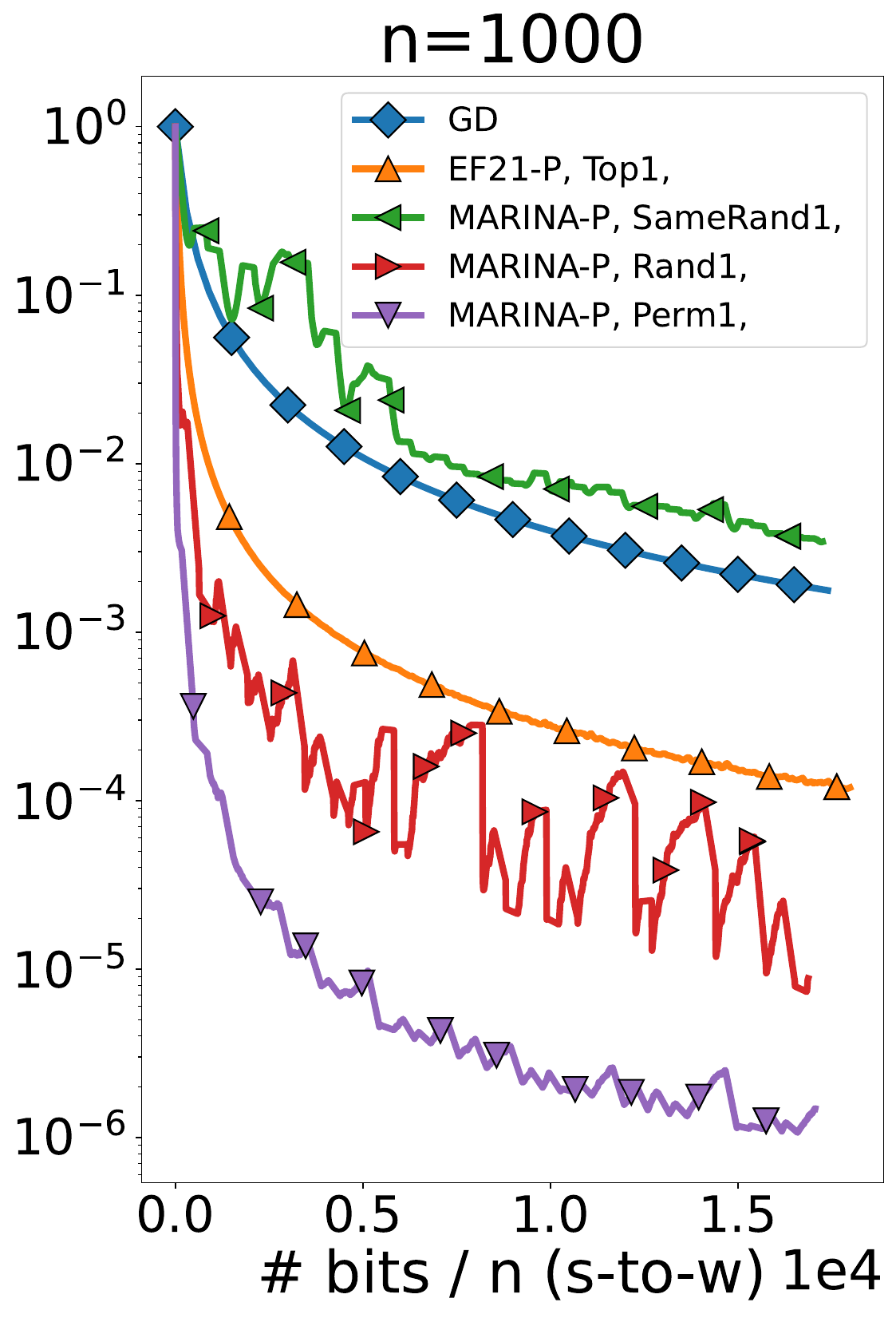}
\end{subfigure}
\caption{Experiments on the quadratic optimization problem from Section~\ref{sec:marinap_quadratic}. We plot the norm of the gradient w.r.t. \# of coordinates sent from the server to the workers.}
\label{fig:marinap_quadratic}
\end{figure}

\begin{ack}
The research reported in this publication was supported by funding from King Abdullah University of Science and Technology (KAUST): i) KAUST Baseline Research Scheme, ii) Center of Excellence for Generative AI, under award number 5940, iii) SDAIA-KAUST Center of Excellence in Artificial Intelligence and Data Science. The work of A.T. was partially supported by the Analytical center under the RF Government (subsidy agreement 000000D730321P5Q0002, Grant No. 70-2021-00145 02.11.2021).
\end{ack}

\bibliography{example_paper}
\bibliographystyle{apalike}

\newpage

\appendix

\tableofcontents

\newpage

\section{Three unbiased ways to compress}\label{sec:3compr}

The main focus of this paper is handling the server-to-worker communication costs. To explain better where the improvements outlined in the main part of this paper come from, let us first consider the scenario where uplink communication cost is negligible but downilnk communication cost is not.
While we do no necessarily say that this is a realistic setup, examining it first enables us to understand how downlink compression should be performed, capturing all the intricacies.

Existing algorithms with lossy s2w and w2s communication have a certain common feature. The compression mechanism employed on the clients is very different from the one used on the server: while each client transmits to the server a different message, specific to the data stored on each device, the server broadcasts the same update to all clients. We want to question this algorithmic step and suggest to the reader that if compression is applied multiple times and each worker receives its individual update, then intuitively more information can be transmitted. A well-designed algorithm should be able to take advantage of this.

One can depart from the usual approach of sending the same update to all workers in two ways: a)~compress the update $n$ times \textit{independently}, or b) produce $n$ such updates in a \textit{correlated} way. Either way, the server broadcasts $n$ different compressed messages rather than one, and sends a \emph{different} update to each worker.
The key discovery here is that both a) and b) are mathematically \emph{provably better} than the prevalent approach of sending the same update to all clients.

This is a crucial improvement in a system where the above setup is a good approximation of reality. And even if it is not, and the current model is not perfectly capturing the reality, we can accept it for now, as it allows us to focus on the novel aspects of the approach. With that said, these considerations can serve as a starting point for thinking about \emph{bidirectional compression}: having focused on the simplified setup and equipped with knowledge on how the compression on the master should be performed, we employ this mechanism in more complex scenarios (see Section \ref{sec:m3}).

Let us now describe the three possible ways to perform compression on the server.

\paragraph{``Same'' compressors.}

The prevalent approach in downlink compression is to transmit the same update to all workers. To illustrate this, let us call a collection $\cC_1,\ldots,\cC_n$ of compressors ``SameRand$K$'' if for all $i \in [n]$ we have $\cC_i = \cC$ for some Rand$K$ compressor $\cC$.
Now, consider one iteration $t$ of \algname{MARINA-P} with SameRand$K$ compressor. The server calculates $\cC_i^t(x^{t+1}-x^t)$ for $i\in[n]$, but in this case, $\cC_1^t(x^{t+1}-x^t)=\ldots=\cC_n^t(x^{t+1}-x^t)=\cC^t(x^{t+1}-x^t)$.
Thus, applying a collection of SameRand$K$ compressors to some vector $x\in\R^d$ is equivalent to using a single Rand$K$ compression operator and transmitting the same message $\cC(x)$ to all workers.

\paragraph{Independent Compressors.}

Rather than setting $\cC_i(x) = \cC(x)$ for all $i \in [n]$, one can break the dependency between the messages and allow the compressors to differ. For illustrational purposes, suppose that $\cC_i, i\in[n]$ are independent Rand$K$ compressors (Assumption~\ref{ass:independent}). Then, applying such a collection of mappings to the vector of interest $x\in\R^d$, one obtains $n$ distinct and independent sparse vectors $\cC_1(x), \ldots, \cC_n(x)$.

\begin{remark}
    We are aware of only one method that uses $n$ distinct compressors in downlink compression, \algname{Rand-MCM} by \citet{philippenko2021preserved}. Given the absence of results in the non-convex case, let us compare the communication complexities of \algname{Rand-MCM} and \algname{M3} under the Polyak-Łojasiewicz condition (Assumption \ref{ass:pl}), which holds under strong convexity.
    In the strongly convex case, the proved iteration complexity of \algname{Rand-MCM} is
    \begin{align*}
        \Omega\parens{\frac{L_{\max}}{\mu} \parens{\omega_P^{3/2} + \frac{\omega_P\omega_D^{1/2}}{\sqrt{n}} + \frac{\omega_D}{n}} \log\frac{\delta^0}{\varepsilon}}.
    \end{align*}
    Assuming for simplicity that the server and the workers use Rand$K$ compressors with $K=\nicefrac{d}{n}$, this gives the total communication complexity of
    \begin{align*}
        \Omega\parens{\frac{d}{n} \times \frac{L_{\max}}{\mu} \parens{\omega_P^{3/2} + \frac{\omega_P\omega_D^{1/2}}{\sqrt{n}} + \frac{\omega_D}{n}} \log\frac{\delta^0}{\varepsilon}}
        &= \Omega\parens{\frac{d \sqrt{n} L_{\max}}{\mu} \log\frac{\delta^0}{\varepsilon}},
    \end{align*}
    which is getting \emph{worse} as the number of workers $n$ increases.
    Meanwhile, by Corollary \ref{cor:m3_totalcomm_perm}, the total communication complexity of \algname{M3} (where $\cC_i^t$ are the Perm$K$ compressors and $\cQ_i^t$ are independent Rand$K$ compressors, both with $K = d / n$) under the Polyak-Łojasiewicz condition is
    \begin{align*}
        \cO\left(\left(\frac{d L_{\max}}{n^{1/3} \mu} + \frac{d L_A}{\mu} + d\right) \log\frac{\delta^0}{\varepsilon}\right).
    \end{align*}
    Since $n$ is typically large, the total communication complexity of \algname{M3} can be much better than that of \algname{Rand-MCM}.
\end{remark}

\paragraph{Correlated Compressors.}

In their work, \citet{szlendak2021permutation} introduce an alternative class of compressors, which satisfy the following condition:

\begin{definition}[AB-inequality \citep{szlendak2021permutation}]\label{def:corr_compr_AB}
    There exist constants $A, B \geq 0$ such that the random operators $\cC_1, \ldots \cC_n$ satisfy
    \begin{align}\label{eq:corr_compr_AB}
        \Exp{\cC_i(x)} &= x, \nonumber \\
        \Exp{\norm{\frac{1}{n} \sum_{i=1}^n \cC_i(x_i) - \frac{1}{n} \sum_{i=1}^n x_i}^2} &\leq A \frac{1}{n} \sum_{i=1}^n \norm{x_i}^2
        - B \norm{\frac{1}{n} \sum_{i=1}^n x_i}^2
    \end{align}
    for all $x, x_1,\ldots,x_n \in \R^d$. If these conditions hold, we write $\brac{\cC_i}_{i=1}^n \in \mathbb{U}(A,B)$.
\end{definition}

Following on this idea, we introduce the concept of a \textit{collection of correlated compressors}.

\begin{definition}[Collection of Correlated Compressors]\label{def:corr_compr}
    There exists a constant $\theta\geq 0$ such that the random operators $\cC_1, \ldots \cC_n$ satisfy:
    \begin{align}\label{eq:corr_compr}
        \Exp{\cC_i(x)} &= x \nonumber \\
        \Exp{\norm{\frac{1}{n} \sum_{i=1}^n \cC_i(x) - x}^2} &\leq \theta\norm{x}^2
    \end{align}
for all $x\in\R^d$. If these conditions hold, we write $\brac{\cC_i}_{i=1}^n \in \mathbb{P}(\theta)$.
\end{definition}
Definition \ref{def:corr_compr} will play a key role in our upcoming advancements. But what makes this assumption reasonable?

First, it is easy to note that condition \eqref{eq:corr_compr} is weaker than \eqref{eq:corr_compr_AB}. Indeed, if $\brac{\cC_i}_{i=1}^n \in \mathbb{U}(A,B)$, then inequality~\eqref{eq:corr_compr} holds with $\theta \eqdef A-B$. It turns out that it is in fact strictly weaker, as the following example shows.

\begin{example}
    Let $n = 2$, $d = 1$. Let $\brac{\zeta_x: x \in \R}$ be a collection of independent Cauchy variables indexed by real numbers. Define $\cC_1(u) = u + \zeta_u$, and $\cC_2(u) = u - \zeta_u$. Then
    \begin{align*}
        \frac{1}{2} \parens{\cC_1(u) + \cC_2(u)} = u,
    \end{align*}
    so $\cC_1(u)$ and $\cC_2(u)$ satisfy Definition \ref{def:corr_compr} with $\theta = 0$. However, for $u_1 \neq u_2$, by the properties of Cauchy distribution we have
    \begin{align*}
        \Exp{\parens{\frac{1}{2} \parens{\cC_1(u) + \cC_2(u)} - \frac{1}{2} \parens{u_1 + u_2}}^2}
        &= \Exp{\parens{\frac{1}{2} \parens{\zeta_1+\zeta_2}}^2} \\
        &= \frac{1}{4} \Exp{\zeta_1^2+\zeta_2^2+2\zeta_1\zeta_2}
        = \infty.
    \end{align*}
    Thus, $\cC_1(u)$ and $\cC_2(u)$ do not satisfy Definition \ref{def:corr_compr_AB}.
\end{example}

In fact, the condition specified in Definition \ref{def:corr_compr} does not impose any restrictions on the compressor class when working with unbiased compressors. This is because, for any set of compressors $\cC_1,\ldots,\cC_n \in \mathbb{U}(\omega)$, there exists $\theta \geq 0$ such that $\brac{\cC_i}_{i=1}^n \in \mathbb{P}(\theta)$, as shown in the following lemma.

\begin{lemma}\label{lemma:theta_indep_same}
    \leavevmode
    \begin{enumerate}
        \item Let $\cC_1,\ldots,\cC_n$  be a collection of compressors such that $\cC_i \in \mathbb{U}(\omega)$ for all $i\in[n]$. Then $\brac{\cC_i}_{i=1}^n \in \mathbb{P}(\omega)$.
        \item Let us further assume that $\cC_1,\ldots,\cC_n$ are independent (Assumption~\ref{ass:independent}). Then $\brac{\cC_i}_{i=1}^n \in \mathbb{P}(\nicefrac{\omega}{n})$.
    \end{enumerate}
\end{lemma}
\begin{proof}
    \begin{enumerate}
        \item Jensen's inequality gives
        \begin{align*}
            \Exp{\norm{\frac{1}{n} \sum_{i=1}^n \cC_i(u) - u}^2}
            &\overset{\eqref{eq:jensen}}{\leq} \frac{1}{n} \sum_{i=1}^n \Exp{\norm{\cC_i(u) - u}^2}
            \overset{\textnormal{Def.}\ref{def:unbiased_compression}}{\leq} \frac{1}{n} \sum_{i=1}^n \omega \norm{u}^2
            = \omega \norm{u}^2,
        \end{align*}
        so $\theta=\omega$.
        \item Using independence of compressors, we have
        \begin{align*}
            \Exp{\norm{\frac{1}{n} \sum_{i=1}^n \cC_i(u) - u}^2}
            &= \frac{1}{n^2} \sum_{i=1}^n \Exp{\norm{\cC_i(u) - u}^2}
            \overset{\textnormal{Def.}\ref{def:unbiased_compression}}{\leq} \frac{1}{n^2} \sum_{i=1}^n \omega \norm{u}^2
            = \frac{\omega}{n} \norm{u}^2.
        \end{align*}
        Thus $\theta=\nicefrac{\omega}{n}$.
    \end{enumerate}
\end{proof}

However, the true advantages of employing correlated compressors become apparent when the definition holds with $\theta=0$, as in the case of Perm$K$ compressors.

\begin{lemma}\label{lemma:3compr_omega_theta}
    Let $\cC_1,\ldots,\cC_n$  be a collection of a) SameRand$K$, b) independent Rand$K$, c) Perm$K$ compressors. Then
    \begin{enumerate}[a)]
        \item $\brac{\cC_i}_{i=1}^n \in \mathbb{P}(\omega)$ and $\cC_i \in \mathbb{U}(\omega)$ where $\omega = \nicefrac{d}{K} - 1$ for all $i \in [n]$,
        \item $\brac{\cC_i}_{i=1}^n \in \mathbb{P}(\nicefrac{\omega}{n})$ and $\cC_i \in \mathbb{U}(\omega)$ where $\omega = \nicefrac{d}{K} - 1$ for all $i \in [n]$,
        \item $\brac{\cC_i}_{i=1}^n \in \mathbb{P}(0)$ and $\cC_i \in \mathbb{U}(\omega)$ where $\omega = n - 1$ for all $i \in [n]$.
    \end{enumerate}
\end{lemma}
\begin{proof}
    Unbiasedness follows easily from definitions of compressors (the proof for Perm$K$ compressors can be found in \citet{szlendak2021permutation}). That Rand$K \in \mathbb{U}(\nicefrac{d}{K} - 1)$ (and hence trivially SameRand$K \in \mathbb{U}(\nicefrac{d}{K} - 1)$) is a well-known fact.
    Next, the fact that a) $\brac{\cC_i}_{i=1}^n \in \mathbb{P}(\omega)$ for SameRand$K$ compressors and b) $\brac{\cC_i}_{i=1}^n \in \mathbb{P}(\nicefrac{\omega}{n})$ for independent Rand$K$ compressors follows directly from Lemma \ref{lemma:theta_indep_same}.
    
    To compute $\omega$ for Perm$K$ compressor, first assume that $d \geq n$. Then $\Exp{\norm{\cC_i(x)}^2} = n\norm{x}^2$ \citep{szlendak2021permutation}, so
    \begin{align*}
        \Exp{\norm{\cC_i(x) - x}^2} \overset{\eqref{eq:vardecomp}}{=} \Exp{\norm{\cC_i(x)}^2} - \norm{x}^2 = (n-1) \norm{x}^2.
    \end{align*}
    Similarly, suppose that $d \leq n$, and write $n$ as $n=qd+r$, where $q\in\mathbb{N}_{>0}$ and $0\leq r<d$. Then $\Exp{\norm{\cC_i(x)}^2} = \nicefrac{n}{q} \norm{x}^2$ \citep{szlendak2021permutation}, and hence
    \begin{align*}
        \Exp{\norm{\cC_i(x) - x}^2}
        \overset{\eqref{eq:vardecomp}}{=} \Exp{\norm{\cC_i(x)}^2} - \norm{x}^2
        = \parens{\frac{n}{q}-1} \norm{x}^2
        \leq \parens{n-1} \norm{x}^2.
    \end{align*}
    In both cases $\omega=n-1$.

    Finally, by construction of Perm$K$, we have $\frac{1}{n} \sum_{i=1}^n \cC_i(x) = x$, implying $\theta = 0$.
\end{proof}

In what follows, when considering the Perm$K$ compressor, we shall assume for simplicity that $d \geq n$. The results for $d < n$ are analogous.

\section{Biased Compressors}
\label{sec:biased_compressors}

In addition to unbiased compressors (Definition~\ref{def:unbiased_compression}), the literature of compressed methods distinguishes another class of mappings:
\begin{definition}
\label{def:biased_compression}
A stochastic mapping $\cC\,:\,\R^d \rightarrow \R^d$ is a \textit{biased compressor} if
there exists $\alpha \in (0,1]$ such that
\begin{align}
    \label{eq:biased_compressor}
    \qquad \Exp{\norm{\cC(x) - x}^2} \leq (1 - \alpha) \norm{x}^2 \quad \forall x \in \R^d.
\end{align}
\end{definition}

The family of such compressors is denoted by $\mathbb{B}(\alpha)$. It is well-known that if $\cC \in \mathbb{U}(\omega),$ then $(\omega + 1)^{-1} \cC \in \mathbb{B}\left((\omega + 1)^{-1}\right)$, meaning that the family of biased compressors is broader. A canonical example is the Top$K \in \mathbb{B}(\nicefrac{K}{d})$ compressor, which preserves the $K$ largest in magnitude coordinates of the input vector \citep{beznosikov2020biased}.

\section{Properties of $L_A$ and $L_B$}

We first prove the results from Section \ref{sec:marina_p}, starting with calculating the constants $L_A$ and $L_B$ from Assumption \ref{ass:functional} in some special cases.

\begin{algorithm}[t]
    \caption{\algname{MARINA-P}}
    \begin{algorithmic}[1]\label{alg:marinap}
    \STATE \textbf{Input:} initial model $x_0\in\R^d$ {\color{gray}(stored on the server)}, initial model shifts $w_1^0=\ldots=w_n^0 = x^0$ {\color{gray}(stored on the workers)}, step size $\gamma > 0$, probability $0 < p \leq 1$, compressors $\cC_1^t,\dots,\cC_n^t \in \mathbb{U}(\omega_P)$
    \FOR{$t = 0, \dots, T$}
    \FOR{$i = 1, \dots, n$ in parallel}
    \STATE Calculate $\nabla f_i(w_i^t)$ and send it to the server
    \hfill{\scriptsize \color{gray} Workers evaluate the gradients at the current model estimate}
    \ENDFOR \\
    \textbf{On the server:}
    \STATE $g^t = \frac{1}{n} \sum_{i=1}^n \nabla f_i(w_i^t)$
    \hfill{\scriptsize \color{gray}Server averages the messages received from the workers}
    \STATE $x^{t+1} = x^t - \gamma g^t$ \hfill{\scriptsize \color{gray}Server takes a gradient-type step to update the global model}
    \STATE Sample $c^t \sim \textnormal{Bernoulli}(p)$
    \IF{$c^t = 0$}
        \STATE Send $\cC_i^t(x^{t+1} - x^t)$ to worker $i$ for $i\in[n]$ \hfill{\scriptsize \color{gray}Server sends compressed messages to all workers w.p. $1-p$}
    \ELSE
        \STATE Send $x^{t+1}$ to worker $i$ for $i\in[n]$ \hfill{\scriptsize \color{gray}Server sends the same uncompressed message to all workers w.p. $p$}
    \ENDIF \\
    \textbf{On the workers:}
    \FOR{$i = 1, \dots, n$ in parallel}
    \STATE $w^{t+1}_i = 
    \begin{cases}
        x^{t+1} & \text{if } c^t = 1,\\
        w_i^t + \cC_i^t(x^{t+1} - x^t) & \text{if } c^t = 0
    \end{cases}$ \alglinelabelmain{line:prob}
    \hfill{\scriptsize\color{gray}Worker $i$ updates its local model shift}
    \ENDFOR
    \ENDFOR
    \end{algorithmic}
\end{algorithm}

\THEOREMLMAX*
\begin{proof}
    From Assumption~\ref{ass:local_lipschitz_constant} it follows that
    \begin{eqnarray*}
        \norm{\frac{1}{n} \sum_{i=1}^n (\nabla f_i(x+u_i) - \nabla f_i(x))}^2
        &\overset{\eqref{eq:jensen}}{\leq}& \frac{1}{n} \sum_{i=1}^n \norm{\nabla f_i(x+u_i) - \nabla f_i(x)}^2 \\
        &\overset{\textnormal{Ass.} \ref{ass:local_lipschitz_constant}}{\leq}& 
        \frac{1}{n} \sum_{i=1}^n L_i^2 \norm{u_i}^2 \\
        &\leq& 
        L_{\max}^2 \parens{\frac{1}{n} \sum_{i=1}^n \norm{u_i}^2},
    \end{eqnarray*}
    so Assumption~\ref{ass:functional} holds with $L_A = L_{\max}$ and $L_B = 0$.
\end{proof}

\THEOREMLQUAD*
\begin{proof}
    It is easy to verify that
    \begin{eqnarray*}
        \norm{\frac{1}{n} \sum_{i=1}^n (\nabla f_i(x+u_i) - \nabla f_i(x))}^2
        &=& \norm{\frac{1}{n} \sum_{i=1}^n \parens{\mA (x+u_i) + b - \parens{\mA x + b}}}^2 \\
        &=& \norm{\mA \parens{\frac{1}{n} \sum_{i=1}^n u_i}}^2 \\
        &\leq& \norm{\mA}^2 \norm{\frac{1}{n} \sum_{i=1}^n u_i}^2,
    \end{eqnarray*}
    meaning that Assumption \ref{ass:functional} holds with $L_A = 0$ and $L_B = \norm{\mA}$.
\end{proof}

\begin{lemma}\label{lemma:lalblmax}
    Let Assumption \ref{ass:local_lipschitz_constant} hold. Then, there exist constants $L_A, L_B \geq 0$ such that Assumption~\ref{ass:functional} holds and $L_A^2 + L_B^2 \leq L_{\max}^2$.
\end{lemma}
\begin{proof}
    Assumption \ref{ass:local_lipschitz_constant} gives
    \begin{eqnarray*}
        \norm{\frac{1}{n} \sum_{i=1}^n (\nabla f_i(x+u_i) - \nabla f_i(x))}^2
        &\overset{\eqref{eq:jensen}}{\leq}& \frac{1}{n} \sum_{i=1}^n \norm{\nabla f_i(x+u_i) - \nabla f_i(x)}^2 \\
        &\overset{\textnormal{Ass.}\ref{ass:local_lipschitz_constant}}{\leq}& 
        \frac{1}{n} \sum_{i=1}^n L_i^2 \norm{u_i}^2 \\
        &\leq& 
        L_{\max}^2 \parens{\frac{1}{n} \sum_{i=1}^n \norm{u_i}^2},
    \end{eqnarray*}
    and hence Assumption \ref{ass:functional} holds with $L_A^2 = L_{\max}^2$ and $L_B^2 = 0$.
\end{proof}

\begin{remark}
    Under Assumption \ref{ass:functional} we have
    \begin{eqnarray*}
        \norm{\frac{1}{n} \sum_{i=1}^n (\nabla f_i(x+u_i) - \nabla f_i(x))}^2
        &\overset{\textnormal{Ass.}\ref{ass:functional}}{\leq}& L_A^2\left(\frac{1}{n} \sum_{i=1}^n \norm{u_i}^2\right) + L_B^2 \norm{\frac{1}{n} \sum_{i=1}^n u_i}^2 \\
        &\overset{\textnormal{Ass.}\ref{ass:local_lipschitz_constant}}{\leq}& \parens{L_A^2 + L_B^2}\left(\frac{1}{n} \sum_{i=1}^n \norm{u_i}^2\right),
    \end{eqnarray*}
    so, in principle, one could always set $L_B^2=0$. However, the bound could be tightened by decreasing $L_A$ and increasing $L_B$. The smaller $L_A$, the better the performance of our algorithms (see Corollaries~\ref{cor:main_corolary_3comp} and~\ref{cor:m3_total_comm}).
\end{remark}

Now, we proceed to prove the result that relates the values of $L_A$ and $L_B$ to the Hessians of the functions $f_i$.

\THEOREMLHESS*
\begin{proof}
    By the fundamental theorem of calculus,
    \begin{eqnarray*}
        \nabla f_i(x+u_i) - \nabla f_i(x)
        = \int_0^1 \nabla^2 f_i(x+tu_i)u_i dt
        = \parens{\int_0^1 \nabla^2 f_i(x+tu_i) dt} u_i
        = \mQ_iu_i,
    \end{eqnarray*}
    where $\mQ_i = \int_0^1 \nabla^2 f_i(x+tu_i) dt$. Letting $\mQ = \frac{1}{n} \sum_{i=1}^n \mQ_i$, we can write
    \begin{eqnarray*}
        \norm{\frac{1}{n} \sum_{i=1}^n \parens{\nabla f_i(x+u_i) - \nabla f_i(x)}}^2
        &=& \norm{\frac{1}{n} \sum_{i=1}^n \mQ_iu_i}^2 \\
        &=& \norm{\frac{1}{n} \sum_{i=1}^n \parens{\mQ_i - \mQ}u_i + \mQ \parens{\frac{1}{n} \sum_{i=1}^n u_i}}^2 \\
        &\overset{\eqref{eq:young}}{\leq}& 2\norm{\frac{1}{n} \sum_{i=1}^n \parens{\mQ_i - \mQ}u_i}^2
        + 2\norm{\mQ \parens{\frac{1}{n} \sum_{i=1}^n u_i}}^2 \\
        &\overset{\eqref{eq:jensen}}{\leq}& 2\frac{1}{n} \sum_{i=1}^n \norm{\parens{\mQ_i - \mQ}u_i}^2
        + 2\norm{\mQ}^2 \norm{\frac{1}{n} \sum_{i=1}^n u_i}^2 \\
        &\leq& 2\frac{1}{n} \sum_{i=1}^n \norm{\mQ_i - \mQ}^2 \norm{u_i}^2
        + 2\norm{\mQ}^2 \norm{\frac{1}{n} \sum_{i=1}^n u_i}^2.
    \end{eqnarray*}
    Further,
    \begin{eqnarray*}
        \norm{\mQ_i - \mQ}
        &=& \norm{\int_0^1 \nabla^2 f_i(x+tu_i) dt - \frac{1}{n} \sum_{j=1}^n \int_0^1 \nabla^2 f_j(x+tu_j) dt} \\
        &=& \norm{\int_0^1 \nabla^2 f_i(x+tu_i) dt - \int_0^1 \frac{1}{n} \sum_{j=1}^n \nabla^2 f_j(x+tu_j) dt} \\
        &=& \norm{\int_0^1 \frac{1}{n} \sum_{j=1}^n \parens{\nabla^2 f_i(x+tu_i) - \nabla^2 f_j(x+tu_j)} dt} \\
        &\leq& \int_0^1 \norm{\nabla^2 f_i(x+tu_i) - \frac{1}{n} \sum_{j=1}^n \nabla^2 f_j(x+tu_j)} dt \\
        &\leq& \int_0^1 D_i dt
        = D_i,
    \end{eqnarray*}
    and
    \begin{eqnarray*}
        \norm{\mQ}
        &=& \norm{\frac{1}{n} \sum_{j=1}^n \int_0^1 \nabla^2 f_j(x+tu_j) dt}
        = \norm{\int_0^1 \frac{1}{n} \sum_{j=1}^n \nabla^2 f_j(x+tu_j) dt} \\
        &\leq& \int_0^1 \norm{\frac{1}{n} \sum_{j=1}^n \nabla^2 f_j(x+tu_j)} dt
        \overset{\eqref{eq:jensen}}{\leq} \int_0^1 \frac{1}{n} \sum_{j=1}^n \norm{\nabla^2 f_j(x+tu_j)} dt \\
        &\leq& \int_0^1 \frac{1}{n} \sum_{j=1}^n L_j dt
        = \frac{1}{n} \sum_{j=1}^n L_j.
    \end{eqnarray*}
    By combining the above, we get
    \begin{eqnarray*}
        \norm{\frac{1}{n} \sum_{i=1}^n \parens{\nabla f_i(x+u_i) - \nabla f_i(x)}}^2
        &\leq& 2\frac{1}{n} \sum_{i=1}^n D_i^2 \norm{u_i}^2
        + 2 \parens{\frac{1}{n} \sum_{j=1}^n L_j}^2 \norm{\frac{1}{n} \sum_{i=1}^n u_i}^2 \\
        &\leq& 2 \parens{\max_i D_i^2} \norm{u_i}^2
        + 2 \parens{\frac{1}{n} \sum_{j=1}^n L_j}^2 \norm{\frac{1}{n} \sum_{i=1}^n u_i}^2,
    \end{eqnarray*}
    which means that Assumption \ref{ass:functional} holds with $L_A^2 = 2(\max_i D_i^2)$ and $L_B^2 = 2\parens{\frac{1}{n} \sum_{i=1}^n L_i}^2$.
\end{proof}

\begin{remark}
    Clearly, if Assumption \ref{ass:local_lipschitz_constant} holds, i.e., if there exists $L_i \geq 0$ such that $\sup_{z_i\in\R^d} \norm{\nabla^2 f_i(z_i)} \leq L_i$ for all $i\in[n]$, then there exists $D_i$ such that $\sup_{z_1,\ldots,z_n\in\R^d} \norm{\nabla^2 f_i(z_i) - \frac{1}{n} \sum_{j=1}^n \nabla^2 f_j(z_j)} \leq D_i$, which means that this latter condition is not restrictive. Indeed,
    \begin{align*}
        \norm{\nabla^2 f_i(z_i) - \frac{1}{n} \sum_{j=1}^n \nabla^2 f_j(z_j)}
        &\leq \norm{\nabla^2 f_i(z_i)} + \norm{\frac{1}{n} \sum_{j=1}^n \nabla^2 f_j(z_j)} \\
        &\leq \norm{\nabla^2 f_i(z_i)} + \frac{1}{n} \sum_{j=1}^n \norm{\nabla^2 f_j(z_j)} \\
        &\leq L_i + \frac{1}{n} \sum_{j=1}^n L_j.
    \end{align*}
    However, $D_i$ can be small even if the constants $\{L_i\}$ are large, as the next theorem shows.
\end{remark}

\THEOREMLGENQAUD*
\begin{proof}
    In this case $\nabla^2 f_i(z_i) \equiv \mA_i$, and the result easily follows from Theorem~\ref{theorem:a_b_hess}.
\end{proof}

\section{Convergence of \algnamebig{MARINA-P} in the General Case}\label{sec:marinap_convergence_theory}

\subsection{Main Results}

As promised, we now present a result generalizing Theorem~\ref{thm:marinap_gen} to all unbiased compressors.

\begin{restatable}{theorem}{THEOREMMARINAPGENERAL}\label{thm:marinap_gen_a}
    Let Assumptions \ref{ass:lipschitz_constant}, \ref{ass:lower_bound} and \ref{ass:functional} be satisfied and suppose that $\brac{\cC_i^t}_{i=1}^n \in \mathbb{P}(\theta)$ (Def.~\ref{def:corr_compr}) and $\cC_i^t \in \mathbb{U}(\omega_P)$ (Def.~\ref{def:unbiased_compression}) for all $i\in[n]$. Let
    \begin{align*}
        0 < \gamma \leq \frac{1}{L + \sqrt{\parens{L_A^2 \omega_P + L_B^2 \theta} \parens{\frac{1}{p}-1}}}.
    \end{align*}
    Letting
    \begin{align*}
        \Psi^t = f(x^t) - f^* + \frac{\gamma L_A^2}{2p} \frac{1}{n} \sum_{i=1}^n \norm{w_i^t - x^t}^2
        + \frac{\gamma L_B^2}{2p} \norm{w^t - x^t}^2,
    \end{align*}
    for each $T \geq 1$ we have
    \begin{align*}
        \sum_{t=0}^{T-1} \frac{1}{T} \Exp{\norm{\nabla f (x^t)}^2} \leq \frac{2 \Psi^0}{\gamma T}.
    \end{align*}
\end{restatable}
Let us provide some important examples:

\begin{restatable}{theorem}{THEOREMMARINAPGENERALEXAMPLES}
    \label{thm:marinap_gen_3comp}
    Let Assumptions~\ref{ass:lipschitz_constant}, \ref{ass:lower_bound} and \ref{ass:functional} be satisfied. Choose
    \begin{align*}
        \gamma =
        \begin{cases}
            \parens{L + \sqrt{\parens{L_A^2 + L_B^2}\nicefrac{\omega_P}{p}}}^{-1} & \text{for \algname{SameRandK} compressors} \\
            \parens{L + \sqrt{\parens{L_A^2 + \nicefrac{L_B^2}{n}} \nicefrac{\omega_P}{p}}}^{-1} & \text{for independent \algname{RandK} compressors} \\
            \parens{L + L_A \sqrt{\nicefrac{\omega_P}{p}}}^{-1} & \text{for \algname{PermK} compressors}
        \end{cases}
    \end{align*}
    and set $w_i^0 = x^0$ for all $i \in [n]$. Then \algname{MARINA-P} finds an $\varepsilon$--stationary point after
    \begin{align*}
        \bar{T} =
        \begin{cases}
            \cO\left(\frac{\delta^0 \left(L + \sqrt{\parens{L_A^2 + L_B^2}\nicefrac{\omega_P}{p}}\right)}{\varepsilon}\right) & \text{for \algname{SameRandK} compressors} \\
            \cO\left(\frac{\delta^0 \left(L + \sqrt{\parens{L_A^2 + \nicefrac{L_B^2}{n}}\nicefrac{\omega_P}{p}}\right)}{\varepsilon}\right) & \text{for independent  \algname{RandK} compressors} \\
            \cO\left(\frac{\delta^0 \left(L + L_A \sqrt{\nicefrac{\omega_P}{p}}\right)}{\varepsilon}\right) & \text{for \algname{PermK} compressors}
        \end{cases}
    \end{align*}
    iterations.
\end{restatable}

\begin{remark}
    \leavevmode
    \begin{itemize}
        \item The result for Perm$K$ compressors proves Theorem \ref{thm:marinap_gen}.
        \item The above theorem demonstrates the complexities for a) SameRand$K$, b) independent Rand$K$ and c) Perm$K$ compressors. However, the result applies to any families of compressors such that for all $t\geq 0$ we have a) $\cC_1^t=\ldots=\cC_n^t=\cC^t \in\mathbb{U}(\omega_P)$, b) $\cC_1^t,\ldots,\cC_n^t \in\mathbb{U}(\omega_P)$ are independent, and c) $\cC_1^t,\ldots,\cC_n^t \in\mathbb{U}(\omega_P)\cap\mathbb{P}(\theta)$, respectively.
    \end{itemize}
\end{remark}

We now derive the communication complexities:

\begin{restatable}{corollary}{THEOREMMARINAPCOR}
    \label{cor:main_corolary_3comp}
    Let us take $p = \nicefrac{1}{n}$ and set $K=\nicefrac{d}{n}$ (corresponding to the sparsification level of a Perm$K$ compressor). Then, in the view of Theorem~\ref{thm:marinap_gen_3comp}, the average s2w communication complexity of \algname{MARINA-P} is
    \begin{align}\label{eq:comm_complexity_3comp}
        \begin{cases}
            \cO\parens{\frac{d \delta^0 L}{n \varepsilon} + \frac{d \delta^0}{\varepsilon} \sqrt{L_A^2 + L_B^2}} & \text{for \algname{SameRandK} compressors} \\
            \cO\parens{\frac{d \delta^0 L}{n \varepsilon} + \frac{d \delta^0}{\varepsilon} \sqrt{L_A^2 + \frac{L_B^2}{n}}} & \text{for independent \algname{RandK} compressors} \\
            \cO\left(\frac{d \delta^0 L}{n \varepsilon} + \frac{d \delta^0}{\varepsilon} L_A\right) & \text{for \algname{PermK} compressors}
        \end{cases}
    \end{align}
\end{restatable}

\begin{remark}
    \leavevmode
    \begin{itemize}
        \item The result for Perm$K$ compressors proves Corollary \ref{cor:main_corolary}.
        \item The key observation from \eqref{eq:comm_complexity_3comp} is the dependence on $L_A$ and $L_B$. 
        In particular, if $L_A \approx 0$ (which is the case, e.g., for homogeneous quadratics), the above communication complexities are
        \begin{align*}
            \begin{cases}
                \cO\parens{\frac{\delta^0}{\varepsilon} d \parens{\frac{L}{n} + L_B}} & \text{for \algname{SameRandK} compressors}, \\
                \cO\parens{\frac{\delta^0}{\varepsilon} d \parens{\frac{L}{n} + \frac{L_B}{\sqrt{n}}}} & \text{for independent \algname{RandK} compressors}, \\
                \cO\left(\frac{\delta^0}{\varepsilon} d \frac{L}{n}\right) & \text{for \algname{PermK} compressors}.
            \end{cases}
        \end{align*}
        Hence, only by sending different messages to different clients, one obtains complexities improving with $n$. In particular, for Perm$K$, the complexity scales linearly with the number of workers.
    \end{itemize}
\end{remark}

\subsection{Proofs}

To prove the results from the previous section, we first establish several identities and inequalities satisfied by the sequences $\{w_1^t,\ldots,w_n^t\}_{t\geq0}$.
We start by studying the evolution of the quantity $\norm{w_i^{t} - x^{t}}^2$. In what follows, $\ExpSub{t}{\cdot}$ denotes the expectation conditioned on the first $t$ iterations.

\begin{lemma}\label{lemma:marinap_sum_w_x}
    Let $\cC_i^t \in \mathbb{U}(\omega_P)$ for all $i\in[n]$. Then
    \begin{align*}
        \frac{1}{n} \sum_{i=1}^n \Exp{\norm{w_i^{t+1} - x^{t+1}}^2}
        \leq (1-p) \frac{1}{n} \sum_{i=1}^n \Exp{\norm{w_i^{t} - x^{t}}^2}
        + (1-p) \omega_P \Exp{\norm{x^{t+1} - x^{t}}^2}.
    \end{align*}
\end{lemma}
\begin{proof}
    In the view of definition of $w^{t+1}$, we get
    \begin{eqnarray*}
        &&\hspace{-1cm}\ExpSub{t}{\norm{w_i^{t+1} - x^{t+1}}^2} \\
        &=& (1-p) \ExpSub{t}{\norm{w_i^t + \cC_i^t(x^{t+1} - x^t) - x^{t+1}}^2} \\
        &\overset{\eqref{eq:vardecomp}}{=}& (1-p) \ExpSub{t}{\norm{\cC_i^t(x^{t+1} - x^t) - (x^{t+1} - x^t)}^2}
        + (1-p) \ExpSub{t}{\norm{w_i^t - x^t}^2} \\
        &\overset{\textnormal{Def.}\ref{def:unbiased_compression}}{\leq}& (1-p) \omega_P \norm{x^{t+1} - x^t}^2
        + (1-p) \norm{w_i^t - x^t}^2.
    \end{eqnarray*}
    Averaging, taking expectation and using the tower property, we get the result.
\end{proof}

This lemma is less powerful: it is \textit{not} an identity, and hence some information is lost. Moreover, it focuses on a single client~$i$, and is therefore not able to take advantage of the correlation among the compressors. On the other hand, it can be used in the convergence analysis without any need to restrict the function class.

Next, we study the evolution of the quantity $\norm{w^{t} - x^{t}}^2$.
\begin{lemma}\label{lemma:w_sum_x}
    Let $\brac{\cC_i^t}_{i=1}^n \in \mathbb{P}(\theta)$. Then
    \begin{align*}
        \Exp{\norm{w^{t+1} - x^{t+1}}^2}
        &\leq (1-p) \Exp{\norm{w^t - x^t}^2}
        + (1-p) \theta \Exp{\norm{x^{t+1} - x^t}^2}.
    \end{align*}
\end{lemma}
\begin{proof}
    In the view of definition of $w^{t+1}$, we get
    \begin{eqnarray*}
        &&\hspace{-1cm}\ExpSub{t}{\norm{w^{t+1} - x^{t+1}}^2} \\
        &=& \ExpSub{t}{\norm{\frac{1}{n} \sum_{i=1}^n w_i^{t+1} - x^{t+1}}^2} \\
        &=& (1-p) \ExpSub{t}{\norm{\frac{1}{n} \sum_{i=1}^n (w_i^t + \cC_i^t(x^{t+1} - x^t)) - x^{t+1}}^2} \\
        &=& (1-p) \ExpSub{t}{\norm{\frac{1}{n} \sum_{i=1}^n \cC_i^t(x^{t+1} - x^t) - (x^{t+1} - x^t) - x^t + w^t}^2} \\
        &\overset{\eqref{eq:vardecomp}}{=}& (1-p) \ExpSub{t}{\norm{\frac{1}{n} \sum_{i=1}^n \cC_i^t(x^{t+1} - x^t) - (x^{t+1} - x^t)}^2}
        + (1-p) \ExpSub{t}{\norm{w^t - x^t}^2} \\
        &\overset{\textnormal{Def.}\ref{def:corr_compr}}{\leq}& (1-p) \theta \norm{x^{t+1} - x^t}^2
        + (1-p) \norm{w^t - x^t}^2.
    \end{eqnarray*}
    Taking expectation and using the tower property, we get the result.
\end{proof}

This lemma is more powerful since it is able to take advantage of the correlation among the compressors.
Indeed, if $\theta = 0$ (as in the case of Perm$K$ compressors), then it becomes an identity:
\begin{align*}
    \Exp{\norm{w^{t+1} - x^{t+1}}^2} = (1-p) \Exp{\norm{w^t - x^t}^2}.
\end{align*}

We now prove convergence of \algname{MARINA-P} in the general case.
\THEOREMMARINAPGENERAL* 
\begin{proof}
    First, combining the inequalities in Lemmas \ref{lemma:marinap_sum_w_x} and \ref{lemma:w_sum_x}, we get
    \begin{align}\label{eq:l_marinap2}
        &\frac{\gamma L_A^2}{2p} \frac{1}{n} \sum_{i=1}^n \Exp{\norm{w_i^{t+1} - x^{t+1}}^2}
        + \frac{\gamma L_B^2}{2p} \Exp{\norm{w^{t+1} - x^{t+1}}^2} \nonumber \\
        &\leq \frac{\gamma L_A^2}{2p} (1-p) \frac{1}{n} \sum_{i=1}^n \Exp{\norm{w_i^{t} - x^{t}}^2}
        + \frac{\gamma L_A^2}{2p} (1-p) \omega_P \Exp{\norm{x^{t+1} - x^{t}}^2} \nonumber \\
        &\quad+ \frac{\gamma L_B^2}{2p} (1-p) \Exp{\norm{w^t - x^t}^2}
        + \frac{\gamma L_B^2}{2p} (1-p) \theta \Exp{\norm{x^{t+1} - x^t}^2} \nonumber \\
        &= \frac{\gamma L_A^2}{2p} (1-p) \frac{1}{n} \sum_{i=1}^n \Exp{\norm{w_i^{t} - x^{t}}^2}
        + \frac{\gamma L_B^2}{2p} (1-p) \Exp{\norm{w^t - x^t}^2} \nonumber \\
        &\quad+ \frac{\gamma}{2p} \parens{L_A^2 \omega_P + L_B^2 \theta} (1-p) \Exp{\norm{x^{t+1} - x^t}^2}.
    \end{align}
    Next, using Assumption \ref{ass:functional}, we have
    \begin{eqnarray*}
        \Exp{\norm{g^t - \nabla f(x^t)}^2}
        &=& \Exp{\norm{\frac{1}{n} \sum_{i=1}^n \parens{\nabla f_i(w_i^t) - \nabla f_i(x^t)}}^2} \\
        &\overset{\textnormal{Ass.}\ref{ass:functional}}{\leq}& L_A^2 \frac{1}{n} \sum_{i=1}^n \Exp{\norm{w_i^t-x^t}^2} + L_B^2 \Exp{\norm{\frac{1}{n} \sum_{i=1}^n w_i^t - x^t}^2}.
    \end{eqnarray*}
    Combining the above inequality with Lemma \ref{lemma:page} gives
    \begin{eqnarray}\label{eq:delta_marinap2}
        \Exp{\delta^{t+1}} &\leq& \Exp{\delta^{t}} - \frac{\gamma}{2} \Exp{\norm{\nabla f(x^t)}^2} - \left( \frac{1}{2\gamma} - \frac{L}{2} \right) \Exp{\norm{x^{t+1} - x^t}^2} + \frac{\gamma}{2} \Exp{\norm{g^t - \nabla f(x^t)}^2} \nonumber\\
        &\leq& \Exp{\delta^{t}} - \frac{\gamma}{2} \Exp{\norm{\nabla f(x^t)}^2} - \left( \frac{1}{2\gamma} - \frac{L}{2} \right) \Exp{\norm{x^{t+1} - x^t}^2} \nonumber\\
        &&+ \frac{\gamma}{2} \parens{L_A^2 \frac{1}{n} \sum_{i=1}^n \Exp{\norm{w_i^t-x^t}^2} + L_B^2 \Exp{\norm{w^t - x^t}^2}}.
    \end{eqnarray}
    By adding inequalities \eqref{eq:l_marinap2} and \eqref{eq:delta_marinap2}, we get
    \begin{eqnarray*}
        \Exp{\Psi^{t+1}} &=& \Exp{\delta^{t+1}}
        + \frac{\gamma L_A^2}{2p} \frac{1}{n} \sum_{i=1}^n \Exp{\norm{w_i^{t+1} - x^{t+1}}^2}
        + \frac{\gamma L_B^2}{2p} \Exp{\norm{w^{t+1} - x^{t+1}}^2} \\
        &\leq& \Exp{\delta^{t}} - \frac{\gamma}{2} \Exp{\norm{\nabla f(x^t)}^2} - \left( \frac{1}{2\gamma} - \frac{L}{2} \right) \Exp{\norm{x^{t+1} - x^t}^2}  \\
        &&+ \frac{\gamma}{2} \parens{L_A^2 \frac{1}{n} \sum_{i=1}^n \Exp{\norm{w_i^t-x^t}^2} + L_B^2 \Exp{\norm{w^t - x^t}^2}} \\
        &&+ \frac{\gamma L_A^2}{2p} (1-p) \frac{1}{n} \sum_{i=1}^n \Exp{\norm{w_i^{t} - x^{t}}^2}
        + \frac{\gamma L_B^2}{2p} (1-p) \Exp{\norm{w^t - x^t}^2} \\
        &&+ \frac{\gamma}{2p} \parens{L_A^2 \omega_P + L_B^2 \theta} (1-p) \Exp{\norm{x^{t+1} - x^t}^2} \\
        &=& \Exp{\Psi^t} - \frac{\gamma}{2} \Exp{\norm{\nabla f(x^t)}^2} \\
        &&- \left( \frac{1}{2\gamma} - \frac{L}{2} - \frac{\gamma}{2p} \parens{L_A^2 \omega_P + L_B^2 \theta} (1-p) \right) \Exp{\norm{x^{t+1} - x^t}^2} \\
        &\leq& \Exp{\Psi^t} - \frac{\gamma}{2} \Exp{\norm{\nabla f(x^t)}^2},
    \end{eqnarray*}
    where in the last line we use the assumption on the step size and Lemma \ref{lemma:step_lemma}. Summing up the above inequality for $t = 0, 1, \ldots, T-1$ and rearranging the terms, we get
    \begin{align*}
        \frac{1}{T} \sum_{t=0}^{T-1} \Exp{\norm{\nabla f (x^t)}^2}
        \leq \frac{2}{\gamma T} \parens{\Exp{\Psi^{0}} - \Exp{\Psi^{T}}}
        \leq \frac{2 \Psi^0}{\gamma T}.
    \end{align*}
\end{proof}

With the above result, we can establish the iteration and communication complexities of \algname{MARINA-P} for three different compression schemes described in Appendix \ref{sec:3compr}. First, let us prove a result when independent compressors are used.

\begin{theorem}\label{thm:marinap_gen_indep_compressors}
    Let Assumptions~\ref{ass:lipschitz_constant}, \ref{ass:lower_bound} and \ref{ass:functional} be satisfied and suppose that $\cC_1^t,\ldots,\cC_n^t$ is a collection of independent compressors (Assumption~\ref{ass:independent}) such that $\cC_i^t \in \mathbb{U}(\omega)$ for all $i\in[n]$, $t\in\mathbb{N}$. Choose
    \begin{align*}
        \gamma = \parens{L + \sqrt{\parens{L_A^2 + \nicefrac{L_B^2}{n}} \nicefrac{\omega_P}{p}}}^{-1}
    \end{align*}
    and set $w_i^0 = x^0$ for all $i \in [n]$. Then \algname{MARINA-P} finds an $\varepsilon$--stationary point after
    \begin{align*}
        \bar{T} = \cO\left(\frac{\delta^0 \left(L + \sqrt{\parens{L_A^2 + \nicefrac{L_B^2}{n}}\nicefrac{\omega_P}{p}}\right)}{\varepsilon}\right)
    \end{align*}
    iterations.
\end{theorem}
\begin{proof}
    In view of Theorem \ref{thm:marinap_gen_a}, the step size satisfies the inequality
    \begin{align*}
        \gamma \leq \frac{1}{L + \sqrt{\parens{L_A^2 \omega_P + L_B^2 \theta} \parens{\frac{1}{p}-1}}}.
    \end{align*}
    Since by Lemma \ref{lemma:theta_indep_same}, when the compressors are independent we have $\theta = \nicefrac{\omega_P}{n}$, the algorithm converges in
    \begin{align}\label{eq:it_comp_marinap_indep}
        \bar{T} &= \frac{\Psi^0}{\varepsilon} \parens{L + \sqrt{\parens{L_A^2 \omega_P + L_B^2 \frac{\omega_P}{n}} \parens{\frac{1}{p}-1}}} \nonumber\\
        &= \cO\parens{\frac{\Psi^0}{\varepsilon} \parens{L + \sqrt{\frac{\omega_P}{p} \parens{L_A^2 + \frac{L_B^2}{n}}}}}
    \end{align}
    iterations.
\end{proof}

\THEOREMMARINAPGENERALEXAMPLES*

\begin{proof}
    In view of Theorem \ref{thm:marinap_gen_a}, the step size is such that
    \begin{align*}
        \gamma \leq \frac{1}{L + \sqrt{\parens{L_A^2 \omega_P + L_B^2 \theta} \parens{\frac{1}{p}-1}}}.
    \end{align*}
    We now use Lemma \ref{lemma:3compr_omega_theta} and substitute the vales of $\theta$ specific to each compression type.
    
    For SameRand$K$, we have $\theta = \omega_P$, so the algorithm converges after
    \begin{align}\label{eq:it_comp_marinap_samerand}
        \bar{T} = \frac{\Psi^0}{\varepsilon} \parens{L + \sqrt{\parens{L_A^2 \omega_P + L_B^2 \omega_P} \parens{\frac{1}{p}-1}}}
        = \cO\parens{\frac{\Psi^0}{\varepsilon} \parens{L + \sqrt{\frac{\omega_P}{p} \parens{L_A^2 + L_B^2}}}}
    \end{align}
    iterations.
    Following the same reasoning as in the proof of Theorem \ref{thm:marinap_gen_indep_compressors}, for Rand$K$ we have
    \begin{align}\label{eq:it_comp_marinap_rand}
        \bar{T} = \frac{\Psi^0}{\varepsilon} \parens{L + \sqrt{\parens{L_A^2 \omega_P + L_B^2 \frac{\omega_P}{n}} \parens{\frac{1}{p}-1}}}
        = \cO\parens{\frac{\Psi^0}{\varepsilon} \parens{L + \sqrt{\frac{\omega_P}{p} \parens{L_A^2 + \frac{L_B^2}{n}}}}}.
    \end{align}
    Finally, for Perm$K$ we have $\theta = 0$, so
    \begin{align}\label{eq:it_comp_marinap_perm}
        \bar{T} = \frac{\Psi^0}{\varepsilon} \parens{L + \sqrt{L_A^2 \omega_P \parens{\frac{1}{p}-1}}}
        = \cO\parens{\frac{\Psi^0}{\varepsilon} \parens{L + L_A \sqrt{\frac{\omega_P}{p}}}}.
    \end{align}
    
    The result follows from the fact that $w_i^0=x^0$ for all $i\in[n]$.
\end{proof}

\THEOREMMARINAPCOR*
\begin{proof}
    The expected number of floats a server is relaying to each client at each iteration of \algname{MARINA-P} is
    \begin{align*}
        pd + (1 - p)k
        = \frac{d}{n} + \frac{n-1}{n} k
        \leq \frac{2d}{n}.
    \end{align*}
    Next, using the results from Lemma \ref{lemma:3compr_omega_theta}, our choice of compressors and parameters gives $\omega_P=\nicefrac{d}{K}-1=n-1$ in each of the three cases. Hence, substituting $p = \nicefrac{1}{n}$ in \eqref{eq:it_comp_marinap_samerand}, \eqref{eq:it_comp_marinap_rand} and \eqref{eq:it_comp_marinap_perm}, we obtain the following server-to-worker communication complexities:
    \begin{enumerate}
        \item for SameRand$K$ compressors:
        \begin{align*}
            \frac{d}{n} \times \frac{\delta^0}{\varepsilon} \parens{L + \sqrt{\omega_P \parens{L_A^2 + L_B^2} \parens{\frac{1}{p}-1}}}
            &= \frac{\delta^0}{\varepsilon} \parens{\frac{d}{n}L + \frac{d}{n} \sqrt{\parens{L_A^2 + L_B^2} (n-1)^2}} \\
            &= \cO\parens{\frac{\delta^0}{\varepsilon} \parens{\frac{d}{n}L + d \sqrt{L_A^2 + L_B^2}}},
        \end{align*}
        \item for Rand$K$ compressors:
        \begin{align*}
            \frac{d}{n} \times \frac{\delta^0}{\varepsilon} \parens{L + \sqrt{\omega_P \parens{L_A^2 + \frac{L_B^2}{n}} \parens{\frac{1}{p}-1}}}
            &= \frac{\delta^0}{\varepsilon} \parens{\frac{d}{n}L + \frac{d}{n} \sqrt{\parens{L_A^2 + \frac{L_B^2}{n}}(n-1)^2}} \\
            &= \cO\parens{\frac{\delta^0}{\varepsilon} \parens{\frac{d}{n}L + d \sqrt{L_A^2 + \frac{L_B^2}{n}}}},
        \end{align*}
        \item for Perm$K$ compressors:
        \begin{align*}
            \frac{d}{n} \times \frac{\delta^0}{\varepsilon} \parens{L + \sqrt{L_A^2 \omega_P \parens{\frac{1}{p}-1}}}
            &= \frac{\delta^0}{\varepsilon} \parens{\frac{d}{n}L + \frac{d}{n}L_A \sqrt{(n-1)^2}} \\
            &= \cO\parens{\frac{\delta^0}{\varepsilon} \parens{\frac{d}{n}L + d L_A}}.
        \end{align*}
    \end{enumerate}  
\end{proof}

\subsection{Polyak-\L ojasiewicz condition}
\label{sec:pl}

\subsubsection{Main Results}

To complete the theory, we now establish a convergence result for \algname{MARINA-P} under the Polyak-Łojasiewicz assumption.

\begin{assumption}[Polyak-Łojasiewicz condition]\label{ass:pl}
    The function $f$ satisfies Polyak-Łojasiewicz (PŁ) condition with parameter~$\mu$, i.e., for all $x \in \R^d$ there exists $x^*\in\arg\min_{x\in\R^d} f(x)$ such that
    \begin{align}
        2\mu\parens{f(x) - f(x^*)} \leq \norm{\nabla f(x)}^2.
    \end{align}
\end{assumption}

\begin{restatable}{theorem}{THMPLMARINA}\label{thm:marinap_pl_gen}
    Let Assumptions \ref{ass:lipschitz_constant}, \ref{ass:lower_bound}, \ref{ass:functional} and \ref{ass:pl} be satisfied and suppose that $\brac{\cC_i^t}_{i=1}^n \in \mathbb{P}(\theta)$ and $\cC_i^t \in \mathbb{U}(\omega_P)$ for all $i\in[n]$. Take
    \begin{align}\label{eq:marinap_pl_gen_step}
        0< \gamma \leq \min\left\{ \frac{1}{L + \sqrt{2 \parens{L_A^2 \omega_P + L_B^2 \theta} \parens{\frac{1}{p}-1}}}, \frac{p}{2\mu} \right\}.
    \end{align}
    Letting
    \begin{align}
        \Psi^t = f(x^t) - f^* + \frac{\gamma L_A^2}{p} \frac{1}{n} \sum_{i=1}^n \Exp{\norm{w_i^t - x^t}^2}
        + \frac{\gamma L_B^2}{p} \Exp{\norm{\frac{1}{n} \sum_{i=1}^n w_i^t - x^t}^2},
    \end{align}
    for each $T \geq 1$ we have
    \begin{eqnarray*}
        \Exp{\Psi^{T}} \leq \parens{1 - \gamma\mu}^T \Psi^{0}.
    \end{eqnarray*}
\end{restatable}

\begin{restatable}{corollary}{THMCORMARINAPL}\label{thm:marinap_pl_perm}
    Let $\cC_i^t \in \mathbb{P}(0)$ for all $i\in[n]$ (e.g. Perm$K$), choose $p = 1 / (\omega_P + 1).$ Then, in the view of Theorem~\ref{thm:marinap_pl_gen}, Algorithm~\ref{alg:marinap} ensures that $\Exp{f(x^T) - f^*} \leq \varepsilon$ after
    \begin{align*}
        \cO\left(\max \brac{\frac{L + L_A \omega_P}{\mu}, \omega_P + 1} \log\frac{\Psi^0}{\varepsilon}\right)
    \end{align*}
    iterations.
\end{restatable}

\begin{restatable}{corollary}{CORMARINAPL}\label{cor:main_pl_corolary}
    Let $\cC_i^t$ be the Perm$K$ compressors ($K = \nicefrac{d}{n}$). Then, in the view of Corollary~\ref{thm:marinap_pl_perm}, the s2w communication complexity of \algname{MARINA-P} with Perm$K$ is
    \begin{align*}
        \cO\parens{\left(\frac{d L}{n \mu} + \frac{d L_A}{\mu} + d\right) \log\frac{\Psi^0}{\varepsilon}}.
    \end{align*}
\end{restatable}

\subsubsection{Proofs}

\THMPLMARINA*

\begin{proof}
    We proceed similarly as in the proof of Theorem \ref{thm:marinap_gen_a}. Combining the inequalities in Lemmas~\ref{lemma:marinap_sum_w_x} and~\ref{lemma:w_sum_x} gives
    \begin{align}\label{eq:l_marinap2_pl}
        &\frac{\gamma L_A^2}{p} \frac{1}{n} \sum_{i=1}^n \Exp{\norm{w_i^{t+1} - x^{t+1}}^2}
        + \frac{\gamma L_B^2}{p} \Exp{\norm{w^{t+1} - x^{t+1}}^2} \nonumber \\
        &\leq \frac{\gamma L_A^2}{p} (1-p) \frac{1}{n} \sum_{i=1}^n \Exp{\norm{w_i^{t} - x^{t}}^2}
        + \frac{\gamma L_A^2}{p} (1-p) \omega_P \Exp{\norm{x^{t+1} - x^{t}}^2} \nonumber \\
        &\quad+ \frac{\gamma L_B^2}{p} (1-p) \Exp{\norm{w^t - x^t}^2}
        + \frac{\gamma L_B^2}{p} (1-p) \theta \Exp{\norm{x^{t+1} - x^t}^2} \nonumber \\
        &= \frac{\gamma L_A^2}{p} (1-p) \frac{1}{n} \sum_{i=1}^n \Exp{\norm{w_i^{t} - x^{t}}^2}
        + \frac{\gamma L_B^2}{p} (1-p) \Exp{\norm{w^t - x^t}^2} \nonumber \\
        &\quad+ \frac{\gamma}{p} \parens{L_A^2 \omega_P + L_B^2 \theta} (1-p) \Exp{\norm{x^{t+1} - x^t}^2}.
    \end{align}
    By adding inequalities \eqref{eq:delta_marinap2} and \eqref{eq:l_marinap2_pl}, we get
    \begin{eqnarray*}
        \Exp{\Psi^{t+1}} &=& \Exp{\delta^{t+1}}
        + \frac{\gamma L_A^2}{p} \frac{1}{n} \sum_{i=1}^n \Exp{\norm{w_i^{t+1} - x^{t+1}}^2}
        + \frac{\gamma L_B^2}{p} \Exp{\norm{w^{t+1} - x^{t+1}}^2} \\
        &\leq& \Exp{\delta^{t}} - \frac{\gamma}{2} \Exp{\norm{\nabla f(x^t)}^2} - \left( \frac{1}{2\gamma} - \frac{L}{2} \right) \Exp{\norm{x^{t+1} - x^t}^2}  \\
        &&+ \frac{\gamma}{2} \parens{L_A^2 \frac{1}{n} \sum_{i=1}^n \Exp{\norm{w_i^t-x^t}^2} + L_B^2 \Exp{\norm{w^t - x^t}^2}} \\
        &&+ \frac{\gamma L_A^2}{p} (1-p) \frac{1}{n} \sum_{i=1}^n \Exp{\norm{w_i^{t} - x^{t}}^2}
        + \frac{\gamma L_B^2}{p} (1-p) \Exp{\norm{w^t - x^t}^2} \\
        &&+ \frac{\gamma}{p} \parens{L_A^2 \omega_P + L_B^2 \theta} (1-p) \Exp{\norm{x^{t+1} - x^t}^2} \\
        &=& \Exp{\delta^{t}} - \frac{\gamma}{2} \Exp{\norm{\nabla f(x^t)}^2} \\
        &&- \parens{\frac{1}{2\gamma} - \frac{L}{2} - \frac{\gamma}{p} \parens{L_A^2 \omega_P + L_B^2 \theta} (1-p)} \Exp{\norm{x^{t+1} - x^t}^2} \\
        &&+ \gamma L_A^2 \parens{\frac{1}{p} - \frac{1}{2}} \frac{1}{n} \sum_{i=1}^n \Exp{\norm{w_i^t-x^t}^2}
        + \gamma L_B^2 \parens{\frac{1}{p} - \frac{1}{2}} \Exp{\norm{w^t-x^t}^2} \\
        &\overset{\textnormal{Ass.}\ref{ass:pl}, \eqref{eq:marinap_pl_gen_step}}{\leq}& \parens{1-\gamma\mu} \Exp{\delta^{t}}
        + \frac{\gamma L_A^2}{p} \parens{1-\gamma\mu} \frac{1}{n} \sum_{i=1}^n \Exp{\norm{w_i^t-x^t}^2} \\
        &&+ \frac{\gamma L_B^2}{p} \parens{1-\gamma\mu} \Exp{\norm{w^t-x^t}^2} \\
        &=& \parens{1-\gamma\mu} \Exp{\Psi^t},
    \end{eqnarray*}    
    where the last inequality follows from the Polyak-Łojasiewicz condition, Lemma \ref{lemma:step_lemma} and our choice of $\gamma$. Applying the above inequality iteratively, we finish the proof.
\end{proof}

\THMCORMARINAPL*

\begin{proof}
    In view of Theorem \ref{thm:marinap_pl_gen}, the step size satisfies
    \begin{align*}
        \gamma \leq \min\left\{ \frac{1}{L + \sqrt{2 \parens{L_A^2 \omega_P + L_B^2 \theta} \parens{\frac{1}{p}-1}}}, \frac{p}{2\mu} \right\}.
    \end{align*}
    Therefore, since $\theta = 0$ and $p = 1 / (\omega_P + 1),$ the algorithm converges after
    \begin{align*}
        \bar{T}
        &= \max \brac{\frac{L + \sqrt{2 \parens{L_A^2 \omega_P + L_B^2 \theta} \parens{\frac{1}{p}-1}}}{\mu}, \frac{2}{p}} \log\frac{\Psi^0}{\varepsilon} \\
        &= \cO\parens{\max \brac{\frac{L + L_A \omega_P}{\mu}, \omega_P + 1} \log\frac{\Psi^0}{\varepsilon}}
    \end{align*}
    iterations.
\end{proof}

\CORMARINAPL*

\begin{proof}
    For Perm$K$, $\omega_P = n-1.$ Therefore, the iteration complexity is
    \begin{align*}
        \cO\parens{\max \brac{\frac{L + L_A n}{\mu}, n} \log\frac{\Psi^0}{\varepsilon}}.
    \end{align*}
    Since the expected number of floats the server is relaying to each client is
    \begin{align*}
        pd + (1 - p)k
        = \frac{d}{n} + \frac{n-1}{n} k
        \leq \frac{2d}{n},
    \end{align*}
    the server-to-worker communication complexity is
    \begin{align*}
        \cO\parens{\max \brac{\frac{\frac{d}{n}L + d L_A}{\mu}, d} \log\frac{\Psi^0}{\varepsilon}}.
    \end{align*}
\end{proof}

\newpage

\section{Convergence of \algnamebig{\newmethod} in the General Case} \label{sec:m3_convergence_theory}

We now move on to the bidirectionally compressed method. Below is a generalization of Theorem~\ref{thm:marinap_gen_m} to all unbiased compressors.

\subsection{Main Results}
\begin{restatable}{theorem}{THEOREMMTHREE}
    \label{theorem:mthree}
    Let Assumptions \ref{ass:lipschitz_constant}, \ref{ass:lower_bound}, \ref{ass:local_lipschitz_constant} and \ref{ass:functional} hold and suppose that the compressors $\cQ_i^t \in \mathbb{U}(\omega_D)$ satisfy Assumption~\ref{ass:independent}, $\brac{\cC_i^t}_{i=1}^n \in \mathbb{P}(\theta)$ and $\cC_i^t \in \mathbb{U}(\omega_P)$ for all $i\in[n]$. Let $\gamma>0$ be such that
    \begin{align*}
        \textstyle \gamma \leq \left(L + \sqrt{288 \parens{\left(\frac{\theta}{p_P} + \frac{1 + \theta p_P}{\beta^2}\right) L_B^2 + \left(\frac{\omega_P}{p_P} + \frac{1 + \omega_P p_P}{\beta^2}\right) L_A^2 + \left(\frac{\omega_D \omega_P \beta}{n p_D} + \frac{\omega_D(1 + \omega_P p_P)}{n p_D}\right) L_{\max}^2}}\right)^{-1}.
    \end{align*}
    Letting
    \begin{align*}
        \Psi^t &= \delta^{t}
        + \kappa \norm{g^{t} - \frac{1}{n} \sum_{i=1}^{n} \nabla f_i(z_i^{t})}^2
        + \eta \norm{z^{t} - w^{t}}^2
        + \nu \frac{1}{n} \sum_{i=1}^n \norm{z^{t}_i - w^{t}_i}^2 \\
        &\quad+ \rho \norm{w^{t} - x^{t}}^2
        + \mu \frac{1}{n} \sum_{i=1}^n \norm{w_i^{t} - x^{t}}^2,
    \end{align*}
    where $\kappa = \frac{\gamma}{p_D}$, $\eta = \frac{4\gamma L_B^2}{\beta}$, $\nu = \frac{4\gamma L_A^2}{\beta} + \frac{6 \gamma \omega_D \beta L_{\max}^2}{n p_D}$, $\rho = 32\gamma L_B^2 \left(\frac{1}{p_P} + \frac{p_P}{\beta^2}\right)$ and $\mu = 32 \gamma L_A^2\left(\frac{1}{p_P} + \frac{p_P}{\beta^2}\right) + \frac{48 \gamma \omega_D L_{\max}^2}{n p_D}\left(\beta + p_P\right)$, \algname{M3} ensures that
    \begin{align*}
        \frac{1}{T} \sum_{t=0}^{T-1} \Exp{\norm{\nabla f (x^t)}^2} = \cO\left(\frac{\Psi^0}{\gamma T}\right).
    \end{align*}
\end{restatable}

We now simplify the above result by considering $\theta = 0.$ 

\begin{restatable}{corollary}{CORMTHREE}
    \label{cor:mthree_cor}
    Let $\cC_i^t \in \mathbb{P}(0)$ for all $i\in[n]$ (e.g. Perm$K$), choose $p_P = 1 / (\omega_P + 1),$ $p_D = 1 / (\omega_D + 1)$ and 
    \begin{align*}
        \beta = \min\left\{\left(\frac{n}{\omega_D \omega_P (\omega_D + 1)}\right)^{1/3}, 1\right\}.
    \end{align*}
    Then, in the view of Theorem~\ref{theorem:mthree}, the iteration complexity is
    \begin{align*}
        \cO\left(\frac{\Psi^0}{\varepsilon} \left(L_{\max} + \left(\frac{\omega_D \omega_P (\omega_D + 1)}{n}\right)^{1/3} L_{\max} + \sqrt{\frac{\omega_D (\omega_D + 1)}{n}} L_{\max} + \sqrt{\omega_P (\omega_P + 1)} L_A\right)\right).
    \end{align*}
\end{restatable}

We now give the bound for the total communication complexity of \algname{M3}.

\begin{restatable}{corollary}{CORMTHREECOMPL}\label{cor:m3_total_comm}
    Let $\cC_i^t$ be the Perm$K$ compressors and $\cQ_i^t$ be the independent (Assumption~\ref{ass:independent}) Rand$K$ compressors, both with $K = \nicefrac{d}{n}$. Then, in the view of Corollary~\ref{cor:mthree_cor}, the iteration complexity is
    \begin{align*}
        \cO\left(\frac{\Psi^0}{\varepsilon} \left(n^{2/3} L_{\max} + n L_A\right)\right),
    \end{align*}
    and the total communication complexity is
    \begin{align*}
        \cO\left(\frac{\Psi^0}{\varepsilon} \left(\frac{d L_{\max}}{n^{1/3}} + d L_A\right)\right).
    \end{align*}
\end{restatable}

\begin{remark}
    The above result proves the complexities from Theorem \ref{thm:marinap_gen_m}.
\end{remark}

\subsection{Proofs}

Similar to our approach from the previous section, we start by establishing several inequalities satisfied by the sequences $\{w_1^t,\ldots,w_n^t\}_{t\geq0}$, $\{z_1^t,\ldots,z_n^t\}_{t\geq0}$ and $\{g_1^t,\ldots,g_n^t\}_{t\geq0}$.

\begin{algorithm}[t]
    \caption{\algname{\newmethod}}
    \begin{algorithmic}[1]\label{alg:m3}
    \STATE \textbf{Input:} initial model $x_0\in\R^d$ {\color{gray}(stored on the server)}, initial model shifts $w_i^0 = z_i^0 = x^0$, $i\in[n]$ {\color{gray}(stored on the workers)}, initial gradient estimators $g^0 = \nabla f(x^0)$ {\color{gray}(stored on the sever)}, step size $\gamma > 0$, probabilities $0 < p_P, p_D \leq 1$, compressors $\cC_1^t,\ldots,\cC_n^t \in \mathbb{U}(\omega_P) \cap \mathbb{P}(\theta)$, $\cQ_1^t,\ldots,\cQ_n^t \in \mathbb{U}(\omega_D)$ for all $t \geq 0.$
    \FOR{$t = 0, \dots, T$}
    \STATE $x^{t+1} = x^t - \gamma g^t$ \hfill{\scriptsize \color{gray}Server takes a gradient-type step to update the global model}
    \STATE Sample $c_P^t \sim \textnormal{Bernoulli}(p_P)$, $c_D^t \sim \textnormal{Bernoulli}(p_D)$
    \STATE For $i\in[n]$, send $\cC_i^t(x^{t+1} - x^t)$ to worker $i$ if $c_P^t=0$ and $x^{t+1}$ otherwise
    \FOR{$i = 1, \dots, n$ in parallel}
    \STATE $w^{t+1}_i = 
    \begin{cases}
        x^{t+1} & \textnormal{if } c_P^t=1,\\
        w_i^t + \cC_i^t(x^{t+1} - x^t) & \textnormal{if } c_P^t=0,
    \end{cases}$
    \hfill{\scriptsize\color{gray}Worker $i$ updates its local model shift}
    \STATE $z_i^{t+1} = \beta w_i^{t+1} + (1 - \beta) z_i^{t}$ \hfill{\scriptsize\color{gray}Worker $i$ takes the momentum step}
    \STATE Send $\cQ_i^t(\nabla f_i(z_i^{t+1}) - \nabla f_i(z_i^{t}))$ to the server if $c_D^t=0$ and $\nabla f_i(z_i^{t+1})$ otherwise
    \ENDFOR
    \IF{$c_D^t=1$}
        \STATE $g^{t+1} = \frac{1}{n} \sum_{i=1}^n f_i(z_i^{t+1})$
    \ELSE
        \STATE $g^{t+1} = g^{t} + \frac{1}{n} \sum_{i=1}^n \cQ_i^t(\nabla f_i(z_i^{t+1}) - \nabla f_i(z_i^{t}))$
    \ENDIF
    \ENDFOR
    \end{algorithmic}
    (maintaining only the sequence $g^t$ in the implementation is sufficient; the sequences $g^t_i$ from \eqref{eq:mthree} are \emph{virtual})
\end{algorithm}

\begin{lemma}
    \label{lemma:pDguBXTkQyk}
    Let $\cC_i^t \in \mathbb{U}(\omega_P)$ for all $i\in[n]$ and $\brac{\cC_i^t}_{i=1}^n \in \mathbb{P}(\theta)$. Then
    \begin{align*}
        \ExpSub{t}{\norm{w^{t+1}_i - z^t_i}^2} &\leq \norm{(x^{t+1} - x^t) - p_P (w_i^t - x^t) +  (w_i^t - z^t_i)}^2 \\
        &\quad+ p_P \norm{w_i^t - x^t}^2 + \omega_P \norm{x^{t+1} - x^t}^2
    \end{align*}
    for all $i \in [n],$
    and
    \begin{align*}
        \ExpSub{t}{\norm{w^{t+1} - z^t}^2} &\leq \norm{(x^{t+1} - x^t) - p_P (w^t - x^t) +  (w^t - z^t)}^2 \\
        &\quad+ p_P \norm{w^t - x^t}^2 + \theta \norm{x^{t+1} - x^t}^2.
    \end{align*}
\end{lemma}

\begin{proof}
    Using the definition of $w_i^{t+1}$, we have
    \begin{align*}
        \ExpSub{t}{w^{t+1}_i} = x^{t+1} + (1 - p_P) (w_i^t - x^t)
    \end{align*}
    and hence
    \begin{eqnarray*}
        \ExpSub{t}{\norm{w^{t+1}_i - z^t_i}^2}
        &\overset{\eqref{eq:vardecomp}}{=}& \norm{x^{t+1} + (1 - p_P) (w_i^t - x^t) - z^t_i}^2 \\
        &&+ \ExpSub{t}{\norm{w^{t+1}_i - (x^{t+1} + (1 - p_P) (w_i^t - x^t))}^2}\\
        &=& \norm{(x^{t+1} - x^t) - p_P (w_i^t - x^t) +  (w_i^t - z^t_i)}^2 \\
        &&+ \ExpSub{t}{\norm{w^{t+1}_i - (x^{t+1} + (1 - p_P) (w_i^t - x^t))}^2}.
    \end{eqnarray*}
    Using the definition of $w_i^{t+1}$ again, we get
    \begin{eqnarray*}
        &&\hspace{-1cm}\ExpSub{t}{\norm{w^{t+1}_i - z^t_i}^2} \\
        &=& \norm{(x^{t+1} - x^t) - p_P (w_i^t - x^t) +  (w_i^t - z^t_i)}^2 \\
        &&+ p_P \norm{x^{t+1} - (x^{t+1} + (1 - p_P) (w_i^t - x^t))}^2 \\
        &&+ (1 - p_P)\ExpSub{t}{\norm{w_i^t + \cC_i^t(x^{t+1} - x^t) - (x^{t+1} + (1 - p_P) (w_i^t - x^t))}^2} \\
        &=& \norm{(x^{t+1} - x^t) - p_P (w_i^t - x^t) +  (w_i^t - z^t_i)}^2
        + p_P (1 - p_P)^2 \norm{w_i^t - x^t}^2 \\
        &&+ (1 - p_P)\ExpSub{t}{\norm{\cC_i^t(x^{t+1} - x^t) - (x^{t+1} - x^t) + p_P (w_i^t - x^t)}^2} \\
        &\overset{\eqref{eq:vardecomp}}{=}& \norm{(x^{t+1} - x^t) - p_P (w_i^t - x^t) +  (w_i^t - z^t_i)}^2
        + p_P (1 - p_P)^2 \norm{w_i^t - x^t}^2 \\
        &&+ (1 - p_P)p_P^2 \norm{w_i^t - x^t}^2 + (1 - p_P)\ExpSub{t}{\norm{\cC_i^t(x^{t+1} - x^t) - (x^{t+1} - x^t)}^2} \\
        &\overset{\textnormal{Def.}\ref{def:unbiased_compression}}{\leq}& \norm{(x^{t+1} - x^t) - p_P (w_i^t - x^t) +  (w_i^t - z^t_i)}^2 + p_P \norm{w_i^t - x^t}^2 + \omega_P \norm{x^{t+1} - x^t}^2.
    \end{eqnarray*}
    Using the same reasoning, we now prove the second inequality:
    \begin{eqnarray*}
        &&\hspace{-1cm}\ExpSub{t}{\norm{w^{t+1} - z^t}^2} \\
        &\overset{\eqref{eq:vardecomp}}{=}& \norm{x^{t+1} + (1 - p_P) (w^t - x^t) - z^t}^2 + \ExpSub{t}{\norm{w^{t+1} - (x^{t+1} + (1 - p_P) (w^t - x^t))}^2}\\
        &=& \norm{(x^{t+1} - x^t) - p_P (w^t - x^t) +  (w^t - z^t)}^2 \\
        &&+ \ExpSub{t}{\norm{w^{t+1} - (x^{t+1} + (1 - p_P) (w^t - x^t))}^2} \\
        &=& \norm{(x^{t+1} - x^t) - p_P (w^t - x^t) +  (w^t - z^t)}^2
        + p_P(1 - p_P)^2 \norm{w^t - x^t}^2 \\
        &&+ (1 - p_P)\ExpSub{t}{\norm{w^t + \frac{1}{n}\sum_{i=1}^n \cC_i^t(x^{t+1} - x^t) - (x^{t+1} + (1 - p_P) (w^t - x^t))}^2} \\
        &=& \norm{(x^{t+1} - x^t) - p_P (w^t - x^t) +  (w^t - z^t)}^2
        + p_P(1 - p_P)^2 \norm{w^t - x^t}^2 \\
        &&+ (1 - p_P)\ExpSub{t}{\norm{\frac{1}{n}\sum_{i=1}^n \cC_i^t(x^{t+1} - x^t) - (x^{t+1} - x^t) + p_P (w^t - x^t)}^2} \\
        &\overset{\eqref{eq:vardecomp}}{=}& \norm{(x^{t+1} - x^t) - p_P (w^t - x^t) +  (w^t - z^t)}^2
        + p_P(1 - p_P)^2 \norm{w^t - x^t}^2 \\
        &&+ (1 - p_P)\ExpSub{t}{\norm{\frac{1}{n}\sum_{i=1}^n \cC_i^t(x^{t+1} - x^t) - (x^{t+1} - x^t)}^2}
        + p_P^2 (1 - p_P)\norm{w^t - x^t}^2 \\
        &\overset{\textnormal{Def.}\ref{def:corr_compr}}{\leq}& \norm{(x^{t+1} - x^t) - p_P (w^t - x^t) +  (w^t - z^t)}^2 + p_P \norm{w^t - x^t}^2 + \theta \norm{x^{t+1} - x^t}^2.
    \end{eqnarray*}
\end{proof}

\begin{lemma}
    \label{lemma:g_t}
    Let Assumption \ref{ass:local_lipschitz_constant} hold. Furthermore, suppose that the compressors $\cQ_i^t \in \mathbb{U}(\omega_D)$ satisfy Assumption \ref{ass:independent} and that $\cC_i^t \in \mathbb{U}(\omega_P)$ for $i\in[n]$. Then
    \begin{align*}
        &\Exp{\norm{g^{t+1} - \frac{1}{n} \sum_{i=1}^{n} \nabla f_i(z_i^{t+1})}^2}\\
        &\leq \frac{\omega_D L_{\max}^2}{n} \left(4 p_P \beta^2 \Exp{\frac{1}{n}\sum_{i=1}^n \norm{w^{t}_i - x^t}^2} + 3 \beta^2 \Exp{\frac{1}{n}\sum_{i=1}^n \norm{z^{t}_i - w_i^t}^2} + 3 (\omega_P + 1) \beta^2 \Exp{\norm{x^{t+1} - x^t}^2}\right) \\
        &\quad+ (1-p_D) \Exp{\norm{g^t - \frac{1}{n} \sum_{i=1}^n \nabla f_i(z_i^{t})}^2}.
    \end{align*}
\end{lemma}

\begin{proof}
    First, from the definition of $g_i^{t+1},$ we get
    \begin{eqnarray*}
        &&\hspace{-1cm}\Exp{\norm{g^{t+1} - \frac{1}{n} \sum_{i=1}^{n} \nabla f_i(z_i^{t+1})}^2} \\
        &=& (1-p_D) \Exp{\norm{\frac{1}{n} \sum_{i=1}^n \parens{g_i^{t} + \cQ_i^t(\nabla f_i(z_i^{t+1}) - \nabla f_i(z_i^{t})) - \nabla f_i(z_i^{t+1})}}^2}
    \end{eqnarray*}
    and hence
    \begin{eqnarray}
        &&\hspace{-1cm}\Exp{\norm{g^{t+1} - \frac{1}{n} \sum_{i=1}^{n} \nabla f_i(z_i^{t+1})}^2} \nonumber \\
        &\overset{\eqref{eq:vardecomp}}{=}& (1-p_D) \Exp{\norm{\frac{1}{n} \sum_{i=1}^n \cQ_i^t\parens{\nabla f_i(z_i^{t+1}) - \nabla f_i(z_i^{t})} - \frac{1}{n} \sum_{i=1}^n \parens{\nabla f_i(z_i^{t+1}) - \nabla f_i(z_i^{t})}}^2} \nonumber\\
        &&+ (1-p_D) \Exp{\norm{g^t - \frac{1}{n} \sum_{i=1}^n \nabla f_i(z_i^{t})}^2} \nonumber\\
        &\overset{\textnormal{Def.}\ref{def:unbiased_compression}, \eqref{ass:independent}}{\leq}& \frac{\omega_D}{n} \Exp{\frac{1}{n}\sum_{i=1}^n \norm{\nabla f_i(z_i^{t+1}) - \nabla f_i(z_i^{t})}^2} + (1-p_D) \Exp{\norm{g^t - \frac{1}{n} \sum_{i=1}^n \nabla f_i(z_i^{t})}^2} \nonumber\\
        &\overset{\textnormal{Ass.}\ref{ass:local_lipschitz_constant}}{\leq}& \frac{\omega_D L_{\max}^2}{n} \Exp{\frac{1}{n}\sum_{i=1}^n \norm{z_i^{t+1} - z_i^{t}}^2} + (1-p_D) \Exp{\norm{g^t - \frac{1}{n} \sum_{i=1}^n \nabla f_i(z_i^{t})}^2}. \label{eq:XfqBNaTJxvMlGtg}
    \end{eqnarray}
    Let us consider the first term separately:
    \begin{align*}
        \Exp{\frac{1}{n}\sum_{i=1}^n \norm{z_i^{t+1} - z_i^{t}}^2} &=\Exp{\frac{1}{n}\sum_{i=1}^n \norm{\beta w^{t+1}_i + (1 - \beta ) z^{t}_i - z_i^{t}}^2} \\
        &=\beta^2 \Exp{\frac{1}{n}\sum_{i=1}^n \norm{w^{t+1}_i - z^{t}_i}^2}.
    \end{align*}
    Using the result from Lemma~\ref{lemma:pDguBXTkQyk}, we have
    \begin{eqnarray*}
        &&\hspace{-1cm}\Exp{\frac{1}{n}\sum_{i=1}^n \norm{z_i^{t+1} - z_i^{t}}^2} \\
        &\leq& \beta^2 \Exp{\frac{1}{n}\sum_{i=1}^n \norm{(x^{t+1} - x^t) - p_P (w_i^t - x^t) +  (w_i^t - z^t_i)}^2} \\
        &&+ \beta^2 \Exp{p_P \frac{1}{n}\sum_{i=1}^n \norm{w_i^t - x^t}^2 + \omega_P \norm{x^{t+1} - x^t}^2} \\
        &\overset{\eqref{eq:young_2}}{\leq}& \beta^2 \Exp{\frac{3}{n}\sum_{i=1}^n \norm{w_i^t - z^t_i}^2 + 4 p_P \frac{1}{n}\sum_{i=1}^n \norm{w_i^t - x^t}^2 + (\omega_P + 3) \norm{x^{t+1} - x^t}^2},
    \end{eqnarray*}
    where in the last line we use the fact that $p_P \leq 1$. It remains to substitute the above inequality in~\eqref{eq:XfqBNaTJxvMlGtg}.
\end{proof}

\begin{lemma}
    \label{lemma:z_t}
    Let $\cC_i^t \in \mathbb{U}(\omega_P)$ for all $i\in[n]$ and $\brac{\cC_i^t}_{i=1}^n \in \mathbb{P}(\theta)$. Then
    \begin{align*}
        \Exp{\norm{z^{t+1}_i - w^{t+1}_i}^2} &\leq \left(1 - \frac{\beta}{2}\right) \Exp{\norm{z^t_i - w_i^t}^2} + 4\left(\frac{1}{\beta} + \omega_P\right)\Exp{\norm{x^{t+1} - x^t}^2} \\
        &\quad+ 4 p_P \left(1 + \frac{p_P}{\beta}\right)\Exp{\norm{w_i^t - x^t}^2}
    \end{align*}
    for all $i \in [n],$
    and 
    \begin{align*}
        \Exp{\norm{z^{t+1} - w^{t+1}}^2} &\leq \left(1 - \frac{\beta}{2}\right) \Exp{\norm{z^t - w^t}^2} + 4\left(\frac{1}{\beta} + \theta\right)\Exp{\norm{x^{t+1} - x^t}^2} \\
        &\quad+ 4 p_P\left(1 + \frac{p_P}{\beta}\right)\Exp{\norm{w^t - x^t}^2}.
    \end{align*}
\end{lemma}

\begin{proof}
    From the definition of $z^{t+1}_i,$ we get
    \begin{align*}
        \Exp{\norm{z^{t+1}_i - w^{t+1}_i}^2}
        =\Exp{\norm{\beta w^{t+1}_i + (1 - \beta ) z^{t}_i - w^{t+1}_i}^2}
        =(1 - \beta)^2 \Exp{\norm{w^{t+1}_i - z^{t}_i}^2}.
    \end{align*}
    Then, Lemma~\ref{lemma:pDguBXTkQyk} gives
    \begin{eqnarray*}
        &&\hspace{-1cm}\Exp{\norm{z^{t+1}_i - w^{t+1}_i}^2}\\
        &\leq& (1 - \beta)^2 \Exp{\norm{(x^{t+1} - x^t) - p_P (w_i^t - x^t) +  (w_i^t - z^t_i)}^2} \\
        &&+ (1 - \beta)^2 \Exp{p_P \norm{w_i^t - x^t}^2 + \omega_P \norm{x^{t+1} - x^t}^2} \\
        &\overset{{\eqref{eq:young}, \eqref{eq:ineq1}, \eqref{eq:ineq2}}}{\leq}& \left(1 - \frac{\beta}{2}\right) \Exp{\norm{w_i^t - z^t_i}^2} + \frac{2}{\beta}\Exp{\norm{(x^{t+1} - x^t) - p_P (w_i^t - x^t)}^2} \\
        &&+ \Exp{p_P \norm{w_i^t - x^t}^2 + \omega_P \norm{x^{t+1} - x^t}^2} \\
        &\overset{\eqref{eq:young_2}}{\leq}& \left(1 - \frac{\beta}{2}\right) \Exp{\norm{w_i^t - z^t_i}^2} + \frac{4}{\beta}\Exp{\norm{x^{t+1} - x^t}^2} + \frac{4 p_P^2}{\beta}\Exp{\norm{w_i^t - x^t}^2} \\
        &&+ \Exp{p_P \norm{w_i^t - x^t}^2 + \omega_P \norm{x^{t+1} - x^t}^2} \\
        &\leq& \left(1 - \frac{\beta}{2}\right) \Exp{\norm{w_i^t - z^t_i}^2} + 4\left(\frac{1}{\beta} + \omega_P\right)\Exp{\norm{x^{t+1} - x^t}^2} \\
        &&+ 4 p_P \left(1 + \frac{p_P}{\beta}\right)\Exp{\norm{w_i^t - x^t}^2}.
    \end{eqnarray*}
    The second inequality is proved almost in the same way. First,
    \begin{align*}
        \Exp{\norm{z^{t+1} - w^{t+1}}^2} = (1 - \beta)^2 \Exp{\norm{w^{t+1} - z^{t}}^2},
    \end{align*} 
    and using Lemma~\ref{lemma:pDguBXTkQyk}, we obtain
    \begin{eqnarray*}
        &&\hspace{-1cm}\Exp{\norm{z^{t+1} - w^{t+1}}^2} \\
        &\leq& (1 - \beta)^2 \Exp{\norm{(x^{t+1} - x^t) - p_P (w^t - x^t) +  (w^t - z^t)}^2 + p_P \norm{w^t - x^t}^2 + \theta \norm{x^{t+1} - x^t}^2} \\
        &\overset{\eqref{eq:young}}{\leq}& \left(1 - \frac{\beta}{2}\right) \Exp{\norm{w^t - z^t}^2} + \frac{2}{\beta}\Exp{\norm{(x^{t+1} - x^t) - p_P (w^t - x^t)}^2} \\
        &&+ \Exp{p_P \norm{w^t - x^t}^2 + \theta \norm{x^{t+1} - x^t}^2} \\
        &\overset{\eqref{eq:young_2}}{\leq}& \left(1 - \frac{\beta}{2}\right) \Exp{\norm{w^t - z^t}^2} + 4\left(\frac{1}{\beta} + \theta\right)\Exp{\norm{x^{t+1} - x^t}^2} \\
        &&+ \frac{4 p_P^2}{\beta}\Exp{\norm{w^t - x^t}^2} + \Exp{p_P \norm{w^t - x^t}^2} \\
        &\leq& \left(1 - \frac{\beta}{2}\right) \Exp{\norm{w^t - z^t}^2} + 4\left(\frac{1}{\beta} + \theta\right)\Exp{\norm{x^{t+1} - x^t}^2} \\
        &&+ 4 p_P\left(1 + \frac{p_P}{\beta}\right)\Exp{\norm{w^t - x^t}^2}.
    \end{eqnarray*} 
\end{proof}

\THEOREMMTHREE*
\begin{proof}
    Lemma \ref{lemma:page} gives
    \begin{eqnarray*}
        &&\hspace{-1cm}\Exp{\delta^{t+1}}\\
        &\leq& \Exp{\delta^{t}} - \frac{\gamma}{2} \Exp{\norm{\nabla f(x^t)}^2} - \left( \frac{1}{2\gamma} - \frac{L}{2} \right) \Exp{\norm{x^{t+1} - x^t}^2} + \frac{\gamma}{2} \Exp{\norm{g^t - \nabla f(x^t)}^2} \\
        &\overset{\eqref{eq:young}}{\leq}& \Exp{\delta^{t}} - \frac{\gamma}{2} \Exp{\norm{\nabla f(x^t)}^2} - \left( \frac{1}{2\gamma} - \frac{L}{2} \right) \Exp{\norm{x^{t+1} - x^t}^2} \\
        &&+ \gamma \Exp{\norm{g^t - \frac{1}{n} \sum_{i=1}^{n} \nabla f_i(z_i^t)}^2} + \gamma \Exp{\norm{\frac{1}{n} \sum_{i=1}^{n} \parens{\nabla f_i(z_i^t) - \nabla f_i(x^t)}}^2} \\
        &\overset{\eqref{as:AB_assumption}}{\leq}& \Exp{\delta^{t}} - \frac{\gamma}{2} \Exp{\norm{\nabla f(x^t)}^2} - \left( \frac{1}{2\gamma} - \frac{L}{2} \right) \Exp{\norm{x^{t+1} - x^t}^2} \\
        &&+ \gamma \Exp{\norm{g^t - \frac{1}{n} \sum_{i=1}^{n} \nabla f_i(z_i^t)}^2}
        + \gamma L_A^2 \frac{1}{n} \sum_{i=1}^{n} \Exp{\norm{z_i^t - x^t}^2}
        + \gamma L_B^2 \Exp{\norm{z^t - x^t}^2} \\
        &\overset{\eqref{eq:young}}{\leq}& \Exp{\delta^{t}} - \frac{\gamma}{2} \Exp{\norm{\nabla f(x^t)}^2} - \left( \frac{1}{2\gamma} - \frac{L}{2} \right) \Exp{\norm{x^{t+1} - x^t}^2} \\
        &&+ \gamma \Exp{\norm{g^t - \frac{1}{n} \sum_{i=1}^{n} \nabla f_i(z_i^t)}^2}
        + 2\gamma L_A^2 \frac{1}{n} \sum_{i=1}^{n} \left(\Exp{\norm{z_i^t - w_i^t}^2} + \Exp{\norm{w_i^t - x^t}^2}\right) \\
        &&+ 2\gamma L_B^2 \left(\Exp{\norm{z^t - w^t}^2} + \Exp{\norm{w^t - x^t}^2}\right).
    \end{eqnarray*}
    Let $\kappa, \eta, \nu, \rho, \mu \geq 0$ be some non-negative numbers that we define later. Using Lemmas~\ref{lemma:marinap_sum_w_x}, \ref{lemma:w_sum_x}, \ref{lemma:g_t} and \ref{lemma:z_t}, we get
    \begin{align*}
        &\Exp{\delta^{t+1}}
        + \kappa \Exp{\norm{g^{t+1} - \frac{1}{n} \sum_{i=1}^{n} \nabla f_i(z_i^{t+1})}^2}
        + \eta \Exp{\norm{z^{t+1} - w^{t+1}}^2}\\
        &\quad + \nu \Exp{\frac{1}{n} \sum_{i=1}^n \norm{z^{t+1}_i - w^{t+1}_i}^2}
        + \rho \Exp{\norm{w^{t+1} - x^{t+1}}^2} + \mu \Exp{\frac{1}{n} \sum_{i=1}^n \norm{w_i^{t+1} - x^{t+1}}^2} \\
        &\leq \Exp{\delta^{t}} - \frac{\gamma}{2} \Exp{\norm{\nabla f(x^t)}^2} - \left( \frac{1}{2\gamma} - \frac{L}{2} \right) \Exp{\norm{x^{t+1} - x^t}^2}
        + \gamma \Exp{\norm{g^t - \frac{1}{n} \sum_{i=1}^{n} \nabla f_i(z_i^t)}^2} \\
        &\quad+ 2\gamma L_A^2 \frac{1}{n} \sum_{i=1}^{n} \left(\Exp{\norm{z_i^t - w_i^t}^2} + \Exp{\norm{w_i^t - x^t}^2}\right) \\
        &\quad+ 2\gamma L_B^2 \left(\Exp{\norm{z^t - w^t}^2} + \Exp{\norm{w^t - x^t}^2}\right) \\
        &\quad + \kappa \Bigg(\frac{\omega_D L_{\max}^2}{n} \Bigg(4 p_P \beta^2 \Exp{\frac{1}{n}\sum_{i=1}^n \norm{w^{t}_i - x^t}^2} + 3 \beta^2 \Exp{\frac{1}{n}\sum_{i=1}^n \norm{z^{t}_i - w_i^t}^2} \\
        &\qquad\qquad\qquad\qquad+ 3 (\omega_P + 1) \beta^2 \Exp{\norm{x^{t+1} - x^t}^2}\Bigg)\Bigg) \\
        &\quad + \kappa (1-p_D) \Exp{\norm{g^t - \frac{1}{n} \sum_{i=1}^n \nabla f_i(z_i^{t})}^2} \\
        &\quad + \eta \Bigg(\left(1 - \frac{\beta}{2}\right) \Exp{\norm{z^t - w^t}^2} + 4\left(\frac{1}{\beta} + \theta\right)\Exp{\norm{x^{t+1} - x^t}^2} \\
        &\qquad\quad+ 4 p_P\left(1 + \frac{p_P}{\beta}\right)\Exp{\norm{w^t - x^t}^2}\Bigg) \\
        &\quad + \nu \Bigg(\left(1 - \frac{\beta}{2}\right) \Exp{\frac{1}{n} \sum_{i=1}^n\norm{z^t_i - w_i^t}^2} + 4\left(\frac{1}{\beta} + \omega_P\right)\Exp{\norm{x^{t+1} - x^t}^2} \\
        &\qquad\quad+ 4 p_P \left(1 + \frac{p_P}{\beta}\right)\Exp{\frac{1}{n} \sum_{i=1}^n\norm{w_i^t - x^t}^2}\Bigg) \\
        &\quad + \rho (1-p_P) \left(\Exp{\norm{w^t - x^t}^2} + \theta \Exp{\norm{x^{t+1} - x^t}^2}\right) \\
        &\quad + \mu (1-p_P) \left(\Exp{\frac{1}{n} \sum_{i=1}^n \norm{w_i^{t} - x^{t}}^2} + \omega_P \Exp{\norm{x^{t+1} - x^{t}}^2}\right).
    \end{align*}
    Taking $\kappa = \frac{\gamma}{p_D}$ and $\eta = \frac{4\gamma L_B^2}{\beta}$, we get $\gamma + \kappa (1-p_D) = \kappa$ and $2\gamma L_B^2 + \eta (1 - \nicefrac{\beta}{2}) = \eta$, which gives
    \begin{align*}
        &\Exp{\delta^{t+1}}
        + \kappa \Exp{\norm{g^{t+1} - \frac{1}{n} \sum_{i=1}^{n} \nabla f_i(z_i^{t+1})}^2}
        + \eta \Exp{\norm{z^{t+1} - w^{t+1}}^2} \\
        &\quad + \nu \Exp{\frac{1}{n} \sum_{i=1}^n \norm{z^{t+1}_i - w^{t+1}_i}^2}
        + \rho \Exp{\norm{w^{t+1} - x^{t+1}}^2}
        + \mu \Exp{\frac{1}{n} \sum_{i=1}^n \norm{w_i^{t+1} - x^{t+1}}^2} \\
        &\leq \Exp{\delta^{t}} - \frac{\gamma}{2} \Exp{\norm{\nabla f(x^t)}^2} - \left( \frac{1}{2\gamma} - \frac{L}{2} \right) \Exp{\norm{x^{t+1} - x^t}^2}
        + \kappa \Exp{\norm{g^{t} - \frac{1}{n} \sum_{i=1}^{n} \nabla f_i(z_i^{t})}^2} \\
        &\quad + \eta \Exp{\norm{z^{t} - w^{t}}^2}
        + 2\gamma L_A^2 \frac{1}{n} \sum_{i=1}^{n} \left(\Exp{\norm{z_i^t - w_i^t}^2} + \Exp{\norm{w_i^t - x^t}^2}\right) 
        + 2\gamma L_B^2 \Exp{\norm{w^t - x^t}^2} \\
        &\quad + \frac{\gamma}{p_D} \Bigg(\frac{\omega_D L_{\max}^2}{n} \Bigg(4 p_P \beta^2 \Exp{\frac{1}{n}\sum_{i=1}^n \norm{w^{t}_i - x^t}^2} + 3 \beta^2 \Exp{\frac{1}{n}\sum_{i=1}^n \norm{z^{t}_i - w_i^t}^2} \\
        &\qquad\qquad\qquad\qquad\quad+ 3 (\omega_P + 1) \beta^2 \Exp{\norm{x^{t+1} - x^t}^2}\Bigg)\Bigg) \\
        &\quad + \frac{4\gamma L_B^2}{\beta} \left(4\left(\frac{1}{\beta} + \theta\right)\Exp{\norm{x^{t+1} - x^t}^2} + 4 p_P\left(1 + \frac{p_P}{\beta}\right)\Exp{\norm{w^t - x^t}^2}\right) \\
        &\quad + \nu \Bigg(\left(1 - \frac{\beta}{2}\right) \Exp{\frac{1}{n} \sum_{i=1}^n\norm{z^t_i - w_i^t}^2} + 4\left(\frac{1}{\beta} + \omega_P\right)\Exp{\norm{x^{t+1} - x^t}^2} \\
        &\quad\quad\quad+ 4 p_P \left(1 + \frac{p_P}{\beta}\right)\Exp{\frac{1}{n} \sum_{i=1}^n\norm{w_i^t - x^t}^2}\Bigg) \\
        &\quad + \rho \left((1-p_P) \Exp{\norm{w^t - x^t}^2} + \theta \Exp{\norm{x^{t+1} - x^t}^2}\right) \\
        &\quad + \mu \left((1-p_P) \Exp{\frac{1}{n} \sum_{i=1}^n \norm{w_i^{t} - x^{t}}^2} + \omega_P \Exp{\norm{x^{t+1} - x^{t}}^2}\right).
    \end{align*}
    We rearrange the terms to obtain
    \begin{align*}
        &\Exp{\delta^{t+1}}
        + \kappa \Exp{\norm{g^{t+1} - \frac{1}{n} \sum_{i=1}^{n} \nabla f_i(z_i^{t+1})}^2}
        + \eta \Exp{\norm{z^{t+1} - w^{t+1}}^2} \\
        &\quad + \nu \Exp{\frac{1}{n} \sum_{i=1}^n \norm{z^{t+1}_i - w^{t+1}_i}^2}
        + \rho \Exp{\norm{w^{t+1} - x^{t+1}}^2} + \mu \Exp{\frac{1}{n} \sum_{i=1}^n \norm{w_i^{t+1} - x^{t+1}}^2} \\
        &\leq \Exp{\delta^{t}} - \frac{\gamma}{2} \Exp{\norm{\nabla f(x^t)}^2} - \left( \frac{1}{2\gamma} - \frac{L}{2} \right) \Exp{\norm{x^{t+1} - x^t}^2}
        + \kappa \Exp{\norm{g^{t} - \frac{1}{n} \sum_{i=1}^{n} \nabla f_i(z_i^{t})}^2} \\
        &\quad + \eta \Exp{\norm{z^{t} - w^{t}}^2}
        + \left(\nu \left(1 - \frac{\beta}{2}\right) + 2\gamma L_A^2 + \frac{3 \gamma \omega_D \beta^2 L_{\max}^2}{n p_D} \right) \Exp{\frac{1}{n} \sum_{i=1}^n\norm{z^t_i - w_i^t}^2} \\
        &\quad + \Bigg(\frac{3 \gamma \omega_D (\omega_P + 1) \beta^2 L_{\max}^2}{n p_D} + \rho \theta + \frac{16 \gamma L_B^2}{\beta} \left(\frac{1}{\beta} + \theta\right) \\
        &\quad\quad\quad+ 4 \nu \left(\frac{1}{\beta} + \omega_P\right) + \mu \omega_P\Bigg) \Exp{\norm{x^{t+1} - x^t}^2} \\
        &\quad + \left(\rho(1-p_P) + 2\gamma L_B^2 + \frac{16 \gamma L_B^2 p_P}{\beta} \left(1 + \frac{p_P}{\beta}\right)\right) \Exp{\norm{w^t - x^t}^2} \\
        &\quad + \left(\mu (1-p_P) + 2\gamma L_A^2 + 4 \nu p_P \left(1 + \frac{p_P}{\beta}\right) + \frac{4 \gamma \omega_D p_P \beta^2 L_{\max}^2}{n p_D}\right) \Exp{\frac{1}{n}\sum_{i=1}^n \norm{w^{t}_i - x^t}^2}.
    \end{align*}
    We now consider the coefficient of the term $\Exp{\norm{w^t - x^t}^2}$. Using the inequality $xy \leq \frac{x^2 + y^2}{2}$ for all $x,y \geq 0$, we get
    \begin{align*}
        \rho(1-p_P) + 2\gamma L_B^2 + \frac{16 \gamma L_B^2 p_P}{\beta} \left(1 + \frac{p_P}{\beta}\right)
        &\leq \rho(1-p_P) + 16\gamma L_B^2 \left(1 + \frac{p_P}{\beta} + \frac{p_P^2}{\beta^2}\right) \\
        &\leq \rho(1-p_P) + 32\gamma L_B^2 \left(1 + \frac{p_P^2}{\beta^2}\right) \\
        &= \rho
    \end{align*}
    for $\rho = 32\gamma L_B^2 \left(\frac{1}{p_P} + \frac{p_P}{\beta^2}\right).$ With this choice of $\rho,$ we obtain
    \begin{align*}
        &\Exp{\delta^{t+1}}
        + \kappa \Exp{\norm{g^{t+1} - \frac{1}{n} \sum_{i=1}^{n} \nabla f_i(z_i^{t+1})}^2}
        + \eta \Exp{\norm{z^{t+1} - w^{t+1}}^2} \\
        &\quad + \nu \Exp{\frac{1}{n} \sum_{i=1}^n \norm{z^{t+1}_i - w^{t+1}_i}^2}
        + \rho \Exp{\norm{w^{t+1} - x^{t+1}}^2} + \mu \Exp{\frac{1}{n} \sum_{i=1}^n \norm{w_i^{t+1} - x^{t+1}}^2} \\
        &\leq \Exp{\delta^{t}} - \frac{\gamma}{2} \Exp{\norm{\nabla f(x^t)}^2} - \left( \frac{1}{2\gamma} - \frac{L}{2} \right) \Exp{\norm{x^{t+1} - x^t}^2}
        + \kappa \Exp{\norm{g^{t} - \frac{1}{n} \sum_{i=1}^{n} \nabla f_i(z_i^{t})}^2} \\
        &\quad+ \eta \Exp{\norm{z^{t} - w^{t}}^2}
        + \rho \Exp{\norm{w^{t} - x^{t}}^2} \\
        &\quad + \left(\nu \left(1 - \frac{\beta}{2}\right) + 2\gamma L_A^2 + \frac{3 \gamma \omega_D \beta^2 L_{\max}^2}{n p_D} \right) \Exp{\frac{1}{n} \sum_{i=1}^n\norm{z^t_i - w_i^t}^2} \\
        &\quad + \Bigg(\frac{3 \gamma \omega_D (\omega_P + 1) \beta^2 L_{\max}^2}{n p_D} + 32\gamma L_B^2 \left(\frac{1}{p_P} + \frac{p_P}{\beta^2}\right) \theta + \frac{16 \gamma L_B^2}{\beta} \left(\frac{1}{\beta} + \theta\right) \\
        &\quad\quad\quad+ 4 \nu \left(\frac{1}{\beta} + \omega_P\right) + \mu \omega_P\Bigg) \Exp{\norm{x^{t+1} - x^t}^2} \\
        &\quad + \left(\mu (1-p_P) + 2\gamma L_A^2 + 4 \nu p_P \left(1 + \frac{p_P}{\beta}\right) + \frac{4 \gamma \omega_D p_P \beta^2 L_{\max}^2}{n p_D}\right) \Exp{\frac{1}{n}\sum_{i=1}^n \norm{w^{t}_i - x^t}^2}.
    \end{align*}
    Next, taking $\nu = \frac{4\gamma L_A^2}{\beta} + \frac{6 \gamma \omega_D \beta L_{\max}^2}{n p_D}$ gives 
    \begin{align*}
        &\Exp{\delta^{t+1}}
        + \kappa \Exp{\norm{g^{t+1} - \frac{1}{n} \sum_{i=1}^{n} \nabla f_i(z_i^{t+1})}^2}
        + \eta \Exp{\norm{z^{t+1} - w^{t+1}}^2} \\
        &\quad + \nu \Exp{\frac{1}{n} \sum_{i=1}^n \norm{z^{t+1}_i - w^{t+1}_i}^2}
        + \rho \Exp{\norm{w^{t+1} - x^{t+1}}^2} + \mu \Exp{\frac{1}{n} \sum_{i=1}^n \norm{w_i^{t+1} - x^{t+1}}^2} \\
        &\leq \Exp{\delta^{t}} - \frac{\gamma}{2} \Exp{\norm{\nabla f(x^t)}^2} - \left( \frac{1}{2\gamma} - \frac{L}{2} \right) \Exp{\norm{x^{t+1} - x^t}^2}
        + \kappa \Exp{\norm{g^{t} - \frac{1}{n} \sum_{i=1}^{n} \nabla f_i(z_i^{t})}^2} \\
        &\quad + \eta \Exp{\norm{z^{t} - w^{t}}^2}
        + \rho \Exp{\norm{w^{t} - x^{t}}^2}
        + \nu \Exp{\frac{1}{n} \sum_{i=1}^n \norm{z^{t}_i - w^{t}_i}^2} \\
        &\quad + \Bigg(\frac{3 \gamma \omega_D (\omega_P + 1) \beta^2 L_{\max}^2}{n p_D} + 32\gamma L_B^2 \left(\frac{1}{p_P} + \frac{p_P}{\beta^2}\right) \theta + \frac{16 \gamma L_B^2}{\beta} \left(\frac{1}{\beta} + \theta\right) \\
        &\quad\quad\quad + 4 \left(\frac{4\gamma L_A^2}{\beta} + \frac{6 \gamma \omega_D \beta L_{\max}^2}{n p_D}\right) \left(\frac{1}{\beta} + \omega_P\right) + \mu \omega_P\Bigg) \Exp{\norm{x^{t+1} - x^t}^2} \\
        &\quad + \Bigg(\mu (1-p_P) + 2\gamma L_A^2 + 4 \left(\frac{4\gamma L_A^2}{\beta} + \frac{6 \gamma \omega_D \beta L_{\max}^2}{n p_D}\right) p_P \left(1 + \frac{p_P}{\beta}\right) \\
        &\quad\quad\quad + \frac{4 \gamma \omega_D p_P \beta^2 L_{\max}^2}{n p_D}\Bigg) \Exp{\frac{1}{n}\sum_{i=1}^n \norm{w^{t}_i - x^t}^2}.
    \end{align*}
    Let us consider the last bracket:
    \begin{align*}
        &\mu (1-p_P) + 2\gamma L_A^2 + 4 \left(\frac{4\gamma L_A^2}{\beta} + \frac{6 \gamma \omega_D \beta L_{\max}^2}{n p_D}\right) p_P \left(1 + \frac{p_P}{\beta}\right) + \frac{4 \gamma \omega_D p_P \beta^2 L_{\max}^2}{n p_D} \\
        &=\mu (1-p_P) + 2\gamma L_A^2 + \frac{16\gamma p_P L_A^2}{\beta} + \frac{24 \gamma \omega_D p_P \beta L_{\max}^2}{n p_D} + \frac{16\gamma p_P^2 L_A^2}{\beta^2} \\
        &\quad+ \frac{24 \gamma \omega_D p_P^2 L_{\max}^2}{n p_D} + \frac{4 \gamma \omega_D p_P \beta^2 L_{\max}^2}{n p_D} \\
        &\leq \mu (1-p_P) + 16 \gamma L_A^2\left(1 + \frac{p_P}{\beta} + \frac{p_P^2}{\beta^2}\right) + \frac{24 \gamma \omega_D p_P L_{\max}^2}{n p_D}\left(\beta + p_P + \beta^2\right) \\
        &\leq \mu (1-p_P) + 32 \gamma L_A^2\left(1 + \frac{p_P^2}{\beta^2}\right) + \frac{48 \gamma \omega_D p_P L_{\max}^2}{n p_D}\left(\beta + p_P\right) \\
        &=\mu
    \end{align*}
    for $\mu = 32 \gamma L_A^2\left(\frac{1}{p_P} + \frac{p_P}{\beta^2}\right) + \frac{48 \gamma \omega_D L_{\max}^2}{n p_D}\left(\beta + p_P\right).$ For this choice, we get
    \begin{equation}
    \begin{aligned}
        &\Exp{\delta^{t+1}}
        + \kappa \Exp{\norm{g^{t+1} - \frac{1}{n} \sum_{i=1}^{n} \nabla f_i(z_i^{t+1})}^2}
        + \eta \Exp{\norm{z^{t+1} - w^{t+1}}^2} \\
        &\quad + \nu \Exp{\frac{1}{n} \sum_{i=1}^n \norm{z^{t+1}_i - w^{t+1}_i}^2}
        + \rho \Exp{\norm{w^{t+1} - x^{t+1}}^2} + \mu \Exp{\frac{1}{n} \sum_{i=1}^n \norm{w_i^{t+1} - x^{t+1}}^2} \\
        &\leq \Exp{\delta^{t}} - \frac{\gamma}{2} \Exp{\norm{\nabla f(x^t)}^2} - \left( \frac{1}{2\gamma} - \frac{L}{2} \right) \Exp{\norm{x^{t+1} - x^t}^2}
        + \kappa \Exp{\norm{g^{t} - \frac{1}{n} \sum_{i=1}^{n} \nabla f_i(z_i^{t})}^2} \\
        &\quad + \eta \Exp{\norm{z^{t} - w^{t}}^2} + \nu \Exp{\frac{1}{n} \sum_{i=1}^n \norm{z^{t}_i - w^{t}_i}^2}
        + \rho \Exp{\norm{w^{t} - x^{t}}^2}
        + \mu \Exp{\frac{1}{n} \sum_{i=1}^n \norm{w_i^{t} - x^{t}}^2} \\
        &\quad + \Bigg(\frac{3 \gamma \omega_D (\omega_P + 1) \beta^2 L_{\max}^2}{n p_D} + 32\gamma L_B^2 \left(\frac{1}{p_P} + \frac{p_P}{\beta^2}\right) \theta + \frac{16 \gamma L_B^2}{\beta} \left(\frac{1}{\beta} + \theta\right) \\
        &\qquad\quad + 4 \left(\frac{4\gamma L_A^2}{\beta} + \frac{6 \gamma \omega_D \beta L_{\max}^2}{n p_D}\right) \left(\frac{1}{\beta} + \omega_P\right) \\
        &\qquad\quad + 32 \gamma \omega_P L_A^2\left(\frac{1}{p_P} + \frac{p_P}{\beta^2}\right) + \frac{48 \gamma \omega_D \omega_P L_{\max}^2}{n p_D}\left(\beta + p_P\right)\Bigg) \Exp{\norm{x^{t+1} - x^t}^2}.
        \label{eq:GYtsgyPeoxgGCtoIMJKC}
    \end{aligned}
    \end{equation}
    Let us simplify the last bracket.
    \begin{equation}
    \begin{aligned}
        I &\eqdef \frac{3 \gamma \omega_D (\omega_P + 1) \beta^2 L_{\max}^2}{n p_D} + 32\gamma L_B^2 \left(\frac{1}{p_P} + \frac{p_P}{\beta^2}\right) \theta + \frac{16 \gamma L_B^2}{\beta} \left(\frac{1}{\beta} + \theta\right) \\
        &\quad + 4 \left(\frac{4\gamma L_A^2}{\beta} + \frac{6 \gamma \omega_D \beta L_{\max}^2}{n p_D}\right) \left(\frac{1}{\beta} + \omega_P\right) \\
        &\quad + 32 \gamma \omega_P L_A^2\left(\frac{1}{p_P} + \frac{p_P}{\beta^2}\right) + \frac{48 \gamma \omega_D \omega_P L_{\max}^2}{n p_D}\left(\beta + p_P\right) \\
        &\leq \left(\frac{3 \gamma \omega_D (\omega_P + 1) \beta^2}{n p_D} + \frac{24 \gamma \omega_D}{n p_D} + \frac{24 \gamma \omega_D \omega_P \beta}{n p_D} + \frac{48 \gamma \omega_D \omega_P \beta}{n p_D} + \frac{48 \gamma \omega_D \omega_P p_P}{n p_D}\right) L_{\max}^2 \\
        &\quad+ \left(\frac{16 \gamma}{\beta^2} + \frac{16 \gamma \omega_P}{\beta} + \frac{32 \gamma \omega_P}{p_P} + \frac{32 \gamma \omega_P p_P}{\beta^2}\right) L_A^2
        + \left(\frac{32\gamma \theta}{p_P} + \frac{32\gamma p_P \theta}{\beta^2} + \frac{16 \gamma}{\beta^2} + \frac{16 \gamma \theta}{\beta}\right)L_B^2. \label{eq:yIaYGZwqjZqELpEi}
    \end{aligned}
    \end{equation}
    We next consider the coefficients of $L_B^2$, $L_A^2$ and $L_{\max}^2$. First, for $L_B^2$, we have
    \begin{align*}
        \frac{32\gamma \theta}{p_P} + \frac{32\gamma p_P \theta}{\beta^2} + \frac{16 \gamma}{\beta^2} + \frac{16 \gamma \theta}{\beta}
        &\leq 32 \gamma \left(\frac{\theta}{p_P}\left(1 + \frac{p_P}{\beta} + \frac{p_P^2}{\beta^2}\right) + \frac{1}{\beta^2}\right) \\
        &\leq 64 \gamma \left(\frac{\theta}{p_P} + \frac{\theta p_P}{\beta^2} + \frac{1}{\beta^2}\right) \\
        &= 64 \gamma \left(\frac{\theta}{p_P} + \frac{1 + \theta p_P}{\beta^2}\right).
    \end{align*}
    Next, the coefficient of $L_A^2$ can be bounded as
    \begin{align*}
        \frac{16 \gamma}{\beta^2} + \frac{16 \gamma \omega_P}{\beta} + \frac{32 \gamma \omega_P}{p_P} + \frac{32 \gamma \omega_P p_P}{\beta^2}
        &\leq 32 \gamma \left(\frac{1}{\beta^2} + \frac{\omega_P}{p_P} \left(1 + \frac{p_P}{\beta} + \frac{p_P^2}{\beta^2}\right)\right) \\
        &\leq 64 \gamma \left(\frac{\omega_P}{p_P} + \frac{\omega_P p_P}{\beta^2} + \frac{1}{\beta^2}\right) \\
        &\leq 64 \gamma \left(\frac{\omega_P}{p_P} + \frac{1 + \omega_P p_P}{\beta^2}\right),
    \end{align*}
    and for $L_{\max}^2$ we obtain
    \begin{align*}
        &\frac{3 \gamma \omega_D (\omega_P + 1) \beta^2}{n p_D} + \frac{24 \gamma \omega_D}{n p_D} + \frac{24 \gamma \omega_D \omega_P \beta}{n p_D} + \frac{48 \gamma \omega_D \omega_P \beta}{n p_D} + \frac{48 \gamma \omega_D \omega_P p_P}{n p_D} \\
        &\leq 72 \gamma \omega_D \left(\frac{(\omega_P + 1) \beta^2}{n p_D} + \frac{1}{n p_D} + \frac{\omega_P \beta}{n p_D} + \frac{\omega_P p_P}{n p_D}\right) \\
        &\leq 144 \gamma \omega_D \left(\frac{1}{n p_D} + \frac{\omega_P \beta}{n p_D} + \frac{\omega_P p_P}{n p_D}\right) \\
        &= 144 \gamma \left(\frac{\omega_D \omega_P \beta}{n p_D} + \frac{\omega_D(1 + \omega_P p_P)}{n p_D}\right)
    \end{align*}
    since $\frac{(\omega_P + 1) \beta^2}{n p_D} \leq \frac{1}{n p_D} + \frac{\omega_P \beta}{n p_D}.$ Substituting these inequalities to \eqref{eq:GYtsgyPeoxgGCtoIMJKC} and \eqref{eq:yIaYGZwqjZqELpEi}, we get
    \begin{align*}
        &\Exp{\delta^{t+1}}
        + \kappa \Exp{\norm{g^{t+1} - \frac{1}{n} \sum_{i=1}^{n} \nabla f_i(z_i^{t+1})}^2}
        + \eta \Exp{\norm{z^{t+1} - w^{t+1}}^2} \\
        &\quad + \nu \Exp{\frac{1}{n} \sum_{i=1}^n \norm{z^{t+1}_i - w^{t+1}_i}^2}
        + \rho \Exp{\norm{w^{t+1} - x^{t+1}}^2} + \mu \Exp{\frac{1}{n} \sum_{i=1}^n \norm{w_i^{t+1} - x^{t+1}}^2} \\
        &\leq \Exp{\delta^{t}} - \frac{\gamma}{2} \Exp{\norm{\nabla f(x^t)}^2} - \left( \frac{1}{2\gamma} - \frac{L}{2} \right) \Exp{\norm{x^{t+1} - x^t}^2}
        + \kappa \Exp{\norm{g^{t} - \frac{1}{n} \sum_{i=1}^{n} \nabla f_i(z_i^{t})}^2} \\
        &\quad + \eta \Exp{\norm{z^{t} - w^{t}}^2} + \nu \Exp{\frac{1}{n} \sum_{i=1}^n \norm{z^{t}_i - w^{t}_i}^2}
        + \rho \Exp{\norm{w^{t} - x^{t}}^2} + \mu \Exp{\frac{1}{n} \sum_{i=1}^n \norm{w_i^{t} - x^{t}}^2} \\
        &\quad + 144 \gamma\Bigg(\left(\frac{\theta}{p_P} + \frac{1 + \theta p_P}{\beta^2}\right) L_B^2 + \left(\frac{\omega_P}{p_P} + \frac{1 + \omega_P p_P}{\beta^2}\right) L_A^2 \\
        &\qquad\qquad\quad+ \left(\frac{\omega_D \omega_P \beta}{n p_D} + \frac{\omega_D(1 + \omega_P p_P)}{n p_D}\right) L_{\max}^2 \Bigg) \Exp{\norm{x^{t+1} - x^t}^2}.
    \end{align*}
    By collecting all the terms w.r.t. $\Exp{\norm{x^{t+1} - x^t}^2},$ using the step size $\gamma$ from the theorem and Lemma~\ref{lemma:step_lemma}, we obtain
    \begin{align*}
        \Exp{\Psi^{t+1}}&=\Exp{\delta^{t+1}}
        + \kappa \Exp{\norm{g^{t+1} - \frac{1}{n} \sum_{i=1}^{n} \nabla f_i(z_i^{t+1})}^2}
        + \eta \Exp{\norm{z^{t+1} - w^{t+1}}^2} \\
        &\quad + \nu \Exp{\frac{1}{n} \sum_{i=1}^n \norm{z^{t+1}_i - w^{t+1}_i}^2}
        + \rho \Exp{\norm{w^{t+1} - x^{t+1}}^2} \\
        &\quad + \mu \Exp{\frac{1}{n} \sum_{i=1}^n \norm{w_i^{t+1} - x^{t+1}}^2} \\
        &\leq \Exp{\delta^{t}} - \frac{\gamma}{2} \Exp{\norm{\nabla f(x^t)}^2}
        + \kappa \Exp{\norm{g^{t} - \frac{1}{n} \sum_{i=1}^{n} \nabla f_i(z_i^{t})}^2}
        + \eta \Exp{\norm{z^{t} - w^{t}}^2} \\
        &\quad + \nu \Exp{\frac{1}{n} \sum_{i=1}^n \norm{z^{t}_i - w^{t}_i}^2}
        + \rho \Exp{\norm{w^{t} - x^{t}}^2} + \mu \Exp{\frac{1}{n} \sum_{i=1}^n \norm{w_i^{t} - x^{t}}^2} \\
        &= \Exp{\Psi^t} - \frac{\gamma}{2} \Exp{\norm{\nabla f(x^t)}^2}.
    \end{align*}
    It remains to rearrange and sum the last inequality for $t = 0, \dots, T - 1$.
\end{proof}

\CORMTHREE*

\begin{proof}
    By Theorem~\ref{theorem:mthree}, up to a constant factor, the algorithm converges after
    \begin{align*}
        \bar{T} 
        &\eqdef \frac{\Psi^0}{\varepsilon} \left(L + \sqrt{\left(\frac{\theta}{p_P} + \frac{1 + \theta p_P}{\beta^2}\right) L_B^2 + \left(\frac{\omega_P}{p_P} + \frac{1 + \omega_P p_P}{\beta^2}\right) L_A^2 + \left(\frac{\omega_D \omega_P \beta}{n p_D} + \frac{\omega_D(1 + \omega_P p_P)}{n p_D}\right) L_{\max}^2}\right) \\
        &= \frac{\Psi^0}{\varepsilon} \left(L + \sqrt{\frac{1}{\beta^2} L_B^2 + \left(\frac{\omega_P}{p_P} + \frac{1 + \omega_P p_P}{\beta^2}\right) L_A^2 + \left(\frac{\omega_D \omega_P \beta}{n p_D} + \frac{\omega_D(1 + \omega_P p_P)}{n p_D}\right) L_{\max}^2}\right) \\
        &\leq \frac{\Psi^0}{\varepsilon} \left(L + \sqrt{\frac{L_B^2}{\beta^2} + \left(\frac{\omega_P}{p_P} + \frac{2}{\beta^2}\right) L_A^2 + \left(\frac{\omega_D \omega_P \beta}{n p_D} + \frac{2 \omega_D}{n p_D}\right) L_{\max}^2}\right) \\
        &\leq \frac{\Psi^0}{\varepsilon} \left(L + \sqrt{\frac{L_B^2}{\beta^2} + \left(\omega_P (\omega_P + 1) + \frac{2}{\beta^2}\right) L_A^2 + \left(\frac{\omega_D \omega_P (\omega_D + 1) \beta}{n} + \frac{2 \omega_D (\omega_D + 1)}{n}\right) L_{\max}^2}\right) \\
        &\leq \frac{2 \Psi^0}{\varepsilon} \left(L + \sqrt{\frac{L_A^2 + L_B^2}{\beta^2} + \frac{\omega_D \omega_P (\omega_D + 1) \beta}{n} L_{\max}^2 + \frac{\omega_D (\omega_D + 1)}{n} L_{\max}^2 + \omega_P (\omega_P + 1) L_A^2}\right)
    \end{align*}
    iterations, where we use the choice of $p_P$ and $p_D.$ Using Lemma~\ref{lemma:lalblmax}, we have $L_A^2 + L_B^2 \leq L_{\max}^2$ and hence
    \begin{align*}
        \bar{T} 
        &\leq \frac{2 \Psi^0}{\varepsilon} \left(L + \sqrt{\left(\frac{1}{\beta^2} + \frac{\omega_D \omega_P (\omega_D + 1) \beta}{n}\right) L_{\max}^2 + \frac{\omega_D (\omega_D + 1)}{n} L_{\max}^2 + \omega_P (\omega_P + 1) L_A^2}\right) \\
        &\leq \frac{4 \Psi^0}{\varepsilon} \left(L + \sqrt{\left(1 + \left(\frac{\omega_D \omega_P (\omega_D + 1)}{n}\right)^{2/3}\right) L_{\max}^2 + \frac{\omega_D (\omega_D + 1)}{n} L_{\max}^2 + \omega_P (\omega_P + 1) L_A^2}\right) \\
        &\leq \frac{8 \Psi^0}{\varepsilon} \left(L_{\max} + \left(\frac{\omega_D \omega_P (\omega_D + 1)}{n}\right)^{1/3} L_{\max} + \sqrt{\frac{\omega_D (\omega_D + 1)}{n}} L_{\max} + \sqrt{\omega_P (\omega_P + 1)} L_A\right),
    \end{align*}
    where we substitute our choice of $\beta.$
\end{proof}
\CORMTHREECOMPL*
\begin{proof}
    The choice of compressors and parameters ensures that $\omega_P = \omega_D = n - 1$ (Lemma \ref{lemma:3compr_omega_theta}). Thus, the iteration complexity is
    \begin{align*}
        &\cO\left(\frac{\Psi^0}{\varepsilon} \left(L_{\max} + \left(\frac{\omega_D \omega_P (\omega_D + 1)}{n}\right)^{1/3} L_{\max} + \sqrt{\frac{\omega_D (\omega_D + 1)}{n}} L_{\max} + \sqrt{\omega_P (\omega_P + 1)} L_A\right)\right) \\
        &=\cO\left(\frac{\Psi^0}{\varepsilon} \left(L_{\max} + n^{2/3} L_{\max} + \sqrt{n} L_{\max} + n L_A\right)\right) = \cO\left(\frac{\Psi^0}{\varepsilon} \left(n^{2/3} L_{\max} + n L_A\right)\right).
    \end{align*}
    Since $p_P = p_D = 1 / n$ and $K = d / n,$ on average, the algorithm sends $\leq \frac{2 d}{n}$ coordinates in both directions. Therefore, the total communication complexity is 
    \begin{align*}
        \cO\left(\frac{d}{n} \times \frac{\Psi^0}{\varepsilon} \left(n^{2/3} L_{\max} + n L_A\right)\right) = \cO\left(\frac{\Psi^0}{\varepsilon} \left(\frac{d}{n^{1/3}} L_{\max} + d L_A\right)\right).
    \end{align*}
\end{proof}

\subsection{Polyak-Łojasiewicz condition}

\subsubsection{Main Results}

As with \algname{MARINA-P}, we provide the analysis of \algname{M3} under the Polyak-Łojasiewicz condition.

\begin{restatable}{theorem}{THEOREMMPL}\label{thm:m3_pl}
    Let Assumptions \ref{ass:lipschitz_constant}, \ref{ass:lower_bound}, \ref{ass:local_lipschitz_constant}, \ref{ass:functional} and \ref{ass:pl} be satisfied and suppose that the compressors $\cQ_i^t \in \mathbb{U}(\omega_D)$ satisfy Assumption~\ref{ass:independent}, $\brac{\cC_i^t}_{i=1}^n \in \mathbb{P}(\theta)$ and $\cC_i^t \in \mathbb{U}(\omega_P)$ for all $i\in[n]$. Let $\gamma>0$ be such that
    \begin{align}\label{eq:m3_pl_step}
        \gamma &= \textstyle \min\Bigg\{\parens{L + \sqrt{1536\Bigg(\left(\frac{\theta}{p_P} + \frac{1 + \theta p_P}{\beta^2}\right) L_B^2 + \left(\frac{\omega_P}{p_P} + \frac{1 + \omega_P p_P}{\beta^2}\right) L_A^2 + \left(\frac{\omega_D \omega_P \beta}{n p_D} + \frac{\omega_D(1 + \omega_P p_P)}{n p_D}\right) L_{\max}^2 \Bigg)}}^{-1}, \nonumber\\
        &\textstyle \qquad\qquad\qquad \frac{p_P}{2\mu}, \frac{p_D}{2\mu}, \frac{\beta}{4\mu} \Bigg\}.
    \end{align}
    Letting
    \begin{align*}
        \Psi^t &= \delta^{t}
        + \kappa \norm{g^{t} - \frac{1}{n} \sum_{i=1}^{n} \nabla f_i(z_i^{t})}^2
        + \eta \norm{z^{t} - w^{t}}^2
        + \nu \frac{1}{n} \sum_{i=1}^n \norm{z^{t}_i - w^{t}_i}^2 \\
        &\quad+ \rho \norm{w^{t} - x^{t}}^2
        + \tau \frac{1}{n} \sum_{i=1}^n \norm{w_i^{t} - x^{t}}^2,
    \end{align*}
    where $\kappa = \frac{2\gamma}{p_D}$, $\eta = \frac{8\gamma L_B^2}{\beta}$, $\nu = \frac{8 \gamma L_A^2}{\beta} + \frac{24 \gamma \omega_D \beta L_{\max}^2}{n p_D}$, $\rho = 128 \gamma L_B^2 \left(\frac{1}{p_P} + \frac{p_P}{\beta^2}\right)$ and $\tau = 128 \gamma L_A^2\left(\frac{1}{p_P} + \frac{p_P}{\beta^2}\right) + \frac{384 \gamma \omega_D L_{\max}^2}{n p_D}\left(\beta + p_P\right)$, \algname{M3} ensures that for each $T \geq 1$
    \begin{eqnarray*}
        \Exp{\Psi^{T}} \leq \parens{1 - \gamma\mu}^T \Psi^{0}.
    \end{eqnarray*}
\end{restatable}

\begin{restatable}{corollary}{THEOREMMPLCOR}\label{thm:m3_pl_perm}
    Let $\cC_i^t \in \mathbb{P}(0)$ for all $i\in[n]$ (e.g. Perm$K$), choose $p_P = 1 / (\omega_P + 1),$ $p_D = 1 / (\omega_D + 1)$ and 
    \begin{align*}
        \beta = \min\left\{\left(\frac{n}{\omega_D \omega_P (\omega_D + 1)}\right)^{1/3}, 1\right\}.
    \end{align*}
    Then, in the view of Theorem~\ref{thm:m3_pl}, Algorithm~\ref{alg:m3} ensures that $\Exp{f(x^T) - f^*} \leq \varepsilon$ after
    \begin{align*}
      \textstyle \cO\left(\max\left\{ \frac{\left(1 + \left(\frac{\omega_D \omega_P (\omega_D + 1)}{n}\right)^{1/3} + \sqrt{\frac{\omega_D (\omega_D + 1)}{n}}\right) L_{\max} + \sqrt{\omega_P (\omega_P + 1)} L_A}{\mu},
        \omega_P + 1, \omega_D + 1, \left(\frac{\omega_D \omega_P (\omega_D + 1)}{n}\right)^{1/3} \right\} \log\frac{\Psi^0}{\varepsilon} \right)
    \end{align*}
    iterations.
\end{restatable}

\begin{restatable}{corollary}{MPLCOR}\label{cor:m3_totalcomm_perm}
    Let $\cC_i^t$ be the Perm$K$ compressors and $\cQ_i^t$ be the independent (Assumption~\ref{ass:independent}) Rand$K$ compressors, both with $K = \nicefrac{d}{n}$. Then, in the view of Corollary~\ref{thm:m3_pl_perm}, the total communication complexity is
    \begin{align*}
        \cO\left(\left(\frac{d L_{\max}}{n^{1/3} \mu} + \frac{d L_A}{\mu} + d\right) \log\frac{\Psi^0}{\varepsilon}\right).
    \end{align*}
\end{restatable}

\subsubsection{Proofs}

\THEOREMMPL*

\begin{proof}
    Starting as in the proof of Theorem \ref{theorem:mthree}, we have
    \begin{align*}
        &\Exp{\delta^{t+1}}
        + \kappa \Exp{\norm{g^{t+1} - \frac{1}{n} \sum_{i=1}^{n} \nabla f_i(z_i^{t+1})}^2}
        + \eta \Exp{\norm{z^{t+1} - w^{t+1}}^2}\\
        &\quad + \nu \Exp{\frac{1}{n} \sum_{i=1}^n \norm{z^{t+1}_i - w^{t+1}_i}^2}
        + \rho \Exp{\norm{w^{t+1} - x^{t+1}}^2} + \tau \Exp{\frac{1}{n} \sum_{i=1}^n \norm{w_i^{t+1} - x^{t+1}}^2} \\
        &\leq \Exp{\delta^{t}} - \frac{\gamma}{2} \Exp{\norm{\nabla f(x^t)}^2} - \left( \frac{1}{2\gamma} - \frac{L}{2} \right) \Exp{\norm{x^{t+1} - x^t}^2}
        + \gamma \Exp{\norm{g^t - \frac{1}{n} \sum_{i=1}^{n} \nabla f_i(z_i^t)}^2} \\
        &\quad+ 2\gamma L_A^2 \frac{1}{n} \sum_{i=1}^{n} \left(\Exp{\norm{z_i^t - w_i^t}^2} + \Exp{\norm{w_i^t - x^t}^2}\right) \\
        &\quad+ 2\gamma L_B^2 \left(\Exp{\norm{z^t - w^t}^2} + \Exp{\norm{w^t - x^t}^2}\right) \\
        &\quad + \kappa \Bigg(\frac{\omega_D L_{\max}^2}{n} \Bigg(4 p_P \beta^2 \Exp{\frac{1}{n}\sum_{i=1}^n \norm{w^{t}_i - x^t}^2} + 3 \beta^2 \Exp{\frac{1}{n}\sum_{i=1}^n \norm{z^{t}_i - w_i^t}^2} \\
        &\qquad\qquad\qquad\qquad+ 3 (\omega_P + 1) \beta^2 \Exp{\norm{x^{t+1} - x^t}^2}\Bigg)\Bigg) \\
        &\quad + \kappa (1-p_D) \Exp{\norm{g^t - \frac{1}{n} \sum_{i=1}^n \nabla f_i(z_i^{t})}^2} \\
        &\quad + \eta \Bigg(\left(1 - \frac{\beta}{2}\right) \Exp{\norm{z^t - w^t}^2} + 4\left(\frac{1}{\beta} + \theta\right)\Exp{\norm{x^{t+1} - x^t}^2} + 4 p_P\left(1 + \frac{p_P}{\beta}\right)\Exp{\norm{w^t - x^t}^2}\Bigg) \\
        &\quad + \nu \Bigg(\left(1 - \frac{\beta}{2}\right) \Exp{\frac{1}{n} \sum_{i=1}^n\norm{z^t_i - w_i^t}^2} + 4\left(\frac{1}{\beta} + \omega_P\right)\Exp{\norm{x^{t+1} - x^t}^2} \\
        &\qquad\qquad+ 4 p_P \left(1 + \frac{p_P}{\beta}\right)\Exp{\frac{1}{n} \sum_{i=1}^n\norm{w_i^t - x^t}^2}\Bigg) \\
        &\quad + \rho (1-p_P) \left(\Exp{\norm{w^t - x^t}^2} + \theta \Exp{\norm{x^{t+1} - x^t}^2}\right) \\
        &\quad + \tau (1-p_P) \left(\Exp{\frac{1}{n} \sum_{i=1}^n \norm{w_i^{t} - x^{t}}^2} + \omega_P \Exp{\norm{x^{t+1} - x^{t}}^2}\right)
    \end{align*}
    for some $\kappa, \eta, \nu, \rho, \tau \geq 0$. This time, we let $\kappa = \frac{2\gamma}{p_D}$ and $\eta = \frac{8\gamma L_B^2}{\beta}$, which gives $\gamma + \kappa (1-p_D) = \kappa \parens{1-\frac{p_D}{2}}$ and $2\gamma L_B^2 + \eta (1 - \nicefrac{\beta}{2}) = \eta \parens{1-\frac{\beta}{4}}$. Hence
    \begin{align*}
        &\Exp{\delta^{t+1}}
        + \kappa \Exp{\norm{g^{t+1} - \frac{1}{n} \sum_{i=1}^{n} \nabla f_i(z_i^{t+1})}^2}
        + \eta \Exp{\norm{z^{t+1} - w^{t+1}}^2} \\
        &\quad + \nu \Exp{\frac{1}{n} \sum_{i=1}^n \norm{z^{t+1}_i - w^{t+1}_i}^2}
        + \rho \Exp{\norm{w^{t+1} - x^{t+1}}^2}
        + \tau \Exp{\frac{1}{n} \sum_{i=1}^n \norm{w_i^{t+1} - x^{t+1}}^2} \\
        &\leq \Exp{\delta^{t}} - \frac{\gamma}{2} \Exp{\norm{\nabla f(x^t)}^2} - \left( \frac{1}{2\gamma} - \frac{L}{2} \right) \Exp{\norm{x^{t+1} - x^t}^2} \\
        &\quad + \kappa \parens{1-\frac{p_D}{2}} \Exp{\norm{g^{t} - \frac{1}{n} \sum_{i=1}^{n} \nabla f_i(z_i^{t})}^2}
        + \eta \parens{1-\frac{\beta}{4}} \Exp{\norm{z^{t} - w^{t}}^2} \\
        &\quad + 2\gamma L_A^2 \frac{1}{n} \sum_{i=1}^{n} \left(\Exp{\norm{z_i^t - w_i^t}^2} + \Exp{\norm{w_i^t - x^t}^2}\right) 
        + 2\gamma L_B^2 \Exp{\norm{w^t - x^t}^2} \\
        &\quad + \frac{2\gamma}{p_D} \Bigg(\frac{\omega_D L_{\max}^2}{n} \Bigg(4 p_P \beta^2 \Exp{\frac{1}{n}\sum_{i=1}^n \norm{w^{t}_i - x^t}^2} + 3 \beta^2 \Exp{\frac{1}{n}\sum_{i=1}^n \norm{z^{t}_i - w_i^t}^2} \\
        &\qquad\qquad\qquad\qquad\quad+ 3 (\omega_P + 1) \beta^2 \Exp{\norm{x^{t+1} - x^t}^2}\Bigg)\Bigg) \\
        &\quad + \frac{8\gamma L_B^2}{\beta} \left(4\left(\frac{1}{\beta} + \theta\right)\Exp{\norm{x^{t+1} - x^t}^2} + 4 p_P\left(1 + \frac{p_P}{\beta}\right)\Exp{\norm{w^t - x^t}^2}\right) \\
        &\quad + \nu \Bigg(\left(1 - \frac{\beta}{2}\right) \Exp{\frac{1}{n} \sum_{i=1}^n\norm{z^t_i - w_i^t}^2} + 4\left(\frac{1}{\beta} + \omega_P\right)\Exp{\norm{x^{t+1} - x^t}^2} \\
        &\qquad\qquad+ 4 p_P \left(1 + \frac{p_P}{\beta}\right)\Exp{\frac{1}{n} \sum_{i=1}^n\norm{w_i^t - x^t}^2}\Bigg) \\
        &\quad + \rho \left((1-p_P) \Exp{\norm{w^t - x^t}^2} + \theta \Exp{\norm{x^{t+1} - x^t}^2}\right) \\
        &\quad + \tau \left((1-p_P) \Exp{\frac{1}{n} \sum_{i=1}^n \norm{w_i^{t} - x^{t}}^2} + \omega_P \Exp{\norm{x^{t+1} - x^{t}}^2}\right).
    \end{align*}
    Rearranging the terms
    \begin{align*}
        &\Exp{\delta^{t+1}}
        + \kappa \Exp{\norm{g^{t+1} - \frac{1}{n} \sum_{i=1}^{n} \nabla f_i(z_i^{t+1})}^2}
        + \eta \Exp{\norm{z^{t+1} - w^{t+1}}^2} \\
        &\quad + \nu \Exp{\frac{1}{n} \sum_{i=1}^n \norm{z^{t+1}_i - w^{t+1}_i}^2}
        + \rho \Exp{\norm{w^{t+1} - x^{t+1}}^2} + \tau \Exp{\frac{1}{n} \sum_{i=1}^n \norm{w_i^{t+1} - x^{t+1}}^2} \\
        &\leq \Exp{\delta^{t}} - \frac{\gamma}{2} \Exp{\norm{\nabla f(x^t)}^2} - \left( \frac{1}{2\gamma} - \frac{L}{2} \right) \Exp{\norm{x^{t+1} - x^t}^2} \\
        &\quad + \kappa \parens{1-\frac{p_D}{2}} \Exp{\norm{g^{t} - \frac{1}{n} \sum_{i=1}^{n} \nabla f_i(z_i^{t})}^2}
        + \eta \parens{1-\frac{\beta}{4}} \Exp{\norm{z^{t} - w^{t}}^2} \\
        &\quad + \left(\nu \left(1 - \frac{\beta}{2}\right) + 2\gamma L_A^2 + \frac{6 \gamma \omega_D \beta^2 L_{\max}^2}{n p_D} \right) \Exp{\frac{1}{n} \sum_{i=1}^n\norm{z^t_i - w_i^t}^2} \\
        &\quad + \left(\frac{6 \gamma \omega_D (\omega_P + 1) \beta^2 L_{\max}^2}{n p_D} + \rho \theta + \frac{32 \gamma L_B^2}{\beta} \left(\frac{1}{\beta} + \theta\right) + 4 \nu \left(\frac{1}{\beta} + \omega_P\right) + \tau \omega_P\right) \Exp{\norm{x^{t+1} - x^t}^2} \\
        &\quad + \left(\rho(1-p_P) + 2\gamma L_B^2 + \frac{32 \gamma L_B^2 p_P}{\beta} \left(1 + \frac{p_P}{\beta}\right)\right) \Exp{\norm{w^t - x^t}^2} \\
        &\quad + \left(\tau (1-p_P) + 2\gamma L_A^2 + 4 \nu p_P \left(1 + \frac{p_P}{\beta}\right) + \frac{8 \gamma \omega_D p_P \beta^2 L_{\max}^2}{n p_D}\right) \Exp{\frac{1}{n}\sum_{i=1}^n \norm{w^{t}_i - x^t}^2}.
    \end{align*}
    Considering the coefficient of $\Exp{\norm{w^t - x^t}^2}$ and using the inequality $xy \leq \frac{x^2 + y^2}{2}$ for all $x,y \geq 0$, we get
    \begin{align*}
        \rho(1-p_P) + 2\gamma L_B^2 + \frac{32 \gamma L_B^2 p_P}{\beta} \left(1 + \frac{p_P}{\beta}\right)
        &\leq \rho(1-p_P) + 32\gamma L_B^2 \left(1 + \frac{p_P}{\beta} + \frac{p_P^2}{\beta^2}\right) \\
        &\leq \rho(1-p_P) + 64\gamma L_B^2 \left(1 + \frac{p_P^2}{\beta^2}\right) \\
        &= \rho \parens{1-\frac{p_P}{2}},
    \end{align*}
    where we define $\rho = 128\gamma L_B^2 \left(\frac{1}{p_P} + \frac{p_P}{\beta^2}\right)$. Substituting this choice of $\rho$, we obtain
    \begin{align*}
        &\Exp{\delta^{t+1}}
        + \kappa \Exp{\norm{g^{t+1} - \frac{1}{n} \sum_{i=1}^{n} \nabla f_i(z_i^{t+1})}^2}
        + \eta \Exp{\norm{z^{t+1} - w^{t+1}}^2} \\
        & + \nu \Exp{\frac{1}{n} \sum_{i=1}^n \norm{z^{t+1}_i - w^{t+1}_i}^2}
        + \rho \Exp{\norm{w^{t+1} - x^{t+1}}^2} + \tau \Exp{\frac{1}{n} \sum_{i=1}^n \norm{w_i^{t+1} - x^{t+1}}^2} \\
        &\leq \Exp{\delta^{t}} - \frac{\gamma}{2} \Exp{\norm{\nabla f(x^t)}^2} - \left( \frac{1}{2\gamma} - \frac{L}{2} \right) \Exp{\norm{x^{t+1} - x^t}^2} \\
        & + \kappa \parens{1-\frac{p_D}{2}} \Exp{\norm{g^{t} - \frac{1}{n} \sum_{i=1}^{n} \nabla f_i(z_i^{t})}^2}
        + \eta \parens{1-\frac{\beta}{4}} \Exp{\norm{z^{t} - w^{t}}^2} \\
        &+ \rho \parens{1-\frac{p_P}{2}} \Exp{\norm{w^{t} - x^{t}}^2}
        + \left(\nu \left(1 - \frac{\beta}{2}\right) + 2\gamma L_A^2 + \frac{6 \gamma \omega_D \beta^2 L_{\max}^2}{n p_D} \right) \Exp{\frac{1}{n} \sum_{i=1}^n\norm{z^t_i - w_i^t}^2} \\
        & + \Bigg(\frac{6 \gamma \omega_D (\omega_P + 1) \beta^2 L_{\max}^2}{n p_D} + 128 \gamma L_B^2 \left(\frac{1}{p_P} + \frac{p_P}{\beta^2}\right) \theta + \frac{32 \gamma L_B^2}{\beta} \left(\frac{1}{\beta} + \theta\right) \\
        &\qquad+ 4 \nu \left(\frac{1}{\beta} + \omega_P\right) + \tau \omega_P\Bigg) \Exp{\norm{x^{t+1} - x^t}^2} \\
        & + \left(\tau (1-p_P) + 2\gamma L_A^2 + 4 \nu p_P \left(1 + \frac{p_P}{\beta}\right) + \frac{8 \gamma \omega_D p_P \beta^2 L_{\max}^2}{n p_D}\right) \Exp{\frac{1}{n}\sum_{i=1}^n \norm{w^{t}_i - x^t}^2}.
    \end{align*}
    Similarly, taking $\nu = \frac{8 \gamma L_A^2}{\beta} + \frac{24 \gamma \omega_D \beta L_{\max}^2}{n p_D}$ gives $\nu \left(1 - \frac{\beta}{2}\right) + 2\gamma L_A^2 + \frac{6 \gamma \omega_D \beta^2 L_{\max}^2}{n p_D} = \nu \left(1 - \frac{\beta}{4}\right)$, so
    \begin{align*}
        &\Exp{\delta^{t+1}}
        + \kappa \Exp{\norm{g^{t+1} - \frac{1}{n} \sum_{i=1}^{n} \nabla f_i(z_i^{t+1})}^2}
        + \eta \Exp{\norm{z^{t+1} - w^{t+1}}^2} \\
        &\quad + \nu \Exp{\frac{1}{n} \sum_{i=1}^n \norm{z^{t+1}_i - w^{t+1}_i}^2}
        + \rho \Exp{\norm{w^{t+1} - x^{t+1}}^2} + \tau \Exp{\frac{1}{n} \sum_{i=1}^n \norm{w_i^{t+1} - x^{t+1}}^2} \\
        &\leq \Exp{\delta^{t}} - \frac{\gamma}{2} \Exp{\norm{\nabla f(x^t)}^2} - \left( \frac{1}{2\gamma} - \frac{L}{2} \right) \Exp{\norm{x^{t+1} - x^t}^2} \\
        &\quad + \kappa \parens{1-\frac{p_D}{2}} \Exp{\norm{g^{t} - \frac{1}{n} \sum_{i=1}^{n} \nabla f_i(z_i^{t})}^2}
        + \eta \parens{1-\frac{\beta}{4}} \Exp{\norm{z^{t} - w^{t}}^2} \\
        &\quad+ \rho \parens{1-\frac{p_P}{2}} \Exp{\norm{w^{t} - x^{t}}^2}
        + \nu \parens{1-\frac{\beta}{4}} \Exp{\frac{1}{n} \sum_{i=1}^n \norm{z^{t}_i - w^{t}_i}^2} \\
        &\quad + \Bigg(\frac{6 \gamma \omega_D (\omega_P + 1) \beta^2 L_{\max}^2}{n p_D} + 128 \gamma L_B^2 \left(\frac{1}{p_P} + \frac{p_P}{\beta^2}\right) \theta + \frac{32 \gamma L_B^2}{\beta} \left(\frac{1}{\beta} + \theta\right) \\
        &\quad\quad\quad + 4 \left(\frac{8 \gamma L_A^2}{\beta} + \frac{24 \gamma \omega_D \beta L_{\max}^2}{n p_D}\right) \left(\frac{1}{\beta} + \omega_P\right) + \tau \omega_P\Bigg) \Exp{\norm{x^{t+1} - x^t}^2} \\
        &\quad + \Bigg(\tau (1-p_P) + 2\gamma L_A^2 + 4 \left(\frac{8 \gamma L_A^2}{\beta} + \frac{24 \gamma \omega_D \beta L_{\max}^2}{n p_D}\right) p_P \left(1 + \frac{p_P}{\beta}\right) \\
        &\quad\quad\quad+ \frac{8 \gamma \omega_D p_P \beta^2 L_{\max}^2}{n p_D}\Bigg) \Exp{\frac{1}{n}\sum_{i=1}^n \norm{w^{t}_i - x^t}^2}.
    \end{align*}
    Considering the last bracket, we have
    \begin{align*}
        &\tau (1-p_P) + 2\gamma L_A^2 + 4 \left(\frac{8\gamma L_A^2}{\beta} + \frac{24 \gamma \omega_D \beta L_{\max}^2}{n p_D}\right) p_P \left(1 + \frac{p_P}{\beta}\right) + \frac{8 \gamma \omega_D p_P \beta^2 L_{\max}^2}{n p_D} \\
        &= \tau (1-p_P) + 2\gamma L_A^2 + \frac{32 \gamma p_P L_A^2}{\beta} + \frac{32 \gamma p_P^2 L_A^2}{\beta^2} + \frac{96 \gamma \omega_D p_P \beta L_{\max}^2}{n p_D} + \frac{96 \gamma \omega_D p_P^2 L_{\max}^2}{n p_D} \\
        &\quad+ \frac{8 \gamma \omega_D p_P \beta^2 L_{\max}^2}{n p_D} \\
        &\leq \tau (1-p_P) + 32\gamma L_A^2 \parens{1 + \frac{p_P}{\beta} + \frac{p_P^2}{\beta^2}} + \frac{96 \gamma \omega_D p_P L_{\max}^2}{n p_D}\left(\beta + p_P + \beta^2\right) \\
        &\leq \tau (1-p_P) + 64\gamma L_A^2 \parens{1 + \frac{p_P^2}{\beta^2}} + \frac{192 \gamma \omega_D p_P L_{\max}^2}{n p_D}\left(\beta + p_P\right) \\
        &= \tau \parens{1-\frac{p_P}{2}}
    \end{align*}
    for $\tau = 128\gamma L_A^2 \parens{\frac{1}{p_P} + \frac{p_P}{\beta^2}} + \frac{384 \gamma \omega_D L_{\max}^2}{n p_D}\left(\beta + p_P\right)$. Then
    \begin{align}
        &\Exp{\delta^{t+1}}
        + \kappa \Exp{\norm{g^{t+1} - \frac{1}{n} \sum_{i=1}^{n} \nabla f_i(z_i^{t+1})}^2}
        + \eta \Exp{\norm{z^{t+1} - w^{t+1}}^2} \nonumber \\
        &\quad + \nu \Exp{\frac{1}{n} \sum_{i=1}^n \norm{z^{t+1}_i - w^{t+1}_i}^2}
        + \rho \Exp{\norm{w^{t+1} - x^{t+1}}^2} + \tau \Exp{\frac{1}{n} \sum_{i=1}^n \norm{w_i^{t+1} - x^{t+1}}^2} \nonumber \\
        &\leq \Exp{\delta^{t}} - \frac{\gamma}{2} \Exp{\norm{\nabla f(x^t)}^2} - \left( \frac{1}{2\gamma} - \frac{L}{2} \right) \Exp{\norm{x^{t+1} - x^t}^2} \nonumber \\
        &\quad+ \kappa \parens{1-\frac{p_D}{2}} \Exp{\norm{g^{t} - \frac{1}{n} \sum_{i=1}^{n} \nabla f_i(z_i^{t})}^2}
        + \eta \parens{1-\frac{\beta}{4}} \Exp{\norm{z^{t} - w^{t}}^2} \nonumber \\
        &\quad+ \rho \parens{1-\frac{p_P}{2}} \Exp{\norm{w^{t} - x^{t}}^2}
        + \nu \parens{1-\frac{\beta}{4}} \Exp{\frac{1}{n} \sum_{i=1}^n \norm{z^{t}_i - w^{t}_i}^2} \nonumber \\
        &\quad + \tau \parens{1-\frac{p_P}{2}} \Exp{\frac{1}{n} \sum_{i=1}^n \norm{w_i^{t} - x^{t}}^2} \nonumber \\
        &\quad + \Bigg(\frac{6 \gamma \omega_D (\omega_P + 1) \beta^2 L_{\max}^2}{n p_D} + 128 \gamma L_B^2 \left(\frac{1}{p_P} + \frac{p_P}{\beta^2}\right) \theta + \frac{32 \gamma L_B^2}{\beta} \left(\frac{1}{\beta} + \theta\right) \nonumber \\
        &\qquad\quad + 4 \left(\frac{8 \gamma L_A^2}{\beta} + \frac{24 \gamma \omega_D \beta L_{\max}^2}{n p_D}\right) \left(\frac{1}{\beta} + \omega_P\right) \nonumber \\
        &\qquad\quad + 128 \gamma \omega_P L_A^2\left(\frac{1}{p_P} + \frac{p_P}{\beta^2}\right) + \frac{384 \gamma \omega_D \omega_P L_{\max}^2}{n p_D}\left(\beta + p_P\right)\Bigg) \Exp{\norm{x^{t+1} - x^t}^2},
        \label{eq:GYtsgyPeoxgGCtoIMJKC2}
    \end{align}
    where the last bracket can be bounded as
    \begin{equation}
    \begin{aligned}
        I &\eqdef \frac{6 \gamma \omega_D (\omega_P + 1) \beta^2 L_{\max}^2}{n p_D} + 128\gamma L_B^2 \left(\frac{1}{p_P} + \frac{p_P}{\beta^2}\right) \theta + \frac{32 \gamma L_B^2}{\beta} \left(\frac{1}{\beta} + \theta\right) \\
        &\quad + 4 \left(\frac{8 \gamma L_A^2}{\beta} + \frac{24 \gamma \omega_D \beta L_{\max}^2}{n p_D}\right) \left(\frac{1}{\beta} + \omega_P\right) + 128 \gamma \omega_P L_A^2\left(\frac{1}{p_P} + \frac{p_P}{\beta^2}\right) \\
        &\quad + \frac{384 \gamma \omega_D \omega_P L_{\max}^2}{n p_D}\left(\beta + p_P\right) \\
        &= \parens{\frac{6 \gamma \omega_D (\omega_P + 1) \beta^2}{n p_D} + \frac{96 \gamma \omega_D \beta}{n p_D} \left(\frac{1}{\beta} + \omega_P\right) + \frac{384 \gamma \omega_D \omega_P}{n p_D}\left(\beta + p_P\right)} L_{\max}^2 \\
        &\quad+ \parens{\frac{32 \gamma}{\beta} \left(\frac{1}{\beta} + \omega_P\right) + 128 \gamma \omega_P \left(\frac{1}{p_P} + \frac{p_P}{\beta^2}\right)} L_A^2 \\
        &\quad+ \parens{128\gamma \left(\frac{1}{p_P} + \frac{p_P}{\beta^2}\right) \theta + \frac{32 \gamma}{\beta} \left(\frac{1}{\beta} + \theta\right)} L_B^2\\
        &= \parens{\frac{6 \gamma \omega_D (\omega_P + 1) \beta^2}{n p_D} + \frac{96 \gamma \omega_D}{n p_D} + \frac{96 \gamma \omega_D \omega_P \beta}{n p_D} + \frac{384 \gamma \omega_D \omega_P \beta}{n p_D} + \frac{384 \gamma \omega_D \omega_P p_P}{n p_D}} L_{\max}^2 \\
        &\quad+ \parens{\frac{32 \gamma}{\beta^2} + \frac{32 \gamma \omega_P}{\beta} + \frac{128 \gamma \omega_P}{p_P} + \frac{128 \gamma \omega_P p_P}{\beta^2}} L_A^2 \\
        &\quad+ \parens{\frac{128\gamma \theta}{p_P} + \frac{128\gamma p_P \theta}{\beta^2} + \frac{32 \gamma}{\beta^2} + \frac{32 \gamma \theta}{\beta}} L_B^2. \label{eq:yIaYGZwqjZqELpEi2}
    \end{aligned}
    \end{equation}
    We next consider the coefficients of $L_B^2$, $L_A^2$ and $L_{\max}^2$. First, for $L_B^2$, we have
    \begin{align*}
        \frac{128\gamma \theta}{p_P} + \frac{128\gamma p_P \theta}{\beta^2} + \frac{32 \gamma}{\beta^2} + \frac{32 \gamma \theta}{\beta}
        &\leq 128 \gamma \left(\frac{\theta}{p_P}\left(1 + \frac{p_P}{\beta} + \frac{p_P^2}{\beta^2}\right) + \frac{1}{\beta^2}\right) \\
        &\leq 256 \gamma \left(\frac{\theta}{p_P} + \frac{\theta p_P}{\beta^2} + \frac{1}{\beta^2}\right) \\
        &= 256 \gamma \left(\frac{\theta}{p_P} + \frac{1 + \theta p_P}{\beta^2}\right).
    \end{align*}
    Next, the coefficient of $L_A^2$ can be bounded as
    \begin{align*}
        \frac{32 \gamma}{\beta^2} + \frac{32 \gamma \omega_P}{\beta} + \frac{128}{p_P} + \frac{128 \gamma \omega_P p_P}{\beta^2}
        &\leq 128 \gamma \left(\frac{1}{\beta^2} + \frac{\omega_P}{p_P} \left(1 + \frac{p_P}{\beta} + \frac{p_P^2}{\beta^2}\right)\right) \\
        &\leq 256 \gamma \left(\frac{\omega_P}{p_P} + \frac{\omega_P p_P}{\beta^2} + \frac{1}{\beta^2}\right) \\
        &\leq 256 \gamma \left(\frac{\omega_P}{p_P} + \frac{1 + \omega_P p_P}{\beta^2}\right),
    \end{align*}
    and for $L_{\max}^2$ we obtain
    \begin{align*}
        &\frac{6 \gamma \omega_D (\omega_P + 1) \beta^2}{n p_D} + \frac{96 \gamma \omega_D}{n p_D} + \frac{96 \gamma \omega_D \omega_P \beta}{n p_D} + \frac{384 \gamma \omega_D \omega_P \beta}{n p_D} + \frac{384 \gamma \omega_D \omega_P p_P}{n p_D} \\
        &\leq 384 \gamma \omega_D \left(\frac{(\omega_P + 1) \beta^2}{n p_D} + \frac{1}{n p_D} + \frac{\omega_P \beta}{n p_D} + \frac{\omega_P p_P}{n p_D}\right) \\
        &\leq 768 \gamma \omega_D \left(\frac{1}{n p_D} + \frac{\omega_P \beta}{n p_D} + \frac{\omega_P p_P}{n p_D}\right) \\
        &= 768 \gamma \left(\frac{\omega_D \omega_P \beta}{n p_D} + \frac{\omega_D(1 + \omega_P p_P)}{n p_D}\right)
    \end{align*}
    since $\frac{(\omega_P + 1) \beta^2}{n p_D} \leq \frac{1}{n p_D} + \frac{\omega_P \beta}{n p_D}.$ Substituting these inequalities to \eqref{eq:GYtsgyPeoxgGCtoIMJKC2} and \eqref{eq:yIaYGZwqjZqELpEi2}, we get
    \begin{align*}
        &\Exp{\delta^{t+1}}
        + \kappa \Exp{\norm{g^{t+1} - \frac{1}{n} \sum_{i=1}^{n} \nabla f_i(z_i^{t+1})}^2}
        + \eta \Exp{\norm{z^{t+1} - w^{t+1}}^2} \\
        &\quad + \nu \Exp{\frac{1}{n} \sum_{i=1}^n \norm{z^{t+1}_i - w^{t+1}_i}^2}
        + \rho \Exp{\norm{w^{t+1} - x^{t+1}}^2} + \tau \Exp{\frac{1}{n} \sum_{i=1}^n \norm{w_i^{t+1} - x^{t+1}}^2} \\
        &\leq \Exp{\delta^{t}} - \frac{\gamma}{2} \Exp{\norm{\nabla f(x^t)}^2} - \left( \frac{1}{2\gamma} - \frac{L}{2} \right) \Exp{\norm{x^{t+1} - x^t}^2} \\
        &\quad+ \kappa \parens{1-\frac{p_D}{2}} \Exp{\norm{g^{t} - \frac{1}{n} \sum_{i=1}^{n} \nabla f_i(z_i^{t})}^2}
        + \eta \parens{1-\frac{\beta}{4}} \Exp{\norm{z^{t} - w^{t}}^2} \\
        &\quad+ \rho \parens{1-\frac{p_P}{2}} \Exp{\norm{w^{t} - x^{t}}^2}
        + \nu \parens{1-\frac{\beta}{4}} \Exp{\frac{1}{n} \sum_{i=1}^n \norm{z^{t}_i - w^{t}_i}^2} \\
        &\quad+ \tau \parens{1-\frac{p_P}{2}} \Exp{\frac{1}{n} \sum_{i=1}^n \norm{w_i^{t} - x^{t}}^2} \\
        &\quad + 768 \gamma\Bigg(\left(\frac{\theta}{p_P} + \frac{1 + \theta p_P}{\beta^2}\right) L_B^2 + \left(\frac{\omega_P}{p_P} + \frac{1 + \omega_P p_P}{\beta^2}\right) L_A^2 \\
       &\qquad\qquad\quad + \left(\frac{\omega_D \omega_P \beta}{n p_D} + \frac{\omega_D(1 + \omega_P p_P)}{n p_D}\right) L_{\max}^2 \Bigg) \Exp{\norm{x^{t+1} - x^t}^2}.
    \end{align*}
    By collecting all the terms w.r.t. $\Exp{\norm{x^{t+1} - x^t}^2},$ using the step size $\gamma$ from the theorem and Lemma~\ref{lemma:step_lemma}, we obtain
    \begin{align*}
        \Exp{\Psi^{t+1}} &= \Exp{\delta^{t+1}}
        + \kappa \Exp{\norm{g^{t+1} - \frac{1}{n} \sum_{i=1}^{n} \nabla f_i(z_i^{t+1})}^2}
        + \eta \Exp{\norm{z^{t+1} - w^{t+1}}^2} \\
        &\quad + \nu \Exp{\frac{1}{n} \sum_{i=1}^n \norm{z^{t+1}_i - w^{t+1}_i}^2}
        + \rho \Exp{\norm{w^{t+1} - x^{t+1}}^2} \\
        &\quad+ \tau \Exp{\frac{1}{n} \sum_{i=1}^n \norm{w_i^{t+1} - x^{t+1}}^2} \\
        &\leq \Exp{\delta^{t}} - \frac{\gamma}{2} \Exp{\norm{\nabla f(x^t)}^2}
        + \kappa \parens{1-\frac{p_D}{2}} \Exp{\norm{g^{t} - \frac{1}{n} \sum_{i=1}^{n} \nabla f_i(z_i^{t})}^2} \\
        &\quad+ \eta \parens{1-\frac{\beta}{4}} \Exp{\norm{z^{t} - w^{t}}^2}
        + \rho \parens{1-\frac{p_P}{2}} \Exp{\norm{w^{t} - x^{t}}^2} \\
        &\quad+ \nu \parens{1-\frac{\beta}{4}} \Exp{\frac{1}{n} \sum_{i=1}^n \norm{z^{t}_i - w^{t}_i}^2}
        + \tau \parens{1-\frac{p_P}{2}} \Exp{\frac{1}{n} \sum_{i=1}^n \norm{w_i^{t} - x^{t}}^2}.
    \end{align*}
    Lastly, Assumption \ref{ass:pl} gives
    \begin{eqnarray*}
        \Exp{\Psi^{t+1}} &=& \Exp{\delta^{t+1}}
        + \kappa \Exp{\norm{g^{t+1} - \frac{1}{n} \sum_{i=1}^{n} \nabla f_i(z_i^{t+1})}^2}
        + \eta \Exp{\norm{z^{t+1} - w^{t+1}}^2} \\
        && + \nu \Exp{\frac{1}{n} \sum_{i=1}^n \norm{z^{t+1}_i - w^{t+1}_i}^2}
        + \rho \Exp{\norm{w^{t+1} - x^{t+1}}^2} \\
        &&+ \tau \Exp{\frac{1}{n} \sum_{i=1}^n \norm{w_i^{t+1} - x^{t+1}}^2} \\
        &\overset{\textnormal{Ass.}\ref{ass:pl}, \eqref{eq:m3_pl_step}}{\leq}& \parens{1-\gamma\mu} \Exp{\delta^{t}}
        + \kappa \parens{1-\gamma\mu} \Exp{\norm{g^{t} - \frac{1}{n} \sum_{i=1}^{n} \nabla f_i(z_i^{t})}^2} \\
        &&+ \eta \parens{1-\gamma\mu} \Exp{\norm{z^{t} - w^{t}}^2}
        + \nu \parens{1-\gamma\mu} \Exp{\frac{1}{n} \sum_{i=1}^n \norm{z^{t}_i - w^{t}_i}^2} \\
        &&+ \rho \parens{1-\gamma\mu} \Exp{\norm{w^{t} - x^{t}}^2}
        + \tau \parens{1-\gamma\mu} \Exp{\frac{1}{n} \sum_{i=1}^n \norm{w_i^{t} - x^{t}}^2} \\
        &=& \parens{1-\gamma\mu} \Exp{\Psi^{t}}
    \end{eqnarray*}
    It remains to apply the last inequality iteratively to finish the proof.
\end{proof}

\THEOREMMPLCOR*

\begin{proof}
    Note that $\Psi^T \geq f(x^T) - f^*.$ In view of condition \eqref{eq:m3_pl_step} from Theorem \ref{thm:m3_pl}, the step size satisfies 
    \begin{align*}
        \gamma &= \Theta\Bigg(\min\left\{ \left(L + \sqrt{\left(\frac{\theta}{p_P} + \frac{1 + \theta p_P}{\beta^2}\right) L_B^2 + \left(\frac{\omega_P}{p_P} + \frac{1 + \omega_P p_P}{\beta^2}\right) L_A^2 + \left(\frac{\omega_D \omega_P \beta}{n p_D} + \frac{\omega_D(1 + \omega_P p_P)}{n p_D}\right) L_{\max}^2}\right)^{-1}, \right. \\
        &\qquad\qquad\qquad\left. \frac{p_P}{2\mu}, \frac{p_D}{2\mu}, \frac{\beta}{4\mu} \right\}\Bigg).
    \end{align*}
    Therefore, since $\theta = 0$, the algorithm converges after
    \begin{align*}
        \bar{T} 
        &= \textstyle \cO\left(\max\left\{ \frac{L + \sqrt{\frac{1}{\beta^2}L_B^2 + \left(\frac{\omega_P}{p_P} + \frac{1 + \omega_P p_P}{\beta^2}\right) L_A^2 + \left(\frac{\omega_D \omega_P \beta}{n p_D} + \frac{\omega_D(1 + \omega_P p_P)}{n p_D}\right) L_{\max}^2}}{\mu},
        \frac{1}{p_P}, \frac{1}{p_D}, \frac{1}{\beta} \right\} \log\frac{\Psi^0}{\varepsilon} \right)
    \end{align*}
    iterations. Using the choice of $p_P$ and $p_D,$ we have
    \begin{align*}
        \bar{T}
        &= \textstyle \cO\left(\max\left\{ \frac{L + \sqrt{\frac{1}{\beta^2}(L_B^2 + L_A^2) + \omega_P (\omega_P + 1) L_A^2 + \left(\frac{\omega_D (\omega_D + 1) \omega_P \beta}{n} + \frac{\omega_D (\omega_D + 1)}{n}\right) L_{\max}^2}}{\mu},
        \omega_P + 1, \omega_D + 1, \frac{1}{\beta} \right\} \log\frac{\Psi^0}{\varepsilon} \right).
    \end{align*}
    Due to Lemma~\ref{lemma:lalblmax}, we get
    \begin{align*}
        \bar{T}
        &= \textstyle \cO\left(\max\left\{ \frac{L + \sqrt{\omega_P (\omega_P + 1) L_A^2 + \left(\frac{1}{\beta^2} + \frac{\omega_D (\omega_D + 1) \omega_P \beta}{n} + \frac{\omega_D (\omega_D + 1)}{n}\right) L_{\max}^2}}{\mu},
        \omega_P + 1, \omega_D + 1, \frac{1}{\beta} \right\} \log\frac{\Psi^0}{\varepsilon} \right).
    \end{align*}
    Using the choice of $\beta,$ we obtain the result of the theorem.
\end{proof}

\MPLCOR*

\begin{proof}
    The choice of compressors and parameters ensures that $\omega_P = \omega_D = n - 1$ (Lemma \ref{lemma:3compr_omega_theta}). Thus, the iteration complexity is
    \begin{align*}
        &\cO\left(\max\left\{ \frac{\left(1 + n^{2/3} + n^{1/2}\right) L_{\max} + n L_A}{\mu},
        n, n, n^{2/3} \right\} \log\frac{\Psi^0}{\varepsilon} \right) \\
        &=\cO\left(\max\left\{ \frac{n^{2/3}L_{\max} + n L_A}{\mu},
        n\right\} \log\frac{\Psi^0}{\varepsilon} \right).
    \end{align*}
    Since $p_P = p_D = \nicefrac{1}{n}$ and $K = \nicefrac{d}{n},$ on average, the algorithm sends $\leq \nicefrac{2 d}{n}$ coordinates in both directions. Therefore, the total communication complexity is 
    \begin{align*}
        \cO\left(\frac{d}{n} \times \max\left\{ \frac{n^{2/3}L_{\max} + n L_A}{\mu},
        n\right\} \log\frac{\Psi^0}{\varepsilon}\right) = \cO\left(\max\left\{ \frac{\frac{d}{n^{1/3}}L_{\max} + d L_A}{\mu},
        d\right\} \log\frac{\Psi^0}{\varepsilon}\right).
    \end{align*}
\end{proof}

\newpage

\section{Experiments}
\label{sec:experimetns_main}

The experiments were prepared in Python. The distributed environment was emulated on a machine with Intel(R) Xeon(R) Gold 6226R CPU @ 2.90GHz and 64 cores.

\subsection{Experiments with \algnamebig{M3} on quadratic optimization tasks}
\label{sec:core_m3}
We consider close-to-homogeneous quadratic optimization problem with $\mA_i = (1 + \xi_i) \mI_d,$ where $\xi_i \sim \mathcal{N}(0, 0.01)$ for all $i \in [n],$ and $d = 1000.$ We run two algorithms from Table~\ref{table:complexities}, \algname{M3} and \algname{CORE}, and check whether theory matches practice. In \algname{M3}, we use Perm$K$ followed by the natural compressor $\cC_{\textnormal{nat}}$ \citep{horvoth2022natural} (composition of two unbiased compressors) on the server's side, and Rand$K$ followed by $\cC_{\textnormal{nat}}$ on the workers' side. We use $K = \left\lfloor d / n\right\rfloor \in \{1, 10, 100\}$ for $n \in \{1000, 100, 10\}$. In \algname{CORE}, the number of communicated coordinates is set to $10.$ We run each experiment $5$ times with different seeds and plot the average to reduce the noise factor. Only the step size is fine-tuned for each algorithm.

The results are presented in Figure~\ref{fig:core_m3}. As expected, \algname{CORE} does not change its behavior as the number of workers increases from $10$ to $100;$ this is expected since \algname{CORE} does not depend on $n.$ At the same time, \algname{M3} does improve with $n$, which supports our findings from Theorem~\ref{thm:marinap_gen_m}.

\begin{figure}[t]
    \centering
    \begin{subfigure}{.4\columnwidth}
      \centering
      \includegraphics[width=\columnwidth]{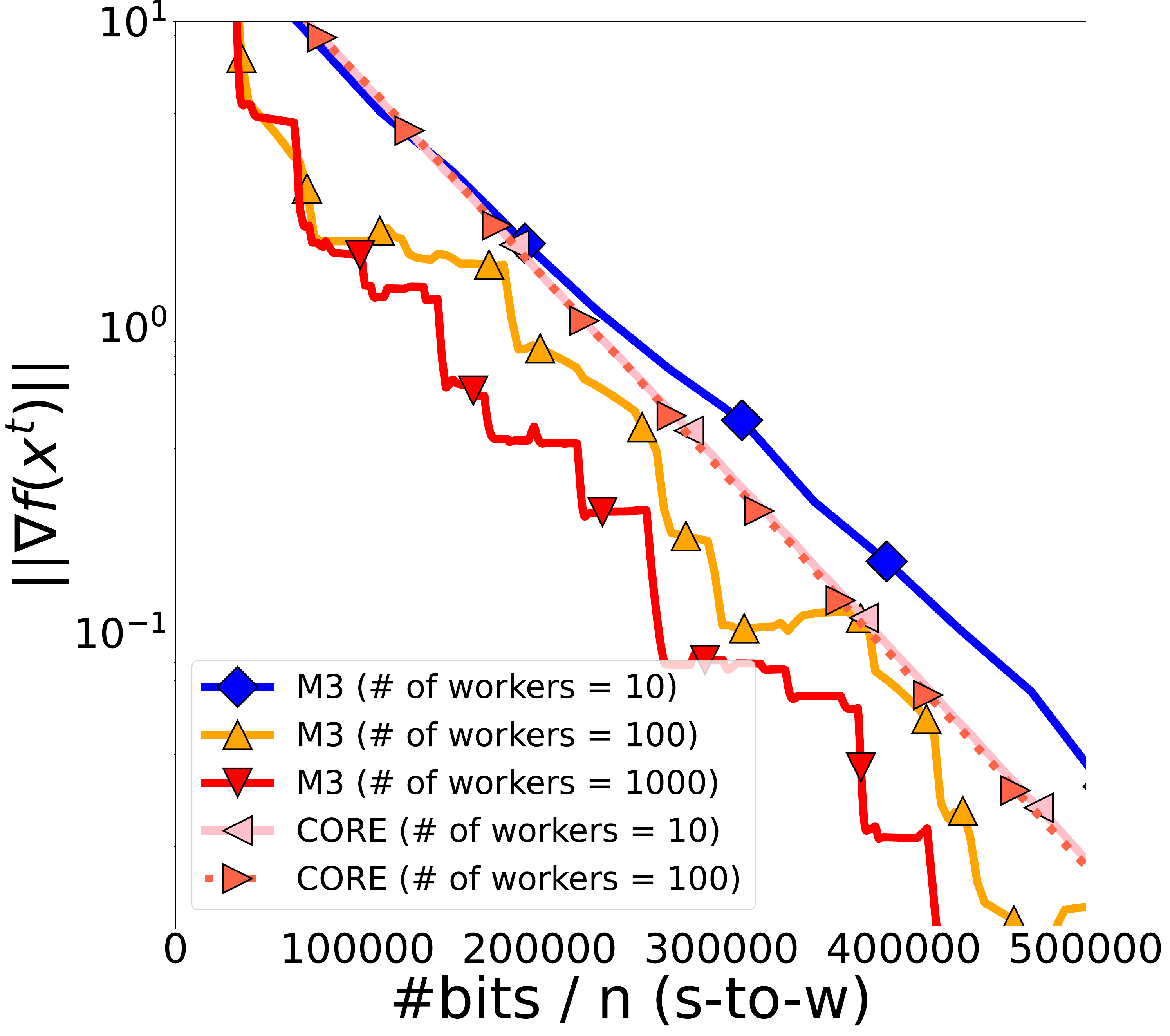}
    \end{subfigure}%
    \begin{subfigure}{.4\columnwidth}
      \centering
      \includegraphics[width=\columnwidth]{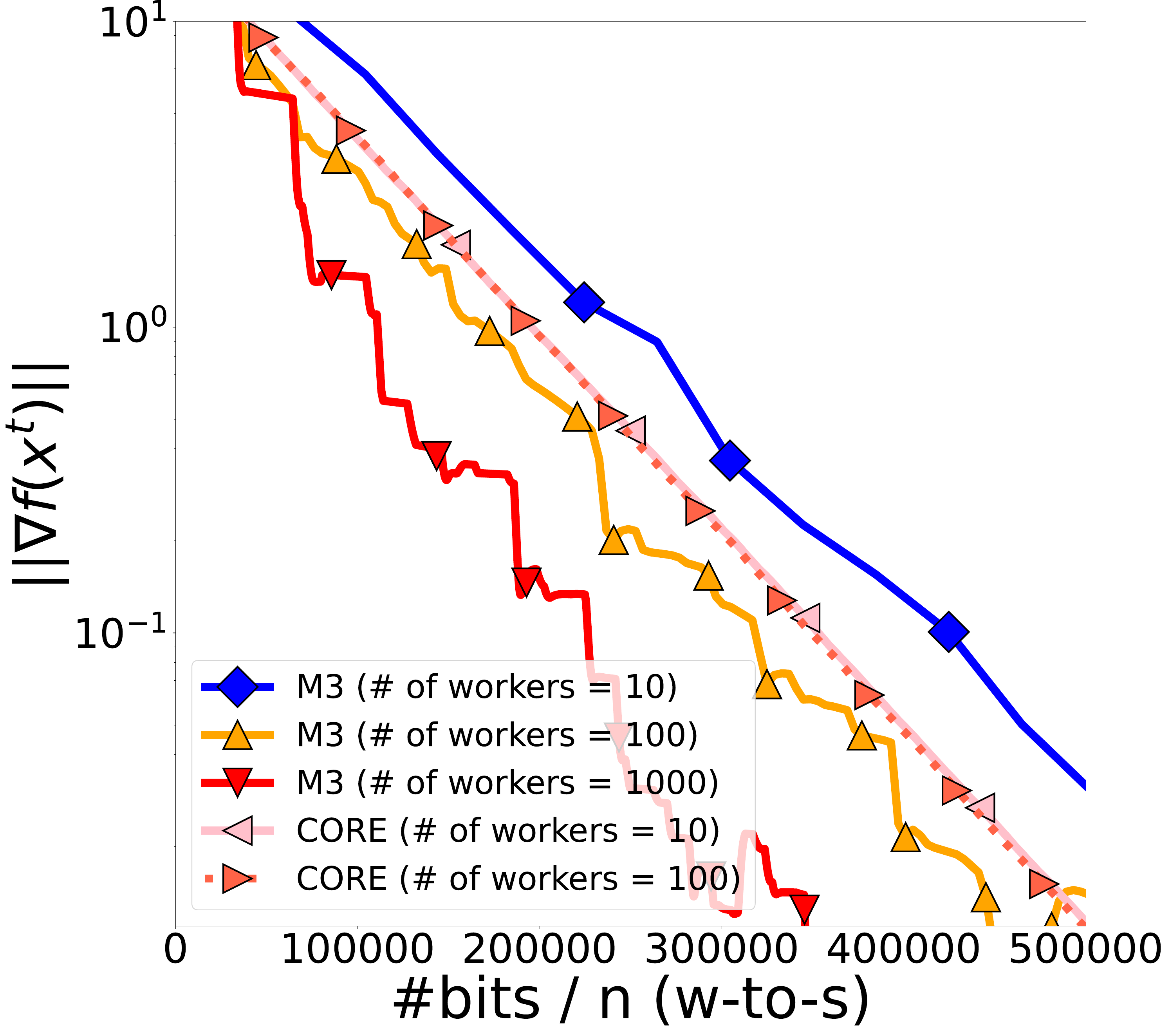}
    \end{subfigure}
    \caption{Experiments on the quadratic optimization problem from Section~\ref{sec:core_m3}.
    We plot the norm of the gradient w.r.t. \# of coordinates sent from the server (s-to-w) and from the workers (w-to-s).}
    \label{fig:core_m3}
\end{figure}

\begin{figure}[t]
    \centering
    \begin{subfigure}{.4\columnwidth}
      \centering
      \includegraphics[width=\columnwidth]{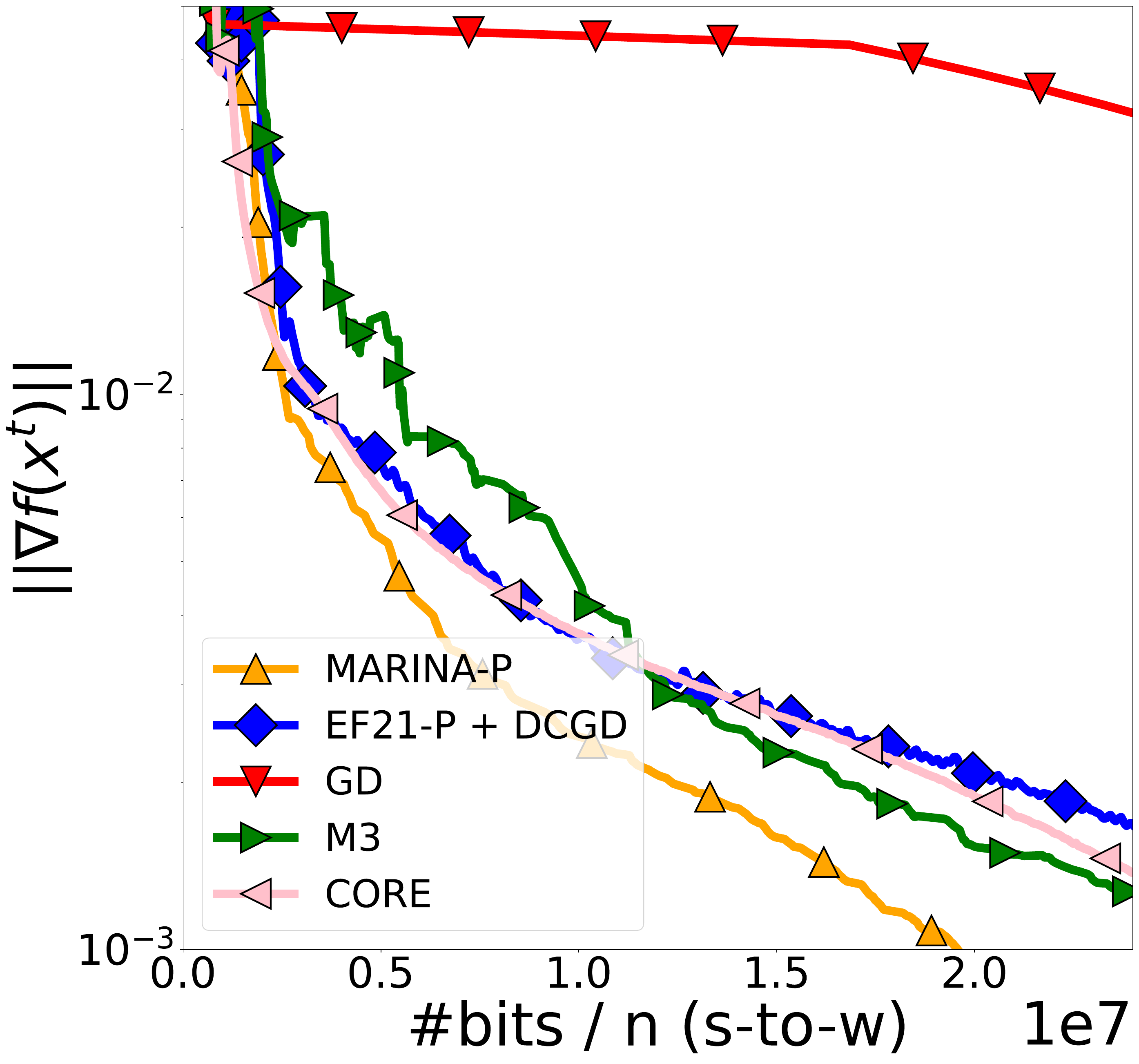}
    \end{subfigure}%
    \begin{subfigure}{.4\columnwidth}
      \centering
      \includegraphics[width=\columnwidth]{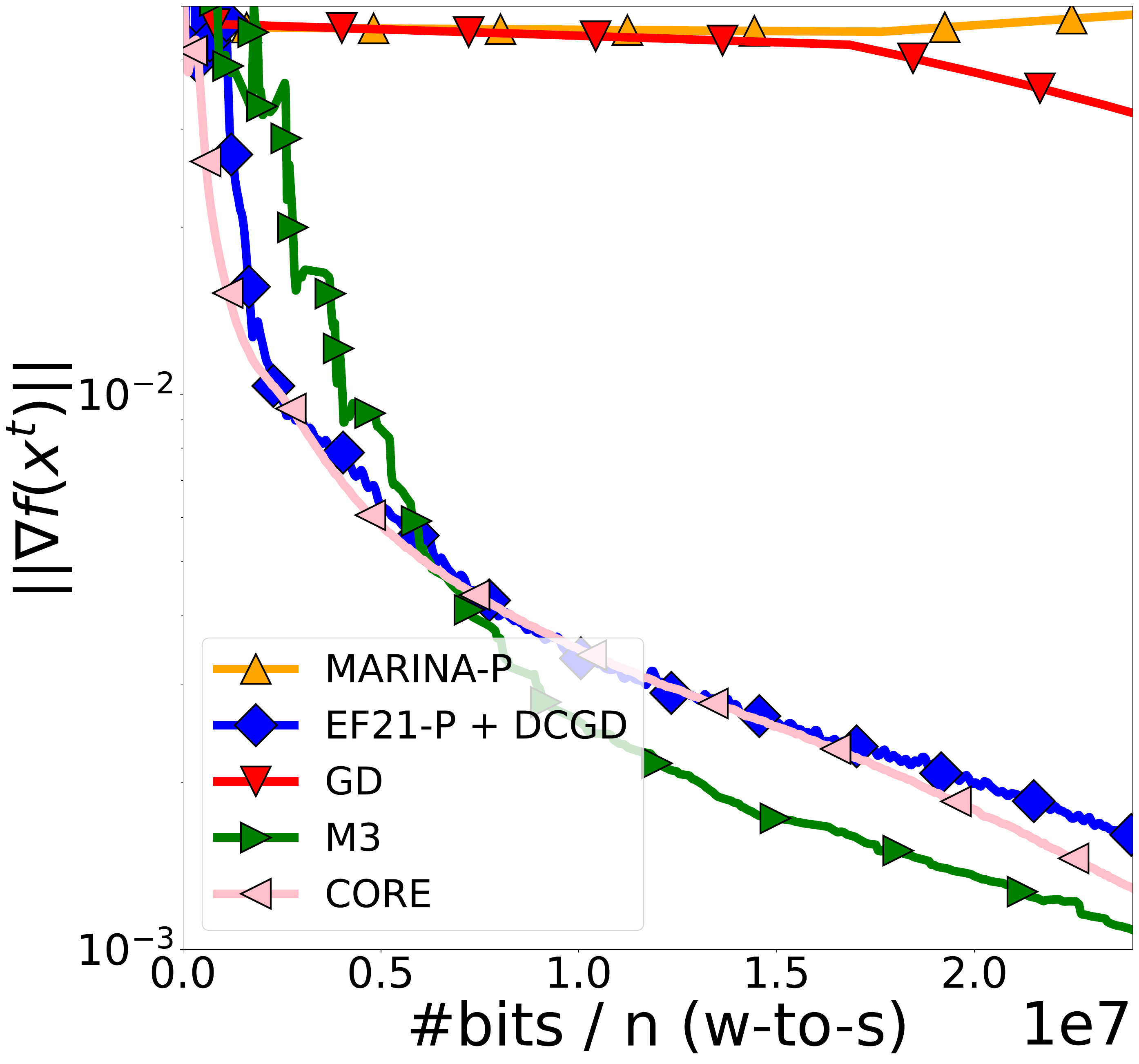}
    \end{subfigure}
    \caption{Experiments on the autoencoder task from Section~\ref{sec:autoencode}. We plot the norm of the gradient w.r.t. \# of coordinates sent from the server (s-to-w) and from the workers (w-to-s).}
    \label{fig:all_exp}
\end{figure}
\subsection{Experiments with an autoencoder and MNIST}
\label{sec:autoencode}
We now compare \algname{MARINA-P}, \algname{M3}, \algname{CORE}, \algname{EF21-P + DCGD}, and \algname{GD} on a non-convex autoencoder problem. We train it on the \texttt{MNIST} dataset \citep{lecun2010mnist} with objective function
$f(\mD, \mE) \eqdef \frac{1}{m}\sum_{i = 1}^m \norm{\mD \mE b_i - b_i}^2 + \frac{\lambda}{2}\norm{\mD \mE - \mI}_F^2,$
where $\mD \in \R^{d_1 \times d_2}$, $\mE \in \R^{d_2 \times d_1}$, $b_i \in \R^{d_1}$ are samples, $d_1 = 784$ is the number of features, 
$d_2 = 16$ is the size of the encoding space, $\lambda = 0.001$ is a regularizer, and $m = 60\,000$ is the number of samples. The dimension of the problem is $d = 25\,088.$ We randomly split the dataset among $n = 100$ workers. For \algname{MARINA-P} and \algname{M3}, we take Perm$K$ followed by the natural compressor $\cC_{\textnormal{nat}}$ on the server's side. On the workers' side, \algname{M3} uses Rand$K$ and $\cC_{\textnormal{nat}}$. For \algname{EF21-P + DCGD}, we take Rand$K$ with $\cC_{\textnormal{nat}}$ on both the workers' and server's sides. In each case, $K = \left\lfloor d / n\right\rfloor = 250.$ For \algname{CORE}, we set the number of communicated coordinates to $100.$ As in previous experiments, we only fine-tune the step size, repeat each experiment $5$ times, and plot the average results.

The results are presented in Figure~\ref{fig:all_exp}. All methods with bidirectional compression: \algname{M3}, \algname{CORE}, and \algname{EF21-P + DCGD}, converge much faster than \algname{GD}. \algname{MARINA-P} converges fastest only in the first plot. This is expected since it compresses only from the server to the workers. \algname{M3}, \algname{CORE}, and \algname{EF21-P + DCGD} have similar convergence rates in both metrics, with \algname{M3} performing better in the low accuracy regime.

\subsection{Extra experiments with quadratic optimization tasks}\label{sec:exp_quad_lalb}

The aim of this set of experiments is to empirically test our results under Assumption \ref{ass:functional}. We consider the problem of quadratic minimization with varying level of heterogeneity between the $n$ functions stored on the workers. The goal is to minimize the squared norm of the gradient of $\sum_{i=1}^n f_i$, where the functions $f_i$ are of form
\begin{align*}
    f_i(x) = \frac{1}{2} x^T\mA_i x + b_i^Tx.
\end{align*}
Here, $\mA_i$ are $d \times d$ matrices generated following the procedure in Algorithm \ref{algorithm:het_quad_problem}, and $b_i$ denotes a standard normal vector in $\R^d$.
The constants $L_A$ and $L_B$ from Assumption \ref{ass:functional} (in this case, by Theorem \ref{theorem:a_b_hess} $L_A = \sqrt{2} \max_{i\in[n]} \norm{\mA_i - \mA}$ and $L_B = \sqrt{2} \left(\frac{1}{n} \sum_{i=1}^n \norm{\mA_i}\right)$) are controlled by parameters $v_i$ and $\sigma_i^2$. In particular, for $\sigma^2_i=0$, all workers hold the same matrix $\mA_i$, and hence in this case $L_A=0$.

We compare the following algorithms:
\begin{enumerate}
    \item \algname{MARINA-P} with Perm$K$ compressors,
    \item \algname{MARINA-P} with Rand$K$ compressors,
    \item \algname{MARINA-P} with SameRand$K$ compressor,
    \item \algname{EF21-P} with Top$K$ compressor,
    \item \algname{GD}.
\end{enumerate}
In all compressed methods, we set $K=\nicefrac{d}{n}$ and use $p = \nicefrac{k}{d}$ in \algname{MARINA-P}.

The step sizes are tuned from $2^i, i\in\mathbb{Z}$ multiples of the values predicted by the theory (indicated by $\times 1, \times 2, \ldots$ in the plots).
We fix $d = 300$ and generate optimization tasks with $n \in \brac{10, 100, 900}$. The results are presented in Figures \ref{fig:grid_n10}, \ref{fig:grid_n100}, \ref{fig:grid_n900}.

\begin{algorithm}[t]
    \caption{\algname{Heterogeneous quadratic problem generation}}
    \label{algorithm:het_quad_problem}
    \begin{algorithmic}[1]
    \STATE \textbf{Parameters:} $v_0,\ldots,v_4 \in \R_+$, $\sigma_0, \ldots, \sigma_4 \in \R_{\geq 0}$.
    \STATE Let
    \begin{align*}
        \mX = \frac{1}{4}
        \begin{bmatrix}
            2  & -1     &        &  0 \\
            -1 & \ddots & \ddots &    \\
               & \ddots & \ddots &  -1  \\
            0  &        & -1     & 2
        \end{bmatrix}
        \in \R^{300\times 300},
    \end{align*}
    \FOR{$k = 0, \dots, 4$}
        \STATE Generate $\xi_i \sim \mathcal{N}(0,\sigma_k^2) \cap [-v_0,v_0]$ for $i\in[n]$
        \FOR{$l = 0, \dots, 4$}
            \STATE Set $\mA_i^{k,l} = (v_l + \xi_i) \mX$ for $i\in[n]$
            \STATE Sample $b_i^{k,l} \sim \mathcal{N}(0,\mathbb{I}_d)$ for $i\in[n]$
        \ENDFOR
        \STATE \textbf{Output:} matrices $\mA_i^{k,l}$, vectors $b_i^{k,l}$, $i\in[n]$, $k,l\in[4]$.
    \ENDFOR
    \end{algorithmic}
\end{algorithm}

The empirical results align well with the theory. Among the algorithms tested, \algname{MARINA-P} with Perm$K$ compressor exhibits the best performance, while \algname{MARINA-P} with SameRand$K$ converges the slowest and comparable to \algname{GD}. \algname{MARINA-P} with Rand$K$ compressor and \algname{EF21-P} achieve performance levels somewhere in between. Notably, the differences between the runs of \algname{MARINA-P} with different compressors become more pronounced as the value of $n$ increases. As anticipated, the performance of \algname{MARINA-P} with Rand$K$ and Perm$K$ compressors improves with an increase in the number of workers, while the performance of \algname{EF21-P} does not follow the same behaviour. Specifically, for $n=10$, \algname{EF21-P} outperforms \algname{MARINA-P} with Rand$K$ compressor, but this pattern reverses for both $n=100$ and $n=1000.$

\begin{figure}
\centering
\begin{subfigure}{0.24\textwidth}
  \centering
  \includegraphics[width=\textwidth]{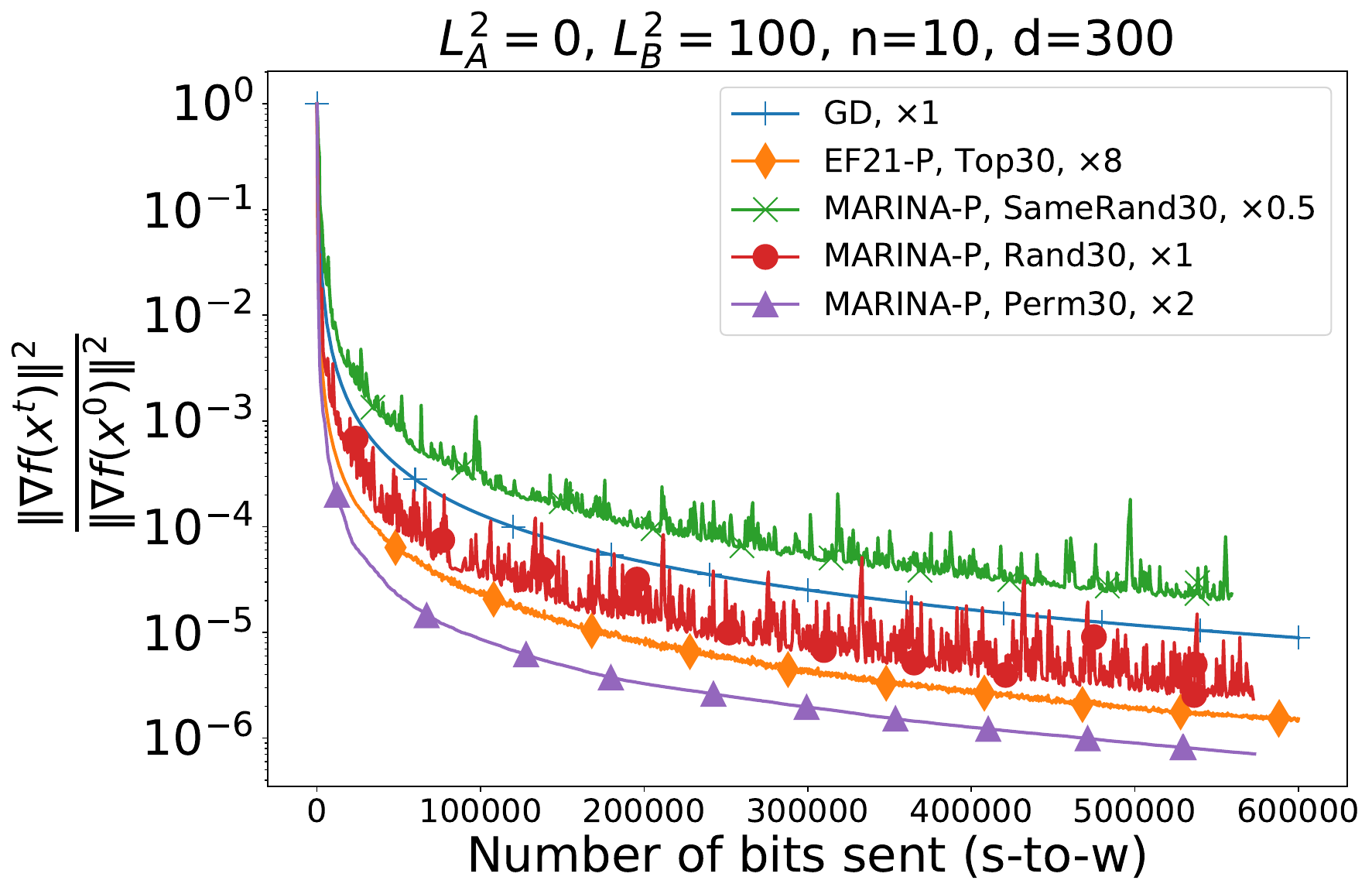}
\end{subfigure}%
\begin{subfigure}{0.24\textwidth}
  \centering
  \includegraphics[width=\textwidth]{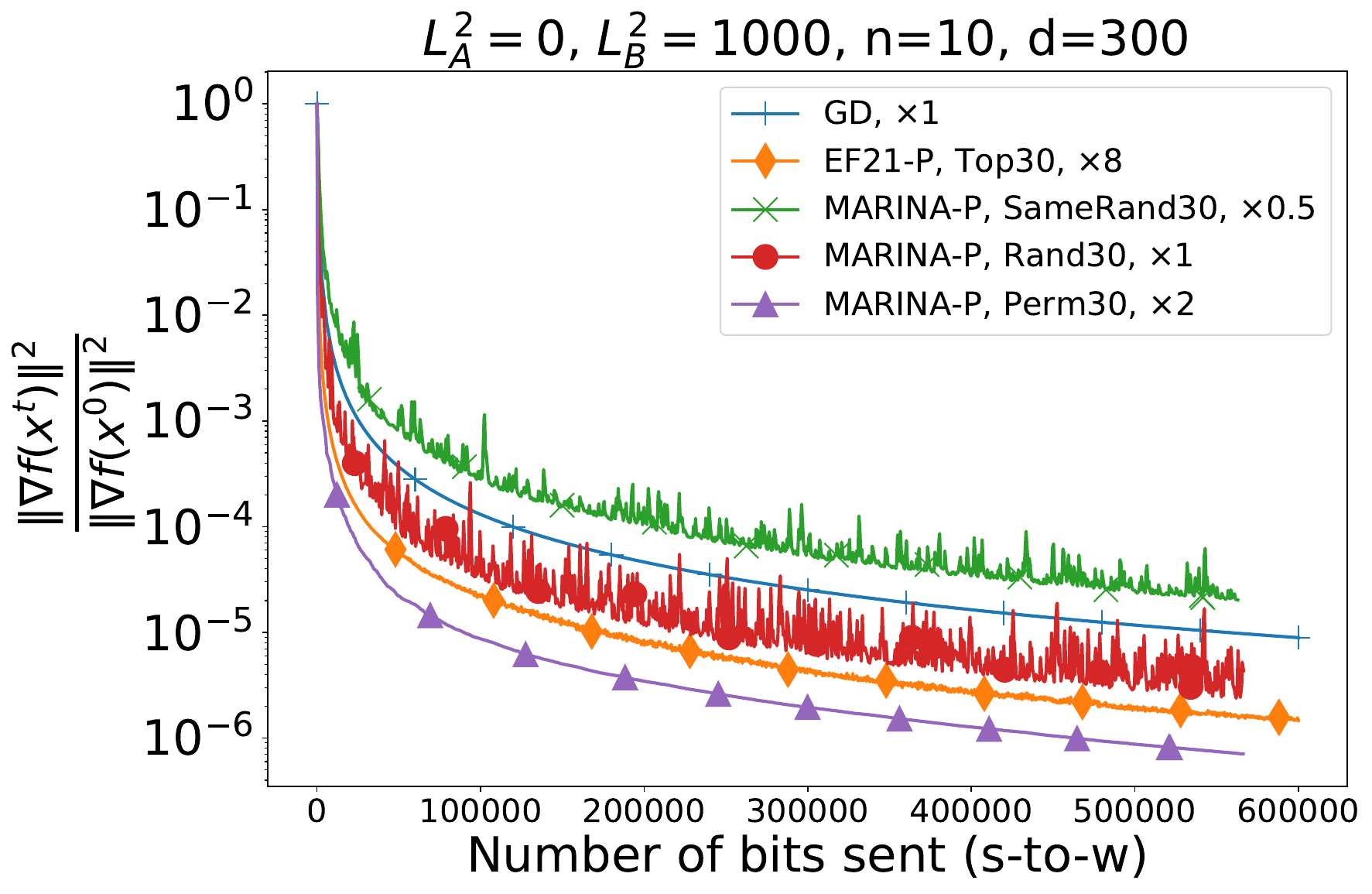}
\end{subfigure}
\begin{subfigure}{0.24\textwidth}
  \centering
  \includegraphics[width=\textwidth]{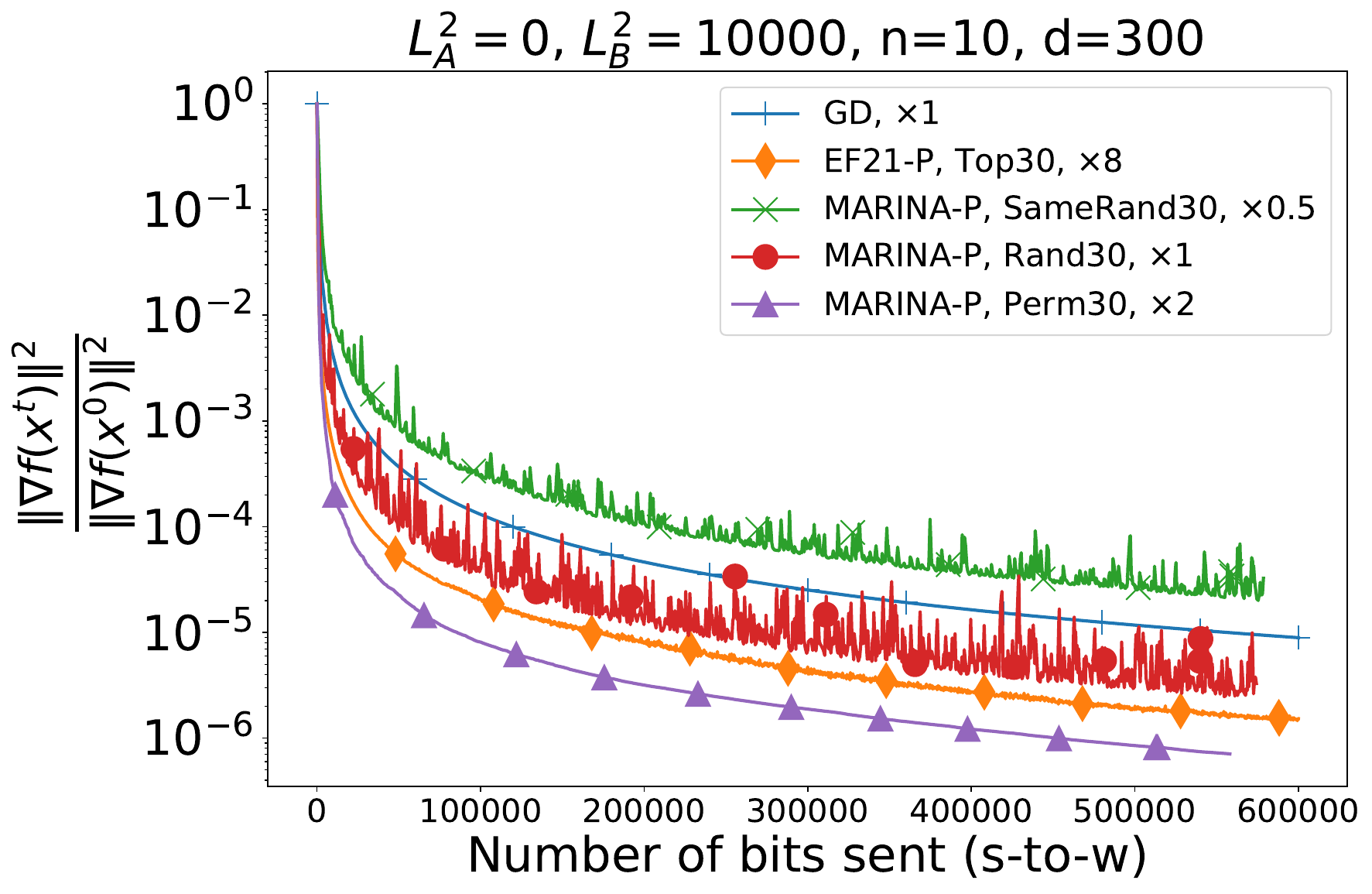}
\end{subfigure}
\begin{subfigure}{0.24\textwidth}
  \centering
  \includegraphics[width=\textwidth]{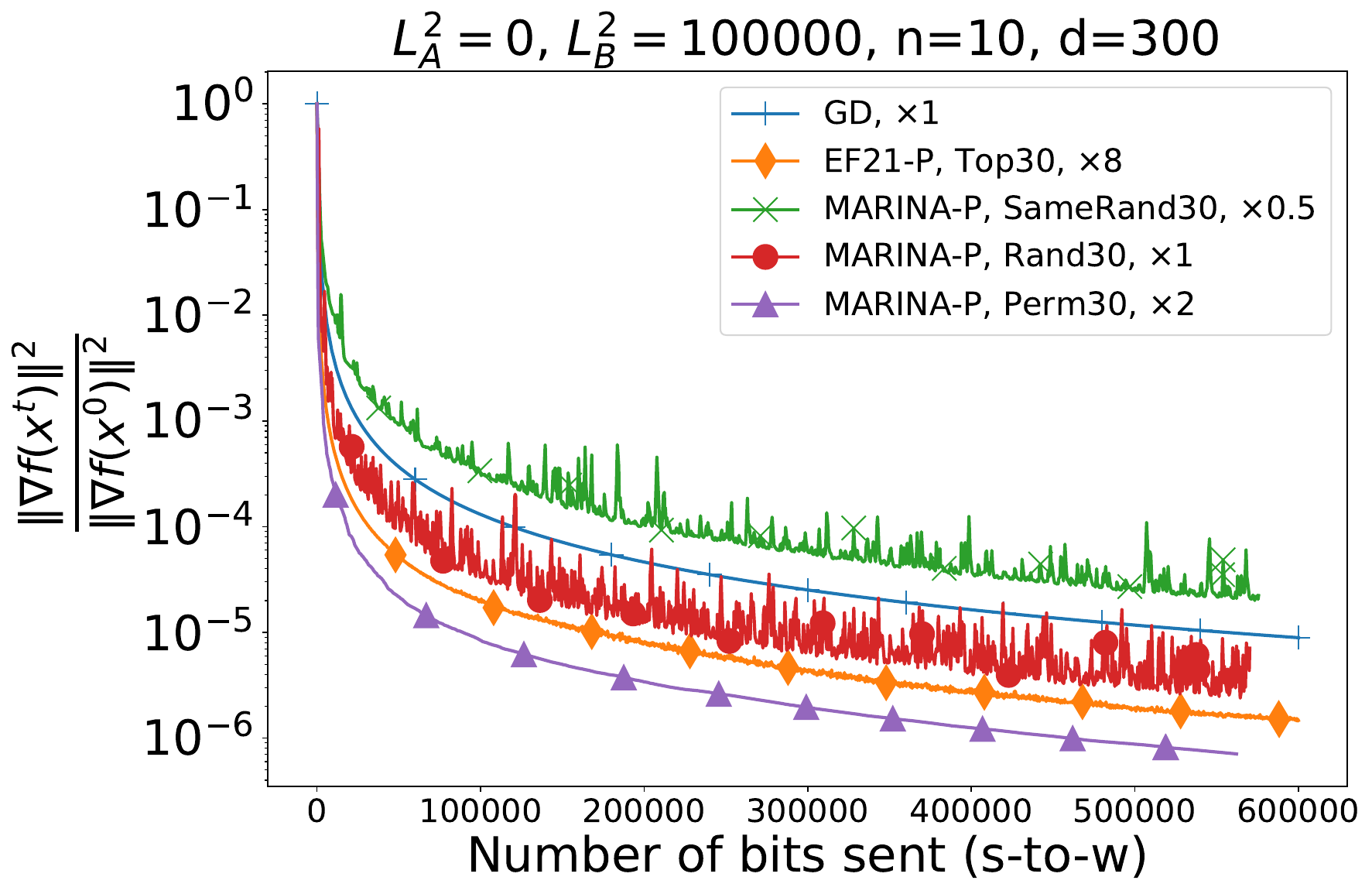}
\end{subfigure}
\begin{subfigure}{0.24\textwidth}
  \centering
  \includegraphics[width=\textwidth]{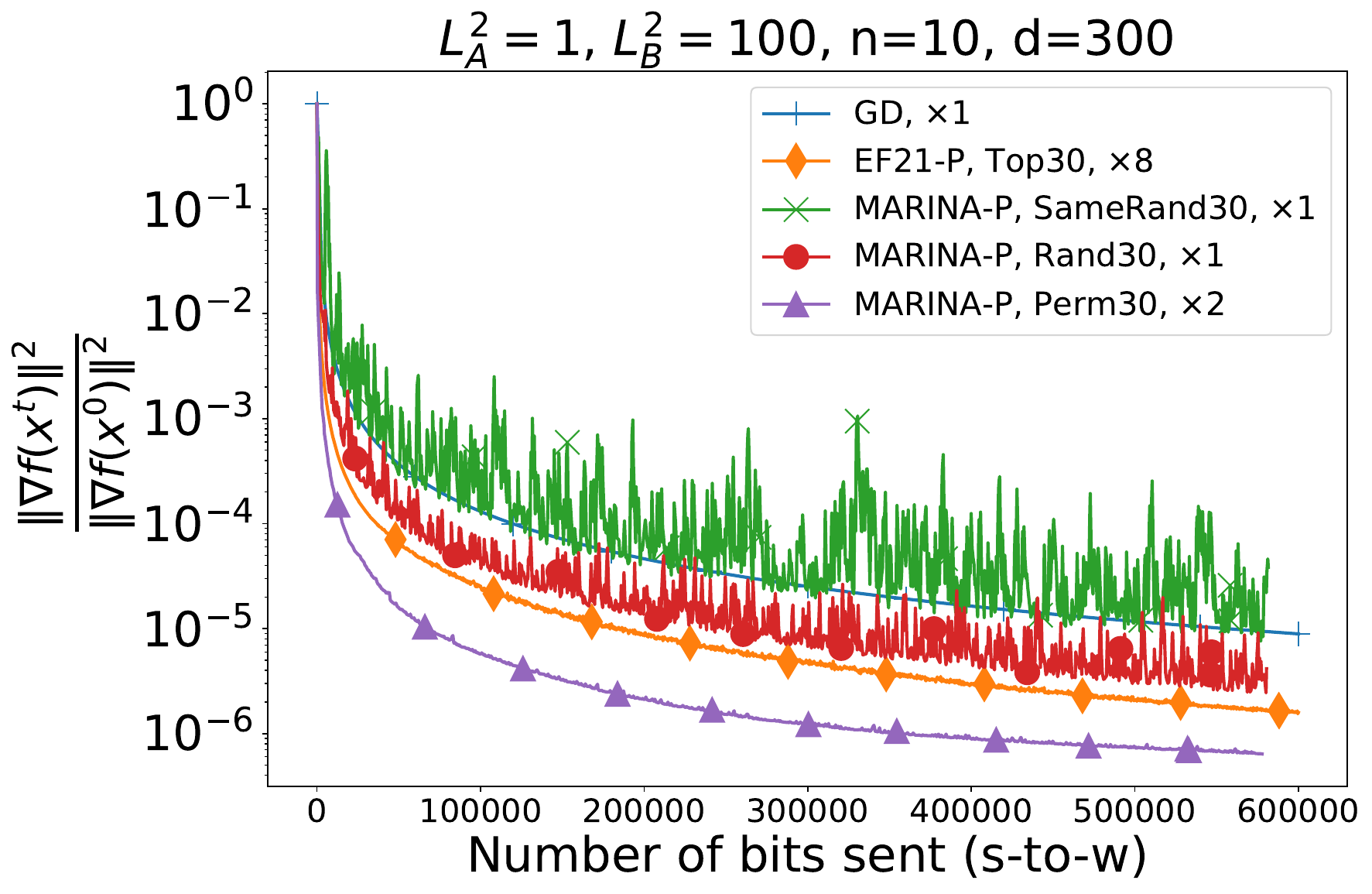}
\end{subfigure}
\begin{subfigure}{0.24\textwidth}
  \centering
  \includegraphics[width=\textwidth]{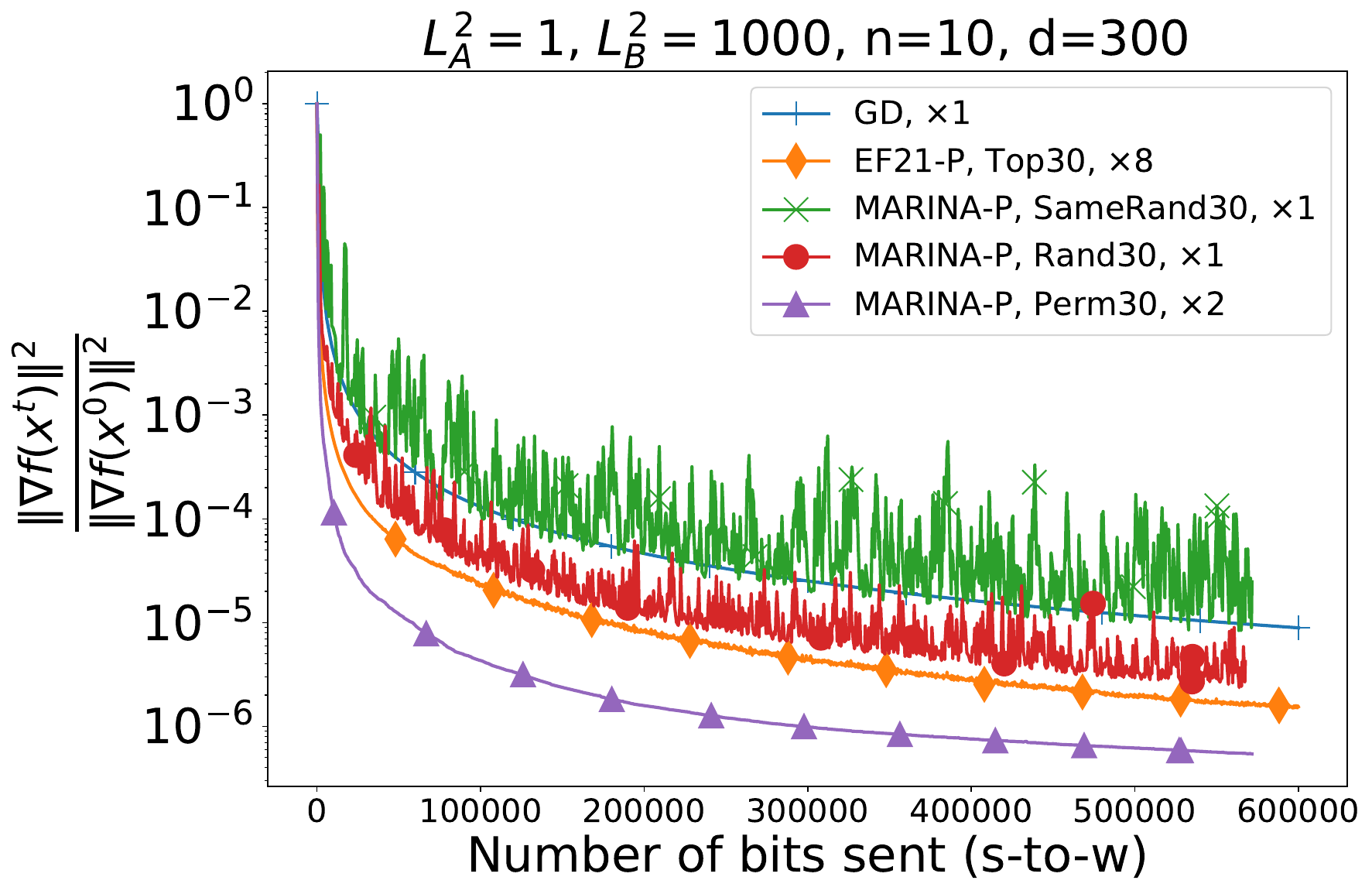}
\end{subfigure}
\begin{subfigure}{0.24\textwidth}
  \centering
  \includegraphics[width=\textwidth]{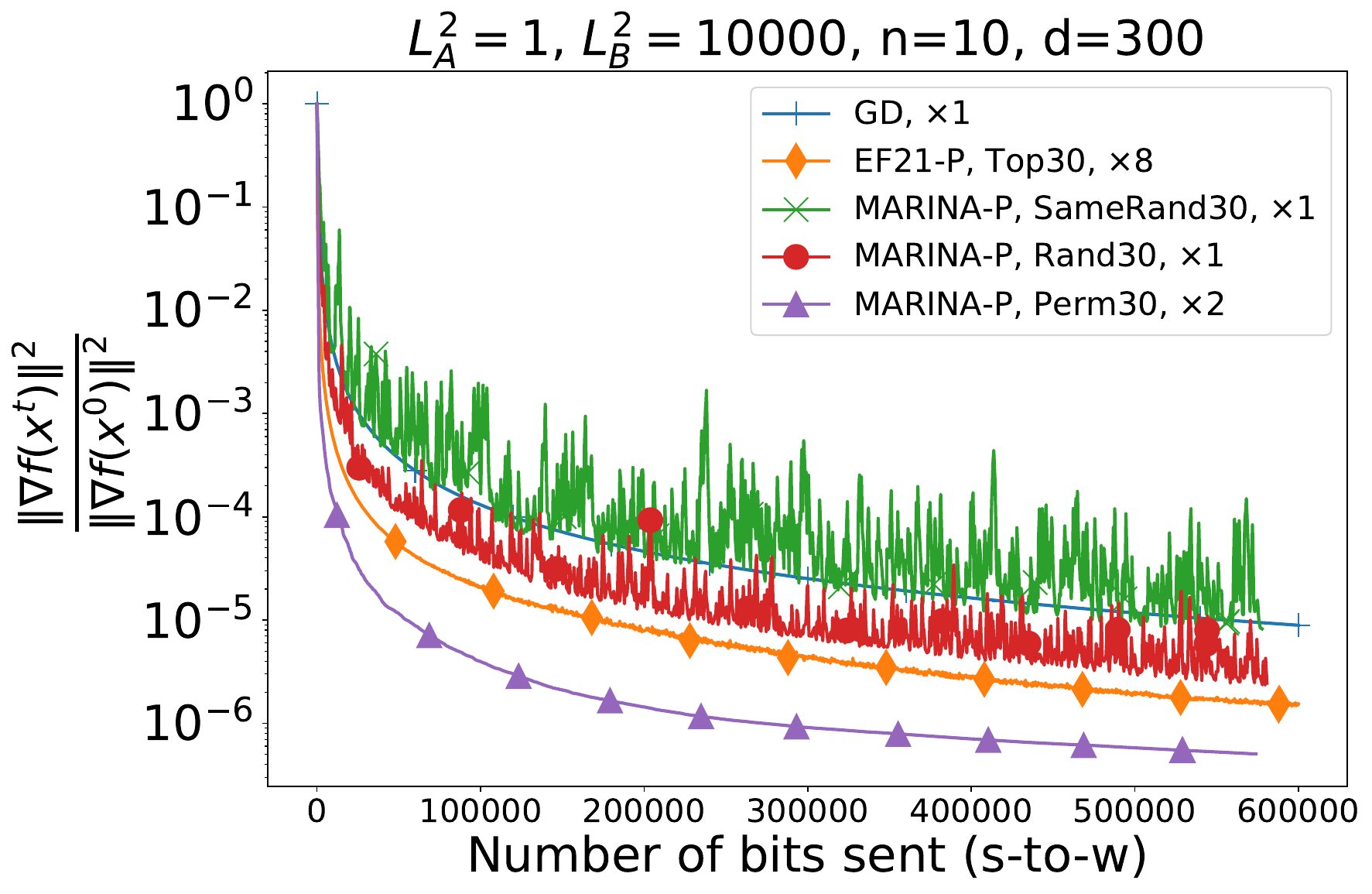}
\end{subfigure}
\begin{subfigure}{0.24\textwidth}
  \centering
  \includegraphics[width=\textwidth]{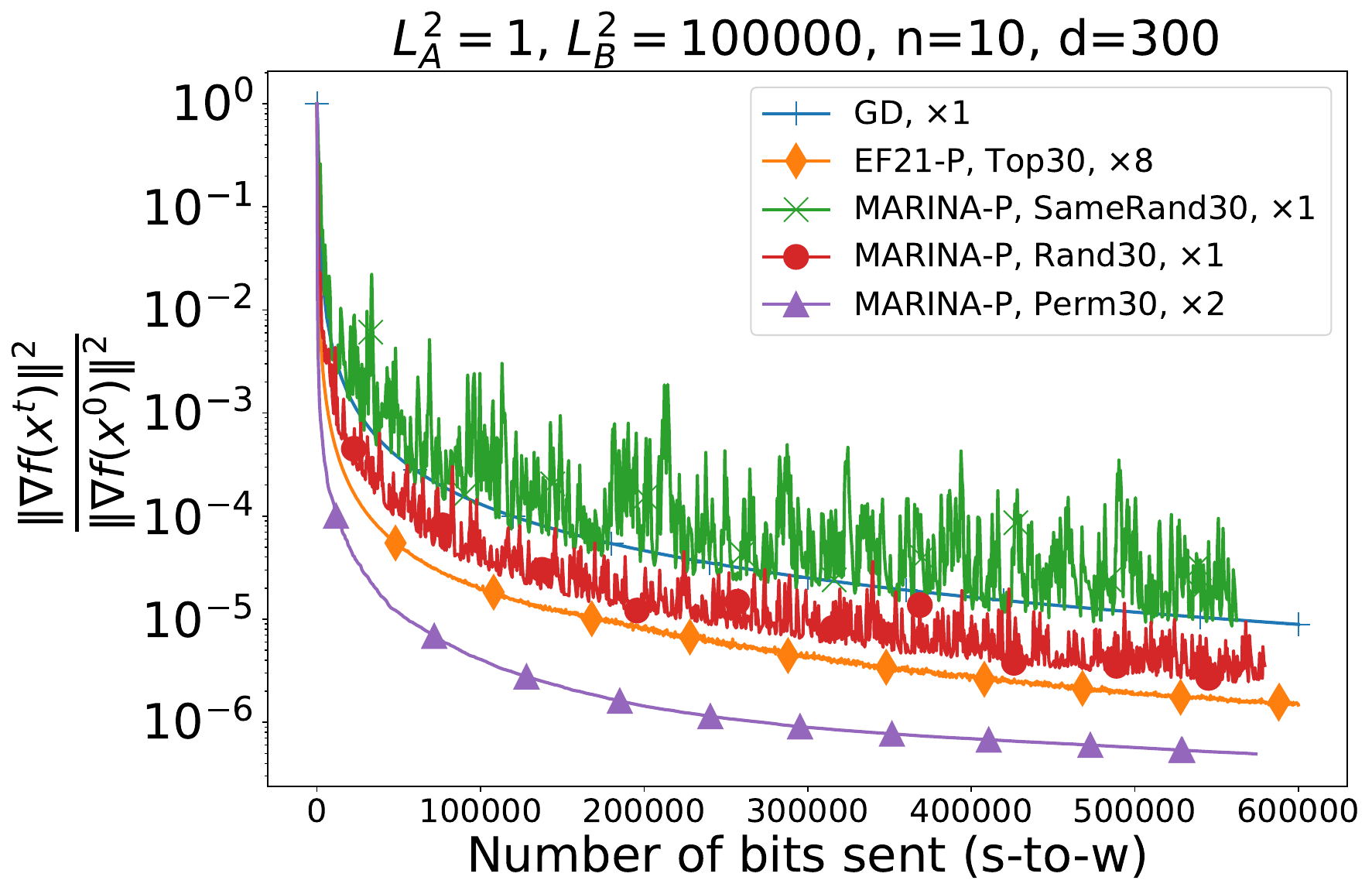}
\end{subfigure}
\begin{subfigure}{0.24\textwidth}
  \centering
  \includegraphics[width=\textwidth]{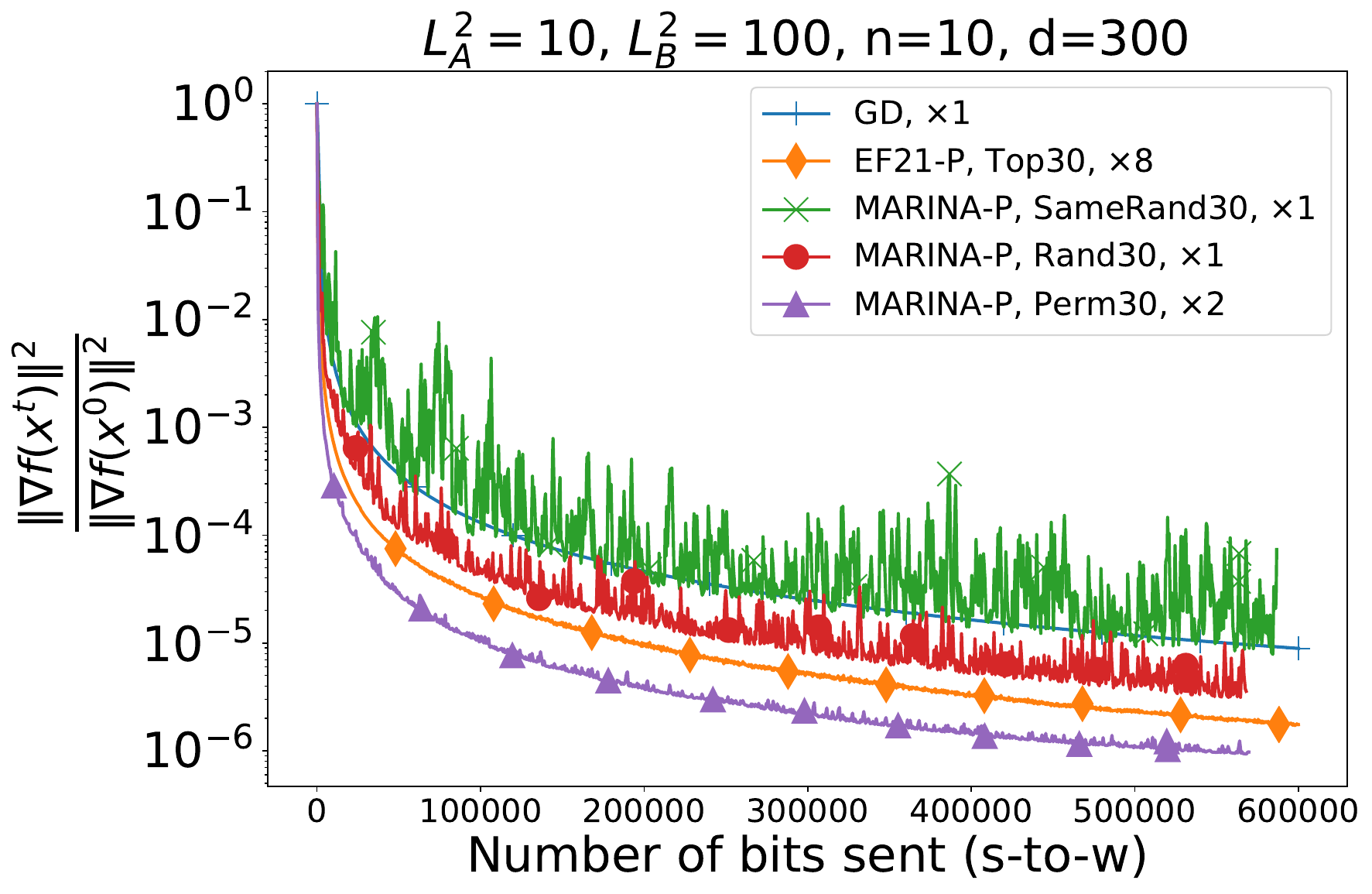}
\end{subfigure}
\begin{subfigure}{0.24\textwidth}
  \centering
  \includegraphics[width=\textwidth]{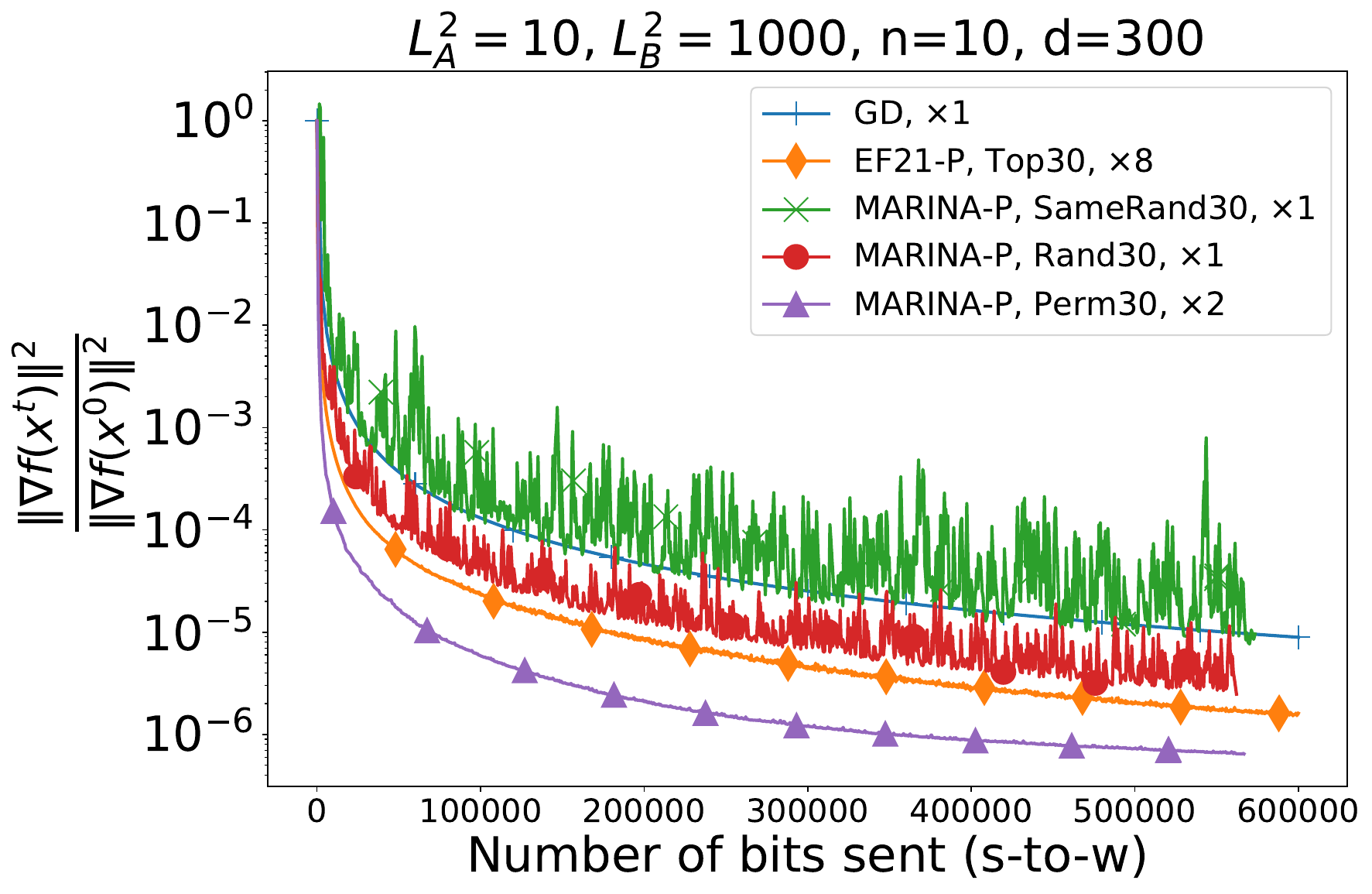}
\end{subfigure}
\begin{subfigure}{0.24\textwidth}
  \centering
  \includegraphics[width=\textwidth]{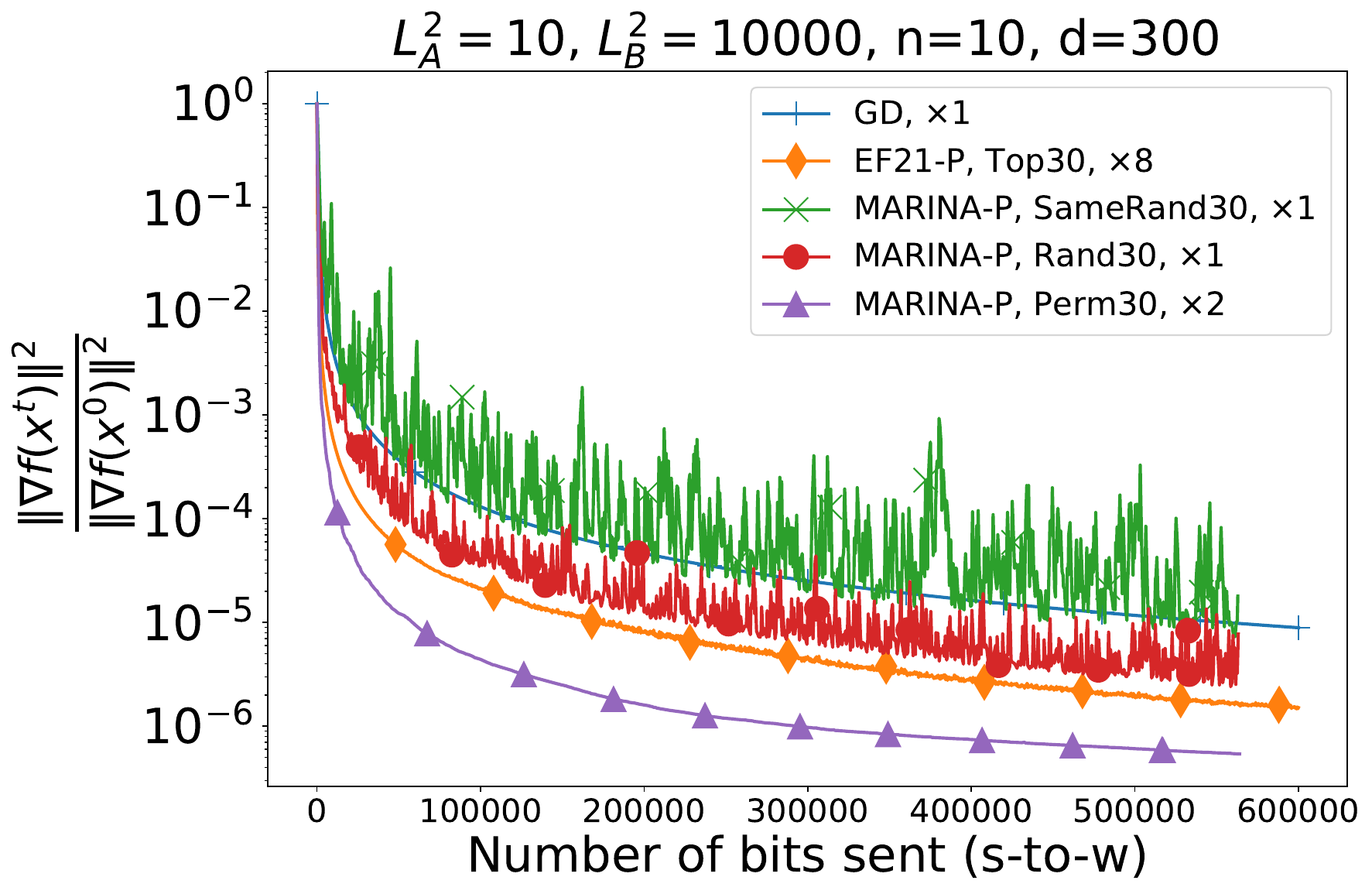}
\end{subfigure}
\begin{subfigure}{0.24\textwidth}
  \centering
  \includegraphics[width=\textwidth]{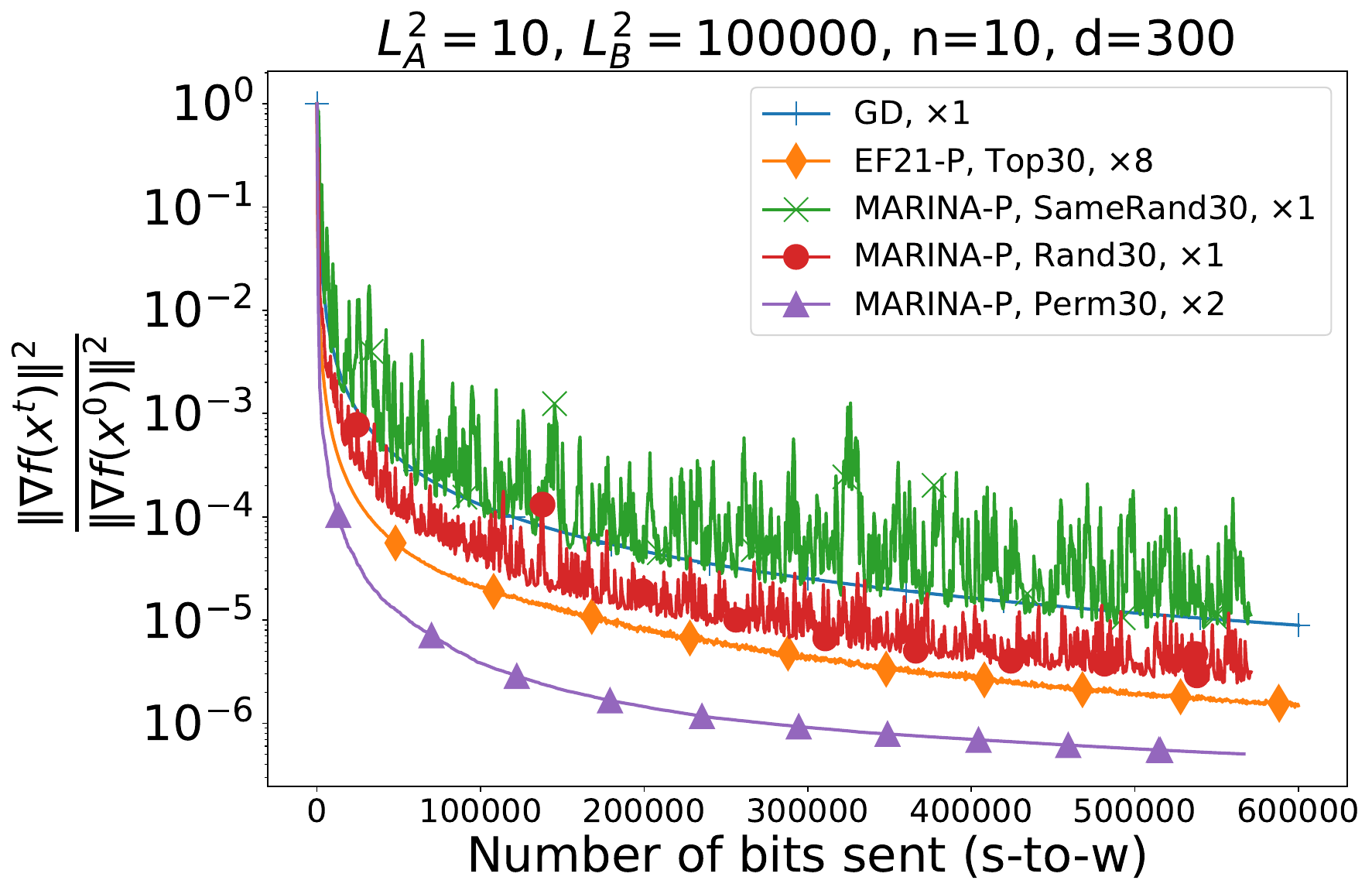}
\end{subfigure}
\begin{subfigure}{0.24\textwidth}
  \centering
  \includegraphics[width=\textwidth]{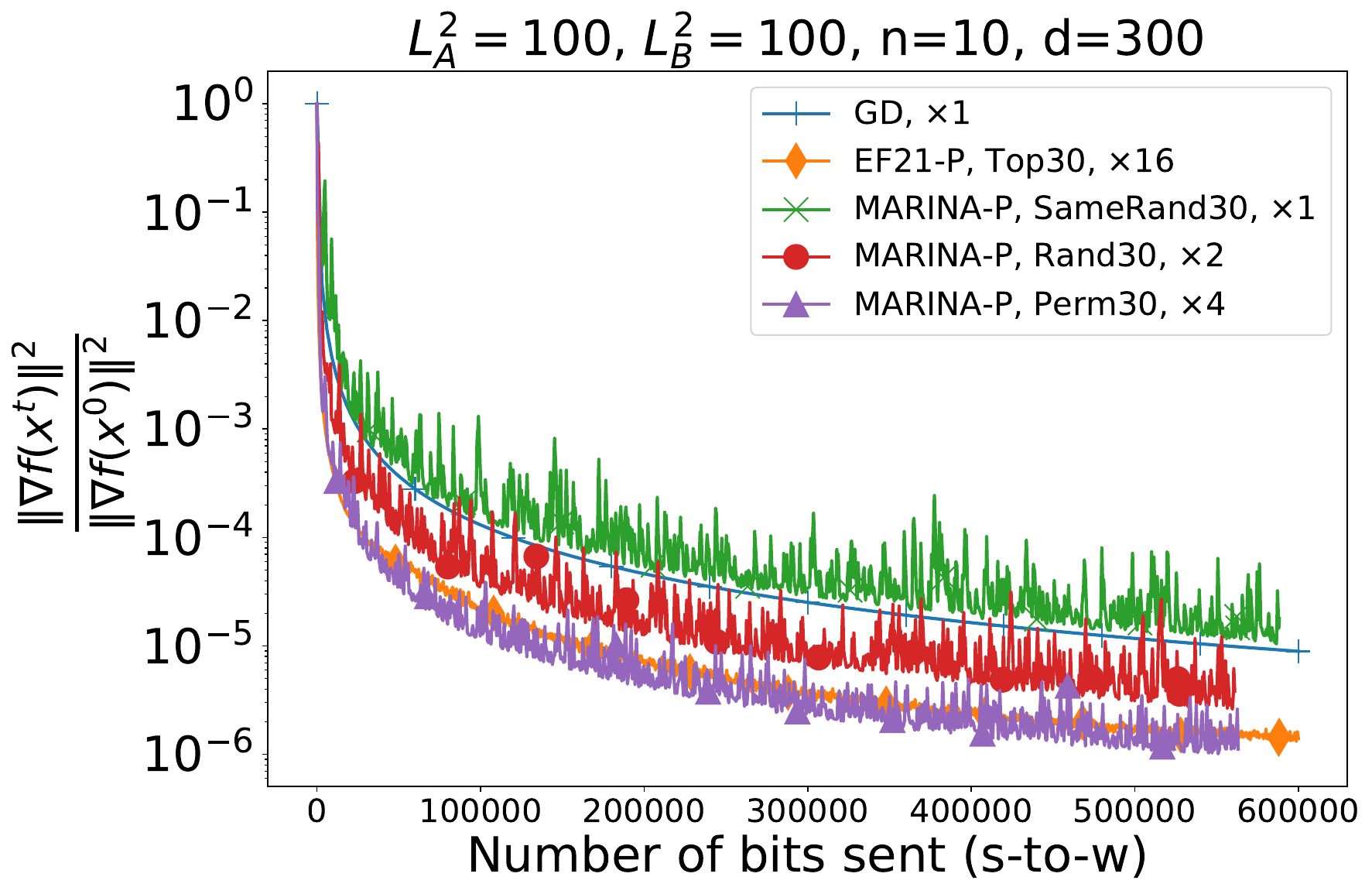}
\end{subfigure}
\begin{subfigure}{0.24\textwidth}
  \centering
  \includegraphics[width=\textwidth]{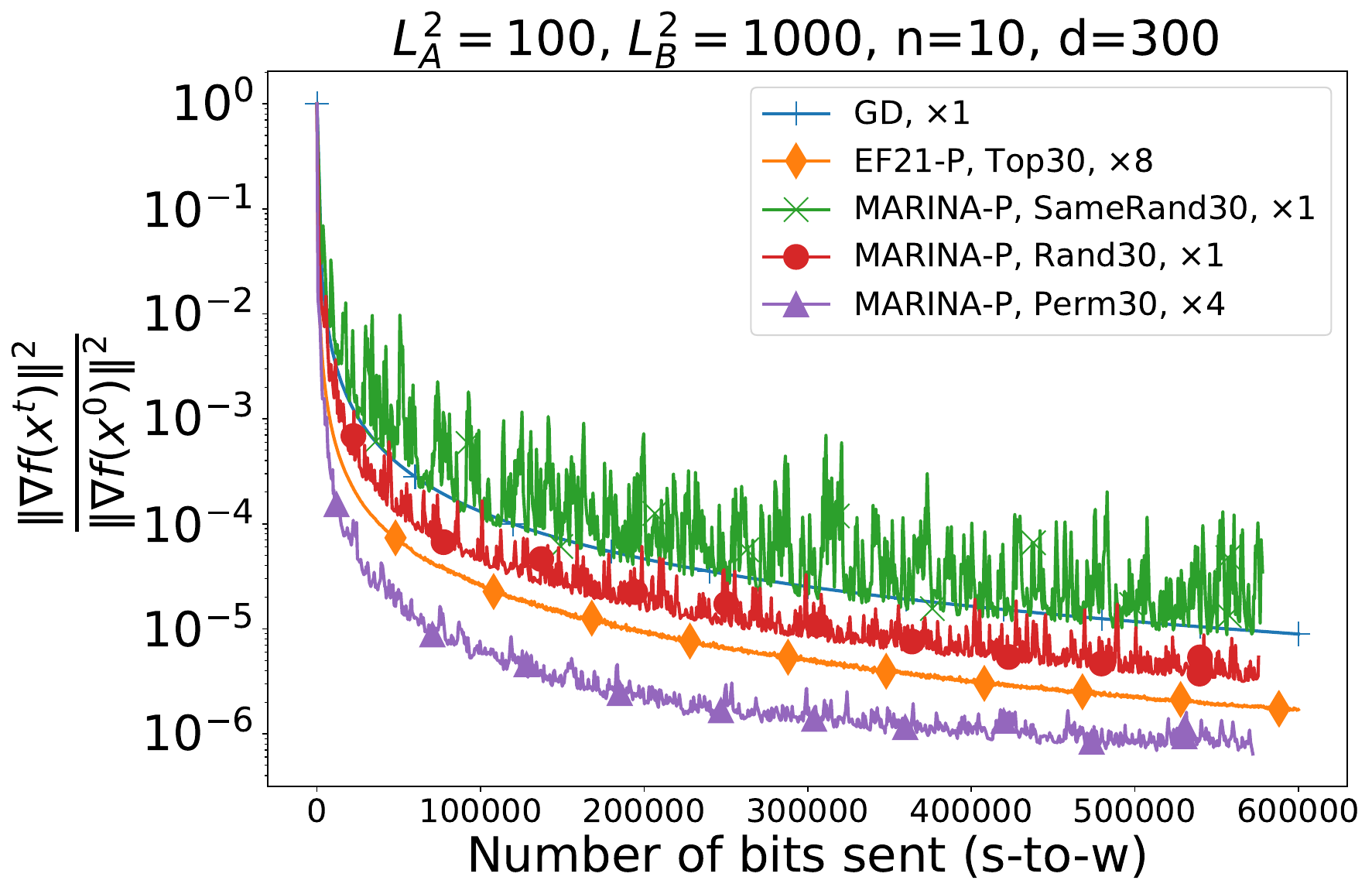}
\end{subfigure}
\begin{subfigure}{0.24\textwidth}
  \centering
  \includegraphics[width=\textwidth]{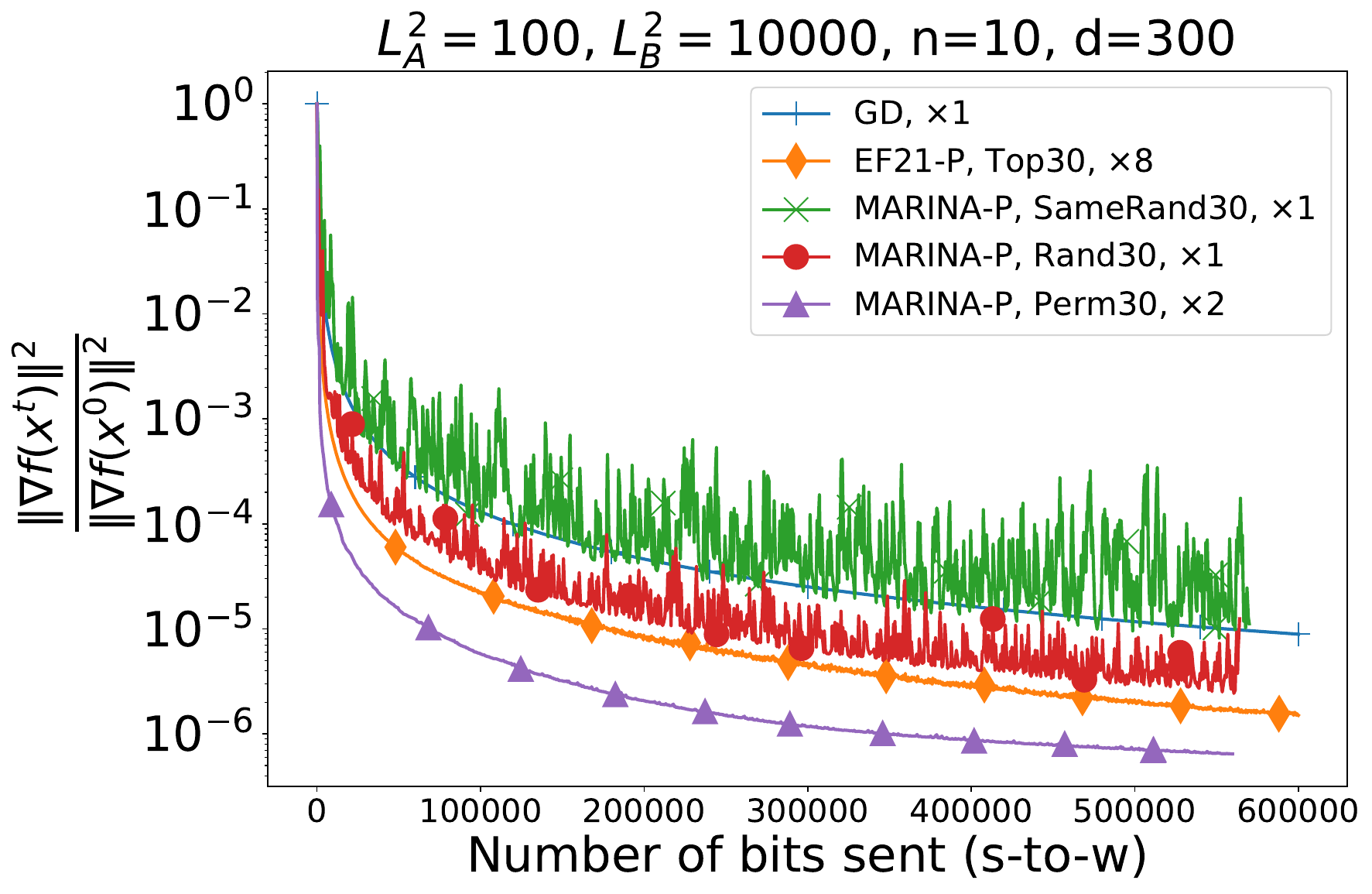}
\end{subfigure}
\begin{subfigure}{0.24\textwidth}
  \centering
  \includegraphics[width=\textwidth]{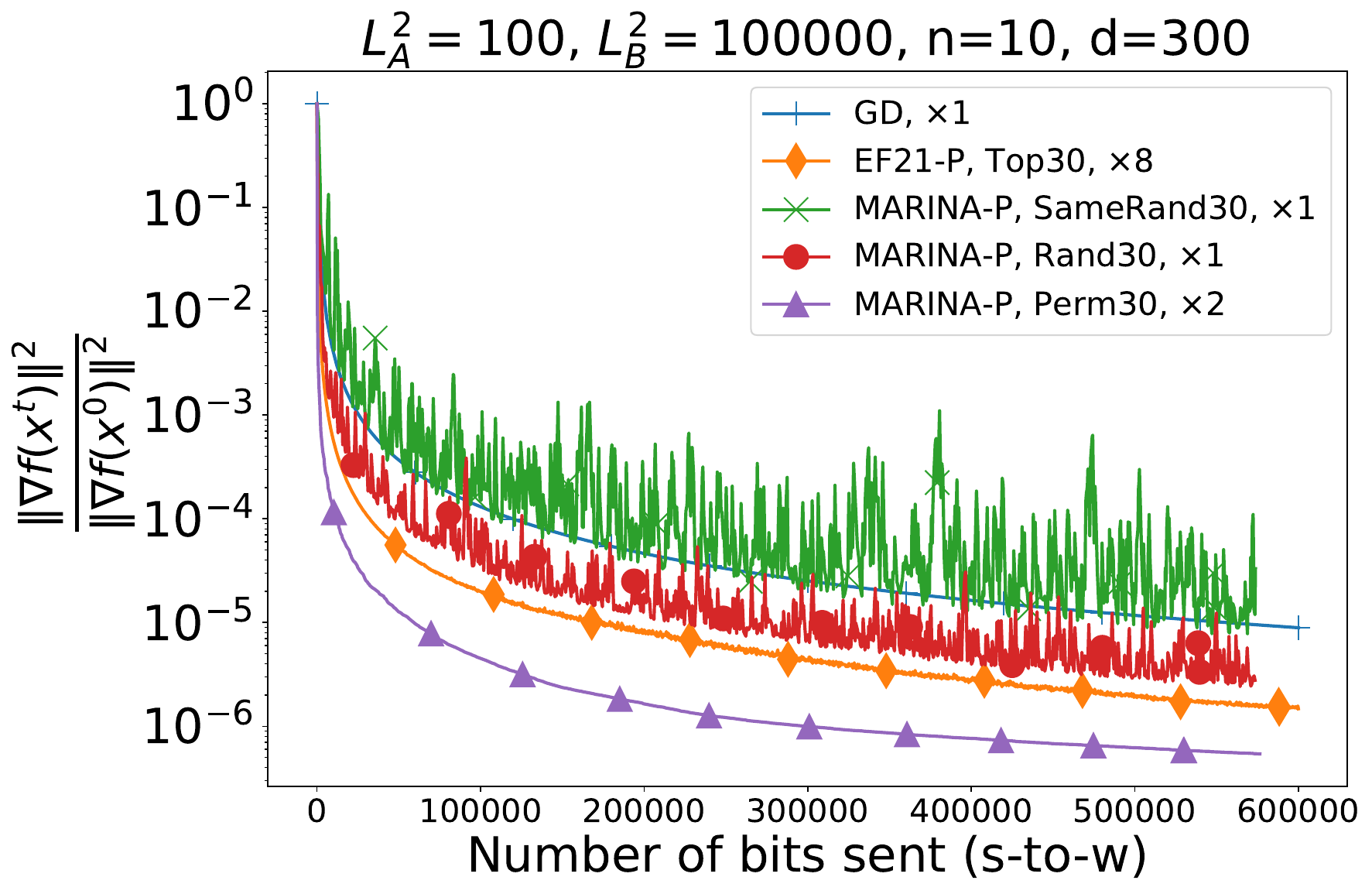}
\end{subfigure}
\caption{Experiments on the quadratic optimization problem from Section~\ref{sec:exp_quad_lalb} with $n=10$ for $L_A^2 \in \brac{0,1,10,100}$ and $L_B^2 \in \brac{100,1000,10000,100000}$.}
\label{fig:grid_n10}
\end{figure}

\begin{figure}[t]
\centering
\begin{subfigure}{0.24\textwidth}
  \centering
  \includegraphics[width=\textwidth]{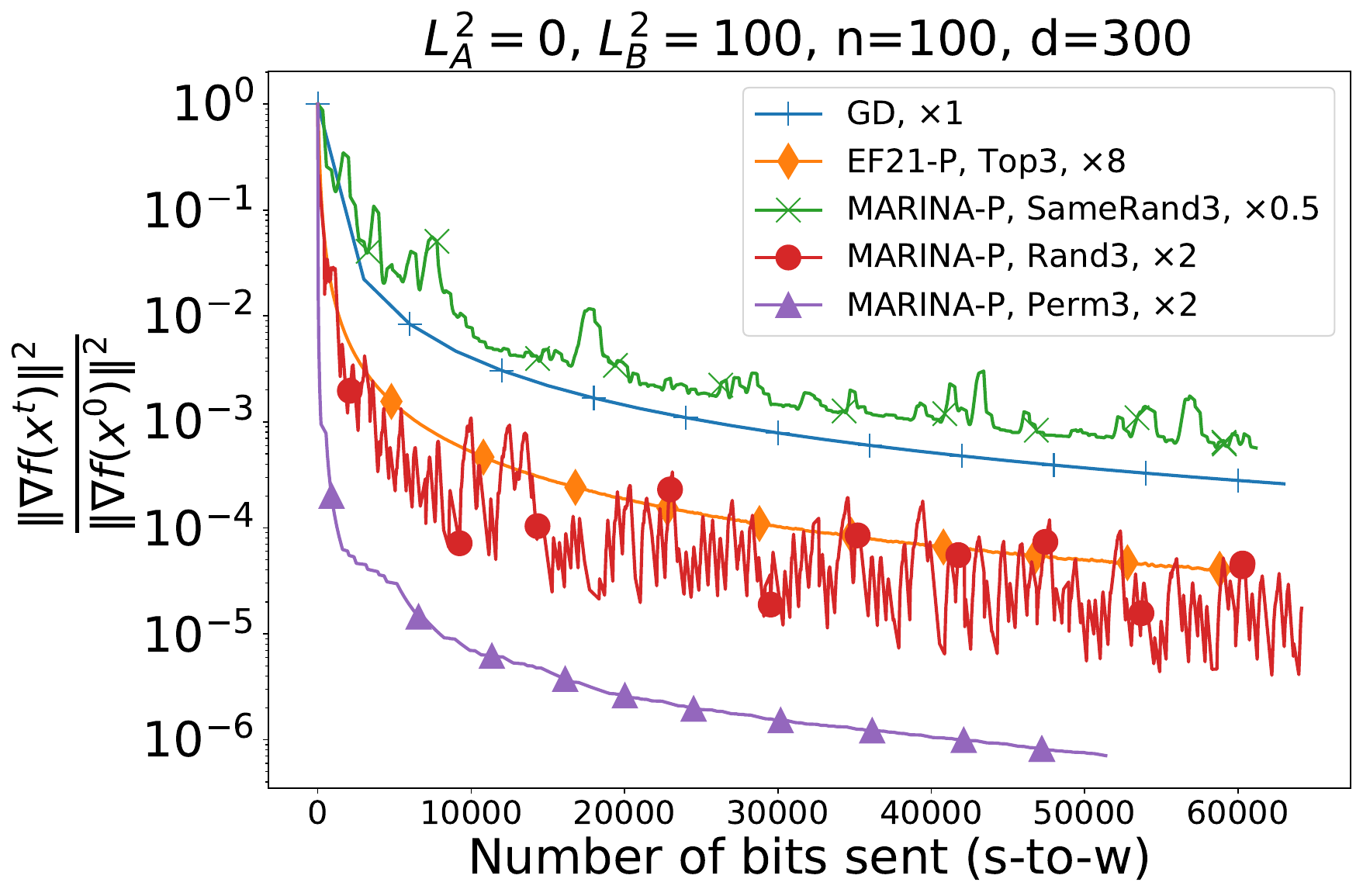}
\end{subfigure}%
\begin{subfigure}{0.24\textwidth}
  \centering
  \includegraphics[width=\textwidth]{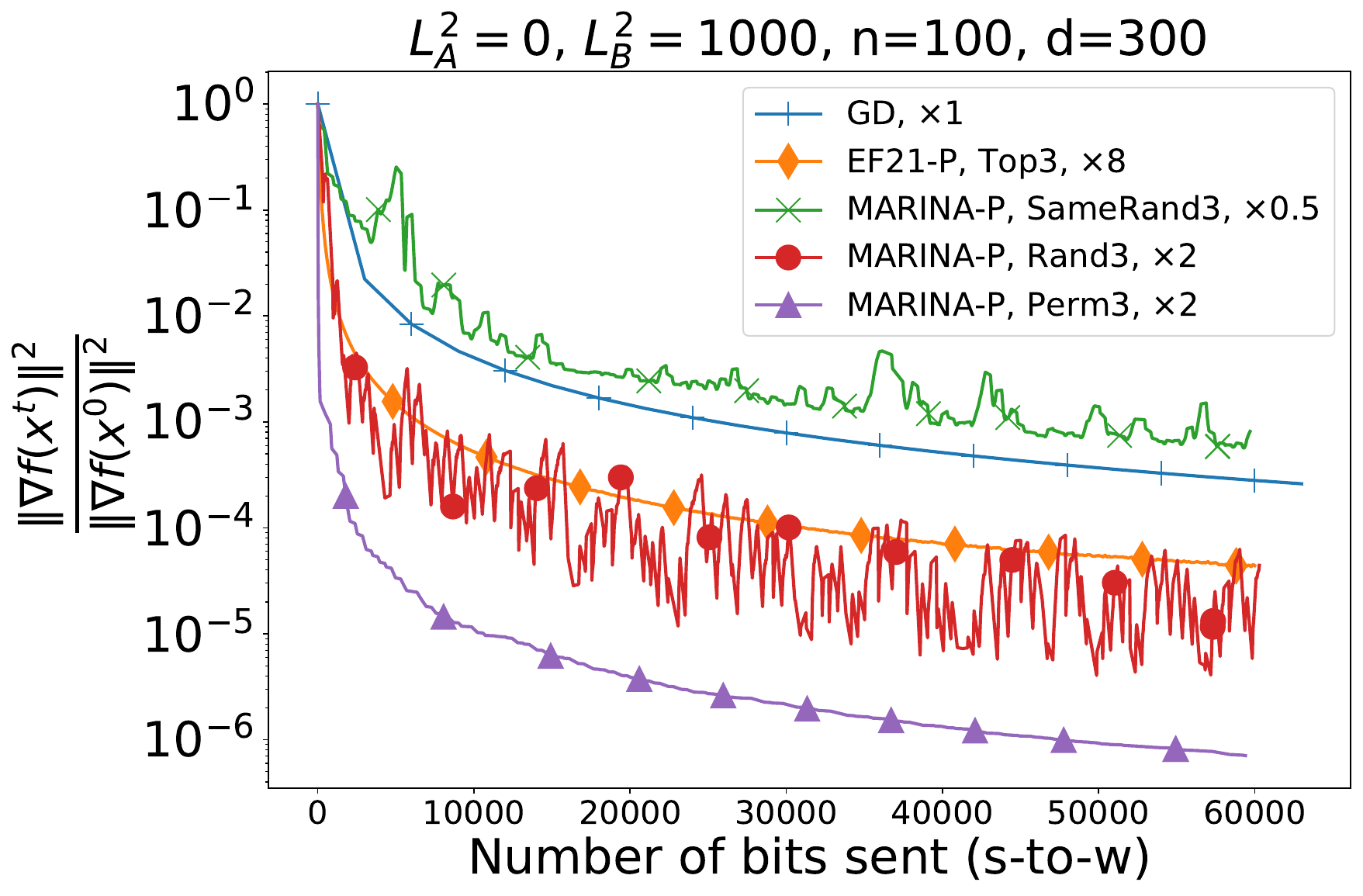}
\end{subfigure}
\begin{subfigure}{0.24\textwidth}
  \centering
  \includegraphics[width=\textwidth]{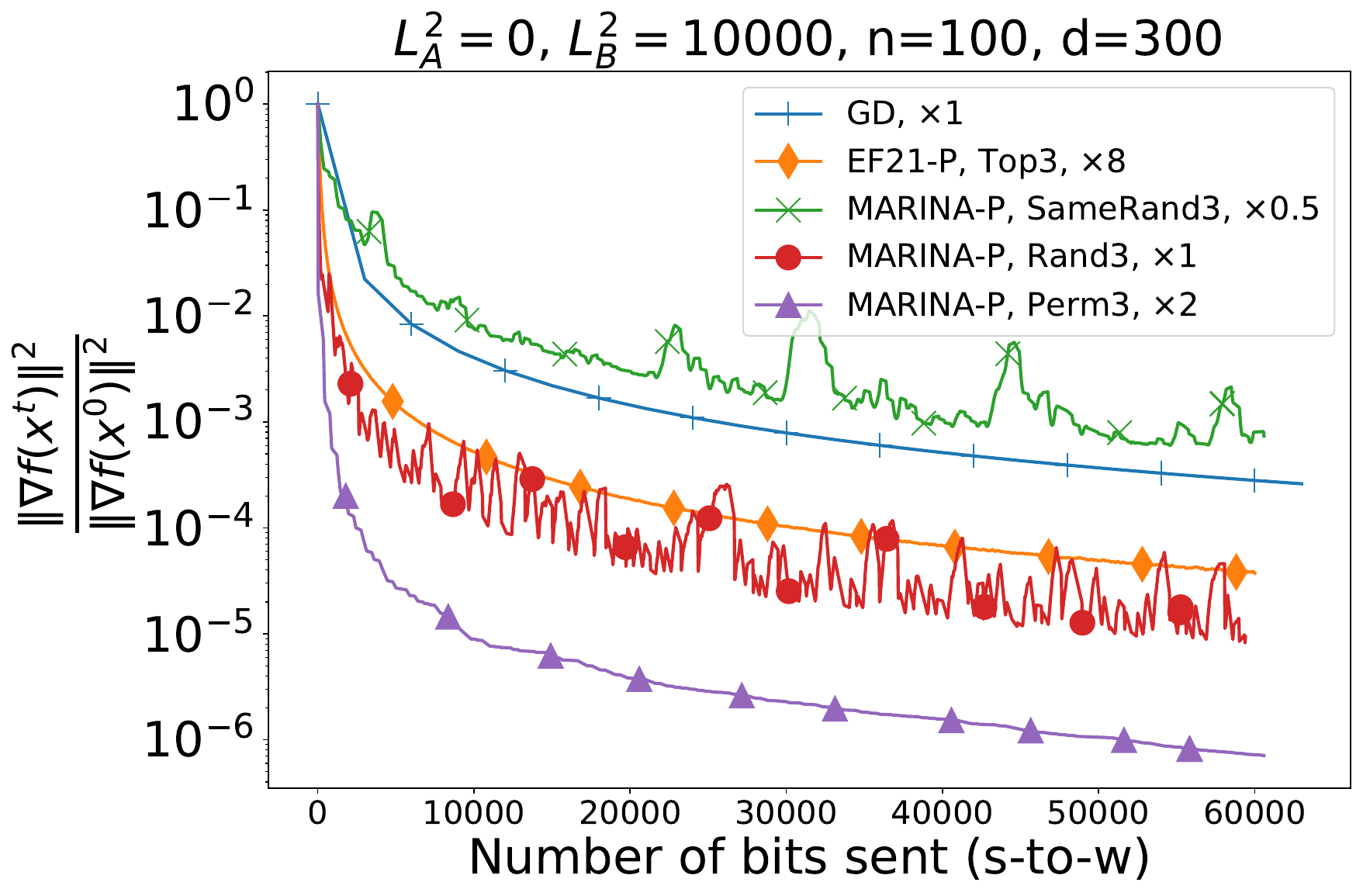}
\end{subfigure}
\begin{subfigure}{0.24\textwidth}
  \centering
  \includegraphics[width=\textwidth]{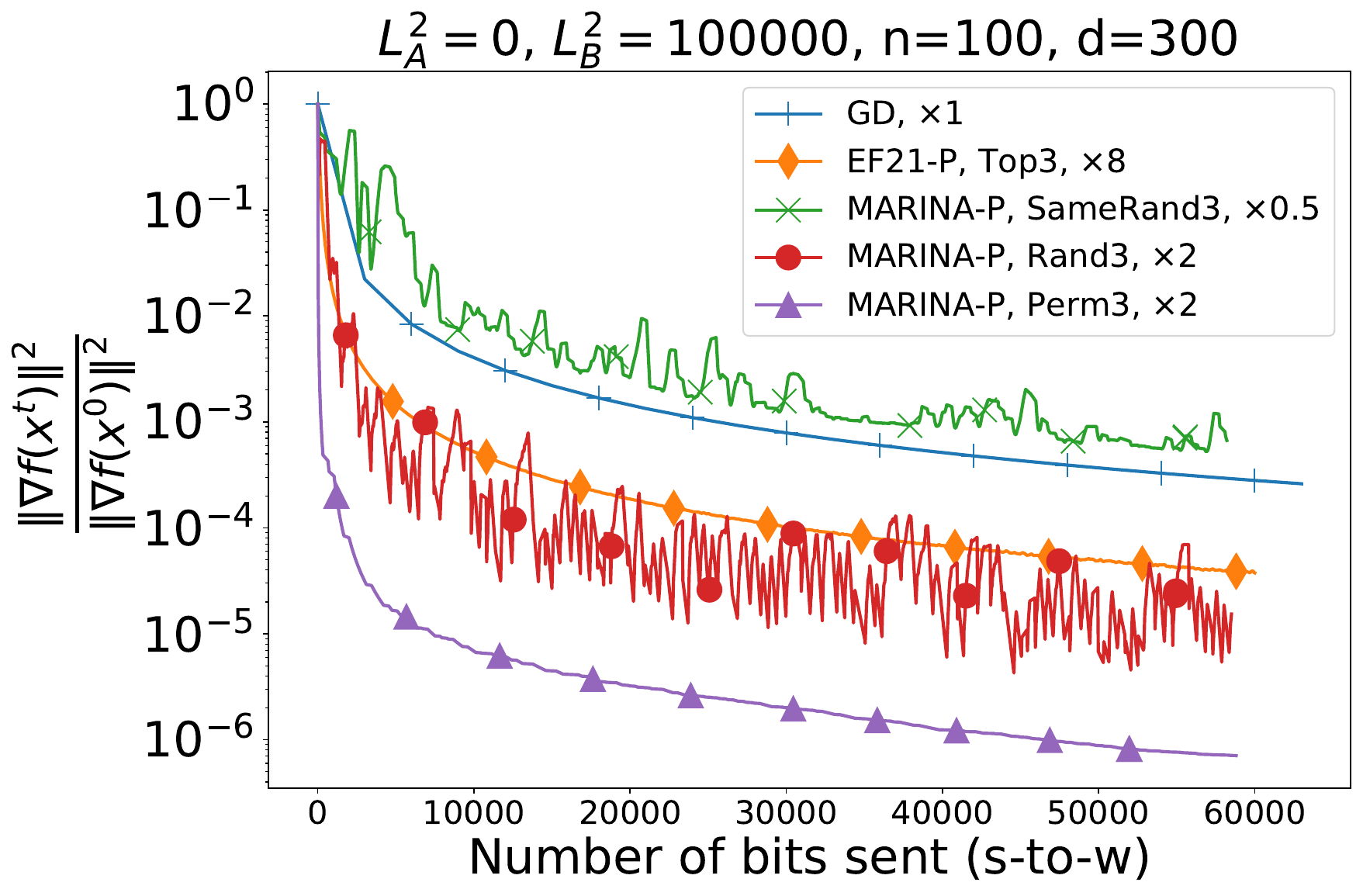}
\end{subfigure}
\begin{subfigure}{0.24\textwidth}
  \centering
  \includegraphics[width=\textwidth]{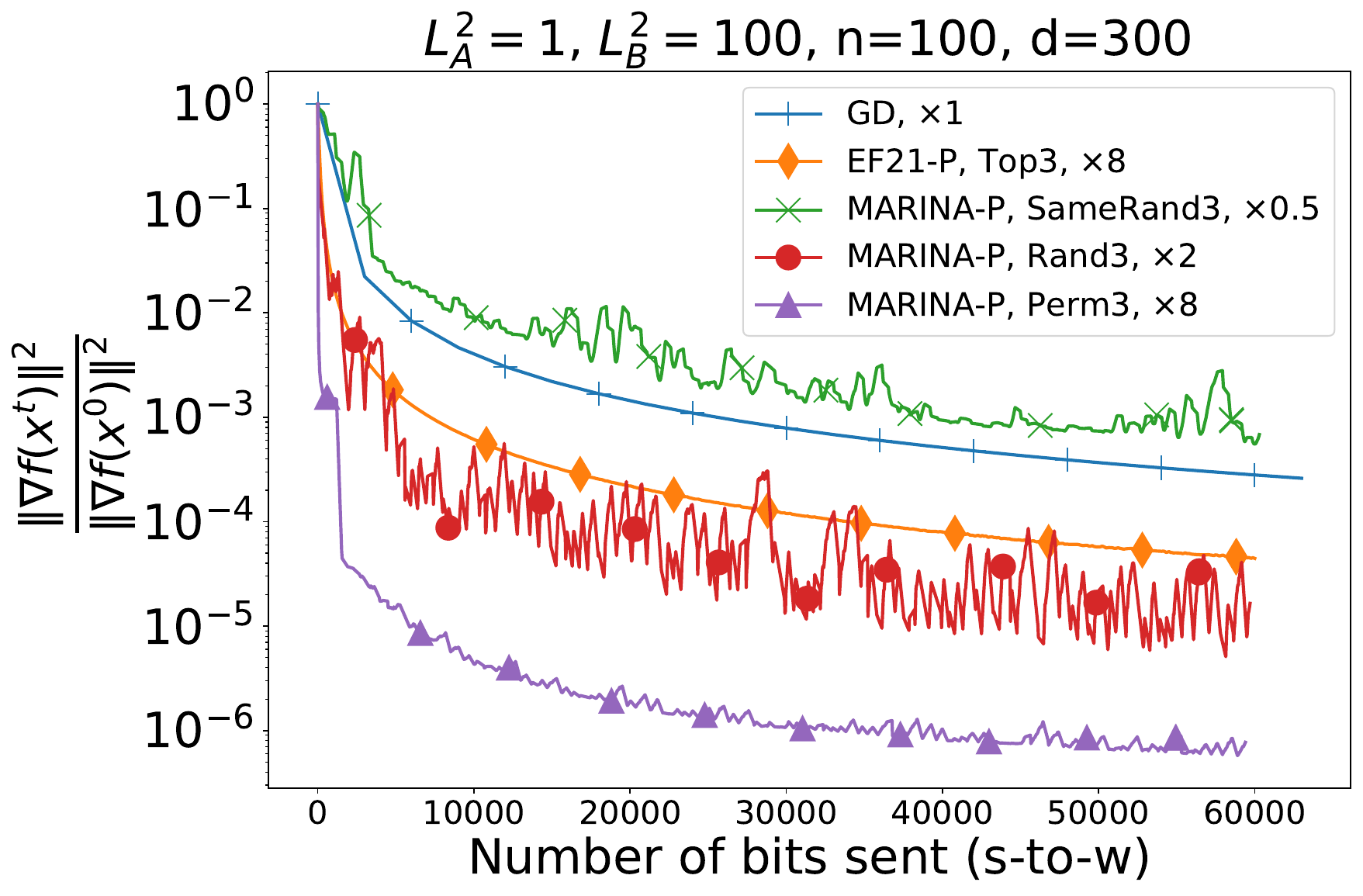}
\end{subfigure}
\begin{subfigure}{0.24\textwidth}
  \centering
  \includegraphics[width=\textwidth]{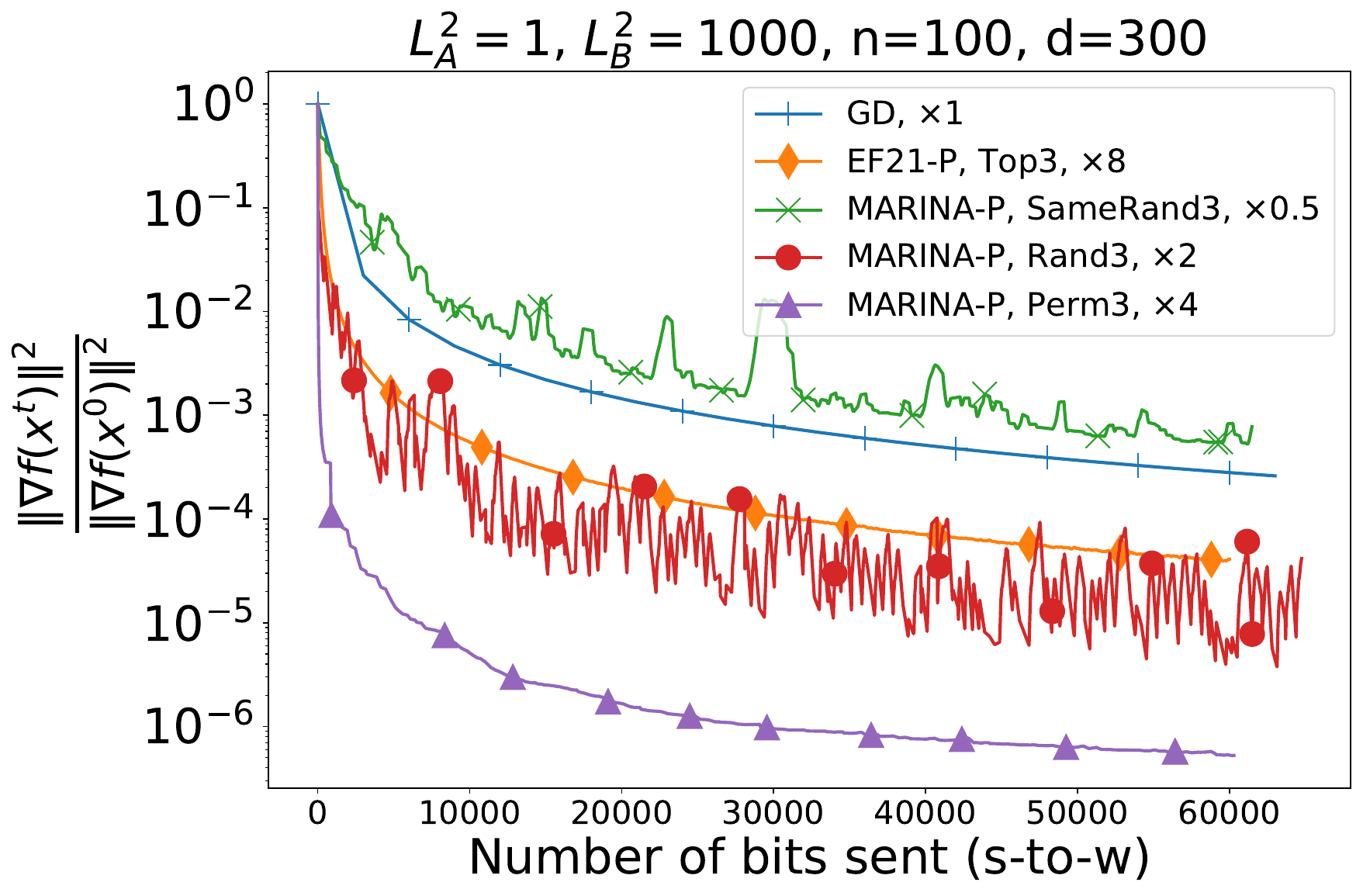}
\end{subfigure}
\begin{subfigure}{0.24\textwidth}
  \centering
  \includegraphics[width=\textwidth]{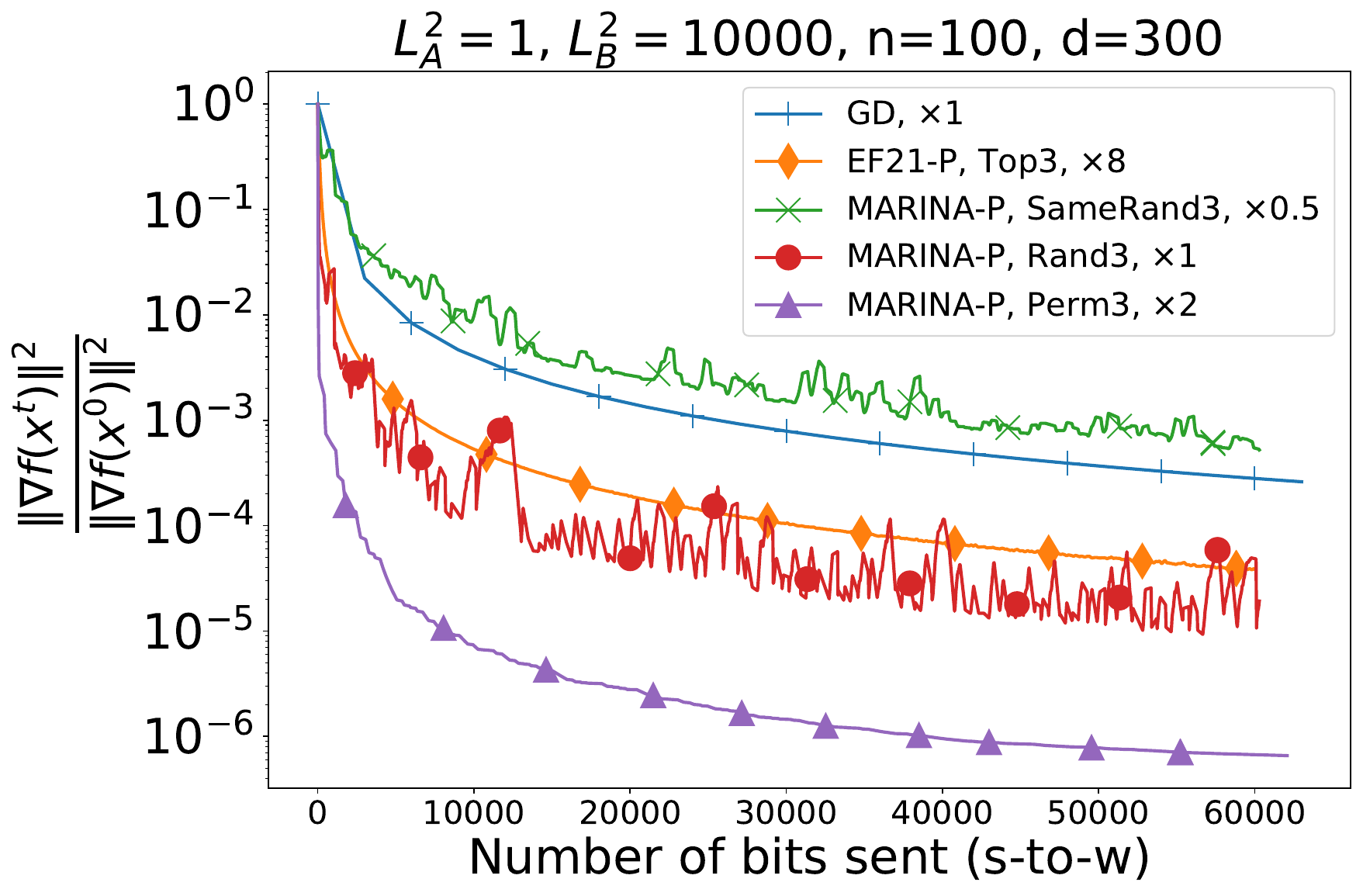}
\end{subfigure}
\begin{subfigure}{0.24\textwidth}
  \centering
  \includegraphics[width=\textwidth]{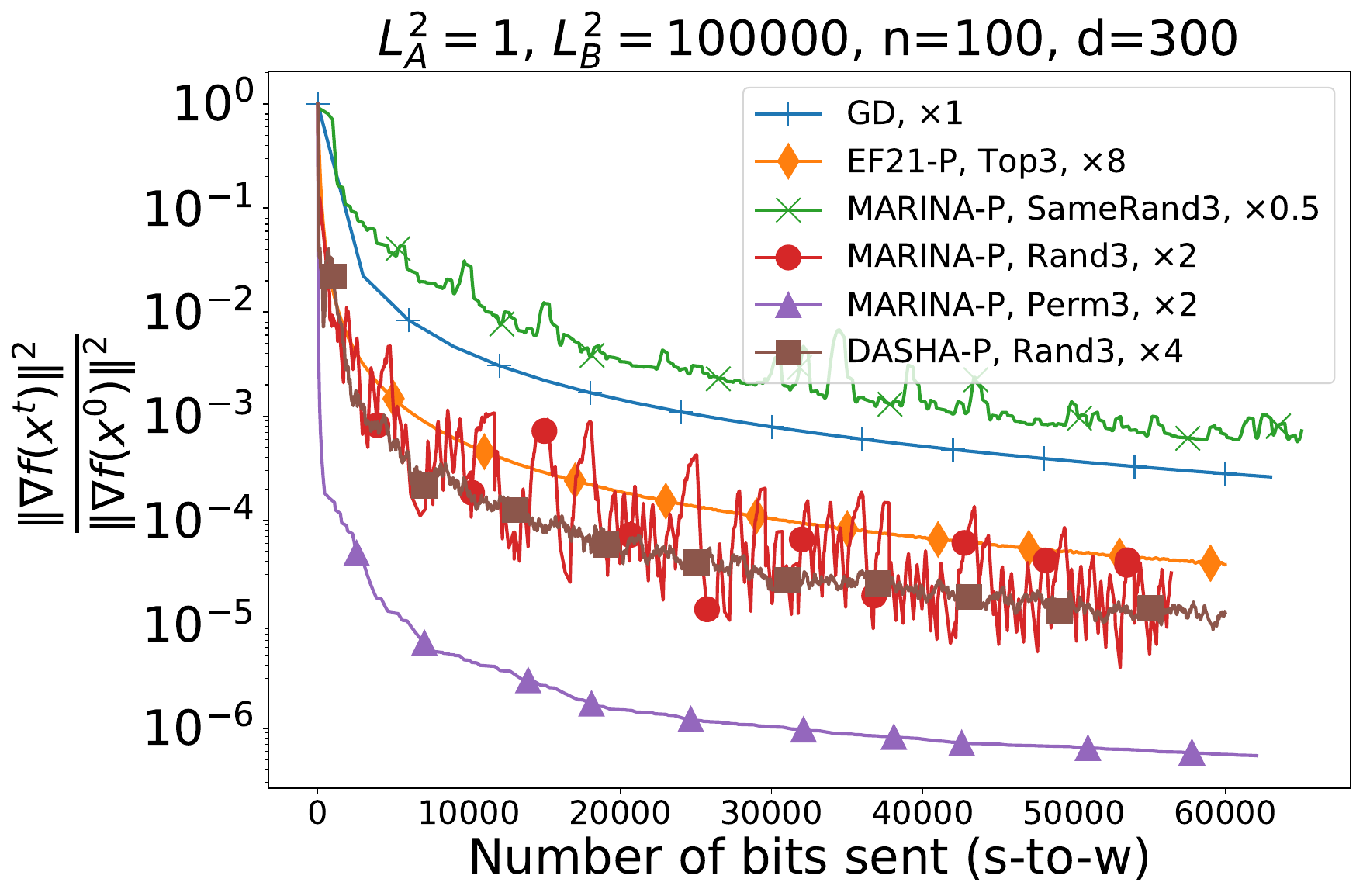}
\end{subfigure}
\begin{subfigure}{0.24\textwidth}
  \centering
  \includegraphics[width=\textwidth]{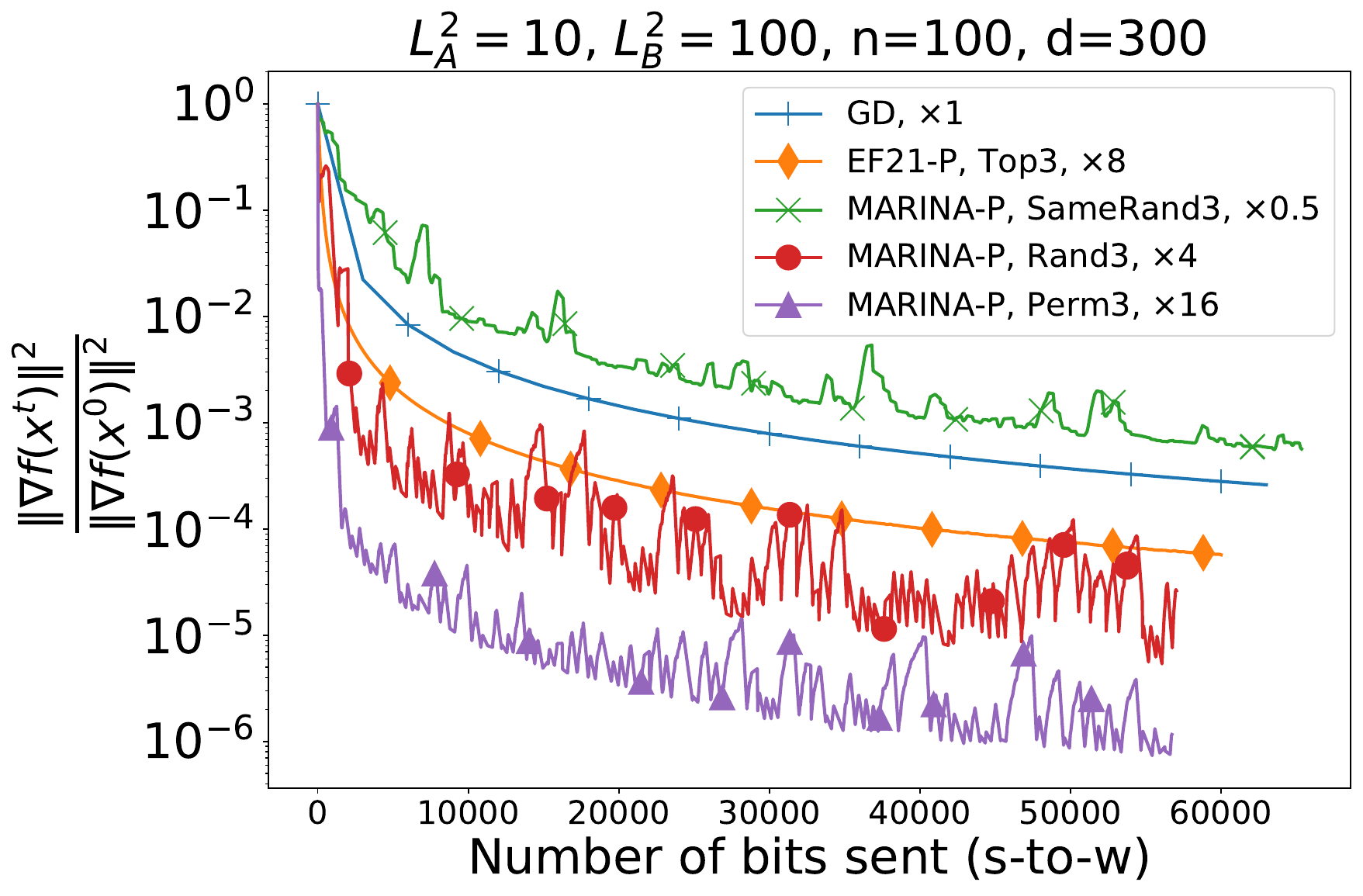}
\end{subfigure}
\begin{subfigure}{0.24\textwidth}
  \centering
  \includegraphics[width=\textwidth]{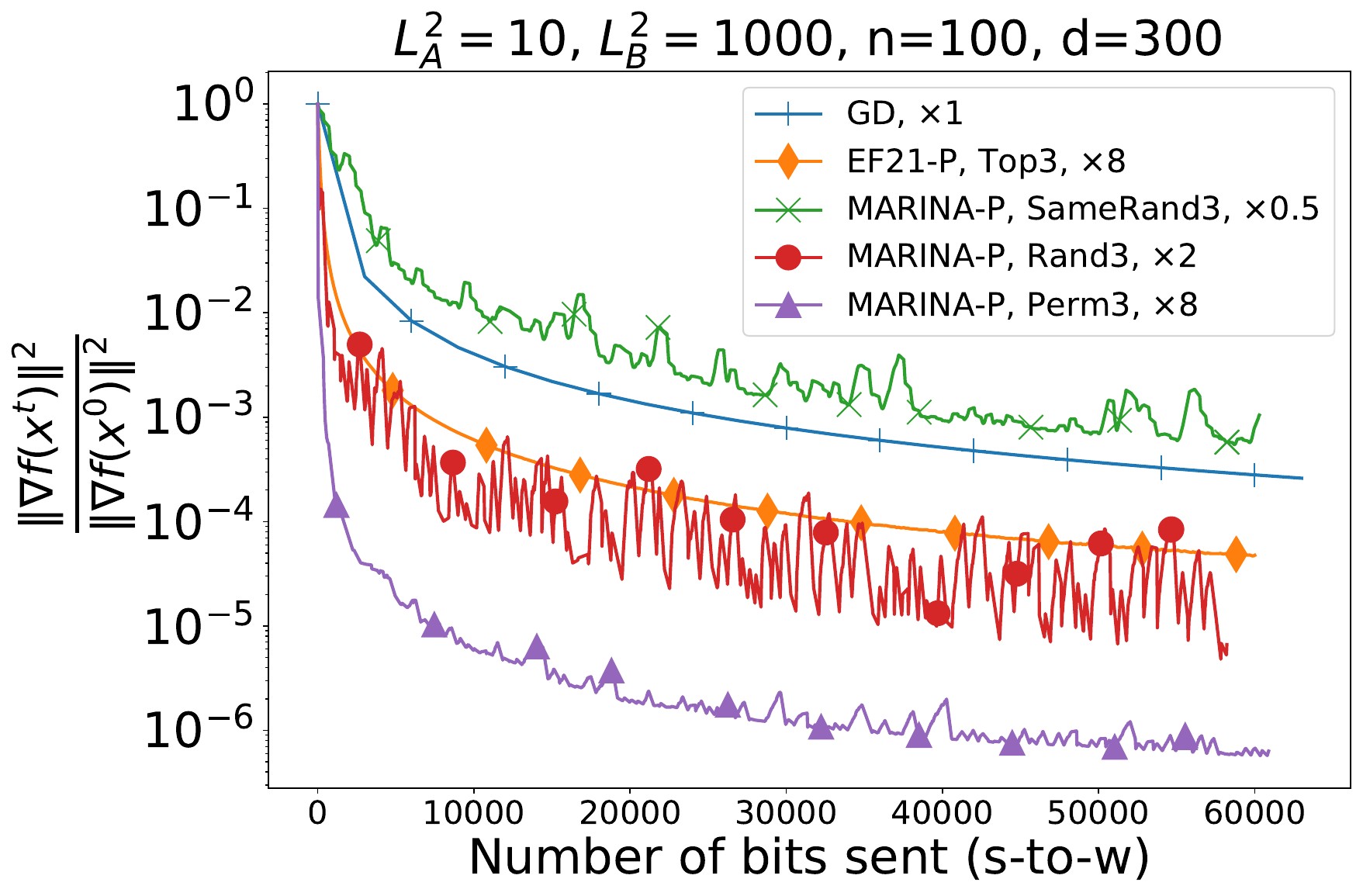}
\end{subfigure}
\begin{subfigure}{0.24\textwidth}
  \centering
  \includegraphics[width=\textwidth]{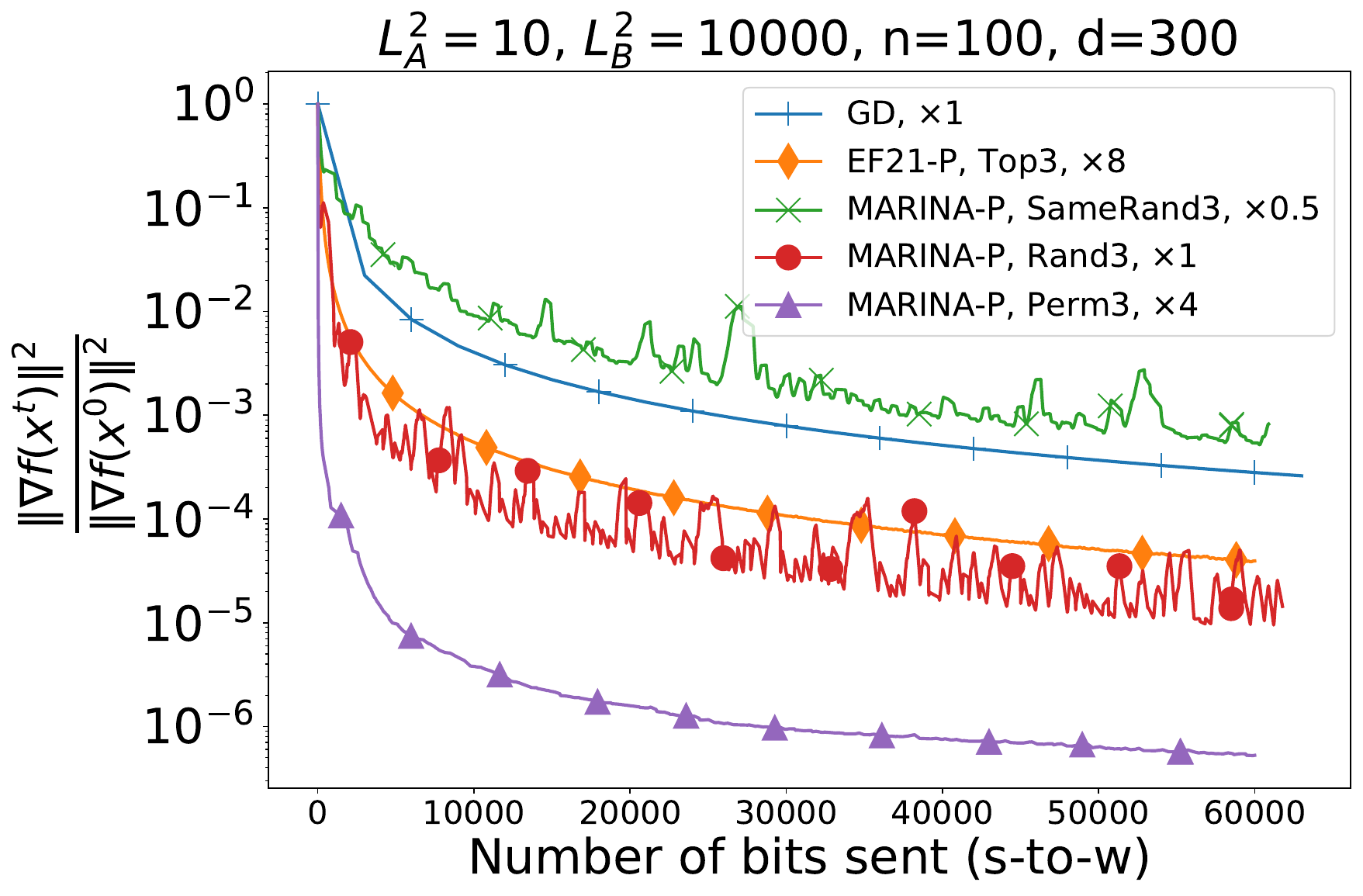}
\end{subfigure}
\begin{subfigure}{0.24\textwidth}
  \centering
  \includegraphics[width=\textwidth]{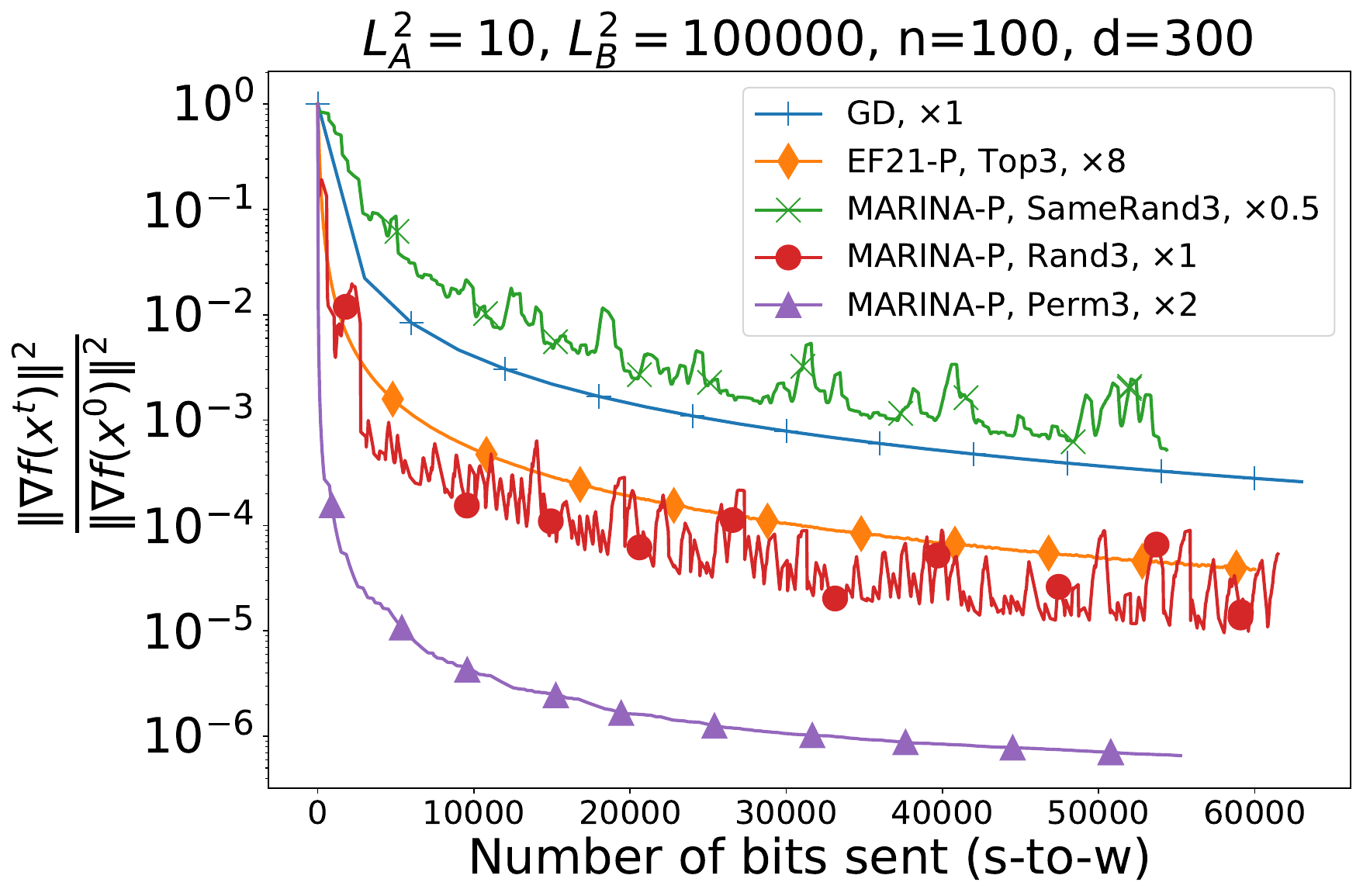}
\end{subfigure}
\begin{subfigure}{0.24\textwidth}
  \centering
  \includegraphics[width=\textwidth]{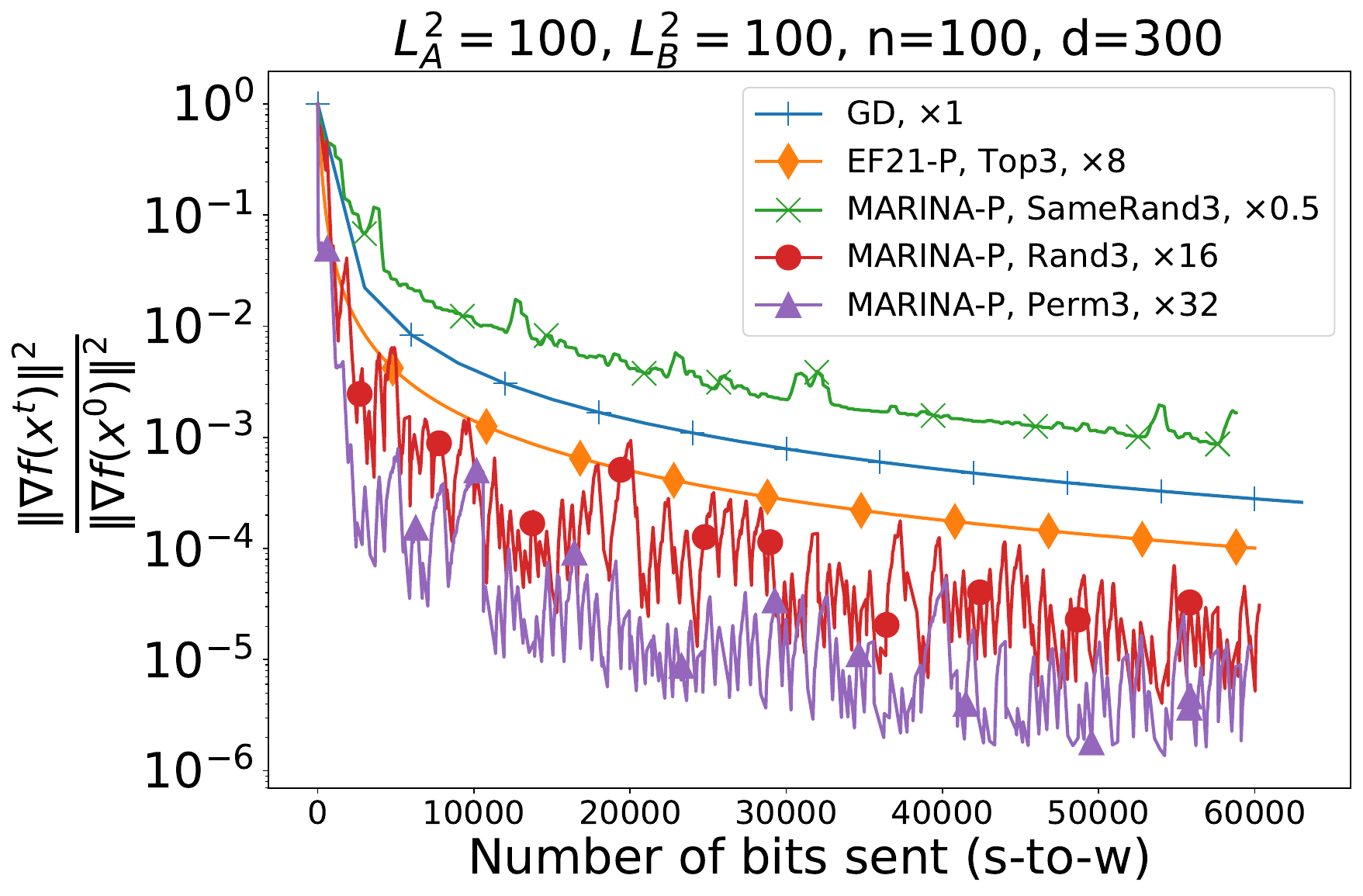}
\end{subfigure}
\begin{subfigure}{0.24\textwidth}
  \centering
  \includegraphics[width=\textwidth]{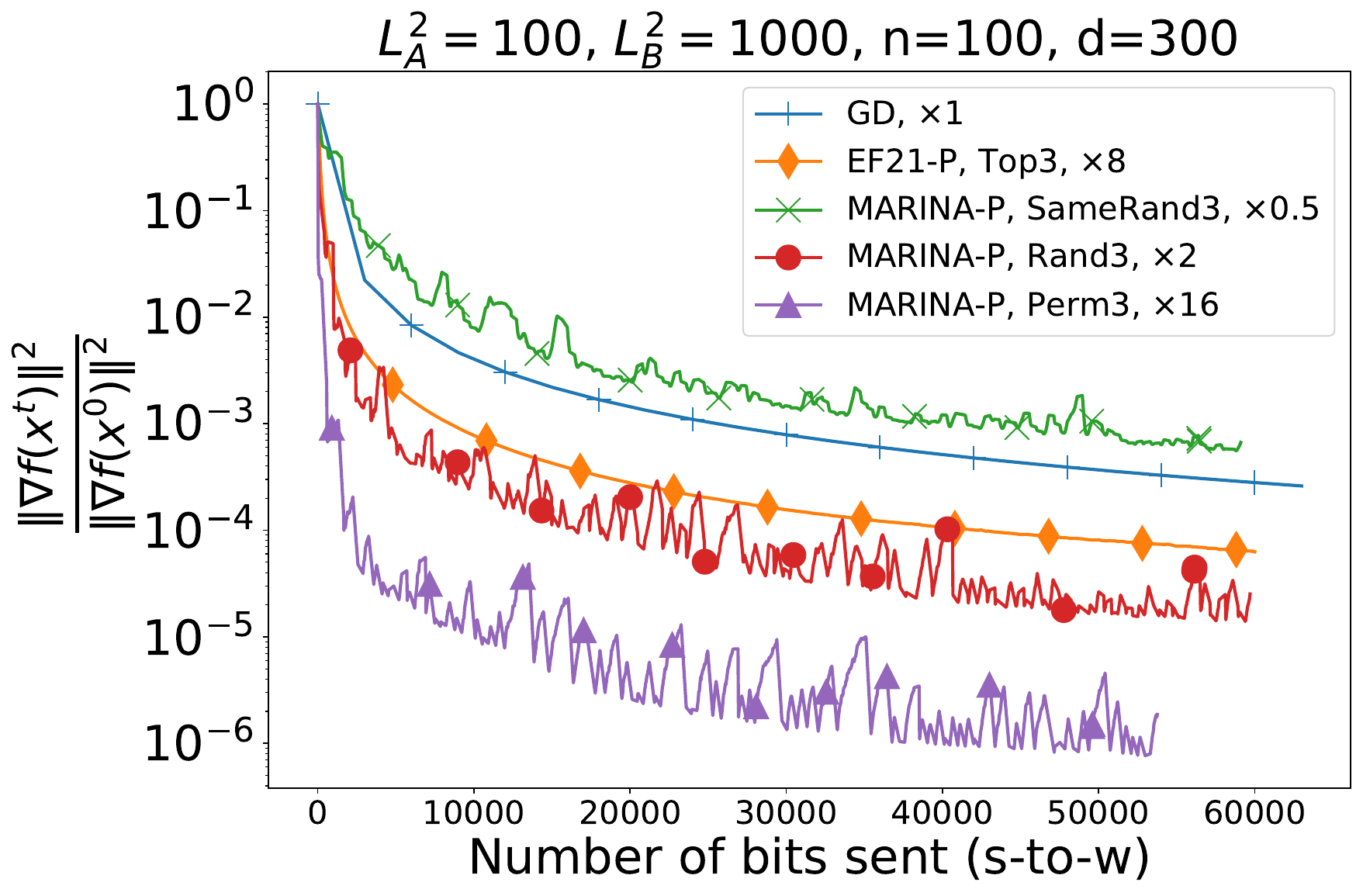}
\end{subfigure}
\begin{subfigure}{0.24\textwidth}
  \centering
  \includegraphics[width=\textwidth]{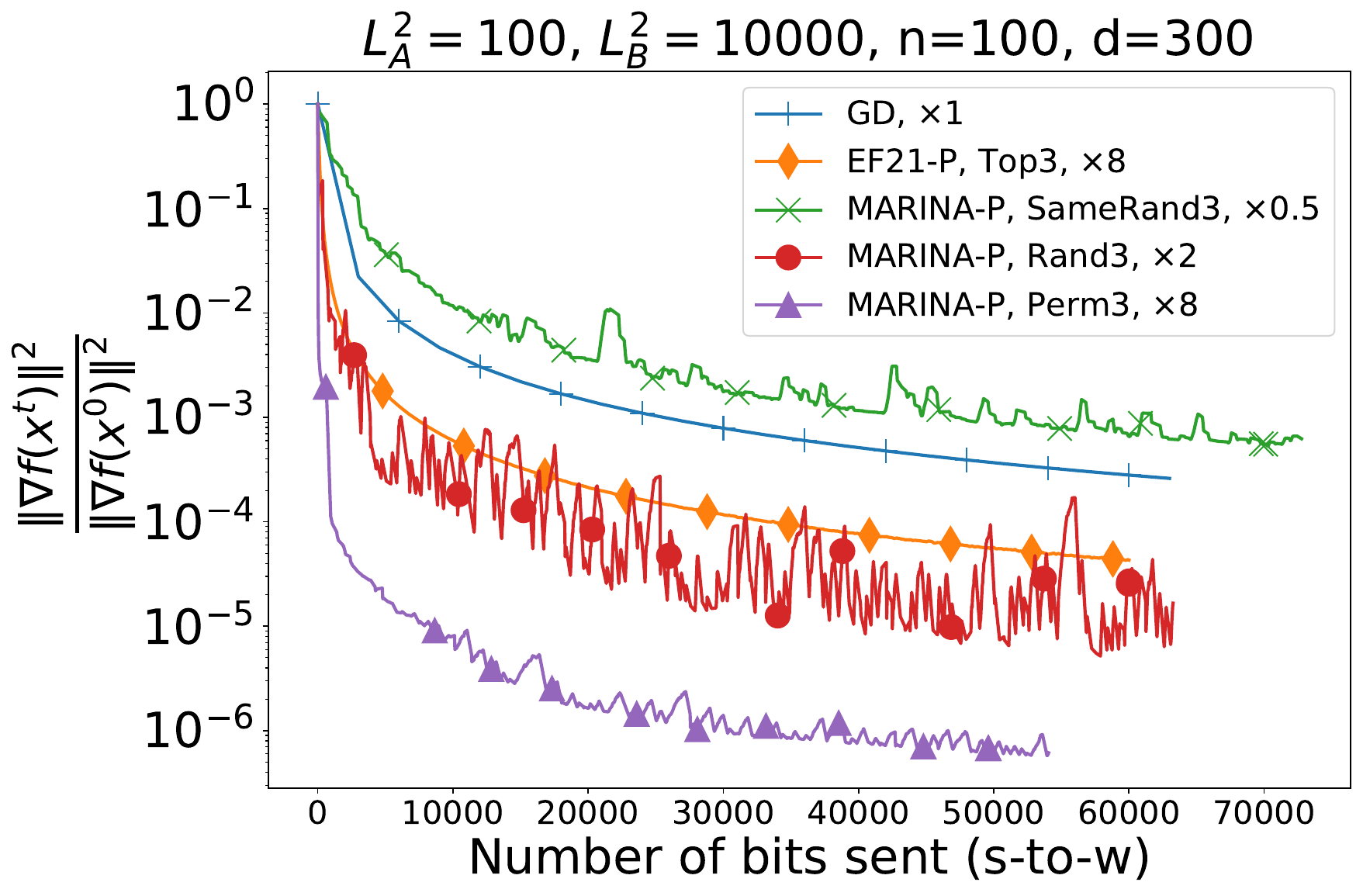}
\end{subfigure}
\begin{subfigure}{0.24\textwidth}
  \centering
  \includegraphics[width=\textwidth]{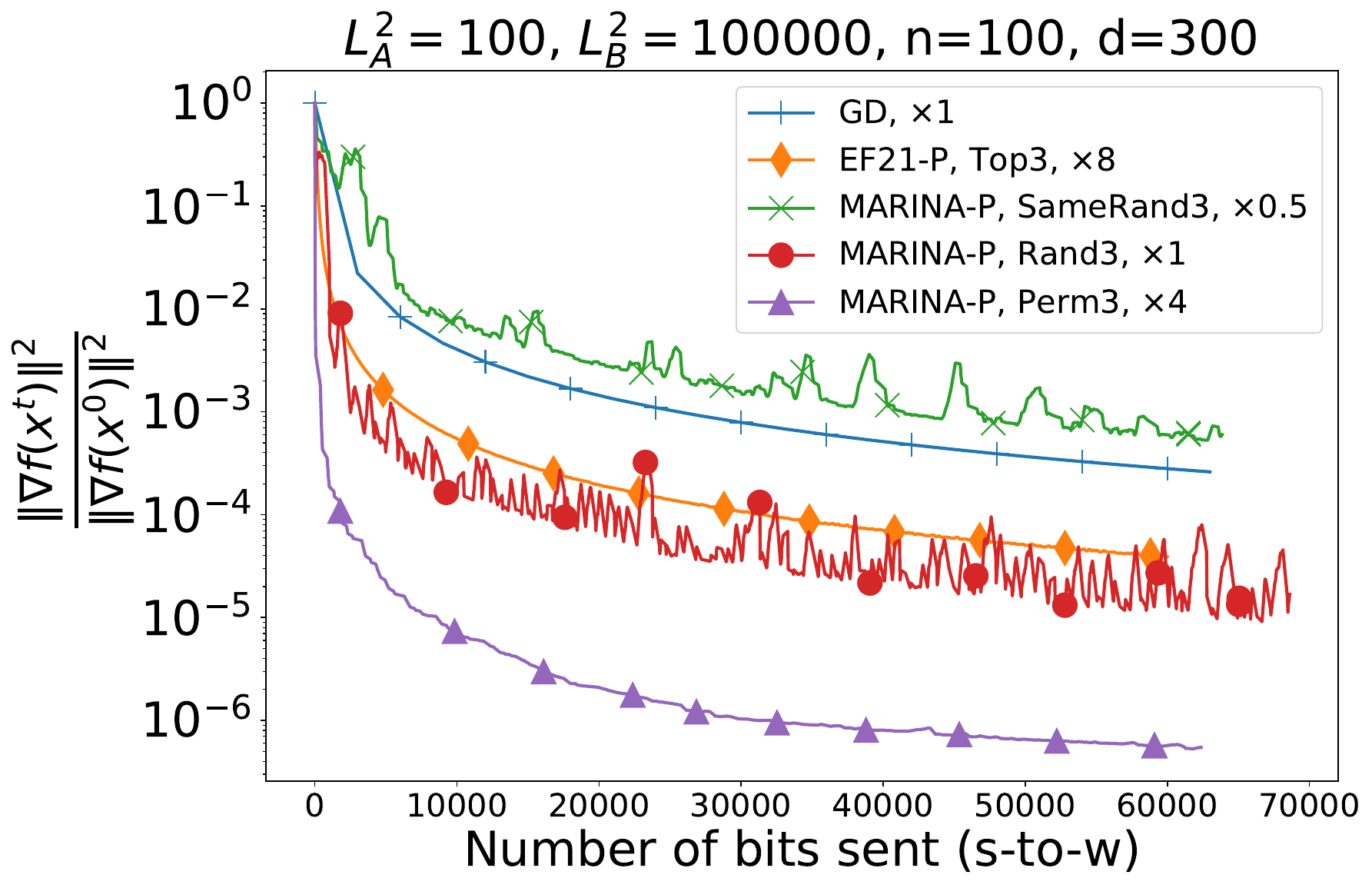}
\end{subfigure}
\caption{Experiments on the quadratic optimization problem from Section~\ref{sec:exp_quad_lalb} with $n=100$ for $L_A^2 \in \brac{0,1,10,100}$ and $L_B^2 \in \brac{100,1000,10000,100000}$.}
\label{fig:grid_n100}
\end{figure}

\begin{figure}[t]
\centering
\begin{subfigure}{0.24\textwidth}
  \centering
  \includegraphics[width=\textwidth]{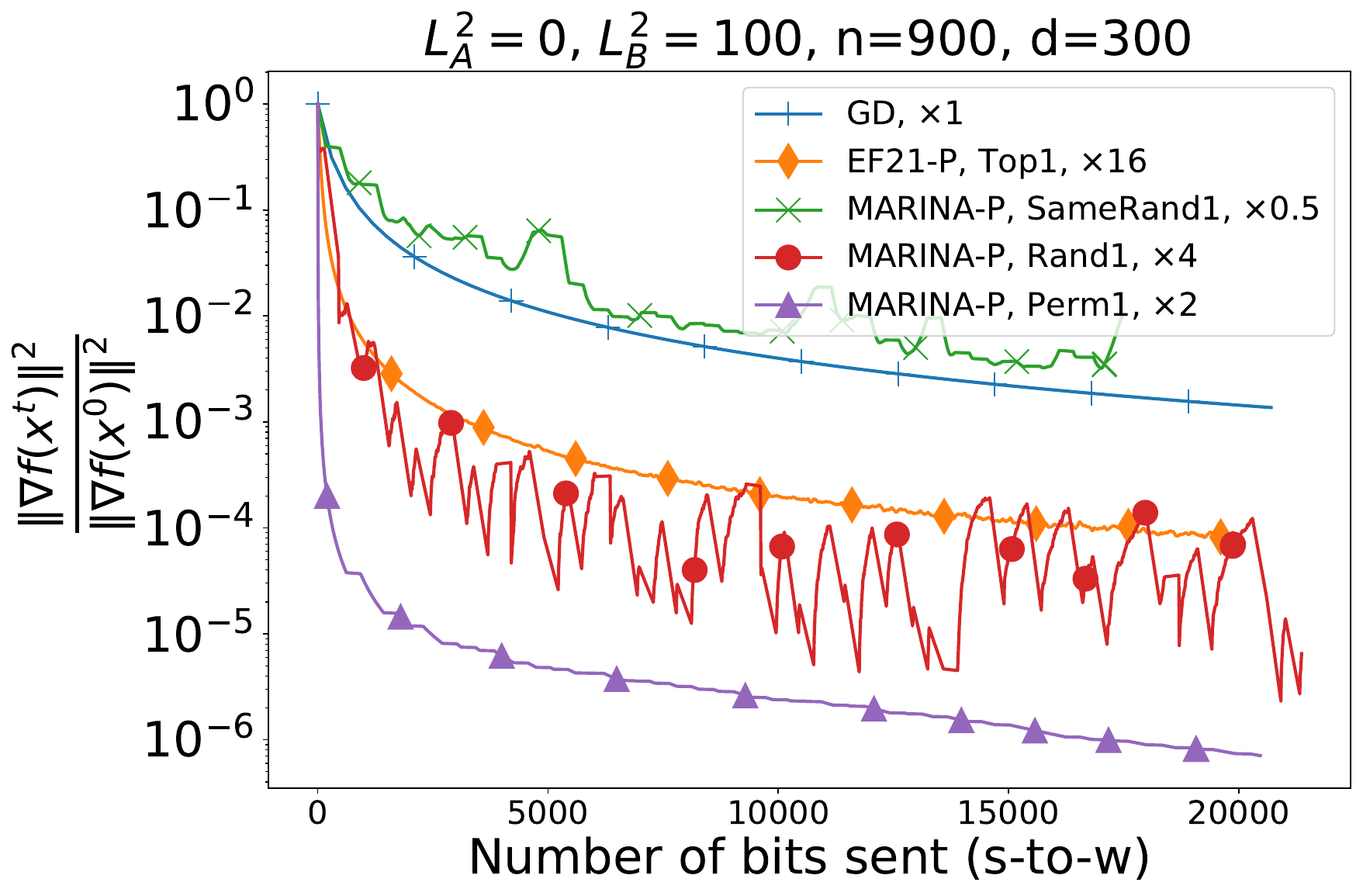}
\end{subfigure}%
\begin{subfigure}{0.24\textwidth}
  \centering
  \includegraphics[width=\textwidth]{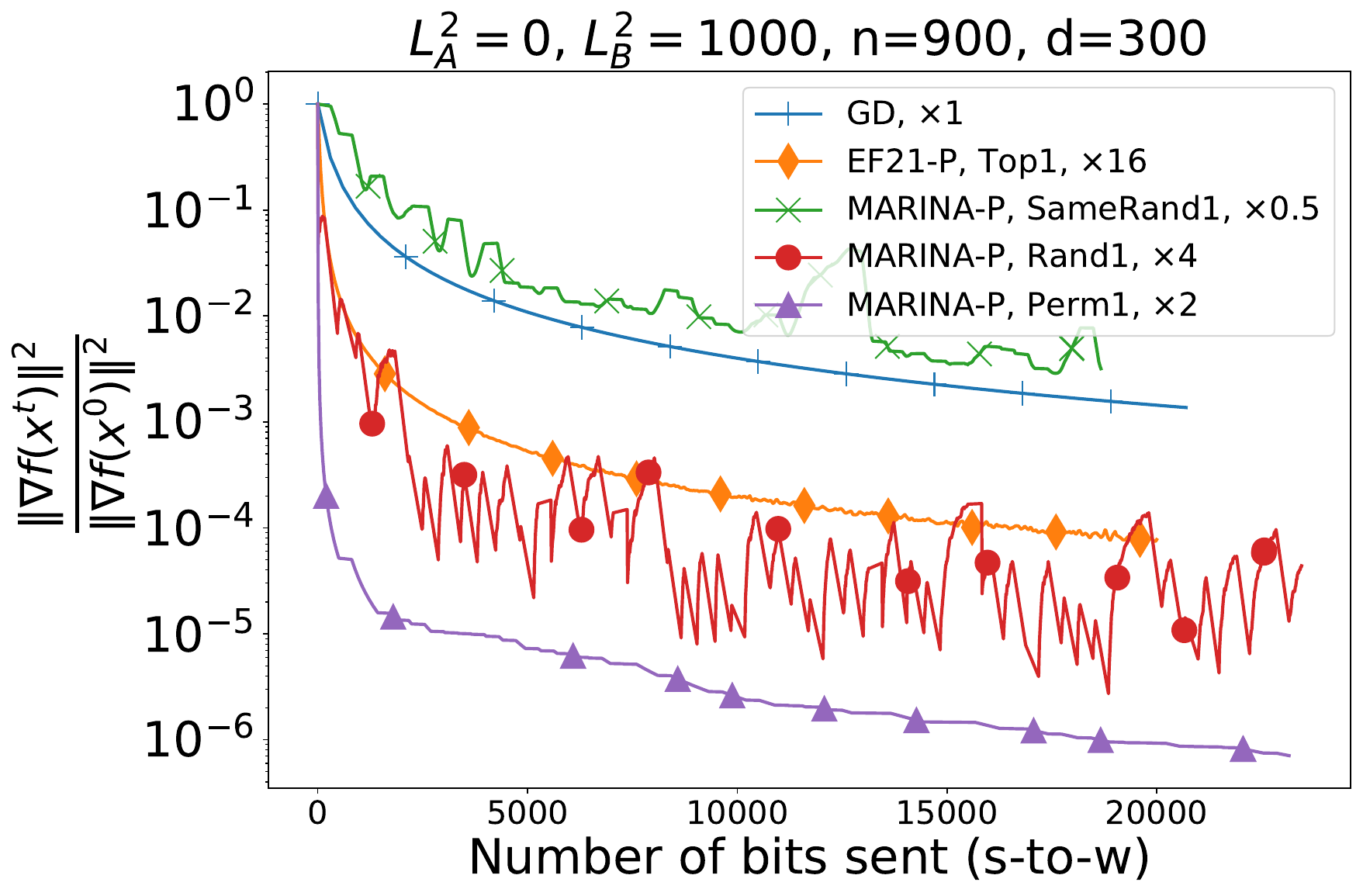}
\end{subfigure}
\begin{subfigure}{0.24\textwidth}
  \centering
  \includegraphics[width=\textwidth]{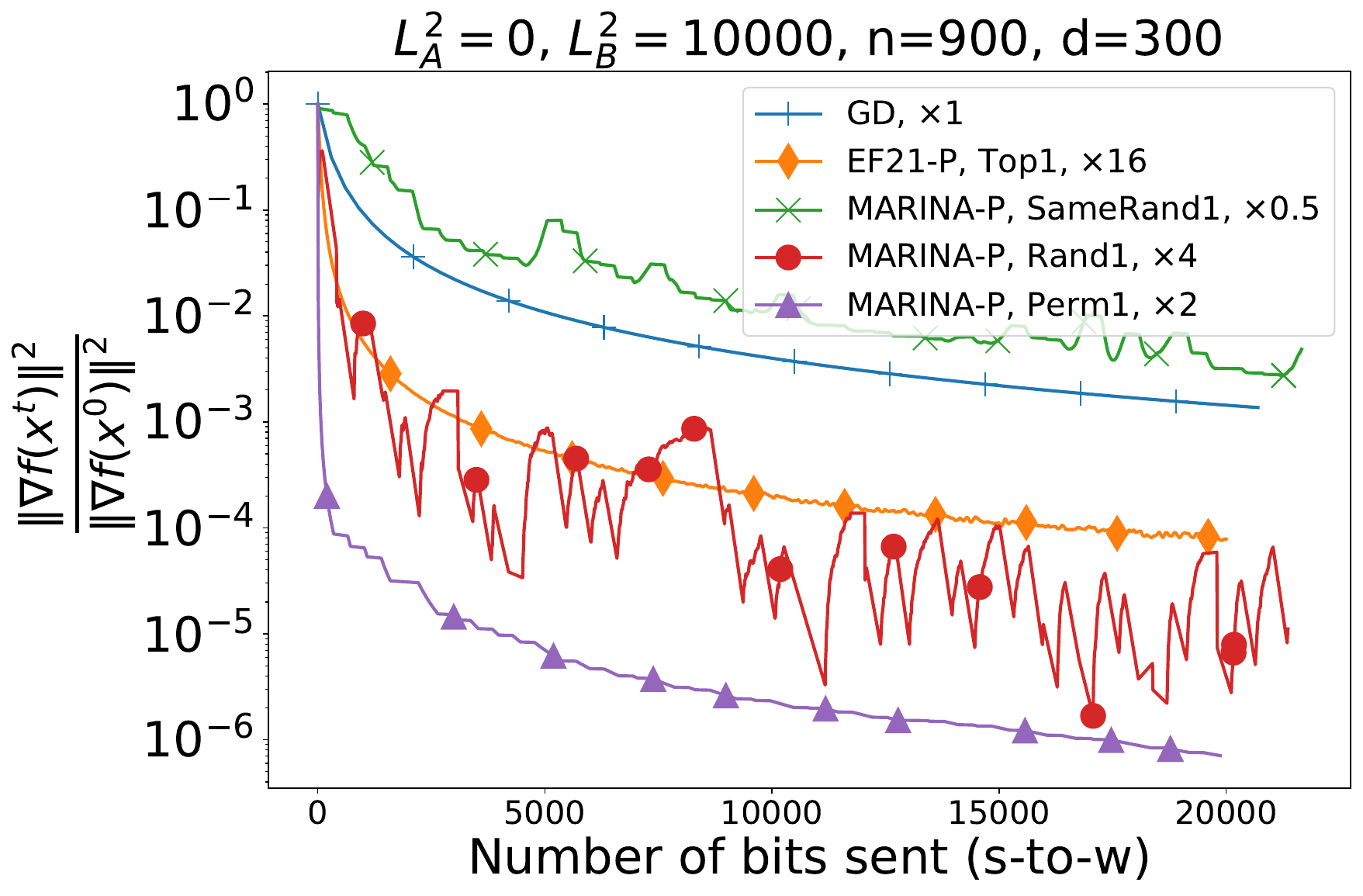}
\end{subfigure}
\begin{subfigure}{0.24\textwidth}
  \centering
  \includegraphics[width=\textwidth]{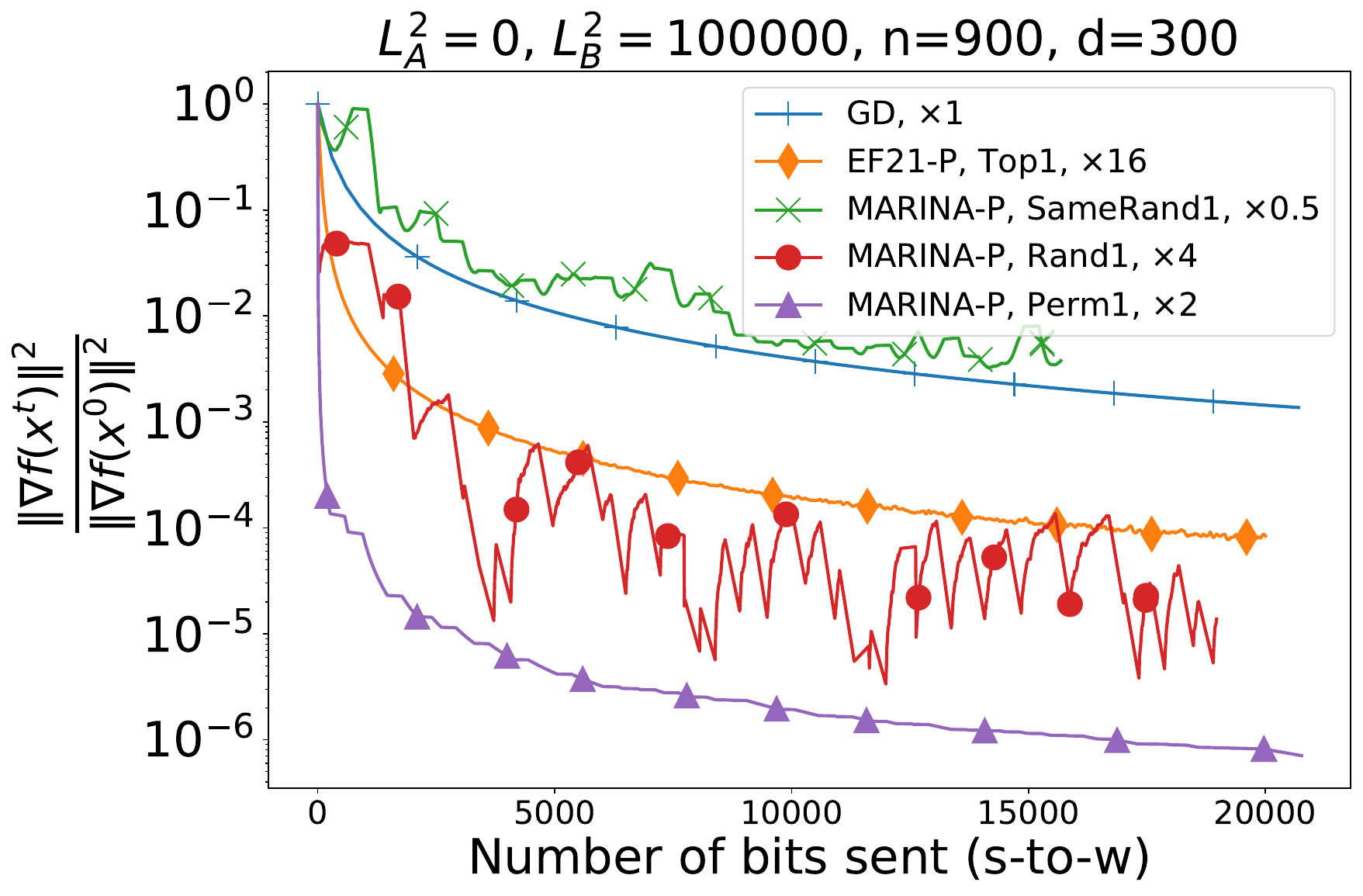}
\end{subfigure}
\begin{subfigure}{0.24\textwidth}
  \centering
  \includegraphics[width=\textwidth]{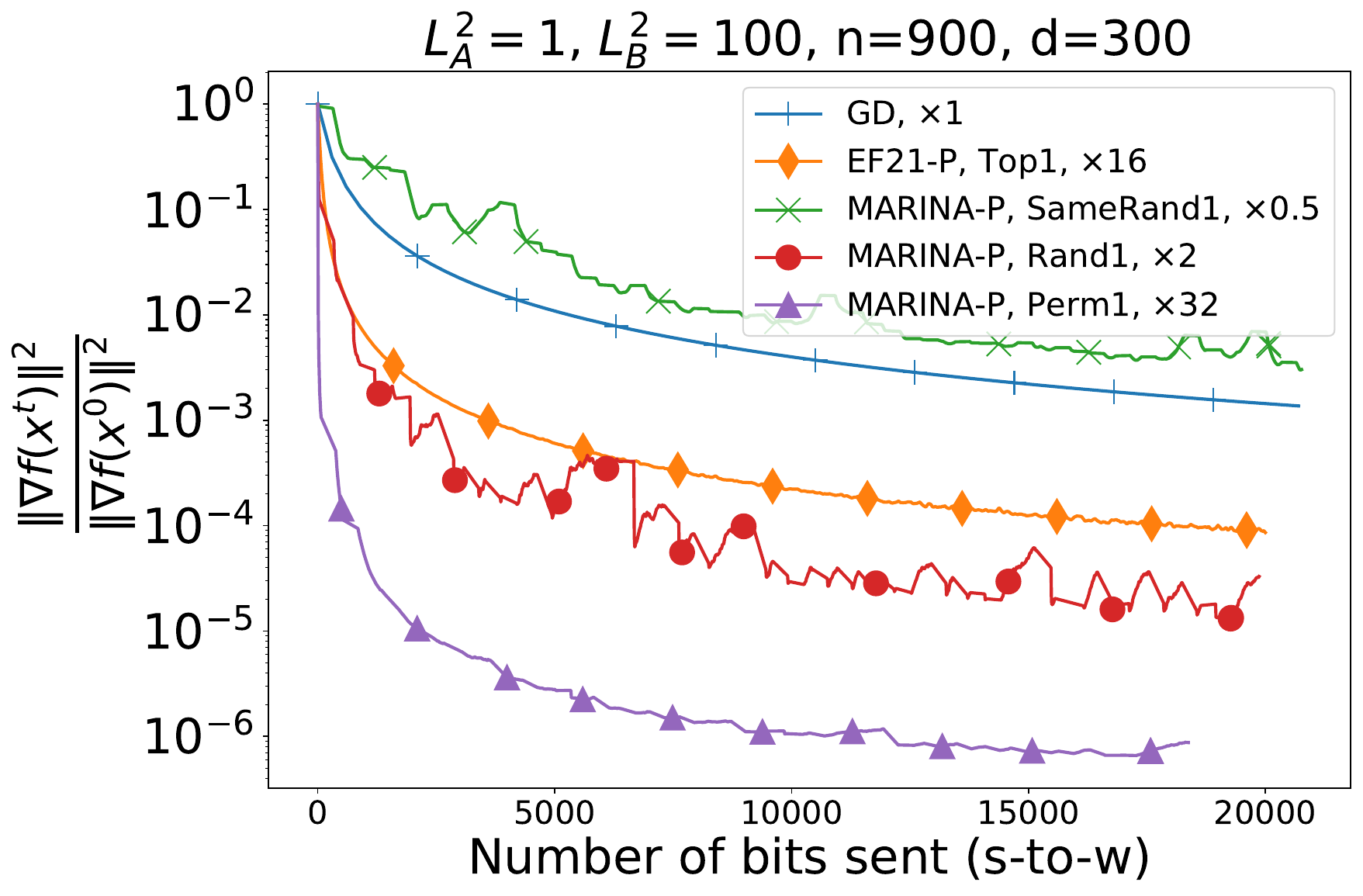}
\end{subfigure}
\begin{subfigure}{0.24\textwidth}
  \centering
  \includegraphics[width=\textwidth]{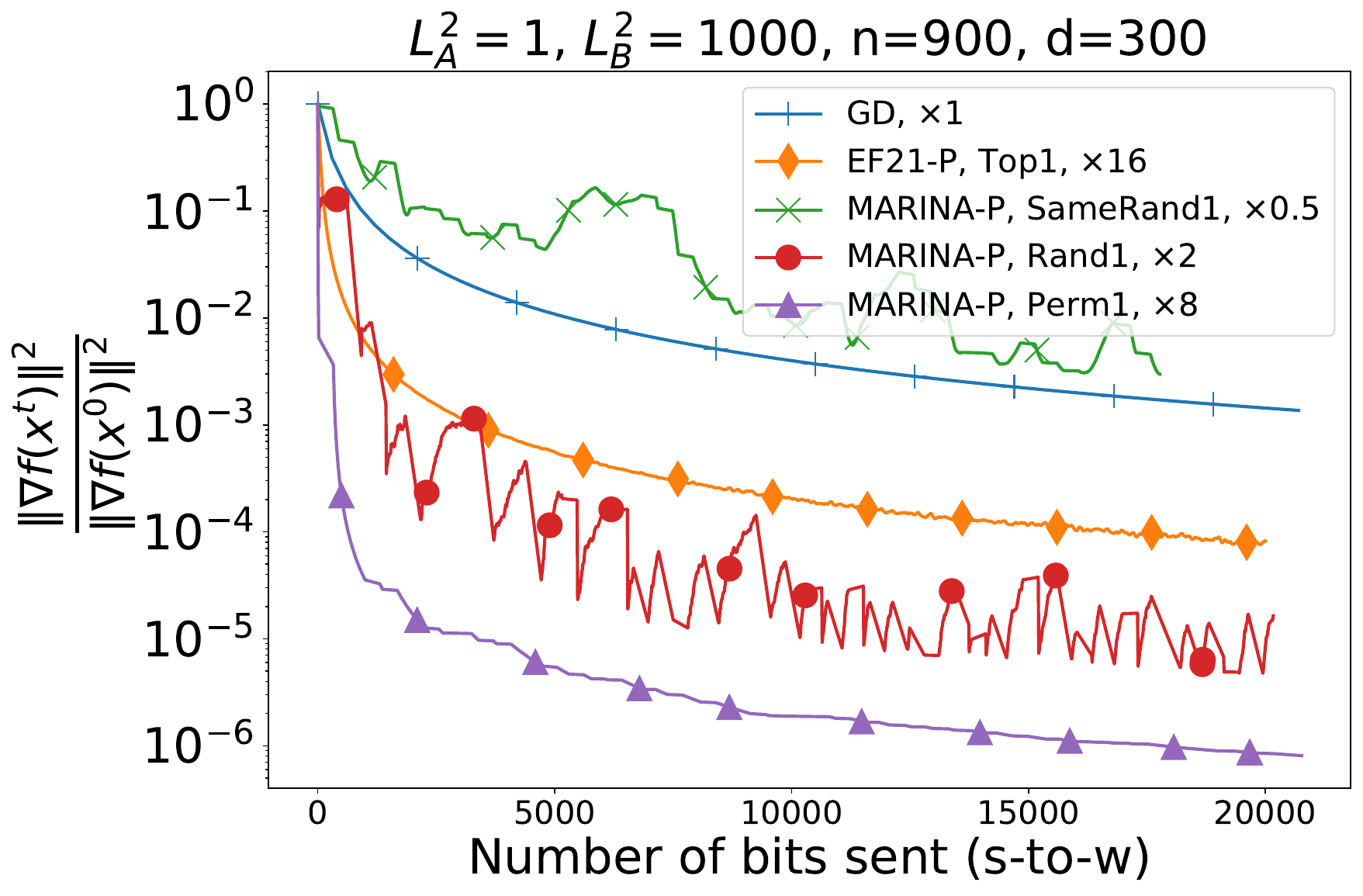}
\end{subfigure}
\begin{subfigure}{0.24\textwidth}
  \centering
  \includegraphics[width=\textwidth]{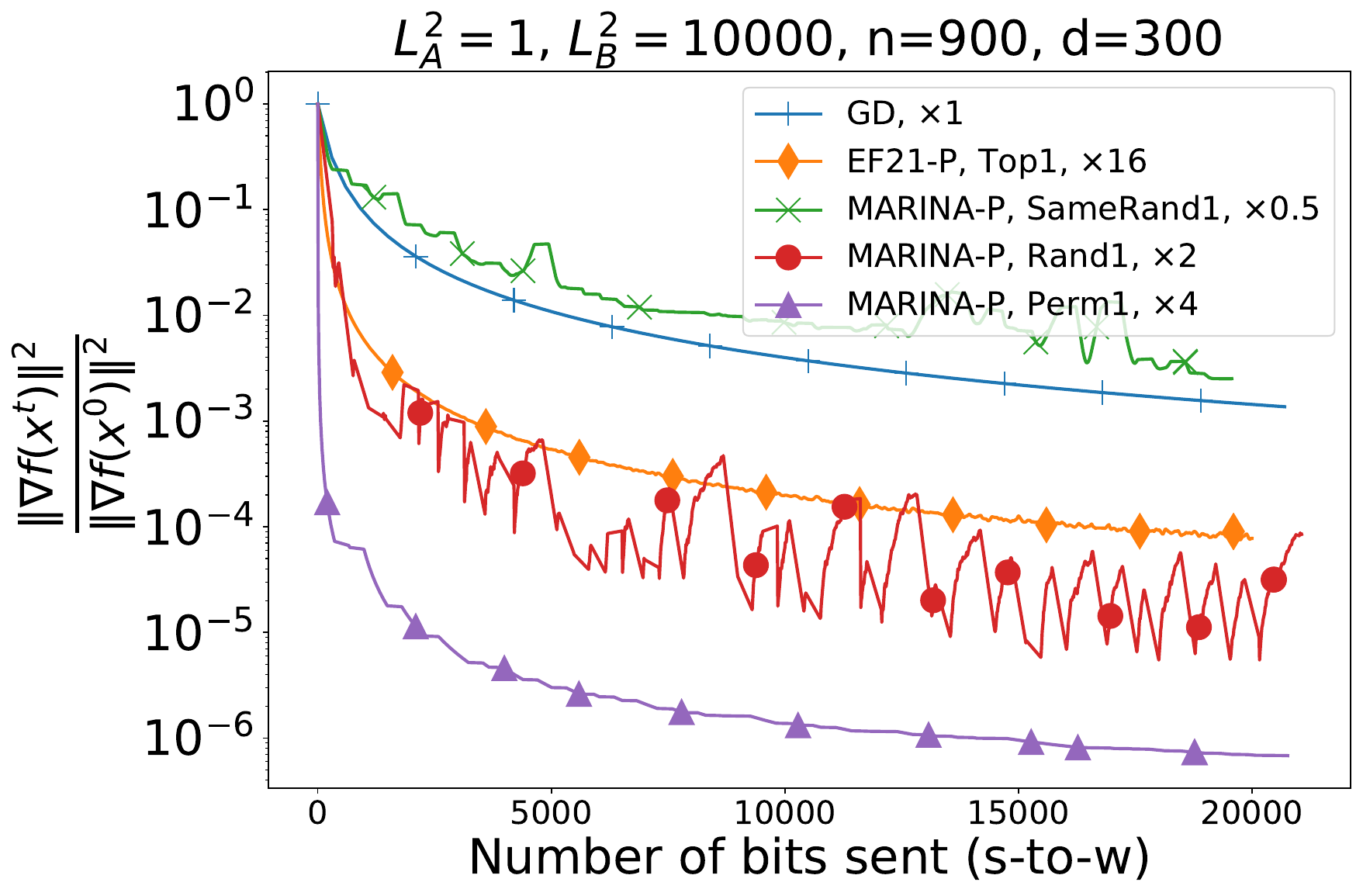}
\end{subfigure}
\begin{subfigure}{0.24\textwidth}
  \centering
  \includegraphics[width=\textwidth]{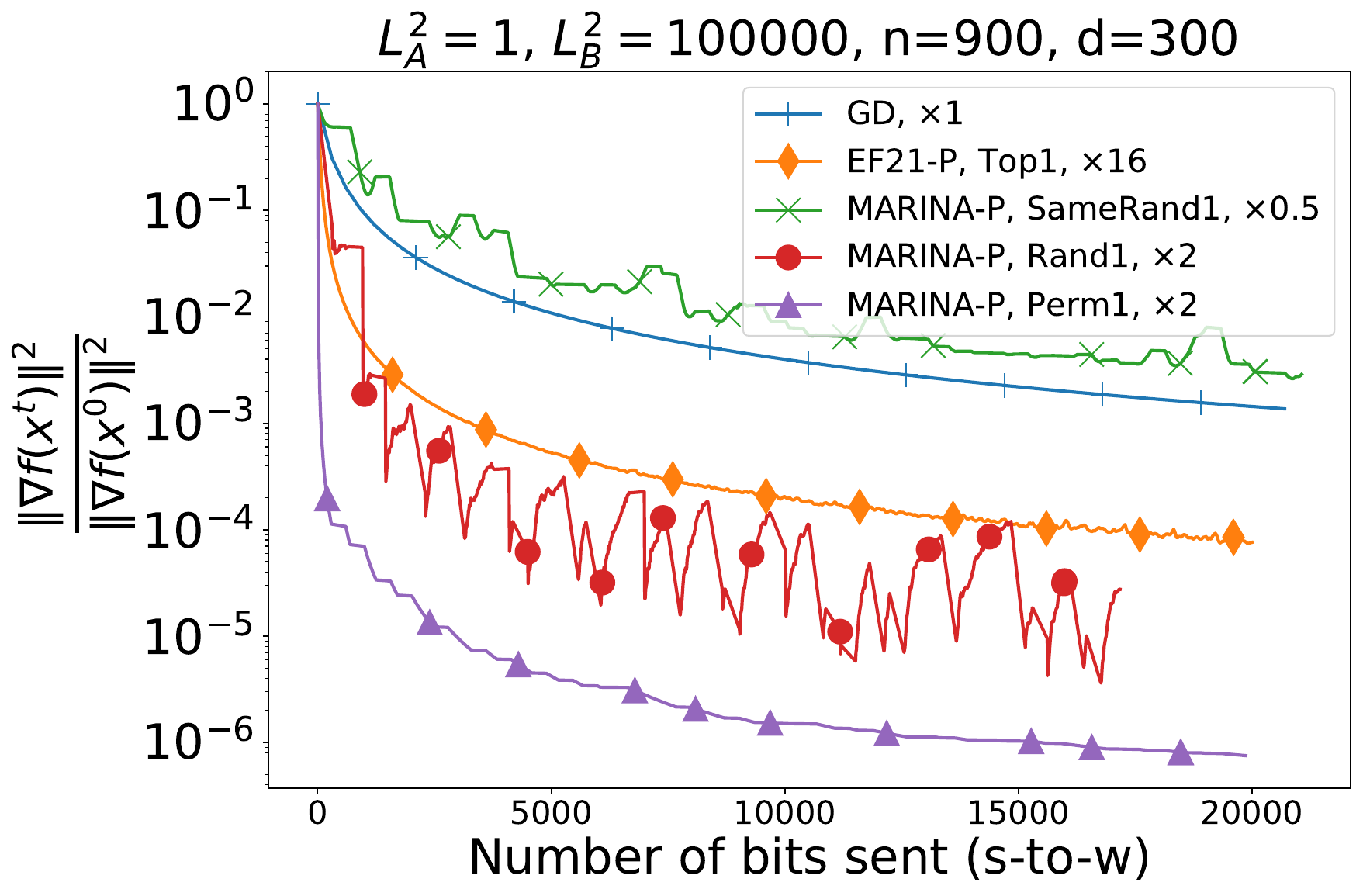}
\end{subfigure}
\begin{subfigure}{0.24\textwidth}
  \centering
  \includegraphics[width=\textwidth]{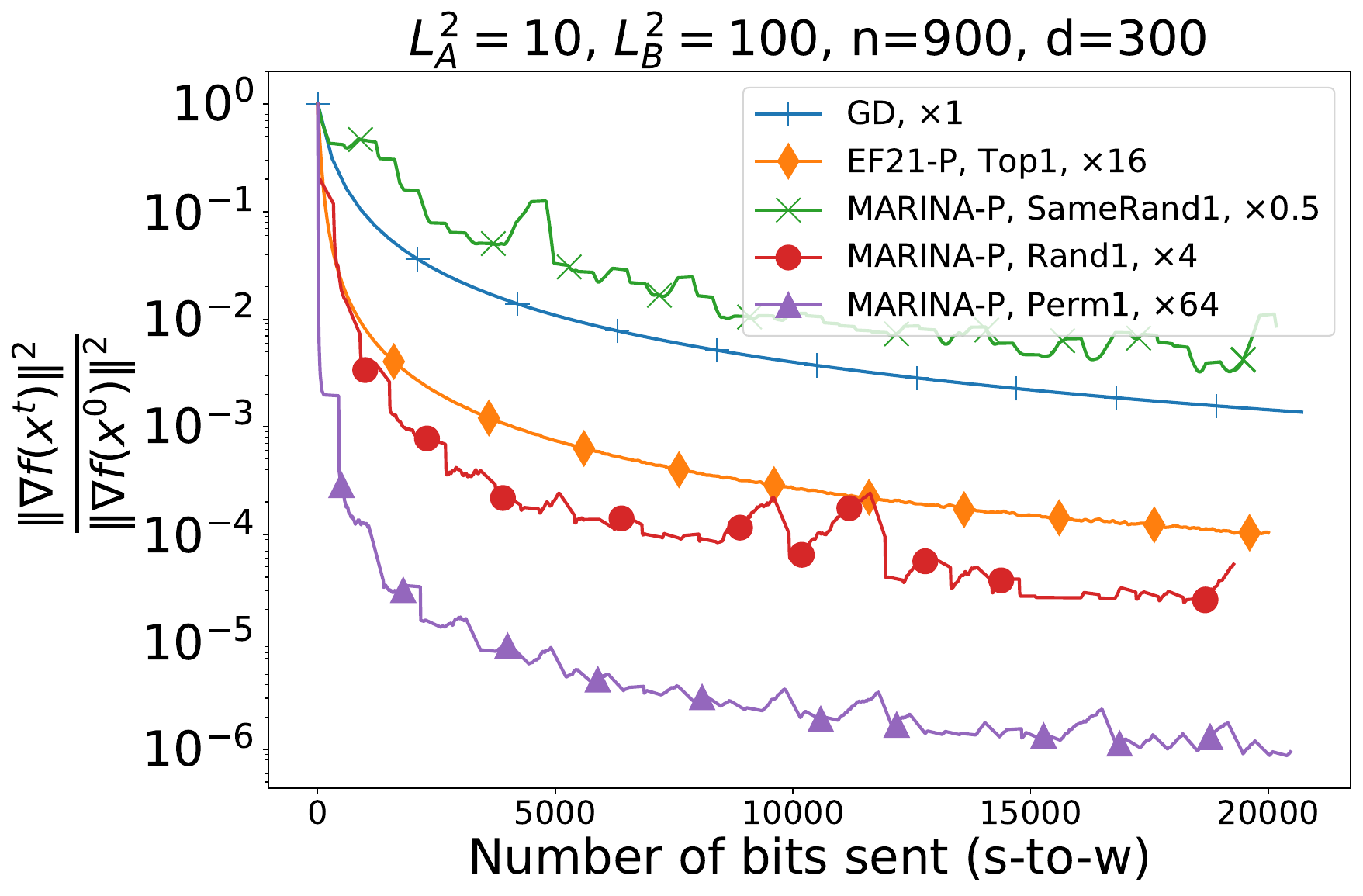}
\end{subfigure}
\begin{subfigure}{0.24\textwidth}
  \centering
  \includegraphics[width=\textwidth]{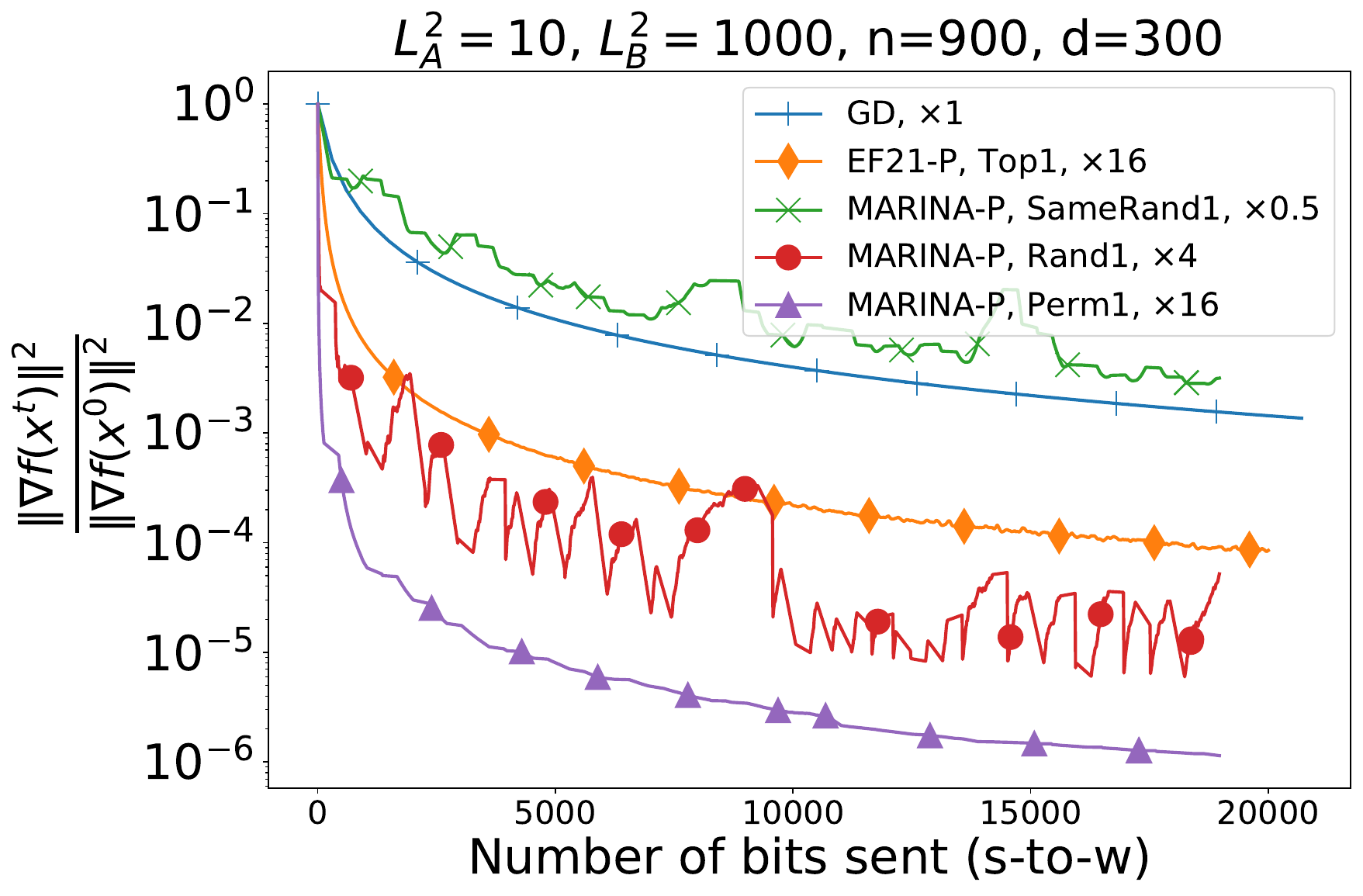}
\end{subfigure}
\begin{subfigure}{0.24\textwidth}
  \centering
  \includegraphics[width=\textwidth]{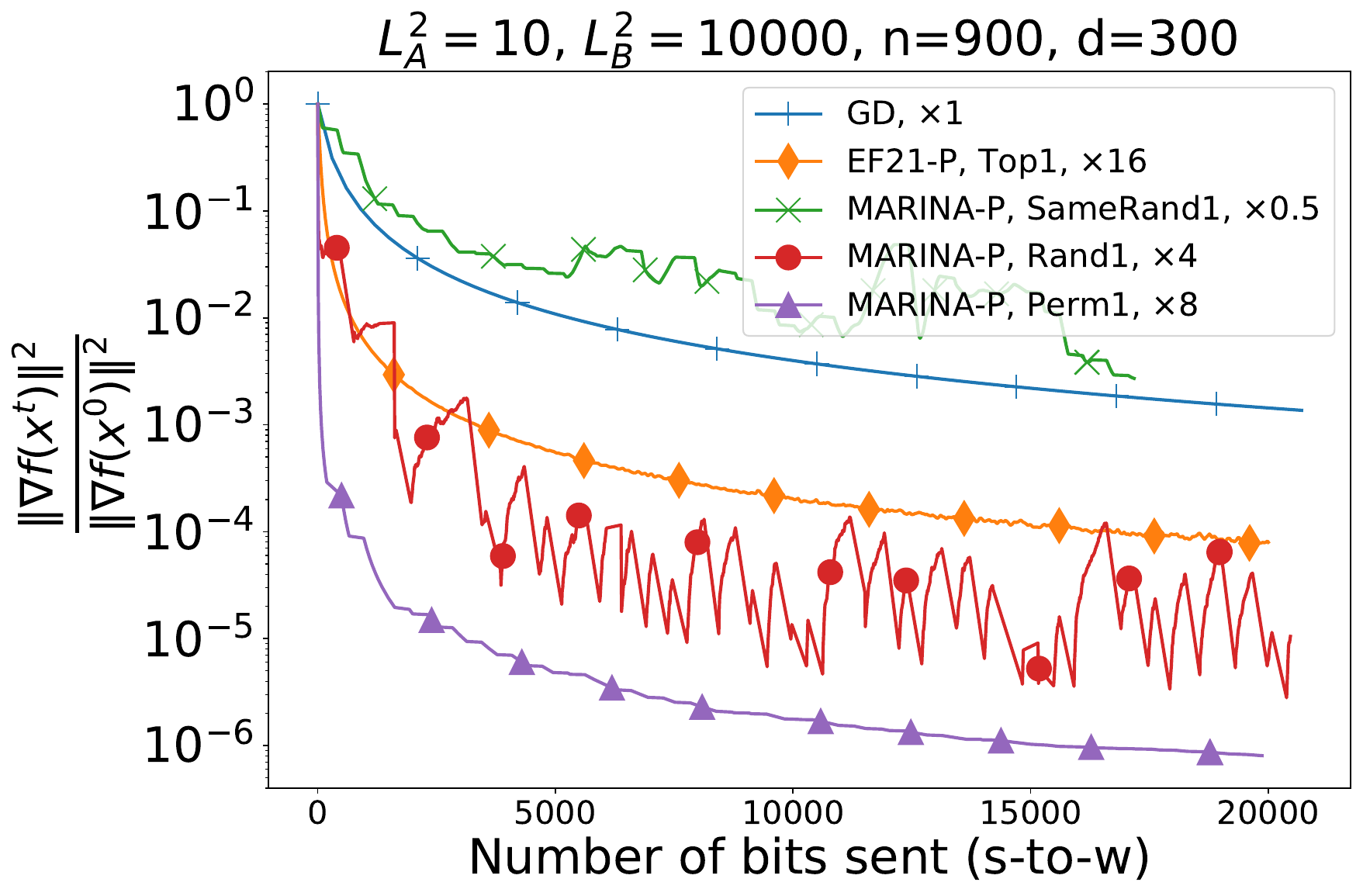}
\end{subfigure}
\begin{subfigure}{0.24\textwidth}
  \centering
  \includegraphics[width=\textwidth]{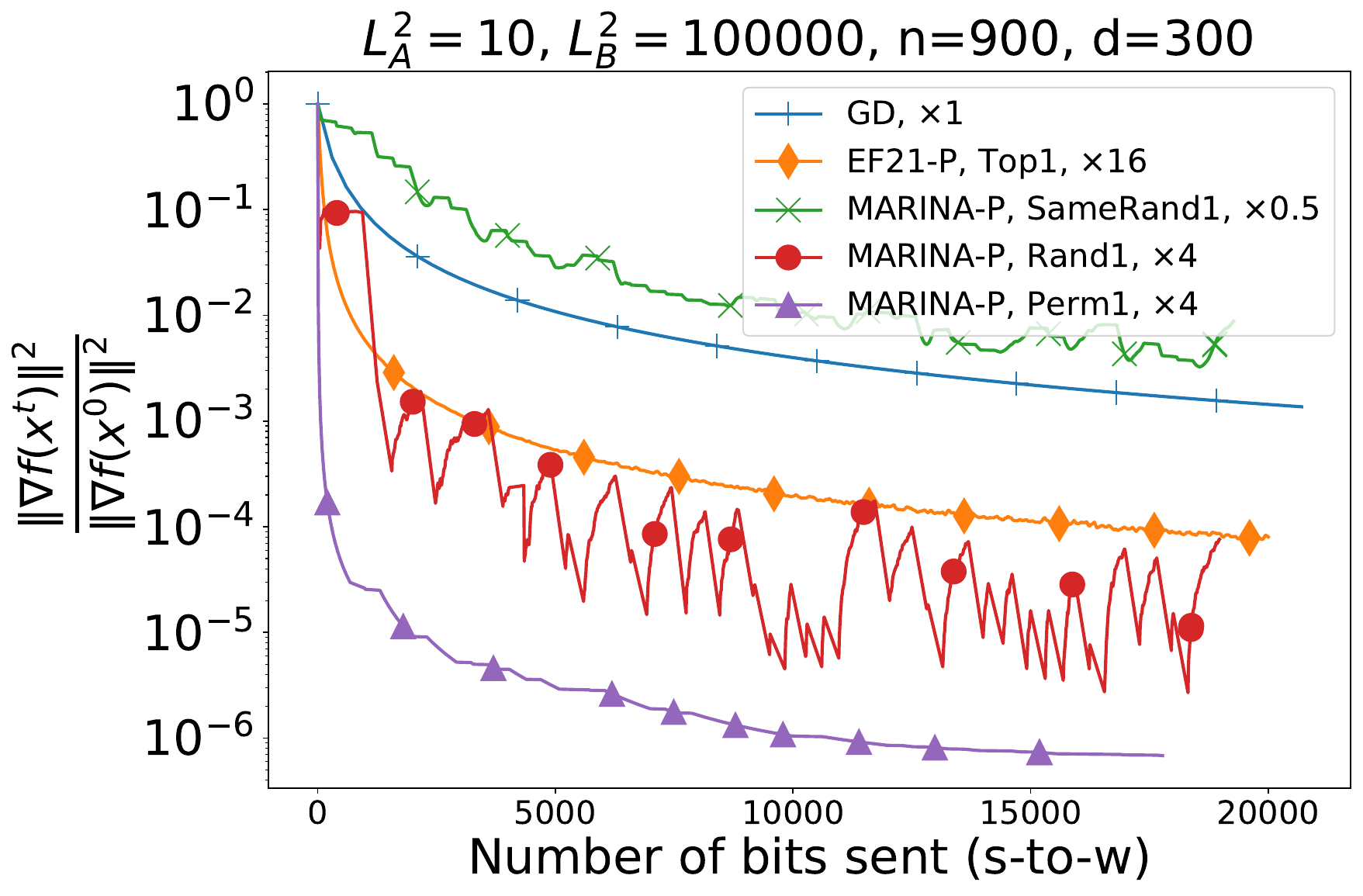}
\end{subfigure}
\begin{subfigure}{0.24\textwidth}
  \centering
  \includegraphics[width=\textwidth]{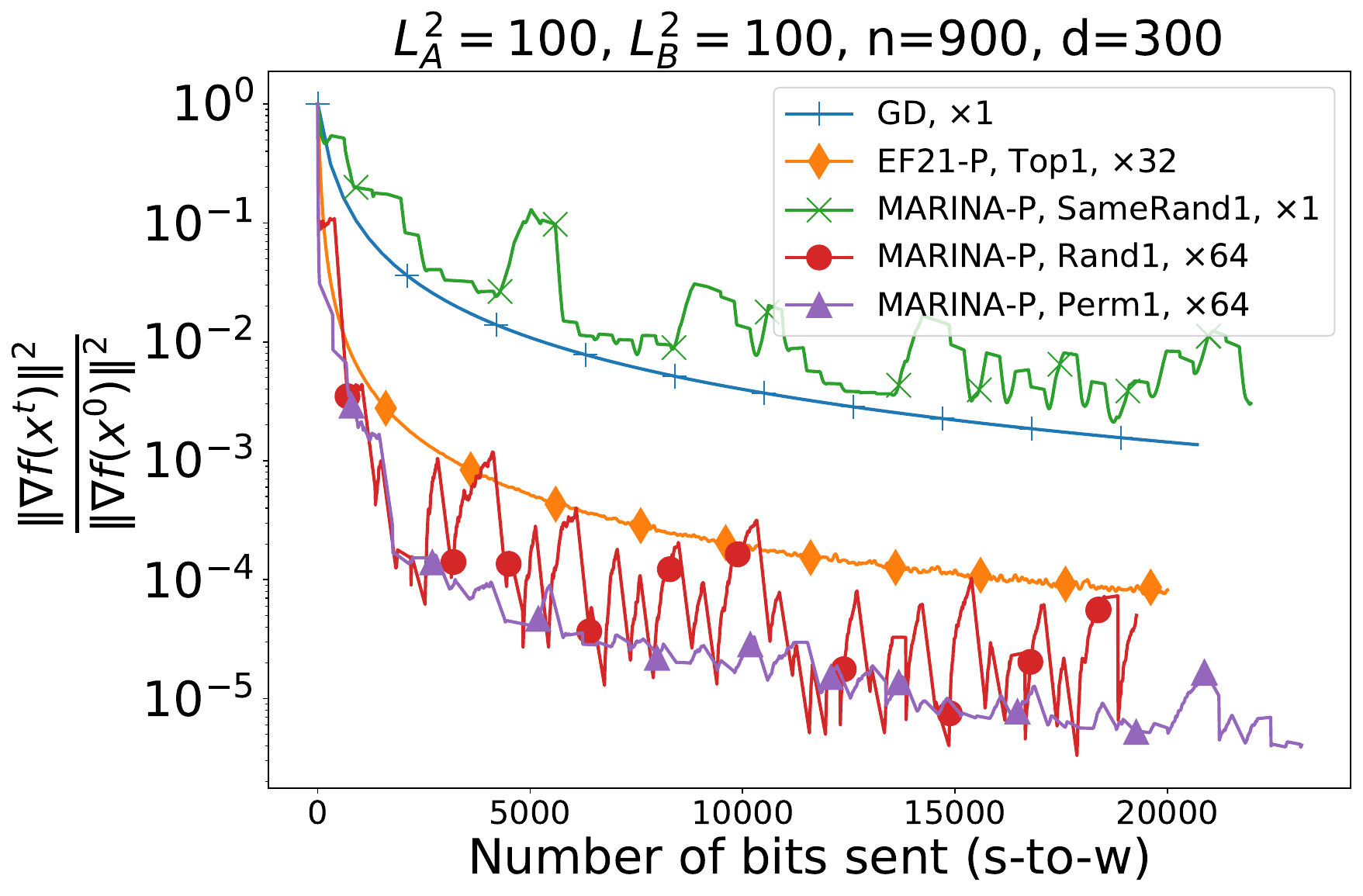}
\end{subfigure}
\begin{subfigure}{0.24\textwidth}
  \centering
  \includegraphics[width=\textwidth]{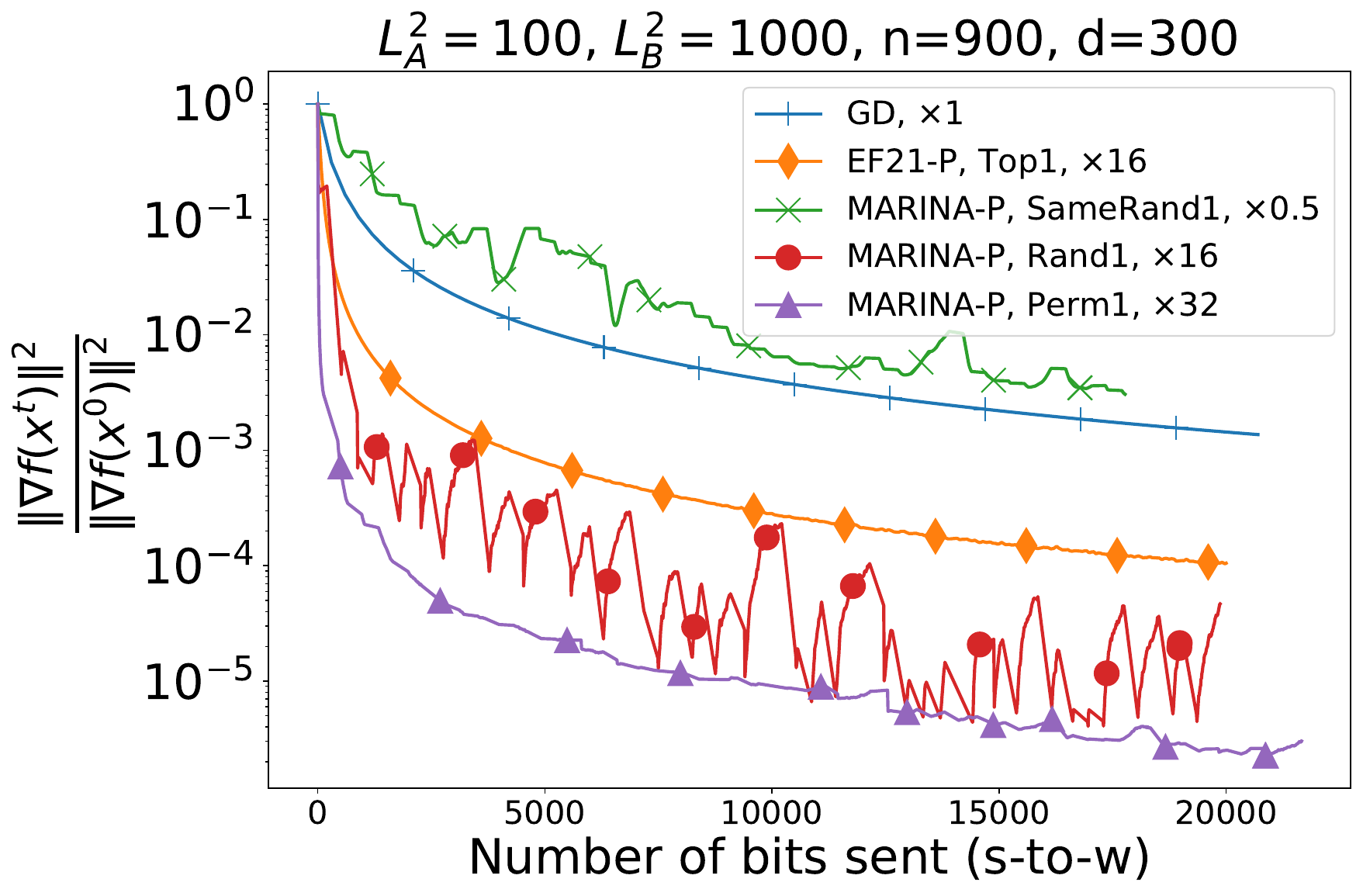}
\end{subfigure}
\begin{subfigure}{0.24\textwidth}
  \centering
  \includegraphics[width=\textwidth]{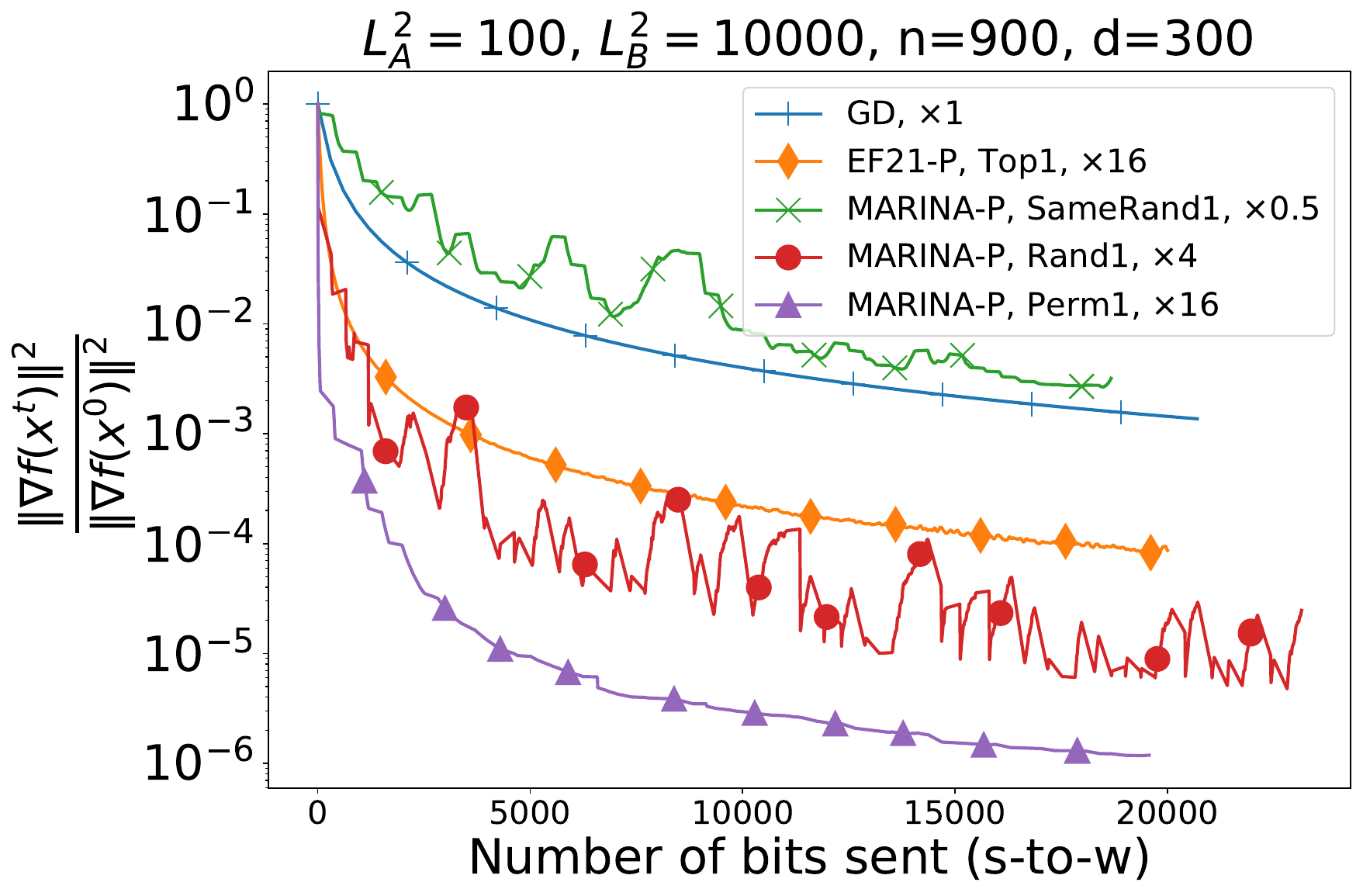}
\end{subfigure}
\begin{subfigure}{0.24\textwidth}
  \centering
  \includegraphics[width=\textwidth]{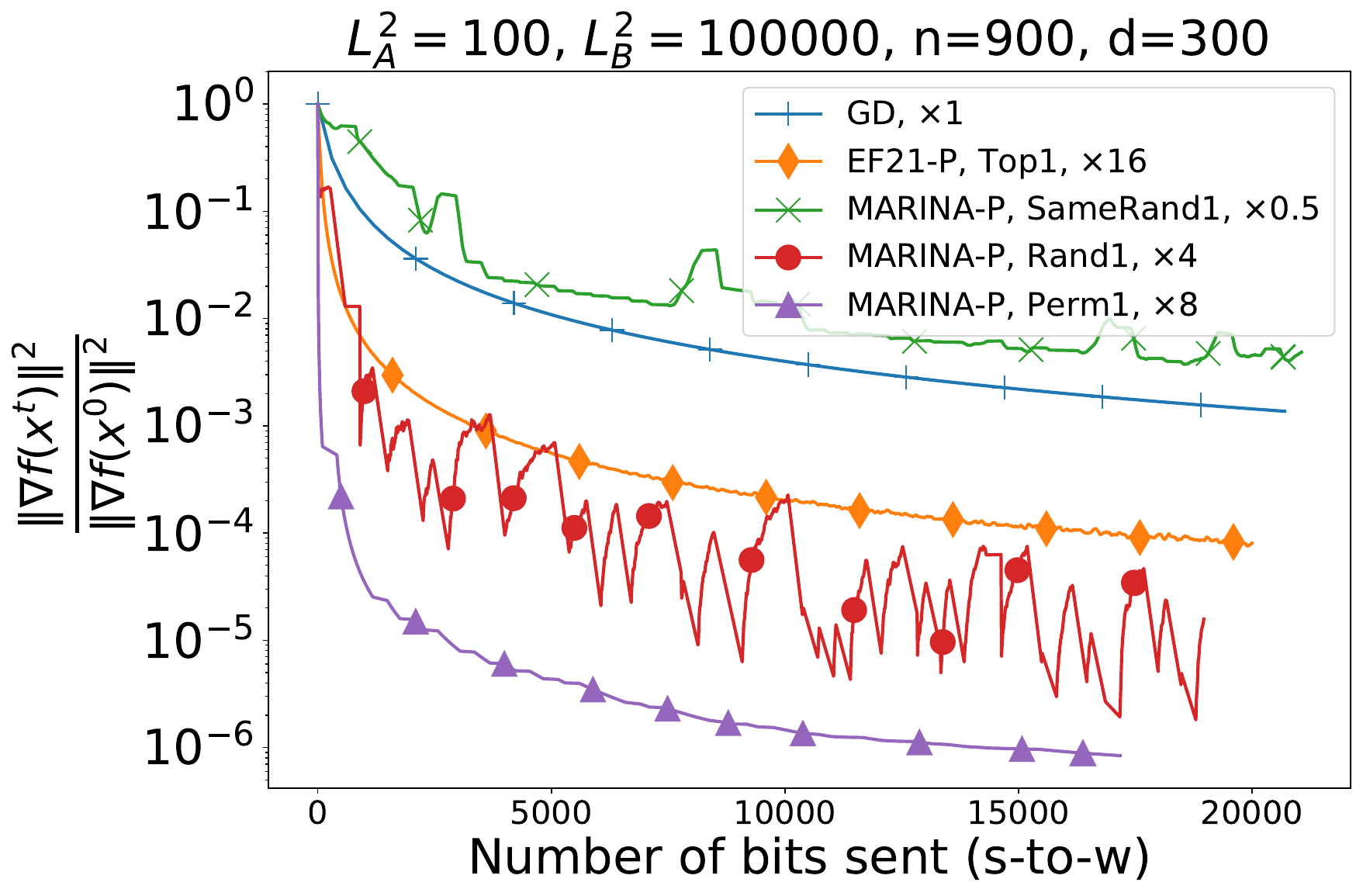}
\end{subfigure}
\caption{Experiments on the quadratic optimization problem from Section~\ref{sec:exp_quad_lalb} with $n=900$ for $L_A^2 \in \brac{0,1,10,100}$ and $L_B^2 \in \brac{100,1000,10000,100000}$.}
\label{fig:grid_n900}
\end{figure}

\clearpage
\newpage

\section{Proof of the Lower Bounds}
\label{sec:lower_bound}

\subsection{The ``difficult'' function from the nonconvex world}
\label{sec:worst_case}
In our lower bound, we use the function from \citet{carmon2020lower,arjevani2022lower}. For any $T \in \N,$ let
\begin{align}
    \label{eq:worst_case}
    F_T(x) \eqdef -\Psi(1) \Phi([x]_1) + \sum_{i=2}^T \left(\Psi(-[x]_{i-1})\Phi(-[x]_i) - \Psi([x]_{i-1})\Phi([x]_i)\right),
\end{align}
where
\begin{align*}
    \Psi(x) = \begin{cases}
        0, & x \leq 1/2, \\
        \exp\left(1 - \frac{1}{(2x - 1)^2}\right), & x \geq 1/2,
    \end{cases}
    \quad\textnormal{and}\quad
    \Phi(x) = \sqrt{e} \int_{-\infty}^{x}e^{-\frac{1}{2}t^2}dt.
\end{align*}

\citet{carmon2020lower, arjevani2022lower} also proved the following properties of the function:

\begin{lemma}[\cite{carmon2020lower, arjevani2022lower}]
    \label{lemma:worst_function}
    The function $F_T$ satisfies:
    \begin{enumerate}
        \item $F_T(0) - \inf_{x \in \R^T} F_T(x) \leq \Delta^0 T,$ where $\Delta^0 = 12.$
        \item The function $F_T$ is $l_1$--smooth, where $l_1 = 152.$
        \item For all $x \in \R^T,$ $\norm{\nabla F_T(x)}_{\infty} \leq \gamma_{\infty},$ where $\gamma_{\infty} = 23.$
        \item For all $x \in \R^T,$ $\textnormal{prog}(\nabla F_T(x)) \leq \textnormal{prog}(x) + 1.$
        \item For all $x \in \R^T,$ if $\textnormal{prog}(x) < T,$ then $\norm{\nabla F_T(x)} > 1,$
    \end{enumerate}
\end{lemma}
where $\textnormal{prog}(x) \eqdef \max \{i \geq 0\,|\,x_i \neq 0\} \quad (x_0 \equiv 1).$

The function is a standard function that is used to establish lower bounds in the nonconvex world \citep{carmon2020lower,arjevani2022lower,lu2021optimal,tyurin2023optimal}.

\subsection{Theorems}

Our lower bound applies to the family of methods with the following structure:

\begin{protocol}[H]
    \caption{Protocol}
    \label{alg:protocol}
    \begin{algorithmic}[1]
    \STATE \textbf{Input: }functions $f_1, \dots, f_n \in \cF,$ algorithm $A,$ probability $p$
    \FOR{$k = 0, \dots, \infty$}
    \STATE Server calculates a new point: $x^k = B_1^k(g_1^1, \dots, g_1^{k}, \dots, g_n^{1}, \dots, g_n^{k})$
    \STATE Server aggregates all available information: $s_i^k = B_{2,i}^k(g_1^1, \dots, g_1^{k}, \dots, g_n^{1}, \dots, g_n^{k})$
    \STATE Server sends sparsified vectors $\bar{s}_i^k$ to the workers, where 
    \begin{align*}
        [\bar{s}_i^k]_j = [s_i^k]_j \times \eta_{i,j}^k,
    \end{align*}
    and $\eta_{i,j}^k$ is a random variable such that $\Prob{\eta_{i,j}^k \neq 0} \leq p$
    for all $j \in [d']$ and for all $i \in [n].$ We define $d' \eqdef \textnormal{dim(dom(} f_1\textnormal{))},$ and $[\cdot]_j$ means the $j$\textsuperscript{th} coordinate. \alglinelabelfirst{line:spar}
    \STATE Workers aggregate all available local information and calculate gradients: $g_i^{k+1} = L_{i}^{k}(\bar{s}_i^0, \dots, \bar{s}_i^k)$ \\
    ($L_{i}^{k}$ has access to the gradient oracle of $f_i$ and can call it as many times as it wants according to the rules \eqref{eq:fCJhEZyUETrJBh} and \eqref{eq:lAXnPlQppJGITqJi})
    \STATE Workers send $g_i^{k+1}$ to the server
    \ENDFOR
    \end{algorithmic}
\end{protocol}

We consider the following standard classes of functions and algorithms:

\begin{definition}
    Let the function $f \,:\,\R^d \rightarrow \R$ be differentiable, $L$-smooth (i.e., $\norm{\nabla f(x) - \nabla f(y)} \leq L \norm{x - y}$ for all $x, y \in \R^d$), and $f(0) - \inf_{x \in \R^d} f(x) \leq \delta^0.$
    We denote the family of functions that satisfy these properties by $\cF_{\delta^0, L}$.
    \label{def:func_class}
\end{definition}

\begin{definition}
\label{def:alg_class}
Consider Protocol~\ref{alg:protocol}. A sequence of tuples of mappings $A = \{(B^k_1, B^k_{2,1}, \dots, B^k_{2,n}, L^k_{1}, \dots, L^k_{n})\}_{k=0}^{\infty}$ is a zero-respecting algorithm, if,
\begin{enumerate}
    \item $B^k_1\,:\, \underbrace{\R^d \times \dots \times \R^d}_{n \times k \textnormal{ times}} \rightarrow \R^d$ for all $k \geq 1,$ and $B^0_1 \in \R^d.$
    \item $B^k_{2,i}\,:\, \underbrace{\R^d \times \dots \times \R^d}_{n \times k \textnormal{ times}} \rightarrow \R^d$ for all $k \geq 1,$ and $B^0_{2,i} \in \R^d$ for all $i \in [n]$
    \item $L_{i}^k\,:\, \underbrace{\R^d \times \dots \times \R^d}_{k + 1 \textnormal{ times}} \rightarrow \R^d$ for all $k \geq 0$ and for all $i \in [n].$ \label{prop:three}
    \item $\textnormal{supp} \left(x^k\right) \subseteq \bigcup_{j = 1}^k \bigcup_{i = 1}^n \textnormal{supp} \left(g_i^j\right),$ $\textnormal{supp} \left(s_i^k\right) \subseteq \bigcup_{j = 1}^k \bigcup_{i = 1}^n \textnormal{supp} \left(g_i^j\right).$ \\
    For all $\hat{g}^{k+1}_{i,1}, \hat{g}^{k+1}_{i,2}, \dots$ such that
    \begin{equation}
    \begin{aligned}
    &\textnormal{supp} \left(\hat{g}^{k+1}_{i,1}\right) \subseteq \bigcup_{j = 0}^{k} \textnormal{supp} \left(\bar{s}_i^j\right), \\
    &\textnormal{supp} \left(\hat{g}^{k+1}_{i,2}\right) \subseteq \bigcup_{j = 0}^{k} \textnormal{supp} \left(\bar{s}_i^j\right) \bigcup \textnormal{supp}(\nabla f_i(\hat{g}^{k+1}_{i,1})), \\
    &\textnormal{supp} \left(\hat{g}^{k+1}_{i,3}\right) \subseteq \bigcup_{j = 0}^{k} \textnormal{supp} \left(\bar{s}_i^j\right) \bigcup \textnormal{supp} (\nabla f_i(\hat{g}^{k+1}_{i,1})) \bigcup \textnormal{supp} (\nabla f_i(\hat{g}^{k+1}_{i,2})), \\
    &\dots
    \label{eq:fCJhEZyUETrJBh}
    \end{aligned}
    \end{equation}
    we have
    \begin{align}
    \label{eq:lAXnPlQppJGITqJi}
    \textnormal{supp} \left(g^{k+1}_{i}\right) \subseteq \bigcup_{j = 1}^{\infty} \textnormal{supp} \left(\hat{g}^{k+1}_{i,j}\right),
    \end{align}
        \label{prop:four}
    for all $k \in \N_0$ and for all $i \in [n],$ where $\textnormal{supp}(x) \eqdef \{i \in [d]\,|\,x_i \neq 0\}.$
\end{enumerate}
We denote the set of all algorithms that satisfy these properties by $\cA_{\textnormal{zr}}.$
\end{definition}

The first three properties define the domains of the mapping. The last property is a standard assumption for a zero-respecting algorithm. Assumption \eqref{eq:lAXnPlQppJGITqJi} allows the mappings $L_{i}^{k}$ to calculate gradients.

\begin{theorem}
    \label{theorem:lower_bound_generic}
    Consider Protocol~\ref{alg:protocol}. Assume that the sets $\{\eta_{i,j}^0\}_{i \in [n],j \in [d']}$, $\{\eta_{i,j}^1\}_{i \in [n],j \in [d']}, \ldots, \{\eta_{i,j}^k\}_{i \in [n],j \in [d']}, \ldots$ are mutually independent (the variables within one set can be dependent). Let $p > 0, L, \delta^0, \varepsilon > 0, n \geq 2$ be any numbers such that $\bar{c} \varepsilon < L \delta^0.$ Then, for any algorithm $A \in \cA_{\textnormal{zr}},$ there exists a function $f \in \cF_{\delta^0, L}$ and functions $f_1, \dots, f_n$ such that $f = \frac{1}{n} \sum_{i=1}^n f_i$ and $\Exp{\norm{\nabla f(x^k)}^2} > \varepsilon$ for all $$k \leq \hat{c} \frac{L \delta^0}{p \varepsilon}.$$
    The quantities $\bar{c}$ and $\hat{c}$ are universal constants.
\end{theorem}

\begin{proof}
    The proof is conceptually the same as in \citet{arjevani2022lower,lu2021optimal,huang2022lower,SPIDER,carmon2020lower,tyurin2023optimal}. We fix $\lambda > 0,$ and consider the following function $f \,:\, \R^{T} \rightarrow \R:$ 
    \begin{align*}
        f(x) \eqdef \frac{L \lambda^2}{l_1} F_{T}\left(\frac{x}{\lambda}\right).
    \end{align*}
    One can show \citep{arjevani2022lower}[Theorem 1] that $f \in \cF_{\delta^0, L}$ if
    \begin{align}
        \label{eq:qBULyjtr}
        T = \left\lfloor \frac{\delta^0 l_1}{L \lambda^2 \Delta^0} \right\rfloor. 
    \end{align}
    Next, we define
    \begin{align*}
        F_i(x) \eqdef \begin{cases}
            -\Psi(1) \Phi([x]_1) + \sum_{2 \leq j \leq T \textnormal{ and } (j - 1) \bmod n = 0} \left(\Psi(-[x]_{j-1})\Phi(-[x]_j) - \Psi([x]_{j-1})\Phi([x]_j)\right), & i = 1\\
            \sum_{2 \leq j \leq T \textnormal{ and } (j - 1) \bmod n = i - 1}^T \left(\Psi(-[x]_{j-1})\Phi(-[x]_j) - \Psi([x]_{j-1})\Phi([x]_j)\right), & i > 1
        \end{cases}
    \end{align*}
    and 
    \begin{align*}
        f_i(x) \eqdef \frac{n L \lambda^2}{l_1} F_i\left(\frac{x}{\lambda}\right).
    \end{align*}
    The idea is that we take the first block from \eqref{eq:worst_case} to the first worker, the second block to the second worker, \dots, $(n + 1)$\textsuperscript{th} block to the first worker, and so on. Then, one can show that
    \begin{align*}
        \frac{1}{n} \sum_{i=1}^{n} f_i(x) = f(x).
    \end{align*}
    Using Lemma~\ref{lemma:worst_function}, we obtain
    \begin{align}
        \label{eq:tIfcURoTFCdxXLWo}
        \norm{\nabla f(x)}^2 = \frac{L^2 \lambda^2}{l_1^2} \norm{\nabla F_{T}\left(\frac{x}{\lambda}\right)}^2 > \frac{L^2 \lambda^2}{l_1^2} \mathbbm{1}[\textnormal{prog} (x) < T].
    \end{align}
    The functions $f_i$ are \emph{zero-chain} \citep{arjevani2022lower}: for all $i \in [n],$ if $\textnormal{prog} (x) = j$ and $(j \bmod n) + 1 = i,$ then $\textnormal{prog} (\nabla f_i(x)) \leq j + 1$, and for all $i \in [n],$ if $\textnormal{prog} (x) = j$ and $(j \bmod n) + 1 \neq i,$ then $\textnormal{prog} (\nabla f_i(x)) \leq j$. Using the zero-chain property and the fact that we consider the family of zero-respecting algorithms:
    \begin{enumerate}
        \item The first non-zero coordinate can be discovered only by the first worker.
        \item Assume that $\max_{j = 1}^k \max_{i = 1}^n \textnormal{prog} \left(g_i^j\right) = j \geq 1.$ An algorithm can discover \emph{one} new non-zero coordinate in the $(j + 1)$\textsuperscript{th} position only if the $(j \bmod n + 1)$\textsuperscript{th} worker gets a \emph{non-zero} $j$\textsuperscript{th} coordinate from \emph{the server}. This is by the construction of the functions $f_i.$ Note that for $n \geq 2,$ one worker cannot discover two \emph{consecutive} coordinates.
    \end{enumerate}
    Let us define
    \begin{align*}
        \xi^j = &\mathbb{I}[\textnormal{In the $j$\textsuperscript{th} iteration, the coordinate with index $\bar{p} \equiv \max_{j = 1}^k \max_{i = 1}^n \textnormal{prog} \left(g_i^j\right)$ is not zeroed out in \begin{NoHyper}Line~\ref{line:spar}\end{NoHyper}} \\
        &\quad\textnormal{of Protocol~\ref{alg:protocol} to the worker with index $(\bar{p} \bmod n + 1)$ AND $T - 1 \geq \bar{p} \geq 1$}] \quad (\bar{p} = 0 \textnormal{ if } k = 0).
    \end{align*}
    Then, we have
    \begin{align*}
        \Prob{\textnormal{prog} (x^k) \geq T} \leq \Prob{\sum_{j=0}^{k-1} \xi^j \geq T - 1}.
    \end{align*}
    Assume that $\mathcal{G}_j$ is the $\sigma$--algebra generated by all randomness up to the $j$\textsuperscript{th} iteration (inclusive). Then, $\xi^j$ is $\mathcal{G}_j$--measurable, and, by the construction of \begin{NoHyper}Line~\ref{line:spar}\end{NoHyper} of Protocol~\ref{alg:protocol}, $\ProbCond{\xi^{j+1} = 1}{\mathcal{G}_j} \leq p,$ where we also use the assumption of the theorem that the sets of random variables are mutually independent.
    Using the standard approach with the Chernoff method \citep{arjevani2022lower,lu2021optimal,huang2022lower}, one can show that
    \begin{align*}
        \Prob{\sum_{j=0}^{k-1} \xi^j \geq T - 1} \leq \rho
    \end{align*}
    for all $$k \leq \frac{T - 1 - \log \frac{1}{\rho}}{2 p}$$
    and $\rho \in (0, 1].$
    Therefore, we get
    \begin{align}
        \label{eq:prob_main}
        \Prob{\textnormal{prog} (x^k) \geq T} \leq \rho
    \end{align}
    for all $$k \leq \frac{T - 1 - \log \frac{1}{\rho}}{2 p}.$$
    Using \eqref{eq:tIfcURoTFCdxXLWo}, we have
    \begin{align*}
        \Exp{\norm{\nabla f(x^k)}^2} > 2 \varepsilon \Prob{\norm{\nabla f(x^k)}^2 > 2 \varepsilon} \geq 2 \varepsilon \Prob{\frac{L^2 \lambda^2}{l_1^2} \mathbbm{1}[\textnormal{prog} (x) < T] \geq 2 \varepsilon}.
    \end{align*}
    Let us take $\lambda = \frac{\sqrt{2 \varepsilon} l_1}{L}$. Then
    \begin{align}
        \label{eq:NxiQZCcbwCtYUBycL}
        \Exp{\norm{\nabla f(x^k)}^2} > 2 \varepsilon \Prob{\frac{L^2 \lambda^2}{l_1^2} \mathbbm{1}[\textnormal{prog} (x) < T] \geq 2 \varepsilon} = 2 \varepsilon \Prob{\textnormal{prog} (x) < T}.
    \end{align}
    From \eqref{eq:prob_main} with $\rho = \frac{1}{2}$, we get
    \begin{align}
        \Exp{\norm{\nabla f(x^k)}^2} > 2 \varepsilon \Prob{\textnormal{prog} (x) < T} \geq \varepsilon
    \end{align}
    for all $$k \leq \frac{T - 1 - \log 2}{2 p}.$$ 
    From \eqref{eq:qBULyjtr}, one can conclude that
    $$T = \left\lfloor \frac{L \delta^0 l_1}{2 \varepsilon l_1^2 \Delta^0} \right\rfloor.$$
    By the theorem's assumption, $L \delta^0 \geq \bar{c} \varepsilon.$ One can choose a universal constant $\bar{c}$ such that \eqref{eq:NxiQZCcbwCtYUBycL} holds for 
    $$k \leq \Theta\left(\frac{T}{p}\right) = \Theta\left(\frac{L \delta^0}{\varepsilon p}\right),$$ where $\Theta$ hides only a universal constant. 
\end{proof}

\subsection{Compressed communication with independent compressors}

Protocol~\ref{alg:protocol_compr} is exactly the same as Protocol~\ref{alg:protocol} except for \begin{NoHyper}Line~\ref{line:comp}\end{NoHyper} and describes the family of methods that send compressed vectors from the server to the workers.

\begin{protocol}[H]
    \caption{Protocol with Compressors}
    \label{alg:protocol_compr}
    \begin{algorithmic}[1]
    \STATE \textbf{Input: }functions $f_1, \dots, f_n \in \cF,$ algorithm $A,$ compressors $\cC_1, \dots, \cC_n$
    \FOR{$k = 0, \dots, \infty$}
    \STATE Server calculates a new point: $x^k = B_1^k(g_1^1, \dots, g_1^{k}, \dots, g_n^{1}, \dots, g_n^{k})$
    \STATE Server aggregates all available information: $s_i^k = B_{2,i}^k(g_1^1, \dots, g_1^{k}, \dots, g_n^{1}, \dots, g_n^{k})$
    \STATE Server sends compressed vectors $\bar{s}_i^k = \cC_i(s_i^k)$ to the workers \alglinelabel{line:comp}
    \STATE Workers aggregate all available local information and calculate gradients: $g_i^{k+1} = L_{i}^{k}(\bar{s}_i^0, \dots, \bar{s}_i^k)$ \\
    ($L_{i}^{k}$ has access to the gradient oracle of $f_i$ and can call it as many times as it wants according to the rules \eqref{eq:fCJhEZyUETrJBh} and \eqref{eq:lAXnPlQppJGITqJi})
    \STATE Workers send $g_i^{k+1}$ to the server
    \ENDFOR
    \end{algorithmic}
\end{protocol}

\begin{theorem}
    \label{theorem:lower_bound}
    Consider Protocol~\ref{alg:protocol_compr}. Let $\omega \geq 0, L, \delta^0, \varepsilon > 0, n \geq 2$ be any numbers such that $\bar{c} \varepsilon < L \delta^0.$ Then for any algorithm $A \in \cA_{\textnormal{zr}},$ there exists a function $f \in \cF_{\delta^0, L},$ functions $f_1, \dots, f_n$ such that $f = \frac{1}{n} \sum_{i=1}^n f_i,$ and i.i.d. compressors $\cC_1, \dots, \cC_n \in \mathbb{U}(\omega)$
    such that $\Exp{\norm{\nabla f(x^k)}^2} > \varepsilon$ for all $$k \leq \hat{c} \frac{(\omega + 1) L \delta^0}{\varepsilon}.$$
    The quantities $\bar{c}$ and $\hat{c}$ are universal constants.
\end{theorem}

\begin{proof}
    We can use the result of Theorem~\ref{theorem:lower_bound_generic}. It is sufficient to construct an appropriate compressor. Let us define $p \eqdef \nicefrac{1}{\omega + 1}.$ We define the following compressor:
    \begin{align*}
    [\cC(x)]_j \eqdef
    \begin{cases}
        \frac{1}{p} x_j,& j \in S, \\
        0, & j \not\in S,
    \end{cases} \quad \forall j \in [T],
    \end{align*}
    where $S$ is a random subset of $[T]$ and each element from $[T]$ appears with probability $p$ independently.
    Then, $\cC$ is unbiased:
    \begin{align*}
    \ExpSub{S}{[\cC(x)]_j} = x_j \quad \forall j \in \R^{T}
    \end{align*}
    and 
    \begin{align*}
    \ExpSub{S}{\norm{\cC(x)}^2} = \ExpSub{S}{\sum_{j=1}^{n T} \mathbbm{1}\left[j \in S\right] \frac{1}{p^2} x_j^2} = \sum_{j=1}^{n T} \Prob{j \in S} \frac{1}{p^2} x_j^2 = \sum_{j=1}^{n T} \frac{1}{p} x_j^2 = \left(\omega + 1\right) \norm{x}^2.
    \end{align*}
    Therefore, we get $\cC \in \mathbb{U}(\omega).$ 
    Let $\cC_i$ be i.i.d. instantiations of $\cC$ for all $i \in [n]$. Since $\cC$ is a sparsifier as in \begin{NoHyper}Line~\ref{line:spar}\end{NoHyper} of Protocol~\ref{alg:protocol}, we can use Theorem~\ref{theorem:lower_bound_generic} with $p = \nicefrac{1}{\omega + 1}$ to finish the proof.
\end{proof}

\newpage

\section{Useful Identities and Inequalities}

For all $x, y, x_1, \dots, x_m \in \R^d$, $s>0$ and $\alpha\in(0,1]$, we have:
\begin{align}
    \label{eq:young}
    \norm{x+y}^2 &\leq (1+s) \norm{x}^2 + (1+s^{-1}) \norm{y}^2, \\
    \label{eq:young_2}
    \norm{\sum_{i=1}^{m} x_i}^2 &\leq m \left(\sum_{i=1}^{m} \norm{x_i}^2\right), \\
    \label{eq:ineq1}
    \left( 1 - \alpha \right)\left( 1 + \frac{\alpha}{2} \right) &\leq 1 - \frac{\alpha}{2}, \\
    \label{eq:ineq2}
    \left( 1 - \alpha \right)\left( 1 + \frac{2}{\alpha} \right) &\leq \frac{2}{\alpha}.
\end{align}

\textbf{Variance decomposition:} For any random vector $X\in\R^d$ and any non-random vector $c\in\R^d$, we have
\begin{align}
\label{eq:vardecomp}
   \Exp{\norm{X-c}^2} = \Exp{\norm{X - \Exp{X}}^2} + \norm{\Exp{X}-c}^2.
\end{align}

\textbf{Tower property:} For any random variables $X$ and $Y$, we have
\begin{align}
\label{eq:tower}
    \Exp{\Exp{X\,|\,Y}} = \Exp{X}.
\end{align}

\textbf{Jensen's inequality:} If $f$ is a convex function and $X$ is a random variable, then
\begin{align}
\label{eq:jensen}
    \Exp{f(X)} \geq f\left(\Exp{X}\right).
\end{align}

\begin{lemma}[Lemma $2$ of \citet{PAGE}]\label{lemma:page}
    Suppose that function $f$ is $L$-smooth and let $x^{t+1} = x^t - \gamma g^t$. Then for any $g^t\in\R^d$ and $\gamma>0$, we have
    \begin{align*}
        f(x^{t+1}) \leq f(x^t) - \frac{\gamma}{2} \norm{\nabla f(x^t)}^2 - \left( \frac{1}{2\gamma} - \frac{L}{2} \right) \norm{x^{t+1} - x^t}^2 + \frac{\gamma}{2}\norm{g^t - \nabla f(x^t)}^2.
    \end{align*}
\end{lemma}

\begin{lemma}[Lemma $5$ of \citet{richtarik2021ef21}]\label{lemma:step_lemma}
   Let $a,b>0$. If $0\leq\gamma\leq\frac{1}{\sqrt{a}+b}$, then $a\gamma^2+b\gamma\leq 1$. 
   Moreover, the bound is tight up to the factor of $2$ since $\frac{1}{\sqrt{a}+b} \leq \min\left\{\frac{1}{\sqrt{a}}, \frac{1}{b}\right\} \leq \frac{2}{\sqrt{a}+b}$.
\end{lemma}

\newpage

\section{Notation}

\bgroup
\def\arraystretch{1.3}
\begin{table}[H]
\label{table:unbalanced}
	\small
	\centering
	\begin{tabular}{|c|p{10.5cm}|}
	\hline
	\multicolumn{2}{|c|}{Algorithms} \\
	\hline
            $n$ & number of workers/nodes/clients/devices \\
            $\gamma$ & stepsize \\
            $\cC^t_1, \ldots \cC^t_n$ & Server-to-workers (primal) compressors \\
            $\cQ^t_1, \ldots \cQ^t_n$ & Workers-to-server (dual) compressors \\
            $\omega_P$, $\omega_D$ & Parameters of server-to-workers (primal) and workers-to-server (dual) compressors \\
            $\theta$ & Correlated compressors parameter (Definition \ref{def:corr_compr}) \\
            $\beta$ & Momentum parameter (see Algorithm \ref{alg:m3}) \\
    \hline
	\multicolumn{2}{|c|}{Definitions} \\
    \hline
            $\mathbb{U}(\omega)$ & The family of unbiased compressors with parameter $\omega$ (Definition \ref{def:unbiased_compression}) \\
            $\mathbb{P}(\theta)$ & The family of correlated compressors with parameter $\theta$ (Definition \ref{def:corr_compr}) \\
            $L$ & Smoothness parameter of $f$ (Assumption \ref{ass:lipschitz_constant} )\\
            $L_i$ & Smoothness parameter of $f_i$ (Assumption \ref{ass:local_lipschitz_constant}) \\
            $L_A$, $L_B$ & Parameters from Assumption \ref{ass:functional} \\
    \hline
	\multicolumn{2}{|c|}{Notation} \\
    \hline
            \multicolumn{2}{|l|}{$[k]=\{1,\ldots,k\}$ for any positive integer $k$} \\
            \multicolumn{2}{|l|}{$\ExpSub{t}{\cdot}$ - expectation conditioned on the first $t$ iterations} \\
            \multicolumn{2}{|l|}{$\delta^{t} \eqdef f(x^t) - f^*$} \\
            \multicolumn{2}{|l|}{$\widehat{L}^2 \eqdef \frac{1}{n} \sum_{i=1}^n L_i^2$, $L_{\max} \eqdef \max_{i \in [n]} L_i$} \\
            \multicolumn{2}{|l|}{$w^{t} \eqdef \nicefrac{1}{n} \sum_{i=1}^n w_i^t$} \\
            \multicolumn{2}{|l|}{$g^{t} \eqdef \nicefrac{1}{n} \sum_{i=1}^n g_i^t$} \\
            \multicolumn{2}{|l|}{$z^{t} \eqdef \nicefrac{1}{n} \sum_{i=1}^n z_i^t$} \\
        \hline
	\end{tabular}
	\caption{Frequently used notation.}
	\label{table:notation}
\end{table}
\egroup

\newpage
\section*{NeurIPS Paper Checklist}

\begin{enumerate}

\item {\bf Claims}
    \item[] Question: Do the main claims made in the abstract and introduction accurately reflect the paper's contributions and scope?
    \item[] Answer: \answerYes{} 
    \item[] Justification: Sections~\ref{sec:lower_bound_main}, \ref{sec:marina_p}, and \ref{sec:m3}

\item {\bf Limitations}
    \item[] Question: Does the paper discuss the limitations of the work performed by the authors?
    \item[] Answer: \answerYes{} 
    \item[] Justification: Sections~\ref{sec:general_case} and \ref{sec:convergence_m}

\item {\bf Theory Assumptions and Proofs}
    \item[] Question: For each theoretical result, does the paper provide the full set of assumptions and a complete (and correct) proof?
    \item[] Answer: \answerYes{} 
    \item[] Justification: The assumptions and the proofs are in Section~\ref{sec:related_work} in the appendix.

    \item {\bf Experimental Result Reproducibility}
    \item[] Question: Does the paper fully disclose all the information needed to reproduce the main experimental results of the paper to the extent that it affects the main claims and/or conclusions of the paper (regardless of whether the code and data are provided or not)?
    \item[] Answer: \answerYes{} 
    \item[] Justification: Section~\ref{sec:experimetns_main}

\item {\bf Open access to data and code}
    \item[] Question: Does the paper provide open access to the data and code, with sufficient instructions to faithfully reproduce the main experimental results, as described in supplemental material?
    \item[] Answer: \answerYes{} 
    \item[] Justification: In the supplementary materials.

\item {\bf Experimental Setting/Details}
    \item[] Question: Does the paper specify all the training and test details (e.g., data splits, hyperparameters, how they were chosen, type of optimizer, etc.) necessary to understand the results?
    \item[] Answer: \answerYes{} 
    \item[] Justification: Section~\ref{sec:experimetns_main}

\item {\bf Experiment Statistical Significance}
    \item[] Question: Does the paper report error bars suitably and correctly defined or other appropriate information about the statistical significance of the experiments?
    \item[] Answer: \answerYes{} 
    \item[] Justification: We run experiments with different seeds and plot averages to reduce noise factors (see the description in Sections~\ref{sec:core_m3} and \ref{sec:autoencode}).

\item {\bf Experiments Compute Resources}
    \item[] Question: For each experiment, does the paper provide sufficient information on the computer resources (type of compute workers, memory, time of execution) needed to reproduce the experiments?
    \item[] Answer: \answerYes{} 
    \item[] Justification: Section~\ref{sec:experimetns_main}
    
\item {\bf Code Of Ethics}
    \item[] Question: Does the research conducted in the paper conform, in every respect, with the NeurIPS Code of Ethics \url{https://neurips.cc/public/EthicsGuidelines}?
    \item[] Answer: \answerYes{} 
    \item[] Justification: We have reviewed the code of ethics, and we are confident that our paper is in compliance with it.

\item {\bf Broader Impacts}
    \item[] Question: Does the paper discuss both potential positive societal impacts and negative societal impacts of the work performed?
    \item[] Answer: \answerNA{} 
    \item[] Justification: Our work considers a mathematical problem from machine learning.
    
\item {\bf Safeguards}
    \item[] Question: Does the paper describe safeguards that have been put in place for responsible release of data or models that have a high risk for misuse (e.g., pretrained language models, image generators, or scraped datasets)?
    \item[] Answer: \answerNA{} 
    \item[] Justification:

\item {\bf Licenses for existing assets}
    \item[] Question: Are the creators or original owners of assets (e.g., code, data, models), used in the paper, properly credited and are the license and terms of use explicitly mentioned and properly respected?
    \item[] Answer: \answerYes{} 
    \item[] Justification: Section~\ref{sec:experimetns_main}

\item {\bf New Assets}
    \item[] Question: Are new assets introduced in the paper well documented and is the documentation provided alongside the assets?
    \item[] Answer: \answerYes{} 
    \item[] Justification: In the supplementary materials.

\item {\bf Crowdsourcing and Research with Human Subjects}
    \item[] Question: For crowdsourcing experiments and research with human subjects, does the paper include the full text of instructions given to participants and screenshots, if applicable, as well as details about compensation (if any)? 
    \item[] Answer: \answerNA{} 
    \item[] Justification: 

\item {\bf Institutional Review Board (IRB) Approvals or Equivalent for Research with Human Subjects}
    \item[] Question: Does the paper describe potential risks incurred by study participants, whether such risks were disclosed to the subjects, and whether Institutional Review Board (IRB) approvals (or an equivalent approval/review based on the requirements of your country or institution) were obtained?
    \item[] Answer: \answerNA{} 
    \item[] Justification: 

\end{enumerate}

\end{document}